\documentclass[12pt, a4paper, abstracton]{scrartcl}

\usepackage{array}
\usepackage{marginnote}
\usepackage{xcolor}
\usepackage{amscd,amssymb,amsfonts,amsmath,latexsym,amsthm}
\usepackage{hyperref}
\usepackage[all,cmtip]{xy}
\usepackage{mathrsfs}
\usepackage{graphicx}
\usepackage{bm} 
\usepackage[nameinlink, capitalise]{cleveref}
\usepackage{bbm}

\usepackage{etoolbox}

\def\hB{\hspace*{\fill}$\qed$}

\usepackage[nottoc]{tocbibind} 

\usepackage{defs_pp1}
\usepackage{slashed}
\usepackage[utf8]{inputenc}
\usepackage{microtype}
\usepackage[english]{babel}
\usepackage{mathtools}
\usepackage{bm}
\usepackage{esvect}
\title{$KK$- and $E$-theory via homotopy theory}
\author{
Ulrich Bunke\thanks{Fakult{\"a}t f{\"u}r Mathematik,
Universit{\"a}t Regensburg,
93040 Regensburg,
ulrich.bunke@mathematik.uni-regensburg.de}, 
}

\numberwithin{equation}{section}
\newtheorem{theorem}{Theorem}[section] 
\newtheorem{prop}[theorem]{Proposition}
\newtheorem{lem}[theorem]{Lemma}

\newtheorem{ddd}[theorem]{Definition}
\newtheorem{kor}[theorem]{Corollary}

\theoremstyle{remark}
\theoremstyle{definition}

\newtheorem{ex}[theorem]{Example}
\newtheorem{rem}[theorem]{Remark}

\newtheorem{construction}[theorem]{Construction}

\newcommand{\coCMon}{\mathbf{coCMon}}
\newcommand{\coCGroups}{\mathbf{coCGroup}}
\newcommand{\UCT}{\mathrm{UCT}}
\newcommand{\se}{\mathrm{se}}
\newcommand{\splt}{\mathrm{splt}}

\newcommand{\rProj}{\mathcal{P}\tilde{\mathrm{roj}}}
\newcommand{\bProj}{\mathcal{P}\bar{\mathrm{roj}}}
\newcommand{\SPC}{\mathbf{SPC}}
\newcommand{\class}{\mathrm{class}}
\newcommand{\KU}{\mathrm{KU}}
\newcommand{\sa}{\mathrm{sa}}

\newcommand{\Nat}{\mathrm{Nat}}
\newcommand{\cProj}{\mathcal{P}\mathrm{roj}}

\newcommand{\sepa}{\mathrm{sep}}

\newcommand{\codiag}{\mathrm{codiag}}
\newcommand{\lex}{\mathrm{lex}}
\newcommand{\lax}{\mathrm{lax}}

\newcommand{\alg}{\mathrm{alg}}
\newcommand{\Bd}{\mathrm{Bd}}

\newcommand{\EE}{\mathrm{E}}
\newcommand{\ee}{\mathrm{e}}

\newcommand{\group}{\mathrm{group}}

\newcommand{\nCalg}{C^{*}\mathbf{Alg}^{\mathrm{nu}}}
\newcommand{\can}{\mathrm{can}}
\newcommand{\F}{\mathbb{F}}

\newcommand{\ho}{\mathrm{ho}}

\newcommand{\cR}{\mathcal{R}}

\newcommand{\Ob}{\mathrm{Ob}}

\newcommand{\CMon}{\mathbf{CMon}}

\newcommand{\Fib}{{\mathrm{Fib}}}

\newcommand{\incl}{\mathrm{incl}}

\newcommand{\cK}{\mathcal{K}}

\newcommand{\CAlg}{\mathbf{CAlg}}

\newcommand{\uli}[1]{\textcolor{red}{#1}}

\newcommand{\exa}{\mathrm{ex}}

\renewcommand{\Proj}{\mathrm{Proj}}

\newcommand{\const}{{\mathtt{const}}}

\newcommand{\cU}{{\mathcal{U}}}

\renewcommand{\Dirac}{\slashed{D}}

\newcommand{\Spc}{\mathbf{Spc}}

\newcommand{\Calg}{C^{\ast}\mathbf{Alg}}

\newcommand{\op}{\mathrm{op}}

\newcommand{\add}{\mathrm{add}}

\newcommand{\kk}{\mathrm{kk}}
\newcommand{\KK}{\mathrm{KK}}

\renewcommand{\1}{\mathbbm{1}}

\newcommand{\nClinAlg}{{}^{*}\mathbf{Alg}^{\mathrm{nu}}_{\C}}

\newcommand{\npreCalg}{{C}^{*}_{\mathrm{pre}}\mathbf{Alg}^{\mathrm{nu}}}

\newcommand{\compl}{\mathrm{compl}}

\newcommand{\Groups}{\mathbf{Groups}}
\newcommand{\CGroups}{\mathbf{CGroups}}
 
\begin{document}

		 \setcounter{tocdepth}{5}
\maketitle

\begin{abstract}
	We provide a homotopy theorist's point of view on $KK$ -and $E$-theory for $C^{*}$-algebras. 
	We construct stable $\infty$-categories representing these theories through a sequence of Dwyer-Kan localizations of the category of $C^{*}$-algebras. Thereby we will reveal the homotopic theoretic meaning of various
	classical constructions from $C^{*}$-algebra theory, in particular of Cuntz' $q$-construction. 
We will also discuss operator algebra $K$-theory in this framework.	
	 	\end{abstract}

\tableofcontents

\section{Introduction}\label{rgijrweiogwrwrefwf}

In this note we describe a construction of stable $\infty$-categories representing $KK$- and $E$-theory  for $C^{*}$-algebras through   sequences of localizations of the category $\nCalg_{\sepa}$ of   separable $C^{*}$-algebras followed by a left Kan-extension along the inclusion of separable into all $C^{*}$-algebras.
 In contrast to the previous constructions of such an $\infty$-category \cite{LN}, \cite{KKG} in the case  of $KK$-theory,   the  description presented here is independent of the classical  group-valued $KK$-theory  introduced \cite{kasparovinvent}, \cite{MR899916} which is   described e.g.  in the textbooks \cite{blackadar}, \cite{MR1077390}.   A stable $\infty$-category representing $E$-theory 
 has not been considered so far.
 
 The main goal of this note is to give a complete account of the basic categorical and functorial properties of  $KK$- and $E$-theory using only the basic elements of $C^{*}$-algebra theory. In this way we   hope   to make these theories more accessible to readers with a homotopy theory background.  
The  approach to $KK$- and $E$-theory  described here  can easily be generalized to the case of $G$-$C^{*}$-algebras for discrete groups $G$ (see e.g.  \cite{Bunke:2024aa} for $E$-theory) or to $C_{0}(X)$-algebras. With more modifications it should be possible to develop a similar approach to the algebraic version of $KK$-theory \cite{Corti_as_2007}, \cite{zbMATH06421353}, \cite{Garkusha_2016}, \cite{Ellis_2014}. It is also an interesting task to provide a homotopy theoretic interpretation of the constructions from \cite{zbMATH01302969} in the spirit of the present paper.

The starting point of our construction is the characterization of the stable $\infty$-category version of $KK$-theory   through a universal property.
\begin{ddd}\label{weijgoewrfrefwre9}
The functor $\kk:\nCalg\to \KK$ is  initial for functors from $\nCalg$ to cocomplete stable $\infty$-categories
which are homotopy invariant, stable, semiexact and $s$-finitary. 
\end{ddd}
This means that $\kk$ has these properties, described in detail in  \cref{wrogjpwrgrfrewf}, and that for any cocomplete stable $\infty$-category
$\bD$ the restriction along $\kk$ induces an equivalence
$$\kk^{*}:\Fun^{\colim}(\KK,\bD)\stackrel{\simeq}{\to} \Fun^{h,s,se,sfin}(\nCalg,\bD)\ .$$
Here the superscripts $colim$ and $h,s,se,sfin$ stand for colimit-preserving and the corresponding properties listed in \cref{weijgoewrfrefwre9}. 

The characterization of $KK$-theory by  \cref{weijgoewrfrefwre9}  was given in \cite{KKG} following \cite{LN}. A  similar characterization of the group-valued $KK$-functor through universal properties  has been known for a long time \cite{higsa}.

The  characterization of  the stable $\infty$-category representing $E$-theory is similar and obtained by replacing in  \cref{weijgoewrfrefwre9}  the condition of semiexactness by exactness. The motivation comes from the universal property of the classical $E$-theory stated in \cite[Thm. 3.6]{MR1068250}.
\begin{ddd}\label{weijgoewrfrefwre91}
The functor $\ee:\nCalg\to \EE$ is  initial for functors from $\nCalg$ to cocomplete stable $\infty$-categories
which are homotopy invariant, stable, exact and $s$-finitary. 
\end{ddd}
In this case we have an equivalence 
$$\ee^{*}:\Fun^{\colim}(\EE,\bD)\stackrel{\simeq}{\to} \Fun^{h,s,ex,sfin}(\nCalg,\bD)\ .$$
Our construction proceeds with the following steps which are designed to force the universal properties stated above:
\begin{enumerate}
\item \label{werjigowegwerfrefwrf}  $L_{h}:\nCalg\to \nCalg_{h}$ is a Dwyer-Kan localization which inverts the homotopy equivalences.
The resulting $\infty$-category $ \nCalg_{h}$  is left-exact (see \cref{lh}) and the functor $L_{h}$  is 
Schochet exact  in the sense that it sends Schochet fibrant cartesian squares to cartesian squares.
\item  \label{werjigowegwerfrefwrf1}  $L_{K}:\nCalg_{h}\to L_{K}\nCalg_{h}$ is a smashing left Bousfield  localization which inverts the left-upper corner inclusions and produces a semi-additive and left-exact $\infty$-category (see \cref{lhK}).
\item\label{step3} We restrict to the full subcategory of separable algebras and form a left-exact Dwyer-Kan localization
$L_{\sepa,!}:L_{K}\nCalg_{\sepa,h}\to L_{K}\nCalg_{\sepa,h,!} $ for $!$ in  $\{\splt,\se,\exa\}$ which forces split exact, semi-split exact or exact sequences to induce fibre sequence (see \cref{wegijowerferferwfrfwdvdfvdf}).
\item For $!$ in $\{\se ,\exa\} $ the two-fold loop functor $\Omega^{2}_{\sepa,!}:L_{K}\nCalg_{\sepa,h,!} \to L_{K}{\nCalg_{\sepa,h,!}}^{\group}$ (where $\bC^{\group}$ denotes the full subcategory of group objects in a semi-additive $\infty$-category $\bC$) turns out to be the right-adjoint of a right Bousfield localization and has a stable target.
 The composition of the localizations above give functors 
 \begin{align*}\kk_{\sepa}:\nCalg_{\sepa}\to \KK_{\sepa}&:=L_{K}{\nCalg_{\sepa,h,\se}}^{\group}  \ ,\\ \ee_{\sepa}:\nCalg_{\sepa}\to \EE_{\sepa}&:=L_{K}{\nCalg_{\sepa,h,\exa}}^{\group}\ .\end{align*} (see \cref{werigjowergerwg9}). 
\item \label{wegkjowerfwerf} We define the presentable stable $\infty$-categories $\KK $ and $\EE $ as the $\Ind$-completions of  the stable $\infty$-categories $\KK_{\sepa}$ and $\EE_{\sepa}$
and the functors \[\kk:\nCalg\to \KK\ , \quad \ee:\nCalg\to \EE\] 
by   left Kan-extending  the compositions \[\nCalg_{\sepa} \xrightarrow{\kk_{\sepa}}\KK_{\sepa}\xrightarrow{y}\KK \ , \quad \mbox{or}\quad \nCalg_{\sepa} \xrightarrow{\ee_{\sepa}}\EE_{\sepa}\xrightarrow{y}\EE\]  along the inclusion of separable $C^{*}$-algebras into all $C^{*}$-algebras (see \cref{fbweroibjdfbsfdbsfb}). \cref{wergijweogrefwrferwfw} states that the functors constructed by this procedure indeed satisfy the conditions of \cref{weijgoewrfrefwre9} and \cref{weijgoewrfrefwre91}.
\end{enumerate}

One interesting consequence of the constructions is that the functors $$\kk_{\sepa}:\nCalg_{\sepa}\to \KK_{\sepa}\ , \:\:\:\quad \mbox{and} \:\:\: \quad  \ee_{\sepa}:\nCalg_{\sepa}\to \EE_{\sepa}$$ are Dwyer-Kan localizations (see \cref{ergiojeroigwerfwrefrefdvs}).
 
Our construction of $KK$- and $E$-theory for separable $C^{*}$-algebras via a sequence of localizations is analogous to the   construction of an additive category representing $E$-theory in \cite{MR1068250}.  The idea of left-Kan  extending  $KK$-theory from separable $C^{*}$-algebras to all $C^{*}$-algebras also appears in \cite{zbMATH04065714}.

The category $\nCalg$ has symmetric monoidal structures $\otimes_{\max}$ and $\otimes_{\min}$.
The $\kk$- and $\ee$-theory functors have symmetric monoidal refinements which are characterized by
symmetric monoidal versions of \cref{weijgoewrfrefwre9} and  \cref{weijgoewrfrefwre91}. We will discuss the universal properties of the
symmetric monoidal  refinements   in the main body of the present  paper.

The categories $\KK$ and $\EE$ whose construction is sketched above  are  stable $\infty$-categories.  For any two
 $C^{*}$-algebras $A$ and $B$ we therefore have  mapping spectra
 \[\KK(A,B):=\map_{\KK}(\kk(A),\kk(B))\ , \quad \EE(A,B):=\map_{\EE}(\ee(A),\ee(B))\ .\]  Taking homotopy groups we get    $\Z$-graded $KK$- and $E$-theory groups
  \[\KK_{*}(A,B):=\pi_{*}\KK(A,B)\ , \quad \EE_{*}(A,B):=\pi_{*}\EE(A.B)\ .\]
 The approach to $KK$- and $E$-theory taken in the present note   turns the classical constructions 
  of these group-valued bifunctors into calculations. Our homotopy theoretic construction of
$KK$- and $E$- theory is straightforward once one knows which universal property one would like to enforce. 
Composition, homotopy invariance, stability and the respective exactness properties come for free. Also Bott periodicity is just a property which holds because of the existence of the Toeplitz extension.
The real problem in our approach is to calculate the homotopy groups of the mapping spaces  in order to see that they
coincide with the classical groups. The latter are defined in terms of Kasparov modules (see \cite{MR899916},  \cite{zbMATH01302969}  \cite{thomsen1}, \cite{Dadarlat_2018} for alternatives)  in the case of $KK$-theory, or in case of $E$-theory, 
by the one-categorical localization procedure as in  \cite{MR1068250} or asymptotic morphisms \cite{zbMATH04182148}.
The comparison of the homotopy groups of the mapping spectra of the categories constructed in the present note with the classical groups  is  not obvious at all  just from the construction. But it  is   crucial if one wants  
 to use  the models proposed in this note as a homotopy theoretic replacement of the classical analytic constructions. 

In the case of $KK$-theory one could argue by a  comparison of 
 the universal properties that  the  functors $\kk_{\sepa}$ for separable algebras
 constructed in the present paper and in  \cite{LN} are canonically equivalent.  Moreover,  in 
  \cite{KKG} we have shown   that 
 the composition
\[\nCalg_{\sepa}\xrightarrow{\kk_{\sepa}} \KK_{\sepa}\xrightarrow{\ho}\ho\KK_{\sepa}\] is equivalent to
the triangulated category valued $KK$-theory of \cite{MR2193334}, and that the $KK$-groups
$\KK_{0}(A,B)$ for separable $C^{*}$-algebras $A,B$ are canonically isomorphic to the $KK$-groups introduced in  \cite{kasparovinvent}.  But this  argument has a   draw back.
 Though the classical definition of 
 $KK$-groups in terms of equivalence classes of Kasparov modules is not very complicated,
 this method of comparison also relies on
  the construction of the composition (i.e., the Kasparov product) and the verification of semiexactness 
  in the classical theory which 
  are deep theorems.   
   It is therefore one of the guiding challenges  of the present paper to give an independent complete proof for the comparison.

 From the perspective of the present notes it is natural to  compare the $KK$- and $E$-theory functors
of the present paper with the classical ones by comparing their universal properties.  This can be done in a model independent way by defining
the classical functors
$$\kk_{\sepa}^{\class}:\nCalg_{\sepa}\to \KK_{\sepa}^{\class}\ , \quad \ee_{\sepa}^{\class}:\nCalg_{\sepa}\to \EE_{\sepa}^{\class}$$
as the universal homotopy invariant, stable and split-exact or half-exact functors, respectively,  to an additive category in the sense of   \cite[Thm. 3.4]{MR1068250} or  \cite[Thm. 3.6]{MR1068250}.
These can directly be compared with the compositions 
\[\ho\kk_{\sepa}:\nCalg_{\sepa}\xrightarrow{\kk_{\sepa}} \KK_{\sepa}\xrightarrow{\ho}\ho\KK_{\sepa}\ ,\quad \ho\ee_{\sepa}:\nCalg_{\sepa}\xrightarrow{\ee_{\sepa}} \EE_{\sepa}\xrightarrow{\ho}\ho\EE_{\sepa}\ .\] 
  The following is a consequence of   \cref{qiurhfgiuewrgwrfrefrfwrefw} 
 and the  automatic semiexactness theorem \cref{wgokjweprgrefwrefwrfw} (which allows to replace $\kk_{\sepa,q}$ appearing  in  \cref{qiurhfgiuewrgwrfrefrfwrefw} )  by $\kk_{\sepa}$).
  
  \begin{kor}\label{wrtgowtgwgrrwfrefw}
  We have commutative squares
  \[\xymatrix{\nCalg_{\sepa}\ar[r]^{\kk_{\sepa}^{\class}}\ar[d]^{\kk_{\sepa}}&  \KK_{\sepa}^{\class}\\ \KK_{\sepa}\ar[r]^{\ho}\  &\ho\KK_{\sepa}\ar@{-->}[u]^{\simeq}}\ , \xymatrix{\nCalg_{\sepa}\ar[r]^{\ee_{\sepa}^{\class}}\ar[d]^{\ee_{\sepa}}&  \EE_{\sepa}^{\class}\\ \EE_{\sepa}\ar[r]^{\ho}&\ho\EE_{\sepa}\ar@{-->}[u]^{\simeq}}\ .\]
where the dashed arrows are equivalences of additive categories. 
  \end{kor}

    One could argue that in the case of $KK$-theory  the proof of \cref{wrtgowtgwgrrwfrefw} has 
    a similar problem as the argument mentioned above since
    we must know that our preferred model 
    of $\kk^{\class}_{\sepa}$ has the universal property  stated in \cite[Thm. 3.4]{MR1068250}. 
    We therefore will provide another,  completely independent  comparison with the Cuntz picture of $KK$ by showing the formula \eqref{fewweqfhwiffqhideewdq} below.
  One could then read the arguments also in a different direction as showing that the Cuntz  model 
  indeed has the universal property  \cite[Thm. 3.4]{MR1068250}.

In our approach the enrichment    of $KK$- and $E$-theory theory   in spectra is a natural consequence  of  the stability of the $\infty$-categories $\KK$ or $\EE$.  But point-set level constructions  of spectral enrichments of $KK$-theory have previously been considered in
   \cite{zbMATH05657126},  \cite{mitmit}.

%
%
%
%
%
%

As can be seen from the description above our approach to a stable $\infty$-category representing $KK$- and $E$- theory is different from other attempts to produce such stable $\infty$-categories which were guided by the methods of motivic homotopy theory  \cite{pao}, \cite{Mahanta_2014}.
There  the idea was to start from the category of presheaves $\Fun((\nCalg_{\sepa})^{\op} ,\Spc)$, to
perform a series of left Bousfield  localizations forcing homotopy invariance,
stability and the desired version of exactness $!$ in $\{\splt, \se,\exa\}$, and finally to apply $-\otimes \Sp$ in presentable $\infty$-categories in order to stablize.
Let us denote the resulting presentable stable $\infty$-category by $\cK\cK_{\sepa,!}
$.  It comes with a functor $\mathit{kk}_{\sepa,!}:\nCalg_{\sepa}\to \cK\cK_{\sepa,!}
$ which by construction  has the universal property that
\[\mathit{kk}_{\sepa,!}^{*}:\Fun^{\colim}( \cK\cK_{\sepa,!},\bD)\stackrel{\simeq}{\to} \Fun^{h,s,!}(\nCalg_{\sepa}
,\bD)\] for any presentable stable $\infty$-category $\bD$. The main non-trivial question is then to understand the relation between $\pi_{*}\map_{\cK\cK_{\sepa,!}}(\mathit{kk}_{\sepa,!}(A),\mathit{kk}_{\sepa,!}(B))$ and the classical $KK$-groups $\KK^{\class}_{\sepa,*}(A,B)$ (for $!=\se$) or $E$-theory groups $\EE^{\class}_{\sepa,*}(A,B)$ (for $!=\exa$). We will not pursue  this direction.

As said above the advantage of the constructions in the present note is that they do not require previous knowledge of $KK$- or $E$-theory. In contrast, in   \cite{LN}, \cite{KKG} the basic idea was to construct the category $\KK_{\sepa}$ as a Dwyer-Kan localization of $\nCalg_{\sepa}$ at the $\kk$-equivalences. The latter notion was imported from the classical theory.
In the present paper we do not have to know from the beginning what a $\kk$-equivalence is. The notion of a $\kk$-equivalence comes out at the end   as a morphism which is sent to an equivalence by the functor $\kk_{\sepa}$. The input for  the construction of $\KK_{\sepa}$ in the present paper are only  simple $C^{*}$-algebraic notions as homotopy of homomorphisms, compact operators and semi-split exact sequences.

The construction of the $\infty$-categories  $\KK $ and $\EE $ via localizations and $\Ind$-completions is very  suitable for understanding functors  out of these categories. This will be employed in some subsequent papers. On the other hand, it is notoriously difficult to understand the  homotopy types of the mapping spaces in a Dwyer-Kan localization  just from the definition.  In \cref{wreijgowertgwrgrwfrefwrf} and \cref{reoijqoeirfewfqwdewd} we will,
with some effort, calculate the mapping spectra $\EE(A,B)$ for $A\cong \C$ or $A\cong S(\C)\cong C_{0}(\R)$ explicitly.  

We first define the commutative ring spectrum $\KU:=\EE(\C,\C )$. We  justify this notation by providing a ring  isomorphism 
$\pi_{*}\KU\cong \Z[b,b^{-1}]$ with $\deg(b)=-2$ and comparing $\Omega^{\infty}\KU$ with the classical constructions of an infinite loop space  with the same name.
Since $\ee(\C)$ is the tensor unit of $\EE$  this category has a canonical   enrichment over the category $\Mod(\KU)$ of $\KU$-module spectra.
 In \cref{wtgijwoergferwfrefw} we then define the  lax symmetric monoidal $\Mod(\KU)$-valued  $K$-theory functor for $C^{*}$-algebras simply  as \begin{equation}\label{referwfweferf}K(-):=\EE(\C,-):\nCalg\to \Mod(\KU)\ .
\end{equation}
This gives an effordless construction of a highly structured version of a $K$-theory functor for $C^{*}$-algebras.
 For previous constructions of spectrum-valued $K$-theory functors see e.g.   \cite{zbMATH02065570}, \cite{DellAmbrogio:2011aa}, \cite{joachimcat}, 
\cite{Dadarlat:2015wh}.

Recall that the classical constructions of $C^{*}$-algebra $K$-theory groups as described e.g. in \cite{blackadar} employ
equivalence classes of projections or components of unitary groups. 
In order to connect our definition \eqref{referwfweferf} with the classical ones and in order to
show that it gives the correct group valued functors after taking homotopy groups we
relate the infinite loop space  valued functor $\Omega^{\infty}K$ with spaces of projections or unitaries. Thereby we take care of the natural commutative monoid or groups structures.

 Using that $L_{K}\nCalg_{h}$ is semi-additive  we can define the commutative monoid (see  \cref{poregjopewrgerg}) \[\cProj^{s}(B ):=\Map_{L_{K}\nCalg_{h}}(\C,B)\] of stable projections  and the commutative group (see \cref{poregjopewrgerg1}.) \[  \cU^{s}(B):= \Map_{L_{K}\nCalg_{h}}(S(\C),B)\] of stable unitaries in $B $.   The following result combines \cref{trgjkowergferfweferfw} and \cref{erijgoerwgregffwfrefw}. 
\begin{kor}\label{erogjkperfrewf}\mbox{}\begin{enumerate} \item \label{wtgijorgferwfrefr3feerfwr}If $B$ is unital, then 
there is a  canonical morphism $ \cProj^{s}(B)\to \Omega^{\infty}K(B)$ in $\CMon(\Spc)$
which presents its target as the group completion.
\item\label{qrojifgpqrfqedeqwdq} We have a canonical equivalence $ \cU^{s}(B)\simeq \Omega^{\infty-1}K(B)$ in $\CGroups(\Spc)$.
\end{enumerate}
 \end{kor}
The canonical morphisms in \cref{erogjkperfrewf} are induced by the Steps \ref{step3} to \ref{wegkjowerfwerf} of the above sequence of localizations. 
The standard modification of \cref{erogjkperfrewf}.\ref{wtgijorgferwfrefr3feerfwr} for non-unital $C^{*}$-algebras will be  stated as \cref{wogkpwgerfwrefwf}. 

If one goes over to connected components in Assertion \ref{wtgijorgferwfrefr3feerfwr} or homotopy groups in
Assertion \ref{qrojifgpqrfqedeqwdq}, and if one interprets  $\pi_{*}K(B)$ as the classical version of $K$-theory of $C^{*}$-algebras, then the assertions of \cref{erogjkperfrewf} are well-known.
The main point of \cref{erogjkperfrewf} is that   $K(B)$ is not given by the classical definitions but is defined through mapping spectra of
  the category $\EE$ which is constructed by  a formal homotopy theoretic procedure of 
Dwyer-Kan localizations. It is only by \cref{erogjkperfrewf} that we know that these mapping spectra have the correct homotopy types to represent the classical $K$-theory of $C^{*}$-algebras. 

An advantage of the definition  of  the  $K$-theory functor for $C^{*}$-algebras by \eqref{referwfweferf} 
is that it
is  homotopy invariant, stable, exact, and  also $s$-finitary by construction. In addition, 
  in \cref{erigoeggerfwerf9} we show, using the equivalence from     \cref{erogjkperfrewf}.\ref{qrojifgpqrfqedeqwdq},
 that it also preserves filtered colimits. Of course, all these properties are well-known  propositions about the classical definition.

Following \cite{schros} we define the $\UCT$-class in $\KK$ as the localizing subcategory generated by the tensor unit $\kk(\C)$.
Using \cref{erogjkperfrewf}.\ref{qrojifgpqrfqedeqwdq} we will see in \cref{wtrkgoerfrefwerfwerf} that  the natural map $\KK(B,-)\to \EE(B,-)$ is an equivalence if $B$ belongs to the $\UCT$-class. Essentially by definition,  the $K$-theory functor
induces a symmetric monoidal equivalence between the $\UCT$-class and   the stable $\infty$-category  $\Mod(\KU)$.  This leads   to a simple picture of the universal coefficient theorem and the Künneth formula also formulated in \cref{wtrkgoerfrefwerfwerf}.


The Cuntz picture \cite{MR899916} of $KK$-theory  is based on the $q$-construction which involves a functor  and a natural transformation
\[q:\nCalg\to \nCalg\ , \quad \iota:q\to \id_{\nCalg}\ .\]  The goal of  \cref{efgijoerfreferfwrf} is to study 
the homotopical features of the $q$-construction. This whole section is essentially a translation of  \cite{MR899916}  from abelian  group valued functors to functors having values in semi-additive or additive $\infty$-categories.   The main insight derived in  this section is that inverting the image of the set  
$\{\iota_{A}:qA\to A\mid A\in \nCalg\}$ in $L_{K}\nCalg_{h}$ (see Step \ref{werjigowegwerfrefwrf1} above)  yields the universal homotopy invariant, stable  split- and Schochet exact functor \[L_{h,K,q}:\nCalg\to L_{K}\nCalg_{h,q}\]  with values in a left-exact additive $\infty$-category. 
By \cref{eigjohwergerfgerwfrwef} it is a Dwyer-Kan localization.
 Using  the deep result    \cite[Thm 1.6]{MR899916} (reproduced in these notes as \cref{erigjwoergerfwerfwref}) we will see in the separable case that  the Dwyer-Kan localization $L_{\sepa,q}:L_{K}\nCalg_{\sepa,h}\to L_{K}\nCalg_{\sepa,h,q}$  is actually a right  Bousfield localization, and    we   obtain the  very simple formula \begin{equation}\label{fweqwefljweofqwedeqdqedqewd}\ell\underline{\Hom}(qA,K\otimes B)\simeq \Map_{L_{K}\nCalg_{\sepa,h,q}}(A,B)
\end{equation}   for the mapping space between two {\em separable} $C^{*}$-algebras $A$ and $B$  in this localization. Here the left-hand side of this equivalence is the  space associated to the topological space of homomorphisms from $qA$ to $K\otimes B$.

By \cite{MR899916} it is known that for  two   separable  $C^{*}$-algebras $A$ and $B$
there is an isomorphism \begin{equation}\label{fewweqfhwiffqhideewdq333}\pi_{0}\underline{\Hom}(qA,K\otimes B)\cong \KK^{\class}_{\sepa}(A,B)\ .\end{equation}
It is probably the deepest challenge of these notes to provide an accessible proof of the analogue \begin{equation}\label{fewweqfhwiffqhideewdq}\pi_{0}\underline{\Hom}(qA,K\otimes B)\cong \KK_{\sepa,0}(A,B)
\end{equation} of this  formula with classical $KK$-theory replaced by the the homotopy theoretic version constructed in the present paper.
   Note that in contrast to  \cite{MR899916}, where  
\eqref{fewweqfhwiffqhideewdq333} is essentially the definition of the right-hand side,
in our situation the group on the right of \eqref{fewweqfhwiffqhideewdq} is defined  as the  group of components of the mapping space
in a certain Dwyer-Kan localization.
By  \eqref{fweqwefljweofqwedeqdqedqewd} and the possibility to replace $B$ by suspensions $S^{i}(B)$ for $i$ in $\nat$ we see that  \eqref{fewweqfhwiffqhideewdq} 
 is equivalent to the fact that the functor $L_{\sepa,h,K,q}:\nCalg_{\sepa}\to L_{K}\nCalg_{\sepa ,h,q}$   is equivalent to the functor $\kk_{\sepa}:\nCalg_{\sepa}\to \KK_{\sepa}$.  An equivalent formulation of this latter fact  is 
 the automatic semiexactness \cref{tkohprtggrtgetrg} stating that for any additive left-exact $\infty$-category $\bD$ the natural inclusion 
\[\Fun^{h,s,se+Sch}(\nCalg_{\sepa},\bD)\to \Fun^{h,s,splt+Sch}(\nCalg_{\sepa},\bD)\]
from homotopy invariant, stable and Schochet- and semiexact functors   to homotopy invariant, stable and Schochet- and split-exact functors   is an equivalence. The proof of the  automatic semiexactness theorem 
will be discussed in detail in \cref{weijgiowergfrfreferfwr}. 

Note that  semiexactness of an exact sequence of $C^{*}$-algebras is defined in terms of the existence of a completely positive contractive (cpc) split. Since this is an analytic condition which somehow has to be exploited  it is not surprising that 
the proof of the automatic semiexactness theorem  in \cref{weijgiowergfrfreferfwr} is not  purely homotopy   theoretic  in  nature  but contains various analytic arguments. But since we will  avoid to use Kasparov products or other deep results from the classical theory it might be  quite accessible to homotopy theorists. In particular note that our proof does not depend on the formula \eqref{fweqwefljweofqwedeqdqedqewd}, whose proof in \cref{erigjwoergerfwerfwref} involves
Pedersen's derivation lifting.
 
But using   \eqref{fweqwefljweofqwedeqdqedqewd} and the automatic semiexactness theorem together  in \cref{weijotgwegferfwrefw} we can show 
 also in the context of the present notes
that    $\KK_{\sepa}$ admits countable  colimits  and is  therefore idempotent complete. In  \cite[Lem. 2.19]{KKG}
this fact has been shown by importing  \cite[Thm. 2.9]{kasparovinvent}.
\footnote{Using the results from  \cite{Bunke:2024aa} one can show that also 
 $\EE_{\sepa}$ admits countable coproducts.}


The classical construction of $KK$-theory is based on the notion of Kasparov modules. Equivalence classes of   Kasparov modules are interpreted as elements of 
  $\KK_{0}(A,B)$. Kasparov modules in  a certain standard form can be captured by the Cuntz picture in terms of the  $q$-construction. 
 Essentially by definition, the left-hand side of \eqref{fweqwefljweofqwedeqdqedqewd} can be interpreted
 as the space of Kasparov $(A,B)$-modules. 
 So the mapping spaces in $L_{K}\nCalg_{\sepa,h,q}$ are   expressed in terms of spaces of Kasparov modules via  \eqref{fweqwefljweofqwedeqdqedqewd}, while the 
   automatic semiexactness theorem
   implies that these are also the mapping spaces in $\KK_{\sepa}$. 
  In the present note we will not discuss  the alternative models  for the group-valued $KK$-theory  based on asymptotic morphisms 
   \cite{thomsen1} or localization algebras \cite{Dadarlat_2018}.
   
   Recall that the classical concrete model of $E$-theory \cite{zbMATH04182148} involves asymptotic morphisms.
   In \cref{ewriogjowegregrewf99} we will show 
that asymptotic morphisms give rise to elements  in $\EE_{0}(A,B)$ 
in a way which is compatible with compositions.

We finally stress that these notes concentrate on the homotopy theoretic and categorical aspects of $KK$- and $E$-theory.
The full power of $KK$-theory to applications e.g. to the classification  programmes for $C^{*}$-algebras 
only unrolls itself if one employs the equivalence of  different cycle-by-relation models  based on Kasparov modules.
This aspect  will not be considered at all  in these notes.
Other applications e.g. to the Baum-Connes or Novikov conjecture require the ability to control the composition of  morphisms in $KK$-theory explicitly. 
If one uses the model based on Kasparov modules, then there are well-developed methods serving
this purpose e.g. using  connections.    In the model given in the present paper it is quite tricky to calculate
compositions of morphisms which do not   simply  come  from morphisms between the $C^{*}$-algebras. We actually do only one non-trivial example of such a composition which is  \cref{fkjgogerfwreferfw} which is already complicated enough. But note that this calculation is absolutely crucial since it provides the last cornerstone for the  automatic semiexactness theorem which also goes into the comparison result \cref{wrtgowtgwgrrwfrefw}.

%

These notes present an expanded version of the lecture notes of a course  on {\em Noncommutative Homotopy Theory}
tought  at the University of Regensburg in the winter term 2022/23. The author thanks the participants for suspicious attention and various critical remarks.
 
 {\em Acknowledgement: The author was supported by the SFB 1085 (Higher Invariants) funded by the Deutsche Forschungsgemeinschaft (DFG). The author profited from discussions with M. Land und U. Pennig which started this project. Special thanks go to B. Dünzinger for carefully reading the draft.
 }
 
\section{$C^{*}$-algebras}\label{qrfuhqeiurfqwfdewfdqwef9}

In this section we collect the basic facts from $C^{*}$-algebra theory which we will use in the present paper.
The material can be found in the introductory chapters of  textbooks  like \cite{MR458185}, \cite{pedersen}, \cite{zbMATH05256855},   \cite{williams}, \cite{Pisier} and many others. 

In order to fix set-theoretic size issues we choose three Grothendieck universes called the small, large, and very large sets.

In \cref{wrogjpwrgrfrewf} we will introduce the notions appearing in \cref{weijgoewrfrefwre9} and  \cref{weijgoewrfrefwre91}. 

We let $\nCalg$ denote the large, but locally small category of  small $C^{*}$-algebras and homomorphisms. By $\Calg$ we denote its subcategory of unital $C^{*}$-algebras and unit-preserving homomorphisms. 
As we are interested in the categorical properties of the categories of $C^{*}$-algebras  we will follow the approach in   \cite{crosscat}.
We consider $\nCalg$ as a full subcategory of the large locally small  category of small $*$-algebras $\nClinAlg$ over $\C$. The latter is the category of small   (possibly non-unital) algebras over $\C$ with an antilinear involution $*$ and structure-preserving maps.

A {\em $C^{*}$-seminorm} on a $*$-algebra $A$ is a submultiplicative seminorm satisfying the $C^{*}$-equality $ \|a^{*}a\|=\|a\|^{2}$. For $a$ in $A$ we define the {\em maximal seminorm} of $a$ by    $\|a\|_{\max}:=\sup_{\|-\|} \|a\|$, where the supremum runs over all $C^{*}$-seminorms on $A$. 

We say that $A$ is a {\em pre-$C^{*}$-algebra} if all its elements have a finite maximal seminorm. The inclusion $\npreCalg\to \nClinAlg$ of the category of  pre-$C^{*}$-algebras into the category of all $*$-algebras
is the left-adjoint of a right Bousfield localization whose right-adjoint is the {\em bounded elements functor} $\Bd^{\infty}$.

A  {\em $C^{*}$-algebra}    is a pre-$C^{*}$-algebra $A$ with the property  that   $(A,\|-\|_{\max})$ is a Banach space. 
The inclusion  $\nCalg\to \npreCalg$ of the category of $C^{*}$-algebras into the category of  pre-$C^{*}$-algebras is the right-adjoint of a left Bousfield localization whose  left adjoint is the {\em completion functor} $\compl$.

In view of its algebraic description the category $  \nClinAlg$ is clearly complete and cocomplete in the sense that it admits all small limits and colimit.
As a consequence of the above  characterization of $C^{*}$-algebras  
the category $\nCalg$ is  complete and cocomplete, too. We furthermore obtain an explicit description of limits and colimits in terms of their algebraic counterparts indicated by a superscript $\alg$. If $A:I\to \nCalg$ is an $I$-diagram of $C^{*}$-algebras for some small index category $I$, then
\begin{equation}\label{fvrvfevefcvsdfvsfdvsfdvsv}\lim_{I}A\cong \Bd^{\infty}({\lim_{I}}^{\alg}A)\ , \quad \colim_{I}A\cong \compl({\colim_{I}}^{\alg}A)\ .
\end{equation} 

 In particular, the coproduct of the $C^{*}$-algebras $A_{0}$ and $A_{1}$ is represented by the free product of $C^{*}$-algebras $A_{0}*A_{1}:=\compl(A_{0}*^{\alg}A_{1}) $
together with the canonical morphisms $\iota_{i}:A_{i}\to A_{0}*A_{1}$. Similarly, the product  of $A_{0}$ and $A_{1}$ is represented by the (algebraic) sum $A_{0}\oplus A_{1}$ together with the canonical projections $\pr_{i}:A_{0}\oplus A_{1}\to A_{i}$.  
If $(A_{i})_{i\in I}$ is a small  infinite family of $C^{*}$-algebras, then \eqref{fvrvfevefcvsdfvsfdvsfdvsv} says that
$\prod_{i\in I}A_{i}\cong \Bd^{\infty}(\prod_{i\in I}^{\alg}A_{i}) $ is the subalgebra of  the algebraic product of families $(a_{i})_{i\in I}$ of elements $a_{i}$ in $A_{i}$ with $\sup_{i\in I} \|a_{i}\|_{A_{i}}<\infty$.

From now on we will suppress the size adjectives large and small  as much as possible.

The category $\nCalg$ is pointed by the zero algebra $0$. 

 The category $\nCalg$ has two canonical symmetric monoidal structures $\otimes_{\max}$ and $\otimes_{\min}$.
 For $C^{*}$-algebras $A,B$   the {\em maximal tensor product}
 is defined by  $A\otimes_{\max}B:=\compl(A\otimes^{\alg}B)$, 
 where we use that the $*$-algebra $A\otimes^{\alg}B$ is actually a pre-$C^{*}$-algebra.
 
  In order to define the {\em minimal tensor product} (also called the spatial tensor product) we equip $A\otimes^{\alg}B$ with the minimal $C^{*}$-norm (not seminorm!) and form the closure.
This minimal norm can alternatively be characterized as the   norm induced by the representation $ A\otimes B\to B(H\otimes L)$ induced by any two faithful representations $A\to B(L)$ and $B\to B(H)$ for Hilbert spaces $L$ and $H$.

We always have a canonical morphism $A\otimes_{\max}B\to A\otimes_{\min}B$, and $A$ is called {\em nuclear} if this morphism is an isomorphism for all $B$. Its is known that commutative $C^{*}$-algebras and the $C^{*}$-algebra $K$ of compact operators on a separable Hilbert space are nuclear.
If one of the tensor factors is nuclear we can safely  write $\otimes$ and omit the subscript specifying the choice.

  \begin{ex}
  The commutative algebra objects $\CAlg(\nCalg)$ (say with respect to $\otimes_{\max}$)
 are precisely the unital  commutative $C^{*}$-algebras. \hB
  \end{ex}

If $X$ is a  compact topological space, then by $C(X)$ we denote the commutative  $C^{*}$-algebra of continuous $\C$-valued  functions on $X$.  
For  $C^{*}$-algebras $A$ and $B$ we let $\underline{\Hom}(A,B)$ denote the compactly generated topological space characterized by the property that for every {\em compact} space $X$ we have a natural bijection
\begin{equation}\label{fqwfewfdwedewdwdq}\Hom_{\Top}(X,\underline{\Hom} (A,B))\cong \Hom_{\nCalg}(A,C(X)\otimes B)\ .
\end{equation}  The topology on $ \underline{\Hom}(A,B)$
 is equivalent to the maximal compactly generated topology  containing the point-norm topology  on 
   $\Hom_{\nCalg}(A,B)$. 
In this way $\nCalg$ becomes a {\em category enriched in topological spaces}.
 
 A homomorphism $f:B\to C$ between $C^{*}$-algebras  is a {\em homotopy equivalence} if there exists a homomorphism $g:C\to B$, called a {\em homotopy inverse}, 
  such that $f\circ g$ is homotopic to $\id_{C}$ in $\underline{\Hom}(C,C)$ and
 $g\circ f$ is homotopic to $\id_{B}$ in $\underline{\Hom}(B,B)$. Equivalently, one could require that the
 induced map $\underline{\Hom}(A,f):
\underline{\Hom}(A,B)\to \underline{\Hom}(A,C)$ is a homotopy equivalence of topological spaces  for all $C^{*}$-algebras $A$.

A {\em left upper corner inclusion} $A\to A\otimes K$ is a homomorphism of the form $a\mapsto a\otimes e$ where $e$ is a minimal non-zero projection in $K$.

\begin{rem}\label{jgioweprgwrfrwrfrf}
If one interprets $K$ and $A\otimes K$ as algebras of $\nat$-indexed matrices with entries in $\C$ or $A$, respectively,  then we can write this map as
\[a\mapsto \left(\begin{array}{cccc}a&0&0&\dots\\0&0&0&\dots\\0&0&0&\cdots\\\vdots&\vdots&\vdots&\ddots \end{array}\right)\ .\] This picture explains the name {\em left upper corner} inclusion.\hB
\end{rem}

An exact sequence \[0\to I\to B\xrightarrow{\pi} Q\to 0\] of $C^{*}$-algebras is called {\em semisplit exact} (or {\em split exact}), if $\pi$ admits a completely positive contractive (cpc) right-inverse (or a right-inverse homomorphism, respectively).
It is known that the functor $A\otimes_{\max}-$ preserves  exact sequences and in addition the condition
of being  semi-split exact or split-exact. The functor $A\otimes_{\min}-$ preserves semi-split exact sequences and split-exact sequences. A cartesian square in $\nCalg$
\[\xymatrix{E\ar[r]\ar[d] &B \ar[d] \\D \ar[r] &C } \] is called 
{\em exact (semisplit)} if the vertical maps are surjective (admit a cpc split). The functor $A\otimes_{\max}-$ preserves exact cartesian squares and also  semisplit cartesian squares, and $A\otimes_{\min}-$ preserves semisplit ones.  Note that the fact that $B\to C$ is surjective or admits a cpc split implies that $E\to D$ has the same property.

A $C^{*}$-algebra is called  {\em separable} if it contains a countable dense subset. 
We let $\nCalg_{\sepa}$ denote the full subcategory of separable $C^{*}$-algebras.
Note that $\nCalg_{\sepa}$ is essentially small.
For a $C^{*}$-algebra $A$ we let $A'\subseteq_{\sepa}A$ denote the poset of separable subalgebras of $A$.
Then we have a canonical isomorphism \begin{equation}\label{adsfoiajoifafadsff}\colim_{A'\subseteq A}A'\cong A .
\end{equation} 

\begin{ex}
The algebra of compact operators $K(H)$ on a separable Hilbert space $H$ is separable. 
If $\dim(H)=\infty$, then the algebra of bounded operators $B(H)$ is not separable. 
If $X$ is a separable metric space, then $C_{0}(X)$ is a separable $C^{*}$-algebra. 
If $X$ is not compact, then the $C^{*}$-algebra of bounded continuous functions $C_{b}(X)$ on $X$ is not separable.
 \hB
\end{ex}

Let $F$ be a functor defined on $\nCalg$ or $\nCalg_{\sepa}$.
\begin{ddd}\label{wrogjpwrgrfrewf}
\mbox{}\begin{enumerate}
\item \label{rqgjfewfqfewf14rrerwg}$F$ is {\em homotopy invariant} if $F$ sends homotopy equivalences to equivalences.
\item\label{erwqgoihrioferfqfewfqewfqwef} $F$ is {\em stable} if it sends left-upper corner inclusions to equivalences.
\item\label{eqrkgoegferwfqfqfwe} $F$ is {\em reduced} if $F(0)$ is a  zero object.
\item\label{trokgperwgrefwerfrefew} $F$ is {\em exact} ({\em semi-split exact}, {\em split-exact}) if  $F$ is reduced and $F$ sends exact (semi-split exact or split-exact) sequences to fibre sequences.
\item\label{trokgperwgrefwerfrefew1} If $F$ is defined on $\nCalg $, then we say that $F$ is {\em $s$-finitary} if for every $C^{*}$-algebra $A$ the canonical morphism
$\colim_{A'\subseteq_{\sepa}A} F(A')\to F(A)$ is an equivalence.
\end{enumerate}
\end{ddd}

Here   in \cref{wrogjpwrgrfrewf}.\ref{eqrkgoegferwfqfqfwe} and \cref{wrogjpwrgrfrewf}.\ref{trokgperwgrefwerfrefew}
we implicitly assume that the target of $F$ is pointed. In 
\cref{wrogjpwrgrfrewf}.\ref{trokgperwgrefwerfrefew1} we further assume that the colimit exists.

\begin{rem}\label{wergijwerogerwfewrferwfw} In order to check that $F$ is homotopy invariant it suffices to check that $F(A)\to F(C([0,1])\otimes A)$ is an equivalence for every $C^{*}$-algebra $A$, where the map is induced by $(\C\to C([0,1]))\otimes\id_{A}$. 
 
 The functor $F$ is $s$-finitary if and only if it represents the left Kan-exaction of the restriction $ F_{|\nCalg_{\sepa}}$
 along the inclusion $\nCalg_{\sepa}\to \nCalg$.
 In view of \eqref{adsfoiajoifafadsff} a filtered colimit preserving functor is $s$-finitary.
   \hB
\end{rem}

\begin{rem} Many constructions of the present paper done for $\nCalg$ have a version for separable algebras.
We will indicate this in the notation  by adding  subscripts $\sepa$ to the categories or functors.
If everything goes through for separable algebras word by word, then we will simply state that we {\em have a separable version.} At some places separability matters, and then we will be explicit. \hB
\end{rem}

\section{Inverting homotopy equivalences}\label{lh}

In this section we study the Dwyer-Kan localization of the category $\nCalg$ at the set of homotopy equivalences.
We will show that the resulting $\infty$-category $\nCalg_{h}$ is presented by the topological enriched version of $\nCalg$ so that we understand the mapping spaces in $\nCalg_{h}$ explicitly. It will turn out that $\nCalg_{h}$
is a pointed left-exact $\infty$-category. 

We start with recalling the $\infty$-categorical background on Dwyer-Kan localizations. 
Let $\bC$ be a   $\infty$-category and $W$ be a set of morphisms in $\bC$. Then we can form the {\em Dywer-Kan localization} $$L:\bC\to \bC[W^{-1}]$$  of $\bC$ at $W$.  It is   characterized by the universal property
that \begin{equation}\label{vfdsviuhsviisvfdvsvsdv}L^{*}:\Fun (\bC[W^{-1}],\bD)\stackrel{\simeq}{\to} \Fun^{W}(\bC,\bD)
\end{equation}  is an equivalence for every $\infty$-category $\bD$, where the superscript $W$ on the right indicates the full subcategory of functors which send the elements of $W$ to equivalences \cite[Def. 1.3.4.1 \& Rem. 1.3.4.2]{HA}. In the present paper we will apply this to $\infty$-categories $\bC$, $\bD$ in the universe of large sets 
so that $\bC[W^{-1}]$ also belongs to this universe.

\begin{rem}\label{werkgoewgrefrfwref}The functor $L$ is essentially surjective and in order to make formulas more readable  we will usually denote the image $L(C)$ or $L(f)$ in $\bC[W^{-1}]$ of an object or morphism in $\bC$
   simply by $C$ or $f$.  This convention in particular applies when we insert them into functors defined on $ \bC[W^{-1}]$. But sometimes we need the longer, more precise notation in order to avoid confusion. 
\hB \end{rem}

If $\bC$ is symmetric monoidal, then we say  that the localization $L$ admits a symmetric monoidal refinement, if 
$\bC[W^{-1}]$ has a symmetric monoidal structure, $L$ has a symmetric monoial refinement, and we have an equivalence
\begin{equation}\label{qwefqwefwefewdwedqewd}  L^{*}:\Fun_{\otimes/\lax}(\bC[W^{-1}],\bD)\stackrel{\simeq}{\to} \Fun^{W}_{\otimes/\lax}(\bC ,\bD) 
 \end{equation} for every symmetric monoidal $\infty$-category $\bD$, where 
 the notation $\otimes/\lax$ indicates   two separate formulas, one  for symmetric monoidal functors and one for lax symmetric monoidal functors.   
 
In order to check that $L$ has a symmetric monoidal refinement by \cite[Def.3.3.2]{hinich}  it suffices to check that the functor $C\otimes-$ preserves $W$ for every object $C$ of $\bC$.

\begin{ddd} We let 
\begin{equation}\label{vfdvsdfvwr3fes}L_{h}:\nCalg\to \nCalg_{h}
\end{equation}   be the Dwyer-Kan localization of the category $\nCalg$ at the homotopy equivalences. \end{ddd}
By definition it is characterized by the universal property that 
 pull-back along $L_{h}$ induces for any $\infty$-category $\bD$ an equivalence  \begin{equation}\label{weqfwdqedeqdqed} L_{h}^{*}:\Fun(\nCalg_{h},\bD)\stackrel{\simeq}{\to} \Fun^{h}(\nCalg,\bD)\ ,\end{equation}
where the superscript $h$ indicates the full subcategory of $\Fun(\nCalg,\bD)$ of homotopy invariant functors (see \cref{wrogjpwrgrfrewf}.\ref{rqgjfewfqfewf14rrerwg}).  

We consider the tensor product  $\otimes_{?}$  on $\nCalg$  for $?$ in $\{\max,\min\}$.
\begin{lem} For  $?$ in $\{\max,\min\}$
the localization $L_{h}$ has a symmetric monoidal refinement. 
\end{lem}
\begin{proof}  It follows from the functoriality and associativity of $\otimes_{?}$ and \eqref{fqwfewfdwedewdwdq} that for 
  every $C^{*}$-algebra $A$    the  functor $A\otimes_{?} - :\nCalg\to \nCalg $ is continuous for the topological enrichment and therefore   preserves 
 homotopy equivalences. This implies that $L_{h}$ has a symmetric monoidal refinement.  \end{proof}

 Thus for every 
  symmetric monoidal $\infty$-category $\bD$ we have an equivalence  \begin{equation}\label{vacdqwefwfd} L_{h}^{*}:\Fun_{\otimes/\lax}(\nCalg_{h} ,\bD)\stackrel{\simeq}{\to} \Fun^{h}_{\otimes/\lax}(\nCalg,\bD)\ .  \end{equation}

\begin{rem}
 Note that on $\nCalg_{h}$ we have two symmetric monoidal structures $\otimes_{?}$, one for $?=\max$ and one for $?=\min$ which will be discussed in a parallel manner. In particular, \eqref{vacdqwefwfd} actually has two versions. \hB
 \end{rem}
   
  
In contrast to general Dwyer Kan localizations, in the present  case we can understand the mapping spaces in $\nCalg_{h}$ explicitly. In fact, we will see that the topologically enriched category $\nCalg$ directly presents the localization.  To this end we  
  apply the singular complex functor $\sing$ to the $\underline{\Hom}$-spaces in order to get a Kan-complex enriched category. Further applying 
 the homotopy coherent nerve we get an $\infty$-category $\nCalg_{\infty}$ together with a functor 
 $\nCalg\to \nCalg_{\infty}$ given by the inclusion of the zero skeleton of the mapping spaces.

 \begin{prop}\label{erjgiowergrwefwerfwref}
 The functor $\nCalg\to \nCalg_{\infty}$ presents the 
 Dwyer-Kan localization of $\nCalg$ at the homotopy equivalences. 
 \end{prop}
\begin{proof}
For every $C^{*}$-algebra $B$ we define the {\em path algebra} \begin{equation}\label{vadfvasdvscsdcsdac}PB:=C(\Delta^{1})\otimes B\ .
\end{equation}  By \eqref{fqwfewfdwedewdwdq}, defining  a map of simplicial sets $[n]\to \sign(\underline{\Hom}(A,B))$   is equivalent to specifying an element in $\Hom_{\nCalg}(A,C(\Delta^{n})\otimes B )$. We let 
$h_{B} :[1]\to \sign(\underline{\Hom}(PB,B))$ correspond to the identity of $PB$. One then checks that  for any $C^{*}$-algebra $A$ the canonically induced map
$$ \Hom_{\sSet}([n],\sing(\underline{\Hom}(A,PB)))\to \Hom_{\sSet}([1]\times [n],\sing(\underline{\Hom}(A,B)))$$
is a bijection. The assertion of \cref{erjgiowergrwefwerfwref} now follows from  \cite[Thm 1.3.4.7]{HA}.
\end{proof}

\begin{kor} \label{jgiwortegrfrefrwefw}The $\infty$-category
$\nCalg_{h}$  is locally small.
\end{kor}

\begin{rem} At various places in this note we will use that small topological spaces present objects in the large $\infty$-category of small spaces\footnote{This name is changed to {\em anima} in recent literature.} $\Spc$. This is achieved by the functor
\begin{equation}\label{gregwegwerfwfer}\ell:\Top\to \Spc
\end{equation}
which presents the $\infty$-category $\Spc$ as the Dwyer-Kan localization of $\Top$
at the set of weak homotopy equivalences. It is one of the fundamental principles called the {\em Grothendieck's homotopy 
hypothesis} that the $\infty$-category  $\Spc$ defined in this way is equivalent to the $\infty$-category of $\infty$-groupoids in which the mapping spaces of locally small $\infty$-categories naturally live.
For a general large $\infty$-category they belong to the very large $\infty$-category of large spaces which we will denote by $\SPC$.

We will use that $\ell$ preserves coproducts, products and sends Serre fibrant  cartesian squares to cartesian squares,
were a cartesian square \[\xymatrix{X\ar[r]\ar[d]&Y\ar[d]^{f}\\Z\ar[r]&U}\] 
in $\Top$ is called {\em Serre fibrant} if $f$ is a Serre fibration.
 \hB\end{rem}

  As an immediate corollary of Proposition \ref{erjgiowergrwefwerfwref} we get an explicit description of the mapping spaces in $\nCalg_{h}$.
 
\begin{kor}\label{erigjeroferferfrefwerf}
For any two $C^{*}$-algebras $A,B$   we have a natural equivalence of spaces
\begin{equation}\label{fijweiof24ffwrefwrefwff}\Map_{\nCalg_{h}}(A,B)\simeq \ell \underline{\Hom}(A,B)\ .
\end{equation} 
\end{kor}
 In the formula above we adopted the conventions from  \cref{werkgoewgrefrfwref}.

We now discuss limits and colimits in $ \nCalg_{h}$.

\begin{prop}\label{weoijgowefgrefwrefrf}
The category $\nCalg_{h}$ admits finite products and arbitrary small coproducts, and the localization $L_{h}$ preserves them.
\end{prop}
\begin{proof} We start with finite products.
Let $(B_{i})_{i\in I}$ be a finite  family of  $C^{*}$-algebras  and $A$ be any $C^{*}$-algebra. Then we must show that the 
map
\[\Map_{\nCalg_{h}}(A,\prod_{i}B_{i})\to \prod_{i\in I} \Map_{\nCalg}(A,B_{i})\] induced by the family of  projections 
$(\prod_{i\in I}  B_{i}\to B_{j})_{j\in I}$
 is an equivalence.
 But this follows from the fact that \begin{equation}\label{sfgsfdgrgergegsfd}\underline{\Hom} (A,\prod_{i}B_{i})\to \prod_{i\in I} \underline{\Hom} (A,B_{i})
\end{equation}
   is actually a homeomorphism.

 We now consider coproducts. Let $(A_{i})_{i\in I}$ be a small family of $C^{*}$-algebras and $B$ be any $C^{*}$-algebra. Then we must show that the map
  \[\Map_{\nCalg_{h}}(\coprod_{i\in I}A_{i}, B)\to \prod_{i\in I} \Map_{\nCalg}(A_{i},B)\]
  induced by the family of inclusions $(A_{j}\to \coprod_{i\in I}A_{i})_{j\in I}$ is an equivalence.
  Again this follows  from the fact that
 \[\underline{\Hom}(\coprod_{i\in I}A_{i}, B)\to \prod_{i\in I} \underline{\Hom}(A_{i},B)\]
is actually a homeomorphism.
\end{proof}

\begin{rem}
In the case of products we assume that the index set $I$ is finite. If it is not finite, then the map
\eqref{sfgsfdgrgergegsfd} is no longer a homeomorphism.
Let $X$ be a compact topological space.
Then the image under \eqref{sfgsfdgrgergegsfd}  of $\Hom_{\Top}(X,\underline{\Hom} (A,\prod_{i}B_{i}))$ in
\[\Hom_{\Top}(X, \prod_{i\in I} \underline{\Hom} (A,B_{i}))\cong \prod_{i\in I} \Hom_{\Top}(X,\underline{\Hom} (A,B_{i}))\] consists of the families of maps
$(\phi_{i}:X\to \underline{\Hom} (A,B_{i}))_{i\in I}$ such that
the family $(\phi_{i}(a):X\to  B_{i})_{i\in I}$ is equicontinuous for every $a$ in $A$.
\hB
\end{rem}

\begin{lem} The functor $L_{h}$ is reduced and
$\nCalg_{h}$ is pointed.
\end{lem}
\begin{proof}
The zero algebra represents the inital and the final object of $\nCalg_{h}$.
\end{proof}

\begin{ex}\label{weoigowerferferf}
Let $A$ be any $C^{*}$-algebra. Then $C_{0}([0,\infty))\otimes A$ also represents the zero object in $\nCalg_{h}$.
\hB
\end{ex}

A morphism $f:B\to C$ in $\nCalg$ is called a {\em Schochet fibration} if the map
$f_{*}:\underline{\Hom}(A,B)\to \underline{\Hom}(A,C)$ is a Serre fibration of topological spaces for every $C^{*}$-algebra $A$  \cite{zicki}.

\begin{ex} \label{rgifqfqwef} If $i:Y\to X$ is a map of compact spaces which has the homotopy extension property, then the restriction map $i^{*}:C(X)\to C(Y)$ is a Schochet fibration
which in addition admits a cpc split.
\hB\end{ex}
A cartesian square \[\xymatrix{A\ar[r]\ar[d] &B \ar[d]^{f} \\D \ar[r]  &C } \]
is called {\em Schochet fibrant} if $f$ is a  Schochet fibration.
Note that a Schochet fibration is automatically surjective.
If $D=0$, then we say that
\[0\to A\to B\to C\to 0\] is a {\em Schochet exact} sequence.

\begin{ddd}\label{fuhquifhiuqwefeqwdewdewdcdfvca}
A functor $\nCalg\to \bC$ will be called {\em Schochet exact}
if  it sends Schochet fibrant cartesian squares to cartesian squares.
\end{ddd}
We will indicate Schochet exact functors by a super script as in  $\Fun^{Sch}$. 

\begin{rem}\label{fuhquifhiuqwefeqwdewdewdcdfvca1}
In contrast to the other notions of exactness introduced in \cref{wrogjpwrgrfrewf}.\ref{trokgperwgrefwerfrefew}, the notion of Schochet exactness is formulated in terms of squares instead of exact sequences.

If $\bC$ is pointed, then a reduced Schochet exact functor sends Schochet exact sequences to fibre sequences.
If $\bC$ is stable, then it is easy to see that the converse is also true. A functor which sends Schochet exact  sequences
to fibre sequences is reduced and Schochet exact, see \cite[Lem. 2.14]{KKG} for analogous statements for semiexact functors and squares. 
\hB
\end{rem}


\begin{rem}\label{werigowrgerfewrfwerfwerf} For the proof of  \cref{weokjgpwerferfwef}.\ref{eiojfgrgfreaedc3} below we need the   mapping cylinder construction. 
  The {\em mapping cylinder} of a map $f:B\to C$ of $C^{*}$-algebras is defined by the     Schochet fibrant and semisplit cartesian square 
\begin{equation}\label{grrefgrefrefrefwerf}\xymatrix{Z(f)\ar[r]\ar[d]^{h_{f}} &PC \ar[d]^{\ev_{0}} \\B \ar[r]^{f} &C } \ ,
\end{equation} where $PC$ is the path algebra    as in \eqref{vadfvasdvscsdcsdac}.
The maps 
 $h_{f}:Z(f)\to B$  and $\ev_{0}$ are   homotopy equivalences.
 We write elements in $Z(f)$ as pairs $(b,\gamma)$ with $b$ in $B$ and $\gamma$ in $PC$ such that $\gamma(0)=f(b)$. The map
$\tilde f: Z(f)\to C$ given by $(b,\gamma) \mapsto \gamma(1)$ is a Schochet fibration and also admits a cpc split $c\mapsto (0,\gamma_{c})$ with $\gamma_{c}(t):=tc$.
We further define the {\em mapping cone} of $f$ by $C(f):=\ker(\tilde f)$.
An element of $C(f) $ is thus a pair $(a,\gamma)$ of an element of $A$ and a path in $C$ with $f(a)=\gamma(0)$ and $\gamma(1)=0$.
 
 The sequence \begin{equation}\label{regegregefw}0\to C(f)\xrightarrow{\iota_{f}} Z(f)  \xrightarrow{\tilde f} C\to 0 
\end{equation}
 is Schochet- and semi-split exact.  
\hB
\end{rem}


Recall that an $\infty$-category is called {\em left-exact} if it admits all finite limits. A functor between left-exact $\infty$-categories is called {\em left-exact} if it preserves finite limits. We use the notation $\Fun^{\lex}$ in order to denote the full subcategory of left-exact functors.

\begin{prop}\label{weokjgpwerferfwef}\mbox{}
\begin{enumerate}
\item \label{eiojfgrgfreaedc} 
The $\infty$-category $\nCalg_{h}$ is left-exact.
\item \label{eiojfgrgfreaedc1} The functor $L_{h}$ sends Schochet fibrant cartesian squares to cartesian squares.
\item\label{elrkijgowlferfwffr9} The pull-back along $L_{h}$ induces for every left-exact $\infty$-category $\bD$ 
an equivalence  \begin{equation}\label{fdvrfdsfsfrewfcsdfvfdv}  L_{h}^{*}:\Fun^{\lex}(\nCalg_{h},\bD)\stackrel{\simeq}{\to}  \Fun^{h,Sch}(\nCalg,\bD)\ . 
\end{equation}\item\label{elrkijgowlferfwffr91}  The pull-back along the symmetric monoidal refinement of $L_{h}$ induces for every  symmetric monoidal left-exact $\infty$-category $\bD$ 
an equivalence  $$L_{h}^{*}:\Fun^{\lex}_{\otimes,\lax}(\nCalg_{h},\bD)\stackrel{\simeq}{\to}  \Fun^{h,Sch}_{\otimes,\lax}(\nCalg,\bD)\ .$$
\item   \label{eiojfgrgfreaedc3}   For  $?$ in $\{\min,\max\}$
the functor $-\otimes_{?} -:\nCalg_{h}\times \nCalg_{h}\to \nCalg_{h}$  is  bi left-exact.
 \end{enumerate}
\end{prop}
\begin{proof}
We let $W$ be the subcategory of  homotopy equivalences in $\nCalg$ and $F$ be the subcategory of  Schochet fibrations. Then
$(\nCalg,W,F)$ is a category of fibrant objects in the sense of \cite[Def. 7.4.12 and Def. 7.5.7]{Cisinski:2017}.
The corresponding verifications are due to  \cite[Thm. 2.19]{Uuye:2010aa}. 
The main point is to see that the pull-back of a Schochet fibration or of a homotopy equivalence is again a Schochet fibration or a homotopy equivalence.

The Assertions \ref{eiojfgrgfreaedc} and \ref{eiojfgrgfreaedc1}  then follow from \cite[Prop. 7.5.6]{Cisinski:2017}.
For  \ref{eiojfgrgfreaedc1}, in view of \cref{erigjeroferferfrefwerf} one could argue alternatively  that $\underline{\Hom}(A,-)$ sends
Schochet fibrant cartesian squares to Serre fibrant cartesian squares, and that $\ell$ sends  Serre fibrant cartesian squares
 to cartesian squares.

 
 We now show Assertion \ref{elrkijgowlferfwffr9}. 
  By Assertion \ref{eiojfgrgfreaedc1} it is clear that $L_{h}^{*}$ in \eqref{fdvrfdsfsfrewfcsdfvfdv} sends left-exact functors to Schochet exact functors. Since it is the restriction of the equivalence in \eqref{weqfwdqedeqdqed} it is fully faithful.
  We argue that it  also essentially surjective. Let $F$ be in $\Fun^{h,Sch}(\nCalg,\bD)$. Then by  \eqref{weqfwdqedeqdqed} there exists $\hat F$ in $ \Fun(\nCalg_{h},\bD)$ such that $L_{h}^{*}\hat F\simeq F$.
  We must show that $\hat F$ is left-exact.  Since it is reduced it suffices to show that it preserves cartesian squares.
 
In view of \cref{erigjeroferferfrefwerf} 
any diagram   of the shape  \[\xymatrix{ & B\ar[d]^{f} \\ D\ar[r]  & C} \] in $\nCalg_{h}$
 is equivalent to the image under $L_{h}$ of  a diagram of this shape in $\nCalg$. 
 We can replace $f$ by 
  $\tilde f:Z(f)\to C$ without changing  the image of the diagram under $L_{h}$ up to equivalence. 
 We now complete the diagram to a cartesian square
 \begin{equation}\label{vfdvsrefvsvwergtwerg} \xymatrix{ P\ar[d]^{f'}\ar[r] & Z(f)\ar[d]^{\tilde f} \\ D\ar[r]  & C}   \end{equation}in $\nCalg$.
 It is Schochet fibrant and  semisplit.  Its image under $L_{h}$ is then a cartesian square in $\nCalg_{h}$, and every cartesian square in $\nCalg_{h}$  is equivalent to one of this form. The image under $\hat F$ of $L_{h}$ of the latter square is equivalent to the image under $F$ of the original Schochet fibrant cartesian square and hence a cartesian square.

Assertion \ref{elrkijgowlferfwffr91}  follows from a combination of Assertion  \ref{elrkijgowlferfwffr9} and the equivalence \eqref{vacdqwefwfd}.

In the following $\otimes$ stands for $\otimes_{\max}$ or $\otimes_{\min}$.
Let $A$ be a $C^{*}$-algebra. It  follows from \cref{weoijgowefgrefwrefrf} that the endofunctor $A\otimes -:\nCalg_{h}\to \nCalg_{h}$ is reduced and preserves finite products. So it suffices to show that it preserves  cartesian squares.
As seen above every cartesian square in $\nCalg_{h}$ is equivalent to the image under $L_{h}$ of a  square of the form \eqref{vfdvsrefvsvwergtwerg}.
 
Using the exactness of $A\otimes_{\max}-$ or the semi-exactness of $A\otimes_{\min}-$ we see that
 \begin{equation}\label{gwerfefewrfwerfwerfw}\xymatrix{ A\otimes P\ar[d]\ar[r]& A\otimes Z(f)\ar[d]^{\id_{A}\otimes \tilde f} \\ A\otimes D\ar[r]  & A\otimes C} 
\end{equation} 
 is again cartesian.  By analysing the application of   $A\otimes- $ to \eqref{grrefgrefrefrefwerf} we can  obtain an 
 isomorphism 
 \[ \xymatrix{A\otimes Z(f)\ar[r]^{\cong} \ar[d]^{\id_{A}\otimes \tilde f} & Z(\id_{A}\otimes f) \ar[d]^{\widetilde{\id_{A}\otimes f}} \\ A\otimes C\ar@{=}[r] &A\otimes C }\ .\]
 We conclude
 that the square \eqref{gwerfefewrfwerfwerfw} is again a Schochet fibrant cartesian square, and  that its image under $L_{h}$ is a cartesian square in $\nCalg_{h}$.
  \end{proof}


%

\begin{ex}\label{rgkopwergerfwrfwrf}
The functor $L_{h}$   sends the sequence \eqref{regegregefw}  to a fibre sequence. Since the square 
\[\xymatrix{Z(f)\ar[r]^{\tilde f} \ar[d]^{h_{f}} &C \ar@{=}[d] \\B \ar[r]^{f} &C }\]
commutes up to a preferred homotopy it provides an equivalence between $L_{h}(\tilde f)$ and $L_{h}(f)$.
In particular, the mapping cone $C(f)$    represents the fibre of $L_{h}(f)$.\hB
\end{ex}

\begin{ex}\label{qerigqerfewfq9}
A pointed  left-exact $\infty$-category $\bC$ has a loop endofunctor
$\Omega:\bC\to \bC$ For an object $C$ of $\bC$ the object $\Omega C$ is  determined by 
the pull-back \begin{equation}\label{sfdvfsdvfdvsfdvsvfdvftregrtgfv}\xymatrix{\Omega C\ar[r]\ar[d] & 0\ar[d] \\ 0\ar[r] &C } \ .
\end{equation}
The category  $\nCalg$ has the suspension endofunctor \[S:=C_{0}(\R)\otimes -:\nCalg\to \nCalg\ .\]
The square of restriction maps
\[\xymatrix{C_{0}(\R)\ar[r]\ar[d] & \ar[d]^{\ev_{0}} C_{0}((-\infty,0])\\C_{0}([0,\infty)) \ar[r]^{\ev_{0}}  & \C}\] 
 is a Schochet fibrant semisplit cartesian square by \cref{rgifqfqwef}.
 Applying $L_{h}( - \otimes A)$ we get a cartesian square whose lower left and upper right corners are zero objects by \cref{weoigowerferferf}. We therefore obtain an equivalence of endofunctors  \begin{equation}\label{terhertgrtgetrgetrgt}L_{h} \circ S\simeq \Omega\circ L_{h}:\nCalg_{h}\to \nCalg_{h}\ .
\end{equation}
 \hB 
 \end{ex}

 \begin{ex}\label{wtghoktopgerwfrefwrefw}In this example we consider the Puppe sequence associated to a  morphism $f:A\to B$.
 The latter  gives rise to the mapping cone sequence \eqref{regegregefw}.
Since $\nCalg_{h}$  is left-exact we can form the following diagram of pull-back squares
  \[\xymatrix{\Omega^{2}  L_{h}(B))\ar[d]\ar[r]^{\Omega L_{h}(\partial_{f})} &\Omega L_{h}(C(f))\ar[d]^{\Omega L_{h}(i_{f})} \ar[r] &0\ar[d]&\\0\ar[r]&\Omega L_{h}(A)\ar[d]\ar[r]^{\Omega L_{h}(f) }&\Omega L_{h}(B)\ar[r]\ar[d]^{\partial_{f}}&0\ar[d]\\&0\ar[r]&L_{h}(C(f))\ar[d]\ar[r]^{L_{h}(i_{f})}&L_{h}(A)\ar[d]^{L_{h}(f)}\\&&0\ar[r]&L_{h}(B)}\] in 
   $\nCalg_{h}$. The lower square is cartesian by \cref{rgkopwergerfwrfwrf}.
    We further use the homotopy invariance of $L_{h}$ applied to the homotopy equivalence $h_{f}$ in order to replace $L_{h}(Z(f))$ by $L_{h}(A)$, and
  we use \eqref{sfdvfsdvfdvsfdvsvfdvftregrtgfv} in order to identify the corners with iterated loops of the objects. E.g. for  the corner
  $\Omega L_{h}(A)$ just observe that the horizontal composition of the two middle squares is again cartesian.
By \eqref{terhertgrtgetrgetrgt} we can express looping in terms of suspension. The sequence
\begin{equation}\label{oijiosfeqrfgerfefsfdfgsgsdfgdsfg}
\dots \to L_{h}(S^{2}(B))\to L_{h}(S(C(f)))\to L_{h}(S(A))\to L_{h}(S(B))\to L_{h}(C(f))\to L_{h}(A)\to L_{h}(B)
\end{equation} in  $\nCalg_{h}$
is called the {\em Puppe sequence} associated to $f$. Each consecutive pair of morphisms is a part of a fibre sequence in $\nCalg_{h}$.
\hB
  \end{ex}

\begin{rem}\label{qertgoqergoijeqwgwegwergw}
If $\bC$ is some $\infty$-category with a set of morphisms $W_{\bC}$, $\bD$ is a full subcategory, and if we set $\bW_{\bD}:=W\cap \bD$, then we have a commutative square \[\xymatrix{\bD\ar[r]\ar[d] &\bC \ar[d] \\ \bD[W_{\bD}^{-1}]\ar[r] &\bC[W_{\bC}^{-1}] } \ ,\]
where the vertical  functors are Dwyer Kan localizations. In general the lower horizontal map is not fully faithful. 
But this is true if
 we specialize to the case  where $\bC=\nCalg$, $W_{\bC}$ are the homotopy equivalences, and $\bD=\nCalg_{\sepa}$.\hB \end{rem} We 
 write $L_{\sepa,h}:\nCalg_{\sepa}\to \nCalg_{\sepa,h}$ for the corresponding Dwyer-Kan localization.
 Thus  for any $\infty$-category $\bD$ pull-back along $L_{\sepa,h}$ induces an equivalence
 \[L_{\sepa,h}^{*}:\Fun(\nCalg_{\sepa,h},\bD)\stackrel{\simeq}{\to} \Fun^{h}(\nCalg_{\sepa},\bD)\ .\]
Note that $\nCalg_{\sepa,h}$ is essentially small.
Using 
 \cref{erigjeroferferfrefwerf}, \cref{weokjgpwerferfwef} and its separable version and the fact that the tensor products preserve separable algebras  we get the following statements.

 \begin{kor}\label{weogpwergerfgwerfwerf9}\mbox{}\begin{enumerate}
 \item \label{powekgpowegref} We have a commutative square 
 \[\xymatrix{\nCalg_{\sepa}\ar[r]\ar[d]^{L_{\sepa, h}} &\nCalg \ar[d]^{L_{h}} \\ \nCalg_{\sepa,h} \ar[r] &\nCalg_{h} } \ ,\] whose vertical arrows are Dwyer-Kan localizations and whose horizontal arrows are fully faithful.
\item The square in \ref{powekgpowegref}.   has a refinement to a diagram of symmetric monoidal categories and symmetric monoidal functors  for $\otimes_{?}$ with $?\in \{\min,\max\}$ such that $L_{\sepa,h}$ becomes a symmetric monoidal localization.
\item $\nCalg_{\sepa,h}$ is pointed and $L_{\sepa,h}$ is reduced.
\item \label{egljkeogpweferfwrefwref} $\nCalg_{\sepa,h}$ admits  finite products and countable coproducts, and $L_{\sepa,h}$ preserves them.
\item $\nCalg_{\sepa,h}$ is left-exact and $L_{\sepa,h}$ sends Schochet fibrant cartesian squares of separable algebras to cartesian squares.
\item\label{elrkijgowlferfwffr9sss} The pull-back along $L_{\sepa,h}$ induces for every left-exact $\infty$-category $\bD$ 
an equivalence  \begin{equation}\label{fdvrfdsfsfrewfcsdfvfdvsss}  L_{\sepa,h}^{*}:\Fun^{\lex}(\nCalg_{\sepa,h},\bD)\stackrel{\simeq}{\to}  \Fun^{h,Sch}(\nCalg_{\sepa},\bD)\ . 
\end{equation}\item\label{elrkijgowlferfwffr91sss}  The pull-back along the symmetric monoidal refinement of $L_{\sepa,h}$ induces for every  symmetric monoidal left-exact $\infty$-category $\bD$ 
an equivalence  $$L_{\sepa,h}^{*}:\Fun^{\lex}_{\otimes/\lax}(\nCalg_{\sepa,h},\bD)\stackrel{\simeq}{\to}  \Fun^{h,Sch}_{\otimes/\lax}(\nCalg_{\sepa},\bD)\ .$$
\item For $?\in \{\min,\max\}$ the functor $-\otimes_{} -:\nCalg_{\sepa,h}\times \nCalg_{\sepa,h}\to \nCalg_{\sepa,h}$  is  bi left-exact.
\end{enumerate}
 \end{kor}
 
In order to ensure separability of the coproducts, in \cref{weogpwergerfgwerfwerf9}.\ref{egljkeogpweferfwrefwref} 
we must restrict to countable families.
%

\section{Stabilization}\label{lhK}

In this section we consider the Dwyer-Kan   localization of the $\infty$-category $\nCalg_{h}$ from \eqref{vfdvsdfvwr3fes} at the set of left-upper corner inclusions
$A\to K\otimes A$ for all $C^{*}$-algebras $A$, where $K$ denotes the algebra of compact operators on a separable Hilbert space. It turns out that this localization  is a left Bousfield localization generated by the tensor  idempotent  $K$. This fact makes is easy to understand the resulting category $L_{K}\nCalg_{h}$. Its main new feature is semi-additivity.

We start with recalling the $\infty$-categorical background. 
  Let $\bC$ be an $\infty$-category with an endofunctor $L:\bC\to \bC$ and a natural transformation 
$\alpha:\id_{\bC}\to L$. If for every object $C$ the morphisms \[\alpha_{L(C)}, L(\alpha_{C}):L(C)\to L(L(C)) \] are equivalences, then $L:C\to L(C)$ is the left-adjoint of a {\em left Bousfield localization} with unit $\alpha$ (see \cite[Prop. 5.2.7.4]{htt}). It is also a Dwyer-Kan localization 
at the set of morphisms $W_{L}:=\{\alpha_{C}\mid C\in \bC\}$.

 Let $\bC$ be left-exact $\infty$-category with a set of morphisms $W$.
 We say that the Dwyer-Kan localization $L:\bC\to \bC[W^{-1}]$   is {\em left-exact}, if $\bC[W^{-1}]$ is left-exact and the functor $L$ is left-exact. In this case, in addition to \eqref{vfdsviuhsviisvfdvsvsdv},  we have an equivalence
\[L^{*}:\Fun^{\lex}(\bC[W^{-1}],\bD)\stackrel{\simeq}{\to} \Fun^{\lex,W}(\bC,\bD)
\] for any left-exact $\infty$-category $\bD$.

In the present section we encounter a smashing  left Bousfield localization which is generated by an 
  idempotent object \cite[Sec. 4.8.2]{HA}.
  An {\em idempotent  object} in a   symmetric monoidal $\infty$-category $\bC$ with tensor unit $\beins$  
 is an
object $(\epsilon:\beins \to A)$ in the slice $\bC_{\beins/}$ such that the map
$\epsilon\otimes \id_{A}:A\simeq \beins\otimes A\to A\otimes A$ is an equivalence.  The inverse of this map is the multiplication of an essentially unique refinement of this object to a commutative algebra object in $ \bC$.

The functor $L_{A}:=A\otimes -:\bC\to  L_{A}\bC$ together with the 
unit $\id\xrightarrow{\epsilon\otimes \id_{-}} L_{A}$ satisfies the conditions above ensuring that
$L_{A}:\bC\to L_{A}\bC$ is
 the left-adjoint of a left Bousfield localization.
The localization
$L_{A}:\bC\to L_{A}\bC$ is also the Dwyer Kan localization inverting the set $W_{A}$ of morphisms
$B\xrightarrow{\epsilon\otimes \id_{B}}A\otimes B$ for all $B$. By associativity of the tensor product the set  $W_{A}$ is preserved by the functor $-\otimes C$ for any object $C$ of $\bC$. It follows that the localization $L_{A}$ has a  symmetric monoidal refinement.

If $\bC$ has arbitrary coproducts, then by general properties of a left Bousfield localization so does $L_{A}\bC$. Given a family $(B_{i})_{i\in I}$ in $\bC$ we have
\[\coprod^{L_{A}\bC}_{i\in I} L_{A}(B_{i})\simeq L_{A}(\coprod_{i\in I}^{\bC} B_{i})\ .\]
Finally, if $\bC$ is left-exact and $\otimes$ is bi-left exact, then  $L_{A}$ is a  left-exact  localization, and the induced tensor product
$-\otimes-:L_{A}\bC\times L_{A}\bC\to L_{A}\bC$ is bi-left exact.

We will apply the above constructions to  the left-exact symmetric monoidal $\infty$-category $\nCalg_{h}$ with one of the bi-left exact structures $\otimes_{\max}$ or $\otimes_{\min}$  and the tensor unit $\C$.

Recall that $K$ is the algebra of compact operators on a separable Hilbert space $H$.
Let $e$ be a minimal non-zero projection in $K$ and $\epsilon:\C\to K$ be the homomorphism $\lambda\mapsto \lambda e$.  \begin{lem}
$(\epsilon:\C\to K)$ is an idempotent object in $\nCalg_{h}$.
\end{lem}
\begin{proof}
For completeness of the presentation we add an argument for this  well-known  fact  from  $C^{*}$-algebra theory.
 For $t$ in $[0,1]$ we define the isometry
 $U_{t}:L^{2}((-\infty,0])\to L^{2}((-\infty,1])$ by 
\[U_{t}(f)(x):=\left\{\begin{array}{cc}f(x)&x\in (-\infty,-t)\\ \frac{1}{\sqrt{2}} f(\frac{x-t}{2})&x\in (-t,t)\\0&x\in [t,1]\end{array} \right.\ .\]
 Then $t\mapsto U_{t}$ is strongly continuous and $U_{1}$ is unitary.
  
 We now construct a  square 
 \[\xymatrix{H\ar[r]^{e\otimes \id_{H}}\ar[d]^{v}&H\otimes H\ar[d]^{w}\\L^{2}((-\infty,0])\ar[r]^{U_{0}}&L^{2}((-\infty,1])}\]  of isometric maps between Hilbert spaces. For 
 $v$ we choose  any unitary  isomorphism. In order to construct $w$ we 
  choose an isomorphism $\im(e)\cong \C$ and
 a unitary  isomorphism $v':  \im(e)^{\perp}\otimes H\to L^{2}([0,1])$. Then we get the isomorhism
 $H\otimes H\cong \im(e)\otimes H\oplus \im(e)^{\perp}\otimes H\cong H\oplus   \im(e)^{\perp}\otimes H$.
 We then define $w_{|H}:= U_{0}\circ v$ and $ w_{|\im(e)^{\perp}\otimes H}:=v'$

 Then $t\mapsto \phi_{t}:=w^{*}U_{t}v (-)v^{*}U_{t}^{*}w:K\to K\otimes K$ is a point-norm continuous homotopy from
 $\epsilon\otimes \id_{K}$ to an isomorphism.
  \end{proof}
 
 For every $C^{*}$-algebra $A$ the map
  \begin{equation}\label{oiwjveoivjowevsvdfvsfdv}\kappa_{A}:A\simeq \C\otimes A\xrightarrow{\epsilon\otimes \id_{A}}  K\otimes A
\end{equation} 
 is a {\em left upper corner inclusion}.

 We let $L_{K}\nCalg_{h}$ denote the image of the functor $K\otimes -$.
 Since we have an isomorphism $K\otimes K\cong K$ in $\nCalg$, it consists precisely of the objects which are represented by {\em $K$-stable $C^{*}$-algebras}, i.e. algebras $A$ satisfying $A\cong K\otimes A$ (note that this isomorphism is not related with the left upper corner inclusion).
 
 \begin{ddd}
 We define the functor
 \[L_{K}:\nCalg_{k}\to L_{K}\nCalg_{h}\ , \quad A\mapsto K\otimes A\ .\]
 and the natural transformation
 $\kappa:\id\to L_{K}$ with components $(\kappa_{A})_{A}$.   \end{ddd}
 Note that $L_{K}\nCalg_{h}$ is locally small by \cref{jgiwortegrfrefrwefw}.
%
%
   
 We define the functor \begin{equation}\label{qwefdwedqdwed}L_{h,K}: \:\:\: \nCalg\xrightarrow{L_{h}} \nCalg_{h}\xrightarrow{L_{K}} L_{K}\nCalg_{h}
\ .\end{equation}

 \begin{kor}\label{wreogkpwegrewferf}\mbox{}
 \begin{enumerate}
 \item\label{wegergoijiowerfgwerferfrewfrfwefrf} 
 The functor $L_{K}$ is the left-adjoint of a left Boufield localization 
 \[L_{K}:\nCalg_{h}\leftrightarrows L_{K}\nCalg_{h}:\incl\] unit $\kappa$.    \item \label{wkogpwtgwerfwrefwr} The localization  $L_{K}$ is   left-exact.
 \item \label{wetokgpwerferwfwref}
 $L_{K}\nCalg_{h}$ admits arbitrary small coproducts and $L_{K}$ preserves them. 
  \item  For $\otimes_{?}$ with $?\in \{\min,\max\}$ the localization $L_{K} $ has a symmetric monoidal refinement.
 
 \item \label{ierjogoergegsegs9} The functor $-\otimes_{?}-:L_{K}\nCalg_{h}\times L_{K}\nCalg_{h}\to L_{K}\nCalg_{h}$ is bi-left exact.
 \end{enumerate}
  \end{kor}
For Assertion \ref{wetokgpwerferwfwref} we also used \cref{weoijgowefgrefwrefrf}.
The fact that $L_{K}$ is a left Bousfield localization yields the following formula for the mapping spaces in $L_{K}\nCalg_{h}$. For $A,B$ in $\nCalg$ we have
\begin{equation}\label{vwihviuewhviwevewre}\Map_{L_{K}\nCalg_{h}}(A,B)\simeq \Map_{\nCalg_{h}}(A,K\otimes B)\stackrel{\eqref{fijweiof24ffwrefwrefwff}}{\simeq} \ell \underline{\Hom}(A,K\otimes B)\ .
\end{equation} 

   The pull-back along $L_{h,K}$ induces for every
$\infty$-category $\bD$ an equivalence \begin{equation}\label{qwfdqedewdqwedd}L_{h,K}^{*}:\Fun(L_{K}\nCalg_{h},\bD)\stackrel{\simeq}{\to}\Fun^{h,s}(\nCalg,\bD)\ ,
\end{equation}
 where the additional subscript indicates the full subcategory of $\Fun^{h}(\nCalg,\bD)$ of stable functors (see \cref{wrogjpwrgrfrewf}.\ref{erwqgoihrioferfqfewfqewfqwef}).
For any symmetric monoidal $\infty$-category the pull-back along the symmetric monoidal refinement of $L_{h,K}$
induces an equivalence \begin{equation}\label{wregwefrefewrfrefw}L_{h,K}^{*}:\Fun_{\otimes/\lax} (L_{K}\nCalg_{h},\bD)\stackrel{\simeq}{\to} \Fun_{\otimes/\lax}^{h,s}(\nCalg,\bD)\ .
\end{equation}
 For every left-exact $\infty$-category $\bD$ the pull-back along $L_{h,K}$ induces an equivalence  \begin{equation}\label{fqwewfedeqwd}  L_{h,K}^{*}:\Fun^{\lex}(L_{K}\nCalg_{h},\bD)\stackrel{\simeq}{\to} \Fun^{h,s,Sch}(\nCalg,\bD)\ .   \end{equation}If $\bD$ is in addition symmetric monoidal, then we have an equivalence 
\[L_{h,K}^{*}:\Fun^{\lex}_{\otimes/\lax}(L_{K}\nCalg_{h},\bD)\stackrel{\simeq}{\to} \Fun^{h,s,Sch}_{\otimes/\lax}(\nCalg,\bD)\ .\]

If $\bC$ is a pointed $\infty$-category with products and coproducts, then for any two objects $C$ and $C'$ in $\bC$ we have a canonical morphism
\[(c\mapsto (c,0))\sqcup (c'\mapsto (0,c')):C\sqcup C'\to C\times C'\ .\] A pointed $\infty$-category $\bC$ with products and coproducts  
is called {\em semi-additive} if this canonical map is an equivalence for every
two objects $C$ and $C'$, see \cite[Def. 6.1.6.13]{HA}. 

If $\bC$ is semi-additive, then its mapping spaces $\Map_{\bC}(C,D)$ have canonical refinements to commutative monoids in $\Spc$. In particular, every object of $\bC$ naturally becomes a commutative monoid and a commutative comonoid object in $\bC$ \cite[Rem.6.1.6.14]{HA}.

Note that $L_{K}\nCalg_{h}$ is pointed and admits products and coproducts by \cref{wreogkpwegrewferf}.
\begin{prop}\label{weotgpwergferferfwref}
The $\infty$-category $L_{K}\nCalg_{h}$ is semi-additive.
\end{prop}
\begin{proof}
We consider two $C^{*}$-algebras $A$ and $B$. We then have a canonical homomorphism
\[c:A*B\to A\oplus B\] induced via the universal property of the free product 
by  the homomorphisms  $A\to A\oplus B$, $a\mapsto (a,0)$ and $B\to A\oplus B$, $b\mapsto (0,b)$.
\begin{lem}[{\cite[Prop. 3.1]{MR899916},\cite[Prop. 5.3]{Meyer:aa}}]\label{weroigjowerfrefw}
$L_{h,K}(c):L_{h,K}(A*B)\to L_{h,K}(A\oplus B)$ is an equivalence in $L_{K}\nCalg_{h}$.
\end{lem}
A proof of \cref{weroigjowerfrefw} will be given below after the completion of the argument for \cref{weotgpwergferferfwref}.
We consider $C^{*}$-algebras $A$ and $B$. Then
the canonical map from the coproduct to the product of $L_{h,K}(A)$ and $L_{h,K}(B)$  in $L_{K}\nCalg_{h}$ has the following factorization
over equivalences:
\begin{eqnarray*}
L_{h,K}(A)\sqcup L_{h,K}(B)&\stackrel{\cref{wreogkpwegrewferf}.\ref{wetokgpwerferwfwref}}{\simeq}& L_{h,K}(A*B)\\&\stackrel{\cref{weroigjowerfrefw}}{\simeq}&L_{h,K}(A\oplus B)\\
&\stackrel{\cref{weoijgowefgrefwrefrf}}{\simeq}&L_{K}(L_{h}(A)\times L_{h}(B))\\
&\stackrel{\cref{wreogkpwegrewferf}.\ref{wkogpwtgwerfwrefwr}}{\simeq}
&L_{h,K}(A)\times L_{h,K}(B) \end{eqnarray*}
This finishes the proof of \cref{weotgpwergferferfwref} assuming \cref{weroigjowerfrefw}.
  \end{proof}

\begin{proof}[Proof of  \cref{weroigjowerfrefw}]
Since we need some details of the proof of    Lemma \ref{weroigjowerfrefw} later we recall the argument.
We first observe, using that $L_{h,K}\simeq L_{h,K}\circ \Mat_{2}(-)$, 
  that for any $C^{*}$-algebra $C$ the left-upper corner inclusion $C\to \Mat_{2}(C) $ is an equivalence in $L_{K}\nCalg_{h}$.
 
  In the following we consider the $C^{*}$-algebras $A$ and $B$ as subsets of $A*B$.
 We define a homomorphism
  \[f:A\oplus B\xrightarrow{(a,b)\mapsto (a,b) }(A*B)\oplus (A*B) \xrightarrow{\incl} \Mat_{2}(A*B)\ ,\quad f(a,b):=\left( \begin{array}{cc} a&0 \\0 &b \end{array} \right)\ . \] Then
  $f\circ c:A*B \to  \Mat_{2}(A*B)$ is determined by
\[a\mapsto \left( \begin{array}{cc} a&0 \\0 &0 \end{array} \right) \ , \quad b\mapsto \left( \begin{array}{cc} 0&0 \\0 &b \end{array} \right)\ .\]
We consider the family of  is unitaries  \begin{equation}\label{qrfqfoiue9foqwefqwedwqewdqewd}U_{t}:=\left( \begin{array}{cc}\cos(\frac{\pi t}{2}) &-\sin(\frac{\pi t}{2}) \\ \sin(\frac{\pi t}{2})&\cos(\frac{\pi t}{2}) \end{array} \right) \end{equation} 
  in $ \Mat_{2}(M(A*B))$
  and define $h_{t}:A*B \to  \Mat_{2}(A*B)$  by
  \[h_{t}(a):=\left( \begin{array}{cc} a&0 \\0 &0 \end{array} \right) \ , \quad h_{t}(b)=U_{t}^{*}\left( \begin{array}{cc} 0&0 \\0 &b \end{array} \right)U_{t}
  \ .\]
  Then $h_{0}=f\circ c$ and 
$h_{1} $ is  an upper corner inclusion.
We have 
\[\Mat_{2}(c)\circ f:A\oplus  B\to \Mat_{2}(A\oplus  B) \ , \quad (a,b)\mapsto  \left( \begin{array}{cc} (a,0)&0 \\0 &(0,b) \end{array} \right)\ .\]
We define the homotopy $h_{t}:A\oplus  B\to \Mat_{2}(A\oplus  B) $ by 
\[h_{t}(a,b):= \left( \begin{array}{cc} (a,b\sin^{2}(\frac{\pi t}{2}))&(0,b\sin(\frac{\pi t}{2})\cos(\frac{\pi t}{2})) \\ (0,b\sin(\frac{\pi t}{2})\cos(\frac{\pi t}{2})) &(0,b\cos^{2}( \frac{\pi t}{2} )) \end{array} \right)\ .\]
 Then $h_{0}=\Mat_{2}(c)\circ f$ and 
  $h_{1}$   is  a upper corner inclusion.
\end{proof}

\begin{rem}
For two $C^{*}$-algebras $A$ and $B$ the mapping space $\Map_{L_{K}\nCalg_{h}}(A,B)$ in the semi-additive $\infty$-category $L_{K}\nCalg_{h}$  is a  commutative monoid in $\Spc$.
 One is often interested
in its group completion. In this case the observation
\cref{erijgowergfrefwrfwref} might be helpful.
\hB
\end{rem}

\cref{weroigjowerfrefw} can be generalized to countable  coproducts and sums as follows.
Recall that 
  $L_{K}\nCalg_{h}$ admits small  coproducts by \cref{wreogkpwegrewferf}.\ref{wetokgpwerferwfwref}. Furthermore,
  for a small  family  $(A_{i})_{i\in I}$   of $C^{*}$-algebras we can also form the sum \begin{equation}\label{fqoiwjefioqwejfoiqwjioewdjq}\bigoplus_{i\in I}A_{i}:= \colim_{F\subseteq I\ ,|F|<\infty} \bigoplus_{i\in F}A_{i}
\end{equation} 
  which for infinite $I$ should not be confused with the coproduct or product in $\nCalg$.
 We still  have a canonical map $c:*_{i\in I}A_{i}\to \bigoplus_{i\in I}A_{i}$. 
\begin{prop}\label{ergoijfofefdewdqed}If $I$ is countable, then 
  the canonical map induces an equivalence
$$L_{h,K}(c):L_{h,K}(*_{i\in I}A_{i})\to L_{h,K}(\bigoplus_{i\in I}A_{i})\ .$$
\end{prop}
\begin{proof} 
 For finite $I$ this follows by a finite induction from \cref{weroigjowerfrefw}.
We now assume that 
 $I=\nat$.
 We define the map \[f:\bigoplus_{i\in I}A_{i}\to K\otimes *_{i\in I}A_{i}\ , \quad f((a_{i})_{i}):= \diag(a_{0},a_{1},a_{2},\dots )\ .\]
Note that  $\lim_{i\to \infty}\|a_{i}\|=0$  so that 
this diagonal matrix really belongs to $K\otimes *_{i\in I}A_{i}$.

 The composition $f\circ c:*_{i\in I}A_{i}\to K\otimes *_{i\in I}A_{i}$ is determined by 
$a_{i}\mapsto \diag(0,\dots 0,a_{i},0\dots )$ with $a_{i}$ at the $i$th place.
We define 
a homotopy $h_{t}:*_{i\in I}A_{i}\to K\otimes *_{i\in I}A_{i}$ such that for 
 $t\in [1-\frac{1}{i+1},1-\frac{1}{i+2}]$ it rotates in the coordinates $0$ and $i$   as in the proof of \cref{weroigjowerfrefw}.
Then  $h_{t}$ is continuous  as a map $[0,1]\to \underline{\Hom} (*_{i\in I}A_{i},K\otimes *_{i\in I}A_{i})$. Indeed, 
$h_{t}$ is    continuous on the subalgebras $*_{i=1}^{n}A_{i}$ for all $n$ in $\nat$. Since 
 their union is dense and 
 $h_{t}$ is uniformly bounded we conclude  continuity.
We have 
$h_{0}=f\circ c$ and 
$h_{1}$ is a  left upper corner embedding. This implies that 
$L_{h,K}(f)\circ L_{h,K}(c)\simeq L_{h,K}(\id_{ *_{i\in I}A_{i}})$.

 The composition
 $(\id_{K}\otimes c)\circ  f: \bigoplus_{i\in I}A_{i}\to  K\otimes \bigoplus_{i\in I}A_{i}$ is determined by 
\[(a_{i})_{i}\mapsto \diag ((a_{0},0,\dots),(0,a_{1},0,\dots),(0,0,a_{2},\dots),\dots)\ .\]
We 
define the  homotopy $l_{t}: \bigoplus_{i\in I}A_{i}\to  K\otimes \bigoplus_{i\in I}A_{i}$
such that 
 for $t\in [1-\frac{1}{i+1},1-\frac{1}{i+2}]$ it rotates in the coordinates $(0,i)$ and $(i,i)$   as in the proof of \cref{weroigjowerfrefw}.
Then 
 $l_{t}$ is continuous as a map $[0,1]\to \underline{\Hom}( \bigoplus_{i\in I}A_{i},L_{K}(\bigoplus_{i\in I}A_{i})$.
 Indeed,  
$l_{t}$ is   continuous on the subalgebras $\bigoplus_{i=1}^{n}A_{i}$ for all $n$  in $\nat$. Since their union 
  is dense and 
$l_{t}$ is uniformly bounded we can conclude  continuity.
We have  
$l_{0}=(\id_{K}\otimes c)\circ f$
and 
$l_{1}$ is a   left upper corner embedding.
This implies that 
 $L_{h,K}(c)\circ L_{h,K}(f)\simeq L_{h,K}(\id_{ \bigoplus_{i\in I}A_{i}})$.
\end{proof}

\begin{ex} \label{poregjopewrgerg}
For any $C^{*}$-algebra $A$ we let $\Proj(A)$ denote the topological space of projections in $A$. The functor $\Proj(-):\nCalg\to \Top$  is corepresented by 
the $C^{*}$-algebra $\C$. Indeed, for  a $C^{*}$-algebra   $A$   we have a homeomorphism  \[\underline{\Hom} (\C,A)\stackrel{\cong}{\to} \Proj(A)\ , \quad  f\mapsto f(1)\ . \]  
We define the topological space of stable projections in $A$ by 
\[\Proj^{s}(A):=\underline{\Hom}(\C,K\otimes A)\]
which becomes an $H$-space with respect to the block sum operations. 
Using the semi-additivity of $L_{K}\nCalg$ its underlying space  \[\cProj^{s}(A):=\ell \Proj^{s}(A)\stackrel{\eqref{vwihviuewhviwevewre}}{\simeq} \Map_{L_{K}\nCalg_{h}}(\C,A)\]  has a refinement to an object of   $\CMon(\Spc)$, i.e., a commutative monoid object in spaces. It will be called the monoid  of {\em stable projections} in $A$.

We will see in \cref{erijgoerwgregffwfrefw}, that for unital $A$ the group completion of the
 commutative monoid $\cProj^{s}(A)$ is equivalent to the $K$-theory space of $A$.
\hB
\end{ex}

\begin{ex}\label{poregjopewrgerg1}
 For any unital $C^{*}$-algebra $A$ we let  $U(A)$ denote the topological group of unitaries in $A$.
The functor $U:\Calg\to \Groups(\Top)$ is corepresented by 
the $C^{*}$-algebra $C(S^{1})$. Let $u:S^{1}\to \C$ be the inclusion considered as an element of
$U(C(S^{1}))$. Then we have an isomorphism
\[\underline{\Hom_{u}}(C(S^{1}),A)\xrightarrow{\cong} U(A)\ , \quad f\mapsto f(u)\]   of topological groups, 
where the subscript $u$ indicates that we consider the subspace of $\underline{\Hom}(C(S^{1}),A)$ of unit-preserving homomorphisms. 

Using the unitalization functor  $(-)^{u}:\nCalg\to \Calg$  we  define the functor
\begin{equation}\label{urfuherufhwerjkfekrf89}U^{s}:\nCalg\to \Groups(\Top)\ , \quad A\mapsto \{U\in U((K\otimes A)^{u})  \mid U-1_{(K\otimes A)^{u}}\in K\otimes A\} 
\end{equation} which associates to $A$ the topological group of {\em stable unitaries}. 
%
The stable unitaries functor is corepresented by the suspension $S(\C)$ of $\C$. Indeed,
 using the  split-exact sequence
 \begin{equation}\label{fwqefwfedwdwedwdwdfwfwfwf} 0\to S(\C)\xrightarrow{i} C(S^{1})\xrightarrow{\ev_{1}} \C\to 0\end{equation} 
 we identify $C(S^{1})$ with the unitalization of $S(\C)$.
 Then the unitalization functor    induces an identification
   \begin{equation}\label{qefeoijoi4r34r34}\underline{\Hom}(S(\C),K\otimes A) \cong U^{s}(A)\ .
\end{equation}
    By 
     \eqref{vwihviuewhviwevewre} we have an equivalence of spaces
 \begin{equation}\label{vfevsdfvrqfccscsdc}\cU^{s}(A):=\Map_{L_{K}\nCalg_{h}}(S(\C),A)\simeq \ell U^{s}(A)
\end{equation} 
 which equips the commutative monoid on the left with a second group structure.
These two structures distribute and therefore coincide by an Eckmann-Hilton  type argument.
It follows that the commutative  monoid $\cU^{s}(A)$ is already an object of $\CGroups(\Spc)$, i.e., a commutative group object in spaces. It will be called the  {\em group of stable unitaries}. 

Unfolding definitions we see that the monoid structure on $\cU^{s}(A)$ comes from the block-sum of unitaries in $(K\otimes A)^{u}$ while the other group structures is  given by the multiplication of unitaries.
In \cref{trgjkowergferfweferfw} we will show that $\cU^{s}(A)$ is equivalent to a one-fold delooping of the $K$-theory space of $A$.
\hB
\end{ex}

\begin{ex}
\label{ewfiuhgwqeuifdewdewqdqewdqd}
We consider two $C^{*}$-algebras $A$ and $B$ and homomorphisms $f,f':A\to B$.
We say that $\im(f)\perp \im(f')$ if $f(a)f'(a')=0=f'(a')f(a)$ for all $a,a'$ in $A$.
In this case we can define a homomorphism \[f+f':A\to B\ , \quad a\mapsto f(a)+ f(a') \ .\]
We then have an equivalence
\[L_{h,K}(f)+L_{h,K}(f')\simeq L_{h,K}(f+f')\] in $\Map_{L_{K}\nCalg_{h}}(A,B)$, where the sum on the left is the monoid structure on the mapping space coming from the semi-additivity of $L_{K}\nCalg_{h}$ (see  \cref{erkgowegergrewfwerf9}).
Indeed, this sum  is represented by  \[ A  \xrightarrow{ \diag(f,f') }  \Mat_{2}(B) \to K\otimes B  \ .\]
There is a rotation homoptopy $\phi_{t}$ from
$\diag(0,f')$ to $\diag(f',0)$ such that $\im(\phi_{t})\perp \im(\diag(f,0))$ for all $t$. Hence
$
\diag(f,0)+\phi_{t}$ is a homotopy from $\diag(f,f')$ to $\diag(f+f',0)$. This implies the assertion.
\hB
\end{ex}

Since $K$ is separable the functor $L_{K}$ restricts to \[L_{\sepa,K}:\nCalg_{\sepa,h}\to L_{K}\nCalg_{\sepa,h}\]
with an essentially small target.
Together with  \cref{weogpwergerfgwerfwerf9}, \cref{wreogkpwegrewferf} and \cref{weotgpwergferferfwref} this implies the following statements:

 \begin{kor}\label{wreogkpwegrewferf1}\mbox{}
 \begin{enumerate}
 \item \label{werigjweogwergwreg}We have a commutative square \[\xymatrix{\nCalg_{\sepa,h}\ar[r]\ar[d]^{L_{\sepa,K}} &\nCalg_{h} \ar[d]^{L_{K}} \\ L_{K} \nCalg_{\sepa,h}\ar[r] &L_{K}\nCalg_{h} }\]
of $\infty$-categories  where the  vertical arrows are left Bousfield localizations and the horizontal arrows are fully faithful.
 \item The square in \ref{werigjweogwergwreg} has a refinement to a diagram of symmetric monoidal $\infty$-categories and symmetric monoidal functors  for $\otimes_{?}$ with $?\in \{\min,\max\}$ such that $L_{\sepa,K}$ becomes a symmetric monoidal localization.
 \item \label{wkogpwtgwerfwrefwr1} The localization $L_{\sepa,K}$ is  left-exact.
 \item \label{wetokgpwerferwfwref1}
 $L_{K}\nCalg_{\sepa,h}$ admits countable coproducts and $L_{\sepa,K}$ preserves them. 
 \item The functor $-\otimes_{?}-:L_{K}\nCalg_{\sepa,h}\times L_{K}\nCalg_{\sepa,h}\to L_{K}\nCalg_{\sepa,h}$ is bi-left exact.
 \item $L_{K}\nCalg_{\sepa,h}$ is semi-additive.
 \end{enumerate}
  \end{kor}

The pull-back along $ L_{\sepa,h,K}:=L_{\sepa,K}\circ L_{\sepa,h}$
 induces for every
$\infty$-category $\bD$ an equivalence
\[L_{\sepa,h,K}^{*}:\Fun(L_{K}\nCalg_{\sepa,h},\bD)\stackrel{\simeq}{\to}\Fun^{h,s}(\nCalg_{\sepa},\bD)\ .\]
For any symmetric monoidal $\infty$-category the pull-back along the symmetric monoidal refinement of $L_{\sepa,h,K} $
induces an equivalence
\[L_{\sepa,h,K}^{*}:\Fun_{\otimes/\lax} (L_{K}\nCalg_{\sepa,h},\bD)\stackrel{\simeq}{\to} \Fun_{\otimes/\lax}^{ h,s}(\nCalg_{\sepa},\bD)\ .\]
For every left-exact $\infty$-category $\bD$ the pull-back along $L_{\sepa,h,K}$ induces an equivalence 
\[L_{\sepa,h,K}^{*}:\Fun^{\lex}(L_{K}\nCalg_{\sepa,h},\bD)\stackrel{\simeq}{\to} \Fun^{h,s,Sch}(\nCalg_{\sepa},\bD)\ .\]
If $\bD$ is in addition symmetric monoidal, then we have the equivalence 
 \[L^{*}_{ \sepa,h,K}:\Fun_{\otimes/\lax}^{\lex}(L_{K}\nCalg_{\sepa,h},\bD)\stackrel{\simeq}{\to} \Fun_{\otimes/\lax}^{h,s,Sch}(\nCalg_{\sepa},\bD)\ .\]

\section{Forcing exactness}\label{wegijowerferferwfrfwdvdfvdf}

In this subsection we describe  left-exact  Dwyer-Kan localizations
\[L_{!}:L_{K}\nCalg_{h}\to L_{K}\nCalg_{h,!} \] for $!$ in $\{\splt, \se,\exa\}$
designed such that the composition 
\[L_{h,K,!}:=L_{!}\circ L_{h,K}:\nCalg\to L_{K}\nCalg_{h,!} \] (see \eqref{qwefdwedqdwed} for $L_{h,K}$)
sends exact (for $!=\exa$) or semi-split exact (for $!=\se$) or split exact (for $!=\splt$) sequences of $C^{*}$-algebras to fibre sequences. 
We further analyse the compatibility of $L_{!}$ with the symmetric monoidal structures.

Given an exact   exact sequence   $0\to A\to B\xrightarrow{f}C\to 0$ of $C^{*}$-algebras, using the mapping cylinder construction described in \cref{werigowrgerfewrfwerfwerf} 
we can produce  
 a diagram of exact sequences of $C^{*}$-algebras
\begin{equation}\label{tbrtbbgdbfert} \xymatrix{&0\ar[d]&&&\\0\ar[r]  &A\ar[r]\ar[d]^{\iota_{f}} &B\ar[r]^{f}\ar[d]^{h_{f}} &C\ar[r]\ar@{=}[d] \ar@/^-1cm/@{..>}[l]^{s}&0 \\ 0\ar[r]  &C(f)\ar[r]\ar[d]^{\pi_{f}} &Z(f)\ar[r]^{\tilde f}\ar[r] &C\ar[r]  \ar@/^1cm/@{..>}[l]^{\tilde s}  &0\\&Q(f)\ar@/^1cm/@{..>}[u]^{\hat s}\ar[d]\\&0&& }\end{equation}
We write    elements of the mapping cone $C(f)$ of $f$ as pairs $(b,\gamma)$ with $b$ in $B$ and $\gamma$ in $PC$ (see \eqref{vadfvasdvscsdcsdac}) such that  $f(b)=\gamma(0)$ and $\gamma(1)=0$ and consider $A$ as a subset of $B$. With this notation  the map $\iota_{f}$ is given by   $\iota_{f}(a):=(a,0)$. This map is the inclusion of an ideal and the $C^{*}$-algebra 
   $Q(f)$ is defined as the quotient. We have an isomorphism 
   \[Q(f)\cong \{(c,\gamma)\in C\times PC\mid  \gamma(0)=c\ , \gamma(1)=0\}\cong C_{0}([0,1))\otimes C\] which implies that
 $Q$ is contractible.
If $f$ admits a cpc split $s$ (or a split), then we get induced cpc splits (or splits) 
  $\tilde s(c):=(s(c),\const_{c})$
and     $\hat s(c,\gamma)=(s(c),\gamma)$ as indicated.

We consider 
a functor $F:\nCalg\to \bC$ to a pointed $\infty$-category. 
Versions of the  following observation were  basic to the approach in  \cite{MR1068250}, \cite{zbMATH03973625}.
\begin{prop}\label{wetigowergferferferwfw}
We assume that   $F$ is homotopy invariant and
  reduced,
 and that it sends mapping cone sequences   to fibre sequences.
 Then the following statements are equivalent:
  \begin{enumerate}
  \item \label{wergrewfweref}
$F$ is exact (semiexact, or split-exact,  respectively).
\item  \label{wergrewfweref1}
 For every exact (sem-split exact, or split exact,  respectively) sequence of $C^{*}$-algebras
$0\to I\xrightarrow{i} A\xrightarrow{\pi} Q\to 0$ with $Q$ contractible the map $F(i)$ is an equivalence.\end{enumerate}
\end{prop}
\begin{proof}\mbox{}
We assume that 
 $F$ satisfies \ref{wergrewfweref}.
Then $F$ sends the exact (semi-split exact, or split exact) sequence  $0\to I\xrightarrow{i} A\xrightarrow{\pi} Q\to 0$  to a fibre sequence
\[F(I)\to F(A)\to F(Q)\] with $F(Q)\simeq 0$.
Consequently, 
 $F(i)$ is an equivalence.

Conversely we assume that  $F$ satisfies \ref{wergrewfweref1}.  
We 
  consider a general  exact  (semi-split exact or split exact) sequence $0\to A\to B\xrightarrow{f}C\to 0$ and form the  
  diagram \eqref{tbrtbbgdbfert}. By assumption   
$F$ sends the lower horizontal  sequence to a fibre sequence and 
$F(\iota_{f})$ is an equivalence. 
By homotopy invariance  of $F$ also $F(h_{f})$ is an equivalence.
We can therefore conclude that 
 $F(A)\to F(B)\to F(C)$ is fibre sequence, too.
 \end{proof}

Recall  the  \cref{fuhquifhiuqwefeqwdewdewdcdfvca}  of Schochet exactness of a functor in terms of Schochet fibrant squares. 
By our experience,   Schochet exactness    is either obvious from the construction or very difficult to verify.  The following result shows that for stable target categories homotopy invariance and semiexactness together imply Schochet exactness. We consider  homotopy invariance and semiexactnes as properties which are much closer
 to the usual $C^{*}$-algebraic business.
 \begin{lem}\label{reijgoerfeffrfqrferfe}
A homotopy invariant and semiexact functor from $\nCalg$ to a stable $\infty$-category is Schochet exact.
  \end{lem}
 \begin{proof}
 Let $F:\nCalg\to \bD$ be a homotopy invariant and semiexact functor to a stable $\infty$-category. 
 Then $F$ is reduced. Since $\bD$ is stable,  in order to show Schochet exactness by \cref{fuhquifhiuqwefeqwdewdewdcdfvca1}  it suffices  
 to show 
 that $F$  sends any  Schochet exact sequence  $0\to A\to B\to C\to 0$ to a fibre sequence.
 
 We apply $F$ to the diagram  \eqref{tbrtbbgdbfert} and get a diagram
  \[\xymatrix{F(A)\ar[d]^{F(\iota_{f})}\ar[r]&F(B)\ar[d]^{F(h_{f})}\ar[r]&F(C)\ar@{=}[d]\\F(C(f))\ar[r]\ar[d]&F(Z(f))\ar[r]&F(C)\\F(Q(f))&&}\ .\]
  The middle horizontal line is a fibre sequence since $F$ is semiexact and the middle horizontal sequence in  \eqref{tbrtbbgdbfert} is semi-split exact.
  The functor $L_{h}$ sends both horizontal sequences in \eqref{tbrtbbgdbfert} to fibre sequences since they are Schochet exact. 
  Since $h_{f}$ is a homotopy equivalence we see that  $L_{h}(h_{f})$ and hence also $L_{h}(\iota_{f})$ are equivalences. Since $F$ is homotopy invariant it factorizes over $L_{h}$ and therefore $F(\iota_{f})$ is also an equivalence. We can now conclude that the upper horizontal sequence in the diagram above is a fibre sequence.
%
%
    \end{proof}

We consider exact sequences 
\[0\to I\xrightarrow{i}A  \xrightarrow{\pi}Q\to 0\]
and define the following sets of morphisms 
 in $\nCalg$:
\begin{eqnarray}\label{erijoglerfqwfwerfrefwrf}
 \hat W_{\splt}&:=&\{ i\mid \mbox{for all  split-exact sequences  with $Q$ contractible}\} \\
 \hat W_{\se}&:=&\{ i\mid \mbox{for all semi-split exact sequences  with $Q$ contractible}\} \nonumber\\
  \hat W_{\exa}&:=&\{i\mid \mbox{for all exact sequences  with $Q$ contractible}\} \nonumber
 \end{eqnarray}
 We will denote the closures  in $L_{K}\nCalg_{h}$  under equivalences of their images by $L_{h,K}$ by the same symbols.

\begin{rem}A natural idea would be to form the Dwyer-Kan localizations of $L_{K}\nCalg_{h}$ at the sets
$\hat W_{!}$ defined above. But there is no reason that these localizations are left-exact. In order to produce   left-exact
localizations we must localize at the closures of the above sets under pull-backs and the $2$-out-of-$3$ property.
\hB
\end{rem}

In the following a set $W$ of morphisms in an $\infty$-category $\bC$ is  always assumed to be closed under equivalences.
The set  $W$    
is {\em closed under pull-backs} if for every cartesian square
\[\xymatrix{A\ar[r]\ar[d]^{g} & B\ar[d]^{f} \\D \ar[r] &C } \] in $\bC$ with $f$  in $W$ also $g$ belongs to $W$.

The set $W$  has the {\em $2$-out-of $3$-property} if the fact that two out of $f,g,g\circ f$
belong to $W$ implies that the third also belongs to  $W$.

\begin{ex}
If $F:\bC\to \bD$ is a functor between $\infty$-categories, then the set $W$ of morphisms in $\bC$ which are
 sent by $F$ to equivalences  has the  $2$-out-of $3$-property. If $F$ is a left-exact functor between left-exact $\infty$-categories, then  $W$ is also closed under pull-backs.  \hB
\end{ex}

Let    
  $\bC$ be a left-exact $\infty$-category with a set of morphisms $W$.
  \begin{lem}\label{qerighqofqwefqew9}
  If $W$ has the $2$-out-of $3$-property and is closed under pull-backs, then the Dwyer-Kan localization $L:\bC\to \bC[W^{-1}]$ is left-exact.
  \end{lem}
\begin{proof}The triple $(\bC, W,\bC)$ is a category of fibrant objects in the sense of \cite[Def. 7.4.12 and Def. 7.5.7]{Cisinski:2017}. The Assertion  \ref{wergiuhweriugrewf}  therefore follows from \cite[Prop. 7.5.6]{Cisinski:2017}
\end{proof}
 
 Let    
  $\bC$ be a left-exact $\infty$-category with a set of morphisms $\hat W$, and let $W$ be the minimal 
subset of morphisms containing $\hat W$ which has the $2$-out-of $3$-property and is closed under pull-backs. Then for every left-exact $\infty$-category $\bD$   the canonical inclusion    \begin{equation}\label{erfwerferfwef} 
\Fun^{\lex,W}( \bC ,\bD) \stackrel{\simeq}{\to} \Fun^{\lex,  \hat W}(\bC ,\bD)
\end{equation}
is an equivalence.

%
%

For $!$ in $\{\exa,\se,\splt\}$ we define
$W_{!}$ as the smallest  set of morphisms in $L_{K}\nCalg_{h}$ which is closed under  pull-backs and has the $2$-out-of $3$-property, and which contains $\hat W_{!}$ from \eqref{erijoglerfqwfwerfrefwrf}.  
\begin{ddd} \label{wegjwoergerwfrferfrwefwf}We  define  the Dwyer-Kan localization
\[L_{!}:L_{K}\nCalg_{h}\to L_{K}\nCalg_{h,!} \]
at the set $W_{!}$.
\end{ddd}
Note that by construction $ L_{K}\nCalg_{h,!}$ is a large $\infty$-category. As we localize at a large set of morphisms we lose the property of being locally small at this point.

We  define the functor
\[L_{h,K,!}:\nCalg\xrightarrow{L_{h}}\nCalg_{h}\xrightarrow{L_{K}}L_{K}\nCalg_{h}\xrightarrow{L_{!}}L_{K}\nCalg_{h,!}\ .\]

We let $\Fun^{h,s,!+Sch}(\nCalg,\bD)$ denote the full subcategory of homotopy invariant and stable functors which are  Schochet exact and in addition
exact for $!=\exa$ (semiexact for $!=\se$ or split exact for $!=\splt$, respectively).  
 \begin{prop}\label{wreoigwtegwtgw9}
\mbox{}
\begin{enumerate}
\item\label{wergiuhweriugrewf}  The localization $L_{!}$ is   left-exact.
 
\item\label{wergiuhweriugrewf2} Pull-back along $L_{h,K,!}$ induces for every left-exact $\infty$-category $\bD$ an equivalence \begin{equation}\label{ewfwqedqewdedqwd}L_{h,K,! }^{*}:\Fun^{\lex}(L_{K}\nCalg_{h,! } ,\bD)\stackrel{\simeq}{\to} \Fun^{h,s,!+Sch}(\nCalg,\bD)\ .
\end{equation}
\item\label{wergiuhweriugrewf1} $L_{K}\nCalg_{h,!} $  is semi-additive and $L_{!}$ preserves coproducts.
\end{enumerate}
\end{prop}
\begin{proof}
The Assertion  \ref{wergiuhweriugrewf} follows from \cref{qerighqofqwefqew9}.

We next show Assertion \ref{wergiuhweriugrewf2}.
For any $\infty$-category $\bD$ we have restriction functors
\[   \Fun (L_{K}\nCalg_{h,!} ,\bD)\stackrel{L_{!}^{*}}{\simeq} 
  \Fun^{ W_{!}}(L_{K}\nCalg_{h} ,\bD)\stackrel{L_{h,K}^{*}}{\to }
\Fun^{h,s}(\nCalg ,\bD)\ ,\]
where the first is an equivalence by the universal property of the Dwyer-Kan localization  $L_{!}$.
The second functor is 
 the  restriction of the equivalence  \eqref{qwfdqedewdqwedd} and hence fully faithful.
  If $\bD$ is left-exact, then by \ref{wergiuhweriugrewf}   
 the first functor  restricts to the first equivalence in
\begin{equation}\label{bdsiobjoibsbsfbsfvsfdv}\Fun^{\lex} (L_{K}\nCalg_{h,!} ,\bD)\stackrel{L_{!}^{*}}{\simeq} 
  \Fun^{\lex, W_{!} }(L_{K}\nCalg_{h} ,\bD)\stackrel{\eqref{erfwerferfwef}}{\simeq}   \Fun^{ \lex,\hat W_{!} }(L_{K}\nCalg_{h} ,\bD)\ . 
\end{equation}  
  
  Finally, by \cref{wetigowergferferferwfw} the equivalence \eqref{fqwewfedeqwd}  restricts to   an equivalence 
  \[ \Fun^{ \lex,\hat W_{!}, }(L_{K}\nCalg_{h} ,\bD)\stackrel{L_{h,K}^{*}}{\simeq}  \Fun^{h,s,!+Sch}(\nCalg ,\bD)\ .\]
  The composition of these equivalences gives \eqref{ewfwqedqewdedqwd}.  
  
  Assertion \ref{wergiuhweriugrewf1} is a general fact about left-exact localizations of left-exact $\infty$-categories which are in addition semi-additive.
\end{proof}

Next we consider the symmetric monoidal structures. We allow the following combinations of $!$ and $?$:
\bigskip

\centerline{\begin{tabular}{c|c|c} ! $\backslash$ ?&min&max\\ \hline 
splt&yes&yes\\
se&yes&yes\\
ex&no&yes
\end{tabular}}
The combination $(\exa,\min)$ is excluded since the minimal tensor product does not preserve exact sequences.

\begin{prop}\label{ewrogpwergfrefwrefwref}\mbox{} \begin{enumerate} 
\item \label{wrthgioweiogiowerg} The localization $L_{!} $ has a symmetric monoidal refinement. 

\item\label{wrthgioweiogiowerg1}  The functor $-\otimes_{?}-:L_{K}\nCalg_{h,!}  \times L_{K}\nCalg_{h,!} \to L_{K}\nCalg_{h,!} $ is bi-left exact.
\item \label{wrthgioweiogiowerg2}  Pull-back along $L_{h,K,!}$ induces for every left-exact $\infty$-category an equivalence \begin{equation}\label{ewfwqedqewdedqwd1}L_{h,K,!}^{*}:\Fun_{\otimes/\lax}^{\lex}(L_{K}\nCalg_{h,!} ,\bD)\stackrel{\simeq}{\to} \Fun^{h,s,!+Sch}_{\otimes/\lax}(\nCalg,\bD)
\end{equation}
for any symmetric monoidal and left-exact $\infty$-category $\bD$.
\end{enumerate}
\end{prop}
\begin{proof}
We first observe that the functor \[A\otimes_{?}-:L_{K}\nCalg_{h}\to L_{K}\nCalg_{h}\] preserves the set $\hat W_{!}$ defined in \eqref{erijoglerfqwfwerfrefwrf}.
Indeed if $!=\exa$ and $?=\max$, then we use that this functor preserves exact sequences and
contractible objects.  It is at this point where we must exclude the combination $!=\exa$ and $?=\min$.

If $!$ is in $\{\se,\splt\}$, then  this functor preserves semi-split exact or split exact sequences and
contractible objects for both $?=\min$ and $?=\max$.

Let $\tilde W_{!}$ be set of morphisms in $L_{K}\nCalg_{h}$ which are sent to equivalences by $L_{!}$.
Then for every left-exact  $\infty$-category $\bD$ we have an equivalence
\begin{equation}\label{vsdfvdfsvfrfes}\Fun^{\lex, \hat W_{!}}(L_{K}\nCalg_{h},\bD)\stackrel{\eqref{erfwerferfwef}}{\simeq}   \Fun^{\lex, W_{!}}(L_{K}\nCalg_{h},\bD)\simeq  \Fun^{\lex, \tilde W_{!}}(L_{K}\nCalg_{h},\bD)\ .
\end{equation}   By \cref{wreogkpwegrewferf}.\ref{ierjogoergegsegs9} and \cref{wreoigwtegwtgw9}.\ref{wergiuhweriugrewf} the composition  \begin{equation}\label{gqerfqqewdqewdwedq}L_{!}\circ (A\otimes_{?}-): L_{K}\nCalg_{h}\to L_{K}\nCalg_{h,!} 
\end{equation}
is left-exact. By the discussion above it inverts $\hat W_{!}$.  It  then follows from \eqref{vsdfvdfsvfrfes} that $A\otimes_{?}-$ preserves the set $\tilde W_{!}$.  Since $L_{!}$ is also the Dwyer-Kan localization of $L_{K}\nCalg_{h}$ at $\tilde W_{!}$
we conclude that the localization  $L_{!}$ has a symmetric monoidal refinement, hence Assertion \ref{wrthgioweiogiowerg}.

For Assertion \ref{wrthgioweiogiowerg1} we note that
the induced functor  \[A\otimes_{?}-:L_{K}\nCalg_{h,!} \to L_{K}\nCalg_{h,!}\] is  left exact since
it is the preimage of \eqref{gqerfqqewdqewdwedq} under the equivalence 
\[L_{!}^{*}:\Fun^{\lex}( L_{K}\nCalg_{h,!} , L_{K}\nCalg_{h,!} )\stackrel{\simeq}{\to} \Fun^{\lex,W_{!}}( L_{K}\nCalg_{h} , L_{K}\nCalg_{h,!})\]
 (see \cref{ewrogpwergfrefwrefwref}.\ref{wergiuhweriugrewf}).

We finally show Assertion \ref{wrthgioweiogiowerg2}.
Assertion  \ref{wrthgioweiogiowerg} implies the first equivalence in 
 \[\Fun_{\otimes/\lax} (L_{K}\nCalg_{h,!},\bD)\stackrel{L_{!}^{*}}{\simeq} \Fun_{\otimes/\lax}^{W_{!}}(L_{K}\nCalg_{h},\bD)\to  \Fun^{h,s}_{\otimes/\lax}(\nCalg,\bD)\]
 whose second arrow is a restriction of the equivalence \eqref{wregwefrefewrfrefw} and  hence fully faithful.
 We now restrict the domain to (lax) symmetric monoidal functors which are in addition left-exact and get 
 \[\Fun_{\otimes/\lax}^{\lex}(L_{K}\nCalg_{h,!} ,\bD)\stackrel{L_{!}^{*}}{\simeq} \Fun_{\otimes/\lax}^{W_{!},\lex}(L_{K}\nCalg_{h},\bD)\stackrel{\simeq}{\to}   \Fun^{h,s,!+Sch}_{\otimes/\lax}(\nCalg,\bD)\ .\]
 The second arrow indeed takes values in the indicated subcategory by \cref{wreoigwtegwtgw9}.\ref{wergiuhweriugrewf2}.
 In order to see that  it is essential surjective consider a functor $F$ in  $\Fun^{h,s,!+Sch}_{\otimes/\lax}(\nCalg,\bD)$.
 Then  by \eqref{wregwefrefewrfrefw} there is an essentially unique (lax) symmetric monoidal functor $\tilde F:L_{K}\nCalg_{h}\to \bD$ such that $L_{h,K}^{*}\tilde F\simeq F$. By \cref{wreoigwtegwtgw9}.\ref{wergiuhweriugrewf2} the functor $\tilde F$ 
 is left-exact and inverts $W_{!}$. Thus $\tilde F$ belongs to $  \Fun_{\otimes/\lax}^{W_{!},\lex}(L_{K}\nCalg_{h},\bD)$.
   \end{proof}

The constructions above have versions for the category of separable $C^{*}$-algebras. We let $\hat W_{\sepa,!}$ denote the  analog of  $\hat W_{!} $ from \eqref{erijoglerfqwfwerfrefwrf} for separable algebras and $W_{\sepa,!}$ be the smallest subset of morphisms containing $\hat W_{\sepa,!}$ which has the $2$-out-of-$3$-property and is closed under pull-backs.
For $!$ in $\{\exa,\se,\splt\}$ we define the Dwyer-Kan localization
\begin{equation}\label{vervweriuheiucececwewce}L_{\sepa,!}:L_{K}\nCalg_{\sepa,h}\to L_{K}\nCalg_{\sepa,h,!} 
\end{equation} 
at $W_{\sepa,!}$.
Since $L_{K}\nCalg_{\sepa,h}$ is locally small and essentially small,
the set of equivalence classes in $W_{\sepa,!}$ is small. This implies that $L_{K}\nCalg_{\sepa,h,!} $ is still
essentially small and locally small.

We define the composition\begin{equation}\label{qwefoifjhwoefcwdwqedqewdqewd}
L_{\sepa,h,K,!}:=L_{\sepa,!}\circ L_{\sepa,h,K}:\nCalg_{\sepa}\to L_{K}\nCalg_{\sepa,h,!} \ .
 \end{equation}
Then we have the following statements.
\begin{prop}\label{wreoigwtegwtgw9111}
\mbox{}
\begin{enumerate}
\item\label{wergiuhweriugrewf111}  The localization $L_{\sepa,!}$ is   left-exact.
\item\label{wergiuhweriugrewf1111} $L_{K}\nCalg_{\sepa,h,!} $  is semi-additive and $L_{\sepa,!}$ preserves finite coproducts.
\item\label{wergiuhweriugrewf2111} Pull-back along $L_{\sepa,h,K,!}$ induces for every left-exact $\infty$-category $\bD$ an equivalence \begin{equation}\label{ewfwqedqewdedqwd111}L_{\sepa,h,K,!}^{*}:\Fun^{\lex}(L_{K}\nCalg_{\sepa,h,!} ,\bD)\stackrel{\simeq}{\to} \Fun^{h,s,!+Sch}(\nCalg_{\sepa},\bD)\ .
\end{equation}

\item \label{wrthgioweiogiowerg111} The localization $L_{\sepa,!} $ has a symmetric monoidal refinement. 
\item We have a commutative square of symmetric monoidal functors
\begin{equation}\label{fdvsdvsdfvfsdvsdfv}\xymatrix{L_{K}\nCalg_{\sepa,h} \ar[r]\ar[d]^{L_{\sepa,!}}&L_{K}\nCalg_{h}\ar[d]^{L_{!}}\\L_{K}\nCalg_{\sepa,h,!} \ar[r]&L_{K}\nCalg_{h,!} }\ . \end{equation} 
\item\label{wrthgioweiogiowerg1111}  The functor $-\otimes_{?}-:L_{K}\nCalg_{\sepa,h,!} \times L_{K}\nCalg_{\sepa,h,!} \to L_{K}\nCalg_{\sepa,h,!} $ is bi-left exact.
\item \label{wrthgioweiogiowerg2111}  Pull-back along $L_{\sepa,h,K,!}$ induces for every left-exact $\infty$-category an equivalence \begin{equation}\label{ewfwqedqewdedqwd1111}L_{\sepa,h,K,!}^{*}:\Fun_{\otimes/\lax}^{\lex}(L_{K}\nCalg_{\sepa,h,!} ,\bD)\stackrel{\simeq}{\to} \Fun^{h,s,!+Sch}_{\otimes/\lax}(\nCalg_{\sepa},\bD)
\end{equation}
for any symmetric monoidal and left-exact $\infty$-category $\bD$.
\end{enumerate}
\end{prop}
\begin{rem} 
In contrast to \cref{wreogkpwegrewferf}.\ref{wetokgpwerferwfwref}
we do no not know whether $L_{K}\nCalg_{h,!} $ admits infinite coproducts.  

In contrast to the upper horizontal arrow in \eqref{fdvsdvsdfvfsdvsdfv} the lower horizontal arrow in this square is not known to be fully faithful, see \cref{qertgoqergoijeqwgwegwergw}.
\hB
\end{rem}
 \begin{rem}
 If the target category $\bD$ is stable and we consider $!$ in $\{\se,\exa\}$, then by \cref{reijgoerfeffrfqrferfe} (or its separable version) we could remove the superscripts $+Sch$ on the right-rand sides of \eqref{ewfwqedqewdedqwd}, \eqref{ewfwqedqewdedqwd1}, \eqref{ewfwqedqewdedqwd111}, and \eqref{ewfwqedqewdedqwd1111}. \hB
  \end{rem}

\section{Bott periodicity} 

In this section we analyse the Toeplitz extension 
\[0\to K\to \cT\to C(S^{1})\to 0\]
from the homotopy theoretic point of view.  The section is essentially an $\infty$-categorical version of \cite[Sec. 4]{MR750677}. The main result is \cref{erjigowergerrqf9}.
 
 We start with recalling some generalities on group objects in $\infty$-categories.
If $\bC$ is an $\infty$-category admitting cartesian products, then we can consider the  $\infty$-category of 
 commutative algebras in $\bC$ for the cartesian monoidal structure which will be called the $\infty$-category
of {\em  commutative  monoids}
$\CMon(\bC)$. Let $C$ be a  commutative  monoid with multiplication map $m:C\times C\to C$. It is called a {\em  commutative  group} if 
the shear map
\[C\times C\xrightarrow{(c,c')\mapsto (c,m(c,c'))}C\times C\] is an equivalence. 
We let $\CGroups(\bC)$ denote the full subcategory of $\CMon(\bC)$ of  commutative groups.
Note that we have used these notions already for $\bC=\Spc$.

Dually, if $\bC$ admits coproducts, then we can consider the  $\infty$-categories of 
  {\em  cocommutative  comonoids}
$\coCMon(\bC)$ and its full subcategory {\em  cocommutative  cogrouops} $\coCGroups(\bC)$.  
\begin{ex}\label{erkgowegergrewfwerf9}
In a semi-additive $\infty$-category  $\bC$   every object  is naturally a commutative monoid and a commutative comonoid. The 
 functors  forgetting the commutative monoid or  comonoid structures are equivalences:
\[\CMon(\bC)\stackrel{\simeq}{\to} \bC\stackrel{ \simeq}{\leftarrow} \coCMon(\bC)\ .\]
  Let $C$ be in an object of $\bC$. Then the multiplication and comultiplication maps of the corresponding monoid or comonoid  are given by
\[C\times C\stackrel{\simeq}{\leftarrow} C\sqcup C \xrightarrow{\codiag} C\ ,\quad C\xrightarrow{\diag} C\times C\stackrel{\simeq}{\leftarrow} C\sqcup C\ . \]
Moreover, the conditions of being a group or a cogroup are equivalent.

For any two objects $C,C'$ in $\bC$ the mapping space
 $\Map_{\bC}(C,C')$ has a  natural refinement to an object of $\CMon(\Spc)$. 
 The object $C'$ is a group if and only if $\Map_{\bC}(C,C')$ is a group for all objects $C$.  
  \hB\end{ex}

\begin{ex}\label{eriojgwegwerfrfw}
The last assertion in \eqref{erkgowegergrewfwerf9} reduces the verification of the group property for an object in a semi-additive category to  
the case of monoids in spaces. In this case we have a simple criterion.
An object $X$ in $\CMon(\Spc)$ is a group if an only if the monoid $\pi_{0}X$ is  a group. \hB
\end{ex}

  \begin{lem}\label{wgtojwregpokerpgergwefewferf}
  If $\bC$ is semi-additive and left-exact, then $\Omega:\bC\to \bC$ (see \cref{qerigqerfewfq9}) takes values in commutative groups.
  \end{lem}
  \begin{proof}
  Let $C'$ be an object of $\bC$. We must show that $\Omega C'$ is a group.  To this end we will show that $\Map_{\bC}(C,\Omega C')\simeq \Omega \Map_{\bC}(C,C')$ is a group in $\Spc$ for any object $C$ of $\bC$. 
  We now use \cref{eriojgwegwerfrfw} in order to reduce the problem  to the set of components.
 
  Note that  $\pi_{0}  \Omega \Map_{\bC}(C,C')\cong \pi_{1} \Map_{\bC}(C,C')$ clearly has a group structure $\sharp$ as a fundamental group. This structure distributes over the commutative monoid structure  $+$  
  on $   \pi_{0}\Map_{\bC}(C,\Omega C')$ coming from the semi-additivity in the sense that
  $(a+b)\sharp (c+d)=(a\sharp c)+(b\sharp d)$. The Eckmann-Hilton argument implies that both structures coincide. In particular, the commutative monoid structure $+$ is a commutative group structure.
  \end{proof}

To every unital $C^{*}$-algebra   $A$  we functorially associate the topological space    
\[I(A):=\{v\in A\mid v^{*}v=1_{A}\] of isometries in $A$. 
By definition, the   Toeplitz algebra $\cT$ is the  isometries classifier  in $\Calg$. It
  contains a universal  isometry $v$ such that  
$$\underline{\Hom}_{u} (\cT,A)\stackrel{\cong}{\to} I(A)\ , \quad f\mapsto f(v)$$
is a homeomorphism for every unital $C^{*}$-algebra $A$.

Recall from \cref{poregjopewrgerg1}  that $C(S^{1})$ is the unitaries classifier in $\Calg$ with the universal unitary $u$.
Since  unitaries are in particular isometries   we have a canonical unital homomorphism
$$\pi:\cT\to C(S^{1})\ ,  \quad \pi(v)=u\ .$$ Since $C(S^{1})$ is generated by $u$, the homomorphism $\pi$
 is  surjective. We let $K$ denote the  kernel of  $\pi$.   We thus have 
 the {\em Toeplitz exact sequence}\begin{equation}\label{gregeerfwffregfergewg} 0\to K\to  \cT \xrightarrow{\pi} C(S^{1})\to 0\ .\end{equation}
It is known that $\cT$ is separable and nuclear.
The projection $e:=1_{\cT}-vv^{*}$ belongs to $K$, and 
the algebra $K$ is generated by the family of minimal pairwise orthogonal projections   $(v^{n}e
v^{*,n})_{n\in \nat}$. This provides an identification of $K$ with the algebra of compact operators on a separable Hilbert space and justifies the notation. Note that $e$ is 
  a minimal projection in $K$.

 Using the universal property of $\cT$ and the unit we define homomorphisms
$$q:\cT\to \C\ , \quad v\mapsto 1\ , \qquad  j:\C\to \cT \ , \quad 1\mapsto 1_{\cT}\ .$$
 We consider a functor 
$F:\nCalg \to \bC$  to a  semi-additive $\infty$-category.
 
 \begin{prop}\label{qerkghqeriffewq}
 If   $F$  is homotopy invariant, stable,  split-exact   and takes values in 
    group objects,
  then $F(j)$ and $F(q)$ are mutually inverse equivalences.
\end{prop}
\begin{proof} The equality 
 $j\circ q=\id_{\C}$
 implies that $F(j)\circ F(q)\simeq \id_{F(\C)}$. It remains to show that 
 $F(q)\circ F(j)\simeq \id_{F(\cT)}$. To this end we construct the following  diagram of $C^{*}$-algebras:
 \[\xymatrix{&\cT\ar[d]^{\kappa}&&&&\\0\ar[r]&\ar@{=}[d]K\otimes \cT\ar[r]^{\iota}&\ar[d]^{r}\bar \cT\ar[r]^{p}&\cT\ar[r]\ar@{-->}[dl]_{\alpha,\phi_{t}}\ar[d]^{\tau }\ar@/^-1cm/@{..>}[l]^{s,\psi_{t}}\ar@/^-1cm/@{-->}[llu]^{j\circ q , \id_{\cT}}&0\\0\ar[r]&K\otimes \cT\ar[r]&\cT\otimes \cT\ar[r]^{\pi\otimes \id_{\cT}}&C(S^{1})\otimes \cT\ar[r]&0}\ .\]
 The lower horizontal  sequence is the tensor product of the Toeplitz sequence with $\cT$.
 The algebra $\bar \cT$ is defined such that the right square is a pull-back. This determines the homomorphisms $r$ and $p$. 
 The maps $\kappa$, $\tau$ and $\alpha$ are determined by the universal property of $\cT$ and the relations \[\kappa: v\mapsto  e\otimes v\ , \quad \alpha:v\mapsto v(1-e)\otimes 1_{\cT}\ ,\quad   \tau:v\mapsto u\otimes 1_{\cT}\ .\]
Since $(\pi\otimes \id_{\cT})\circ \alpha=\tau$ we can define the map $s$ by the universal property of the pull-back
such that $r\circ s=\alpha$ and $p\circ s=\id_{\cT}$.
The last equality implies that the upper horizontal sequence is split-exact. 
By the split-exactness of $F$ we get an equivalence 
 \[F(\iota)\oplus F(s):F(K\otimes \cT)\oplus F(\cT)\stackrel{\simeq}{\to} F(\bar \cT)\ .\]  In particular we can conclude that  $F(\iota)$ is monomorphism. 
 
 Since 
 $ \im(r\iota \kappa) \perp \im(\alpha)$  (see \cref{ewfiuhgwqeuifdewdewqdqewdqd}) we can define  
  following homomorphisms   
\begin{equation}\label{vewrvreververwvffewrfref}\phi_{0}:=\alpha+r\iota\kappa \id_{\cT} \ , \quad \phi_{1}:=\alpha+ r\iota \kappa (j\circ q)
\end{equation}  
from $\cT$ to $\cT\otimes \cT$.
 It has been shown in \cite[Sec. 4]{MR750677} (see \cite{fritz} for a nice presentation, reproduced in \cref{weoighwjeoirgfwrefrfwerfwrfw} below) that $\phi_{0}$ and $\phi_{1}$ are homotopic by  a homotopy $\phi_{t}:\cT\to \cT\otimes \cT$ such that $(\pi\otimes\id_{\cT})\circ \phi_{t}(v)=\tau$ for all $t$.
 By the universal property of the pull-back we get a homotopy $\psi_{t}:\cT\to \bar \cT$ from
  $s+\iota\kappa \id_{\cT}$   to $s+\iota\kappa (j\circ q)$
 such that $r\circ \psi_{t}=\phi_{t}$ and $p\circ \psi_{t}=\id_{\cT}$.   By the homotopy invariance of $F$ we get 
\[F(s+\iota\kappa \id_{\cT})\simeq F(s+ \iota\kappa (j\circ q))\ .\]
In view of \cref{ewfiuhgwqeuifdewdewqdqewdqd} and the fact that by \eqref{ewfwqedqewdedqwd} 
for $!=\splt$
the functor $F$ has a left-exact, and hence additive, factorization  
$$\xymatrix{  \nCalg  \ar[rr]^{F}\ar[dr]^{L_{ h,K,!}}&&\bC\\&L_{K}\nCalg_{ h,!} \ar@{..>}[ur]& }$$ 
we conclude that  
 $$F(s)+F(\iota\kappa \id_{\cT})\simeq F(s)+ F(\iota \kappa (j\circ q))$$
 as morphisms with target $F(\bar \cT)$.
Since  
 $F( \bar \cT)$ is a group object we can 
cancel $F(s)$ and obtain the equivalence
$$F(\iota\kappa \id_{\cT})\simeq F(\iota \kappa (j\circ q))\ .$$
Since, as seen above,
 $F(\iota)$ is a monomorphism, 
 and $F(\kappa)$ is an equivalence by stability of $F$ (note that $\kappa$ is a left upper corner inclusion), we can conclude that 
$F(\id_{\cT})\simeq F( j\circ q )\simeq  F( j)\circ F(q )$ as desired.
\end{proof}

\begin{rem}\label{weoighwjeoirgfwrefrfwerfwrfw}
For completeness of the presentation,  following  \cite{fritz}
 we sketch the construction of the homotopy between $\phi_{0}$ and $\phi_{1}$
from \eqref{vewrvreververwvffewrfref}.  
These homomorphisms are determined via the universal property of $\cT$ by 
   $$\phi_{0}(v)=v(1-e)\otimes 1_{\cT}+e\otimes v\ , \quad 
 \phi_{1}(v)=v(1-e)\otimes 1_{\cT}+ e\otimes 1_{\cT}\ .$$
 We will employ an explicit realization of $\cT$ by bounded operators on $L^{2}(\nat)$. If $(\xi_{i})_{i\in \nat}$ denotes the standard basis, then  $v$ is the isometry given by $v\xi_{i}=\xi_{i+1}$ for all $i$ in $\nat$. Furthermore,  
 $e$ is the projection onto the subspace generated by $\xi_{0}$.
 
 We realize $\cT\otimes \cT$ correspondingly on $L^{2}(\nat\times \nat)$ with basis $(\xi_{i,j})_{i,j\in  \nat\times \nat}$. We define  
  the   selfadjoint unitaries   $$u_{0}:=v(1-e)v^{*}\otimes 1_{\cT}+ev^{*}\otimes v+ve\otimes v^{*}+e\otimes e$$
 and 
 $$u_{1}:=v(1-e)v^{*}\otimes 1_{\cT}+ev^{*}\otimes 1_{\cT}+ve\otimes 1_{\cT}\ .$$
 One then checks that 
  $$u_{0}(v\otimes 1_{\cT})u_{0}^{*}=\phi_{0}(v)\ ,  \quad u_{1}(v\otimes 1_{\cT})u_{1}^{*}= \phi_{1}(v)$$
 in $\cT\otimes \cT$.
On basis vectors the unitary $u_{0}$ is given by 
 $$\xi_{i,j}\mapsto  \left\{\begin{array}{cc}
\xi_{0,0}&(i,j)=(0,0) \\\xi_{1,j-1} &i=0,j\ge 1  \\ \xi_{0,j+1} &i=1,j\ge 0\\
\xi_{i,j}& i\ge 2\end{array} \right.$$
We connect  
 $u_{0}$ by a homotopy $u_{0,t}$ with $1_{\cT\otimes \cT}$ by a path in $\cT\otimes \cT$
which rotates (with constant speed) in the each of the two-dimensional subspaces $\C\langle \xi_{1,j-1},\xi_{0,j} \rangle $  for $j\ge 1$ from flip to the identity.
Similarly, the action of 
 $u_{1}$  on basis vectors is given by  $$\xi_{i,j}\mapsto  \left\{\begin{array}{cc}
\xi_{1,j}&i=0 \\\xi_{0,j} &i=1  \\ \xi_{i,j} &i\ge 2 \end{array} \right.$$
 We connect  
 $u_{1}$ by a homotopy $u_{1,t}$ with $1_{\cT\otimes \cT}$ by a path in $\cT\otimes \cT$
which rotates (with constant speed) in the each of the two-dimensional subspaces 
  $\C\langle \xi_{0,j},\xi_{1,j} \rangle $  for $j\ge 0$ from the flip to  the identity. 
  Then $u_{0,t}^{*}\phi_{0}u_{0,t}$ is a homotopy from $\phi_{0}$ to the map determined by $v\mapsto v\otimes 1_{\cT}$. Similarly, $u_{1,t}^{*}\phi_{1}u_{1,t}$ is a homotopy from $\phi_{1}$ to the same map.
  The concatenation of the first with the inverse of the second homotopy is the desired homotopy $\phi_{t}$.
  One checks from the explicit formulas, that  $(\pi\otimes\id_{\cT})\circ \phi_{t}(v)=\tau$ for all $t$.
  
  Note that the selfadjointness of $u_{i}$ is not relevant here, but it would be important for a version for real $C^{*}$-algebras. \hB
  \end{rem}

We now consider the split-exact sequence
 \begin{equation}\label{asdvcdcadcqefr}  0\to \cT_{0}\to \cT\xrightarrow{q} \C\to 0  \end{equation} 
 defining 
$\cT_{0}$ as an ideal in $\cT$.
As above we consider  a functor 
$F:\nCalg \to \bC$  to a  semi-additive $\infty$-category.

\begin{kor}\label{wtriojgopwrtgrwewfref}
 If  $F$  is homotopy invariant, stable,  split-exact   and takes values in 
    group objects,
  then 
 $F(\cT_{0})\simeq 0$.
\end{kor}
\begin{proof}\mbox{}
Since $F$ is split-exact  it sends the split-exact exact sequence \eqref{asdvcdcadcqefr} to a fibre sequence.
Since $F(q)$ is an equivalence by \cref{qerkghqeriffewq} we conclude that its fibre $F(\cT_{0})$ is a zero object.
  \end{proof}

The Toeplitz sequence \eqref{gregeerfwffregfergewg} is semi-split exact. This can be seen either by an application of the Choi-Effros lifting theorem \cite{choi-effros} using that $K$ is nuclear,  or by an explicit construction of a cpc right inverse $s$  of $\pi$, see \cref{weriogjoergrefwrf}.

\begin{rem}\label{weriogjoergrefwrf}
For completeness of the presentation we provide a cpc split for 
the Toeplitz extension \eqref{gregeerfwffregfergewg}. We consider $L^{2}(\Z)$ with the standard basis $(\xi_{i})_{i\in \Z}$
and realize the Toeplitz algebra $\cT$ on the subspace 
$L^{2}(\nat)$ as in \cref{weoighwjeoirgfwrefrfwerfwrfw}.
We let $w$ be the  unitary 
 shift operator determined by   $ \xi_{i}\mapsto \xi_{i+1}$  for all $i$ in $\Z$ and
  $P:L^{2}(\Z)\to L^{2}(\nat)$ be the orthogonal projection.
Then we have  $v=PwP$.  Since $C(S^{1})$ classifies unitaries (see  \cref{poregjopewrgerg1})
 we  have a unique  unital homomorphism $\phi: C(S^{1})\to B(L^{2}(\Z))$
determined by   $\phi(u):= w$.
 By an explicit calculation  of its action on basis vectors one checks that $
 [P,\phi(u^{n})]$ is finite-dimensional and therefore belongs to $K$ for all $n$ in $\Z$. 
 Since $u$ generates $C(S^{1})$
 we conclude that $[P,\phi(f)]\in K$
 for all $f$ in $C(S^{1})$.
We define the linear map
$s:C(S^{1})\to B(L^{2}(\nat))$ by $s(f)=P\phi(f)P$.
Using the discussion above one checks that it takes vales in $\cT$. Moreover, since  $\pi(s(u))=\pi(v)=u$
we conclude that $\pi\circ s=\id_{\cT}$. Since it is the compression of a homomorphism
it is completely positive. 
  \hB

\end{rem}

\begin{rem}
The Toeplitz extension does not admit a split. For this reason in the constructions below we must assume that $!$ belongs to $\{\se,\exa\}$ and exclude the case $\splt$. \hB
\end{rem}

We consider the diagram of vertical exact sequence
 \begin{equation}\label{dfbpojpwortbfdgbdfgbdgfb}\xymatrix{&0\ar[d]&0\ar[d]&&\\&K\ar[d]\ar@{=}[r]&K\ar[d]&&\\
&
\cT_{0}\ar[r]\ar[d]^{\pi_{0}}&\cT\ar[d]^{\pi}
&
&
\\
&
\ar@/^0.5cm/@{..>}[u]^{s_{0}}
S(\C)\ar[r]^{i}\ar[d]&\ar[d]\ar@/^0.5cm/@{..>}[u]^{s}
C(S^{1})
&
&
\\&0&0&&}
\end{equation} 
where the lower square is a cartesian and the map $i$ is an in \eqref{fwqefwfedwdwedwdwdfwfwfwf}.  The  cpc split $s$ induces a cpc split  $s_{0}$ as indicated.

%
%
%
%
%
%
%
%
%
%
%
%
%
%
%
%
%
%
%
%
%
%
%
%
%
%
%
%
%
%
%
%
%

%
%
%
%
%
%
%
%
%
%
%

We let $!$ in $\{\se,\exa\}$ and   apply the semiexact functor $L_{h,K,!}$ to the  left vertical semi-split exact  sequence in order to
get a fibre sequence \begin{equation}\label{svfsfvfdvfvdsv}  L_{ h,K,!}( S^{2}(\C))\xrightarrow{\beta_{!,\C}}  L_{ h,K,!}(\C)\to L_{h,K,!}(\cT_{0})\to   L_{ h,K,!}( S(\C))
\end{equation}
defining $\beta_{!, \C}$,
where we used  
 $L_{ h,K,!}(\C)\simeq  L_{ h,K,!}(K)$ by stability of $ L_{ h,K,!}$.
 
 Since we want to speak about the two-fold loop functor in different left-exact $\infty$-categories we add subscripts indicating which category is meant in each case.
 Note that for all $k$ in $\nat$, by  \cref{qerigqerfewfq9} and the left-exactness of $L_{!}\circ L_{K}$  the $k$-fold loop functor $\Omega_{!}^{k}$  on $L_{K}\nCalg_{h,!}$ can be represented by the $k$-fold suspension
 on the level of $C^{*}$-algebras
 $$\Omega^{k}_{!}(-) \simeq  L_{ h,K,!}(S^{k} (\C))\otimes-:L_{K}\nCalg_{h,!} \to L_{K}\nCalg_{h,!} \ .$$  
 Recall  the tensor unit constraint $$  \id_{L_{K}\nCalg_{h,!} }\simeq   L_{ h,K,!}( \C)\otimes- :L_{K}\nCalg_{ h,!} \to L_{K}\nCalg_{ h,!} \ .$$

 The following definition implicitly uses these identifications.
\begin{ddd}\label{qeorigjiowergwerfwfrewfwr} We define a natural transformation of endofunctors
$$\beta_{!} :=\beta_{ !,\C}\otimes- :\Omega_{!}^{2} \to \id_{L_{K}\nCalg_{ h,!} }: L_{K}\nCalg_{ h,!} \to L_{K}\nCalg_{ h,!} $$
 \end{ddd}

We consider a functor 
$E:L_{K}\nCalg_{ h,!}\to \bC$  to a  semi-additive $\infty$-category and 
let $A$ be an object of $L_{K}\nCalg_{ h,!} $.
 \begin{kor}\label{wtriojgopwrtgrwewfref222}
 If  $E$  is left-exact and $E(-\otimes A)$ takes values in 
    group objects, then 
   $E(\beta_{ !, A}): E(\Omega_{!}^{2}(A))\to E(A)$ is an equivalence.
\end{kor}
\begin{proof}\mbox{}
We consider the functor 
$  F(-):=E(-\otimes_{\max} A)$.
Then $  F(\beta_{ !,\C})\simeq E(\beta_{!,A})$. Using \cref{ewrogpwergfrefwrefwref}.\ref{wrthgioweiogiowerg1}
 we observe that $ L_{h,K,!}^{*}  F$ belongs to $\Fun^{h,s,!}(\nCalg ,\bC)$. Since it also
  takes values in group objects we can apply  \cref{wtriojgopwrtgrwewfref}  in order to conclude that
 $ F(L_{h,K,!}(\cT_{0}))\simeq 0$.
 
  The functor $ F$ sends the fibre sequence  
\eqref{svfsfvfdvfvdsv}   to a  fibre sequence
$$   F( \Omega^{2}L_{ h,K,!}(\C))\xrightarrow{F(\beta_{ !,\C})}  F(L_{h,K,!}(\C))\to   F(L_{ h,K,!}(\cT_{0}))\ .$$
 Hence
 $ F(\beta_{!,\C})$ is an equivalence.
 \end{proof}

 Recall that $!$ is in $\{\se, \exa\}$.
\begin{kor}\label{erjigowergerrqf9}
If $A$ is a group object in $ L_{K}\nCalg_{h,!} $, then
$\beta_{ !,A}:\Omega^{2}_{!}(A)\to A$ is an equivalence.
\end{kor}
\begin{proof}
We apply \cref{wtriojgopwrtgrwewfref222} to the identity functor in place of $E$.
We further use that if  $A$ is a group object, then so is $B\otimes A$ for every $B$ in $ L_{K}\nCalg_{h,!}$.
  \end{proof}

  The statements of  \cref{qerkghqeriffewq}, \cref{wtriojgopwrtgrwewfref},   \cref{wtriojgopwrtgrwewfref222}, and  \cref{erjigowergerrqf9} all have  separable versions which are obtained by adding subscripts $\sepa$ appropriately.

\section{Group objects and $\KK_{\sepa}$ and $\EE_{\sepa}$}\label{werigjowergerwg9}

In this section we consider the full subcategories of group objects in the
semi-additive $\infty$-categories  $L_{K}\nCalg_{h,!}$ for $!$ in $\{\se ,\exa\}$ and their separable versions. 
They are the targets of the two-fold loop functor and turn out the be stable $\infty$-categories.
 This two-fold loop functor is the right-adjoint of a right Bousfield localization. It is the last step of the chain of localizations described in \cref{rgijrweiogwrwrefwf}. 
In the separable case,  the composition of all four localizations yields 
 the functors  $\kk_{\sepa}:\nCalg_{\sepa}\to \KK_{\sepa}$  and 
$\ee_{\sepa}:\nCalg_{\sepa}\to \KK_{\sepa}$ whose universal properties will be stated in 
 \eqref{vdfvqr3ferffafafds}
 and \eqref{vdfvqr3ferffafafds1}.

Dually to the situation described at the beginning of \cref{lhK} let $\bC$ be an $\infty$-category with an endofunctor $R:\bC\to \bC$ and a natural transformation 
$\beta:R\to \id_{\bC}$. If for every object $C$ the morphisms \[\beta_{R(C)}, R(\beta_{C}):R(R(C))\to R(C)\] are equivalences, then $R$ is the right-adjoint of a right Bousfield localization with counit $\beta$. The functor $R:\bC \to R(\bC)$ is also a Dwyer-Kan localization 
at the set of morphisms $W_{R}:=\{\beta_{C}\mid C\in \bC\}$.  
If $\bC$ is left-exact, then the localization $R$ is automatically left exact.

If $\bC$ is semi-additive, then we let $\bC^{\group}$ denote the full subcategory of group objects in $\bC$.
A full subcategory of a semi-additive $\infty$-category which is closed under products is again semi-additive.
A semi-additive $\infty$-category is called {\em additive} if all its objects are groups.  If $\bC$ is semi-additive, then $\bC^{\group}$ is additive.  

\begin{ex}
A   stable $\infty$-category is additive. \hB
\end{ex}


We consider 
$!$ in $\{\se,\exa\}$ and $?$ in $\{\min,\max\}$  allowing the  following combinations: \bigskip

\centerline{\begin{tabular}{c|c|c} ! $\backslash$ ?&min&max\\ \hline 
se&yes&yes\\
ex&no&yes
\end{tabular}}
 We consider the  two-fold loop endofunctor  $\Omega^{2}_{!}$ on $ L_{K}\nCalg_{h,!}$ and the natural transformation  $\beta_{!} :\Omega_{!}^{2}\to \id_{L_{K}\nCalg_{h,!} }$  from 
 \cref{qeorigjiowergwerfwfrewfwr}.  
\begin{prop}\label{wtgkwotpgkelrf09i0r4fwefwerf}\mbox{}
\begin{enumerate} \item \label{wetiogjwegferwferfwrefw8805} The essential image of $\Omega_{!}^{2} $ is $ L_{K}{\nCalg_{h,!}}^{\group}$. 
\item \label{wetiogjwegferwferfwrefw88}The functor $\Omega_{!}^{2}$
is the right-adjoint of a right Bousfield localization with counit $\beta_{!}:\Omega_{ !}^{2}\to \id_{L_{K}\nCalg_{ h,!} }$.
In addition, the localization $\Omega_{ !}^{2}$ is left-exact.

\item  \label{wetiogjwegferwferfwrefw881}The $\infty$-category $ L_{K}{\nCalg_{h,!}}^{\group}$ is stable.
\item  \label{wetiogjwegferwferfwrefw883}For every left-exact and additive $\infty$-category  $\bD$ we have an equivalence \begin{equation}\label{gwerfgewrferferwfwrefwref} (\Omega_{!}^{2}\circ  L_{h,K,!})^{*}:\Fun^{\lex}(L_{K}{\nCalg_{h,!}}^{\group},\bD)\stackrel{\simeq}{\to}\Fun^{h,s,!+Sch}(\nCalg, \bD)\ .\end{equation}  
\item \label{wetiogjwegferwferfwrefw884} The localization $\Omega_{!}^{2}$ admits a symmetric monoidal refinement.
\item  \label{wetiogjwegferwferfwrefw885}The functor $$-\otimes_{?}-:L_{K}{\nCalg_{h,!}}^{\group}\times L_{K}{\nCalg_{h,!}}^{\group}\to L_{K}{\nCalg_{h,!}}^{\group}$$ is bi-exact.
\item  \label{wetiogjwegferwferfwrefw886}For every symmetric monoidal,   left-exact and additive $\infty$-category $\bD$ we have an equivalence \begin{equation}\label{vasdvasvdsavasdvasvsadvds}  (\Omega_{!}^{2}\circ  L_{h,K,!})^{*}:\Fun^{\lex}_{\otimes/\lax}(L_{K}{\nCalg_{h,!}}^{\group},\bD)\stackrel{\simeq}{\to}\Fun_{\otimes/\lax}^{h,s,!+Sch}(\nCalg, \bD)\ .\end{equation}  
\end{enumerate}
\end{prop}
\begin{proof}
 We start with Assertion \ref{wetiogjwegferwferfwrefw8805}.
By \cref{wgtojwregpokerpgergwefewferf} the functor $\Omega_{!}^{2}$ takes values in group objects.
If $A$ belongs to $ L_{K}{\nCalg_{h,!}}^{\group}$, then $\beta_{!,  A}:\Omega_{!}^{2}(A)\to A$ is an equivalence by \cref{erjigowergerrqf9}. Hence the essential image of $\Omega_{!}^{2} $ is precisely $ L_{K}{\nCalg_{h,!}}^{\group}$.

For Assertion \ref{wetiogjwegferwferfwrefw88} we first note that 
  $\beta_{!,\Omega_{!}^{2}(A)}$ is an equivalence again by \cref{erjigowergerrqf9}. 
We furthermore employ   the symmetry  of the tensor product and $\beta_{ !,A}\simeq \beta_{!, \C}\otimes A$ in order to see that 
$\Omega_{!}^{2}(\beta_{!,A})\simeq \beta_{!,\Omega_{!}^{2}(A)} $ is an equivalence, too.

In order to show Assertion \ref{wetiogjwegferwferfwrefw881} we show that the loop functor
\[\Omega_{!}:L_{K}{\nCalg_{h,!}}^{\group}\to L_{K}{\nCalg_{h,!}}^{\group}\] is an equivalence. Indeed,  by \cref{erjigowergerrqf9} the restriction of the natural transformation   $\beta_{!}:\Omega_{!}^{2}\to \id$ from \cref{qeorigjiowergwerfwfrewfwr}
to group objects is an equivalence which exhibits $\Omega_{!}$ as its own inverse.

  Assertion \ref{wetiogjwegferwferfwrefw883} follows from 
\begin{eqnarray}
\Fun^{\lex}(L_{K}{\nCalg_{ h,!}}^{\group},\bD) &\stackrel{\Omega^{2,*}_{ !}}{\simeq}&\Fun^{\lex}(L_{K}\nCalg_{ h,!} ,\bD)\label{ekogpwefrfwvgre}\\
&\stackrel{ L_{ h,K,!}^{*}}{\simeq}&\Fun^{h,s,!+Sch}(\nCalg ,\bD) \ . \nonumber
\end{eqnarray}
The functor $\Omega_{!}^{2,*}$   preserves left-exact functors by Assertion \ref{wetiogjwegferwferfwrefw88}. It is fully faithful and essentially surjective  since any  left-exact functor from $L_{K}\nCalg_{ h,!} $ to an additive category $\bD$ automatically
inverts by \cref{wtriojgopwrtgrwewfref222} the set $W_{\Omega^{2}_{!}} :=\{\beta_{A,!}| A\in L_{K}\nCalg_{ h,!}\}$  of generators of the Dwyer-Kan localization $\Omega_{!}^{2}$.
 The second equivalence in \eqref{ekogpwefrfwvgre} is 
 \eqref{ewfwqedqewdedqwd}.


For Assertion \ref{wetiogjwegferwferfwrefw884} we observe that  
the equivalence $\id_{A}\otimes \beta_{!,B}\simeq \beta_{!,A\otimes B}$ for all $A$ and $B$  in $L_{K}\nCalg_{h,!}$ implies that the endofunctor
$A\otimes-$ of $ L_{K}\nCalg_{h,!} $ preserves  the set  $W_{\Omega^{2}_{!}}$.   Consequently,
$\Omega^{2}_{ !}$ admits a symmetric monoidal refinement.  Furthermore, for 
 $A$ in $ L_{K}\nCalg_{ h,!} $  the endofunctor $A\otimes-$ descends   to a left-exact  endofunctor on
 $ L_{K}{\nCalg_{h,!}}^{\group}$. This implies Assertion \ref{wetiogjwegferwferfwrefw885}.
 
 Finally, Assertion \ref{wetiogjwegferwferfwrefw886} follows from
\begin{eqnarray*}
\Fun^{\lex}_{\otimes/\lax}(L_{K}{\nCalg_{h,!}}^{\group},\bD) &\stackrel{\Omega_{!}^{2,*}}{\simeq}&\Fun^{\lex}_{\otimes/\lax}(L_{K}\nCalg_{h,!} ,\bD)\\
&\stackrel{ L_{h,K,!}^{*}}{\simeq}&\Fun^{h,s,!+Sch}_{\otimes/\lax}(\nCalg,\bD)\ ,
\end{eqnarray*}
where the first follows from the left-exactness of $\Omega_{!}^{2}$ shown in Assertion \ref{wetiogjwegferwferfwrefw884}, and the second is \eqref{ewfwqedqewdedqwd1}. 
\end{proof}

The \cref{wtgkwotpgkelrf09i0r4fwefwerf} has a separable version which we state for later reference.
\begin{prop}\label{wtgkwotpgkelrf09i0r4fwefwerfsepa}\mbox{}
\begin{enumerate} \item \label{wetiogjwegferwferfwrefw8805sepa} The essential image of $\Omega_{\sepa,!}^{2} $ is $ L_{K}{\nCalg_{\sepa,h,!}}^{\group}$. 
\item \label{wetiogjwegferwferfwrefw88sepa}The functor $\Omega_{\sepa,!}^{2}$
is the right-adjoint of a right Bousfield localization with counit $\beta_{\sepa,!}:\Omega_{ \sepa,!}^{2}\to \id_{L_{K}\nCalg_{\sepa, h,!} }$.
In addition, the localization $\Omega_{ \sepa,!}^{2}$ is left-exact.

\item  \label{wetiogjwegferwferfwrefw881sepa}The $\infty$-category $ L_{K}{\nCalg_{\sepa,h,!}}^{\group}$ is stable.
\item  \label{wetiogjwegferwferfwrefw883sepa}For every left-exact and additive $\infty$-category  $\bD$ we have an equivalence \begin{equation}\label{gwerfgewrferferwfwrefwrefsepa} (\Omega_{\sepa,!}^{2}\circ  L_{\sepa,h,K,!})^{*}:\Fun^{\lex}(L_{K}{\nCalg_{\sepa,h,!}}^{\group},\bD)\stackrel{\simeq}{\to}\Fun^{h,s,!+Sch}(\nCalg_{\sepa}, \bD)\ .\end{equation}  
\item \label{wetiogjwegferwferfwrefw884sepa} The localization $\Omega_{\sepa,!}^{2}$ admits a symmetric monoidal refinement.
\item  \label{wetiogjwegferwferfwrefw885sepa}The functor $$-\otimes_{?}-:L_{K}{\nCalg_{\sepa,h,!}}^{\group}\times L_{K}{\nCalg_{\sepa,h,!}}^{\group}\to L_{K}{\nCalg_{\sepa,h,!}}^{\group}$$ is bi-exact.
\item  \label{wetiogjwegferwferfwrefw886sepa}For every symmetric monoidal,   left-exact and additive $\infty$-category $\bD$ we have an equivalence \begin{equation}\label{vasdvasvdsavasdvasvsadvds1}  (\Omega_{\sepa,!}^{2}\circ  L_{\sepa,h,K,!})^{*}:\Fun^{\lex}_{\otimes/\lax}(L_{K}{\nCalg_{\sepa,h,!}}^{\group},\bD)\stackrel{\simeq}{\to}\Fun_{\otimes/\lax}^{h,s,!+Sch}(\nCalg_{\sepa}, \bD)\ .\end{equation}  
\end{enumerate}
\end{prop}

We note that $L_{K}{\nCalg_{h,!}}^{\group}$ is a large $\infty$-category while $L_{K}{\nCalg_{\sepa,h,!}}^{\group}
$ is essentially small and locally small.

\begin{rem} \label{werogwegferfwefref} If $\bD$ is in addition stable, then by   \cref{reijgoerfeffrfqrferfe},
 in the right-hand sides of  \eqref{gwerfgewrferferwfwrefwref}, \eqref{vasdvasvdsavasdvasvsadvds}, \eqref{gwerfgewrferferwfwrefwrefsepa} and \eqref{vasdvasvdsavasdvasvsadvds1} we can omit the superscript $+Sch$. \hB
\end{rem}

 \begin{prop}\label{ergiojeroigwerfwrefrefdvs}
 The functors \[ \hspace{-0.2cm}\Omega_{ !}^{2}\circ L_{ h,K,!}:\nCalg \to {L_{K}\nCalg_{h,!}}^{\group}\ ,\quad \Omega_{\sepa,!}^{2}\circ L_{\sepa,h,K,!}:\nCalg_{\sepa}\to  {L_{K}\nCalg_{\sepa,h,!}}^{\group}\]  are Dwyer-Kan localizations.
 \end{prop}
 \begin{proof} We write out the details in the separable case. The non-separable case is analogous.  The functor $L_{\sepa,h,K,!}:\nCalg_{\sepa}\to L_{K}\nCalg_{\sepa,h,!}$
 is constructed as an iterated Dwyer-Kan localization at sets of morphisms in
 $\nCalg_{\sepa}$.  By definition, the last step $\Omega_{\sepa,!}^{2}:L_{K}\nCalg_{\sepa,h,!}\to {L_{K}\nCalg_{\sepa,h,!}}^{\group}$ is a Dwyer-Kan localization
 at the set of morphisms $\beta_{\sepa,!,A}:L_{\sepa,h,K,!}(S^{2}(A))\to L_{\sepa,h,K,!}(A)$ for all $A$ in $\nCalg_{\sepa}$. These morphisms only exist in the localization $ L_{K}\nCalg_{\sepa,h,!}$.
 But we can replace them by a collection of images of  morphisms  in $\nCalg_{\sepa} $.  For $A$ in $\nCalg_{\sepa}$ we  have the commutative diagram
 \[\xymatrix{L_{\sepa,h,K,!}(S^{2}(A))\ar[d]^{L_{\sepa,h,K,!}(\lambda_{A})}\ar[rr]^{\beta_{\sepa,!,A}}&& L_{\sepa,h,K,!}( A)\ar[d]_{L_{\sepa,h,K,!}(\kappa_{A})}^{\simeq}\\L_{\sepa,h,K,!}( C(\pi_{A})) &&\ar[ll]^{L_{\sepa,h,K,!}(\iota_{\pi_{A}})}_{\simeq} L_{\sepa,h,K,!}(A\otimes K) }\]
  in $L_{K}\nCalg_{\sepa,h,!}$, 
 where $\kappa_{A}$ is the left upper corner inclusion \eqref{oiwjveoivjowevsvdfvsfdv}, $\iota_{\pi_{A}} $ is the canonical inclusion is associated to  the semi-split exact sequence
 \[0\to A\otimes K\to A\otimes \cT_{0}\xrightarrow{\pi_{A}} S(A)\to 0\] (the tensor product of the left column in \eqref{dfbpojpwortbfdgbdfgbdgfb} with $A$) as in \eqref{tbrtbbgdbfert}, and $\lambda_{A}:S^{2}(A)\to C(\pi_{A})$ is the canonical inclusion.
 Hence $\Omega_{\sepa,!}^{2}$ is also the Dwyer-Kan localization at the collection of morphisms
 $(L_{\sepa,h,K,!}(\lambda_{A}))_{A\in \nCalg_{\sepa}}$.
 We can conclude that the composition $ \Omega_{\sepa,!}^{2}\circ L_{\sepa,h,K,!}$ is a Dwyer-Kan localization.
  \end{proof}

  
\begin{ddd}We define the {\em $KK$-theory functor for separable $C^{*}$-algebras}
  \begin{equation}\label{fweeqwdewdewdqwed} \kk_{\sepa}:\nCalg_{\sepa}\to \KK_{\sepa}\end{equation}  to be the functor $$\Omega^{2}_{\sepa,\se}\circ L_{\sepa,h,K,\se}:\nCalg_{\sepa}\to {L_{K}\nCalg_{\sepa,h,\se}}^{\group}\ .$$
  \end{ddd}
So $\KK_{\sepa}$ is a locally small and essentially small stable $\infty$-category.
The functor $\kk_{\sepa}$ has the universal property that 
\begin{equation}\label{vdfvqr3ferffafafds} \kk_{\sepa}^{*}:\Fun^{\lex}(\KK_{\sepa},\bD)\stackrel{\simeq}{\to} \Fun^{h,s,se+Sch}(\nCalg_{\sepa},\bD)\end{equation} 
for any left-exact and additive   $\infty$-category $\bD$. 
Since we know by \cref{wtgkwotpgkelrf09i0r4fwefwerfsepa}.\ref{wetiogjwegferwferfwrefw881sepa}  that $\KK_{\sepa}$ is stable the restriction of this universal property to stable $\infty$-categories $\bD$ (where by \cref{werogwegferfwefref} we can omit the superscript $+Sch$) already characterizes the functor $\kk_{\sepa}:\nCalg_{\sepa}\to \KK_{\sepa}$ up to equivalence. 
\begin{rem}By \cite[Thm. 1.5]{KKG} the functor 
 denoted by the same symbol  in \cite[Def. 1.2]{KKG} (for the trivial group $G$)
 has the same universal property and therefore 
 is canonically equivalent to the functor defined above.
We can   conclude by  \cite[Thm. 1.3]{KKG} that the functor
 \[\ho\circ \kk_{\sepa}:\nCalg_{\sepa}\to \ho\KK_{\sepa}\]
 with values in the triangulated category  $\ho\KK_{\sepa}$ is canonically equivalent to 
 the classical functor considered in   \cite{MR2193334} and the $C^{*}$-literature elsewhere.  As explained in the introduction  in the present paper we will give an independent proof for this fact, see \cref{wrtgowtgwgrrwfrefw}.
 \hB\end{rem}

\begin{rem}Since $\KK_{\sepa}$ is stable it admits all finite colimits. After some hard work,  in \cref{weijotgwegferfwrefw} below we will see that $\KK_{\sepa}$ admits countable colimits and is thus  idempotent complete. We do not have a direct proof of this fact just from the constructions. 
 \hB
\end{rem}

\begin{ddd}\label{wopgrgerwfgewrfw}We define the {\em $E$-theory functor for separable $C^{*}$-algebras}  
  \begin{equation}\label{fweeqwdewdewdqwed1} \ee_{\sepa}:\nCalg_{\sepa}\to \EE_{\sepa}\end{equation}  to be the functor $$\Omega^{2}_{\sepa,\exa}\circ L_{\sepa,h,K,\exa}:\nCalg_{\sepa}\to {L_{K}\nCalg_{\sepa,h,\exa}}^{\group}\ .$$
  \end{ddd}
  So $\EE_{\sepa}$ is a  locally small and essentially small stable $\infty$-category.
The functor in \eqref{fweeqwdewdewdqwed1} is the inital homotopy invariant, stable and exact functor to a left-exact and additive $\infty$-category, i.e.  for every any  left-exact and additive $\infty$-category $\bD$  we have the equivalence \begin{equation}\label{vdfvqr3ferffafafds1} \ee_{\sepa}^{*}:\Fun^{\lex}(\EE_{\sepa},\bD)\stackrel{\simeq}{\to} \Fun^{h,s,ex+Sch}(\nCalg_{\sepa},\bD)\ .\end{equation} 
  
\begin{rem}
The justification for calling the functor defined in \cref{wopgrgerwfgewrfw} the $E$-theory functor is that it has an analogous universal property as the additive $1$-category-valued $E$-theory functors considered 
\cite{MR1068250}, \cite{zbMATH04182148}. 
In fact, in  \cref{qiurhfgiuewrgwrfrefrfwrefw} we show that after going to the homotopy category the functor
$\ee_{\sepa}$ becomes equivalent to the classical $E$-theory functor for separable algebras.

Asymptotic morphisms will be discussed in \cref{ewriogjowegregrewf99} below. \hB
 \end{rem}

\begin{rem} Since $\KK_{\sepa}$ and $\EE_{\sepa}$ are stable, at a first glance  it looks more natural
to 
  formulate the universal properties for stable targets $\bD$. But we will take advantage of the more general version for left-exact additive $\infty$-categories in \cref{reoijqoeirfewfqwdewd} below. \hB
\end{rem}

 \section{$s$-finitary functors}\label{fbweroibjdfbsfdbsfb}
 
 In this section we extend the  $KK$ and $E$-theory functors from separable to   all $C^{*}$-algebras and
 characterize these extensions by their universal properties.

To any essentially small and locally small  stable  $\infty$-category $\bC$  we can associate its $\Ind$-completion \[y:\bC\to \Ind(\bC)\ .\]
 As a model, using that $\bC$ has mapping spectra,  one can take 
the Yoneda embedding \begin{equation}\label{erguhiurfhirefrfwdewd}y:\bC\to \Fun^{\lex}(\bC,\Sp)\ .
\end{equation} 
  The large stable $\infty$-category $\Ind(\bC)$ is presentable, and the fully faithful and exact functor $y$ has the universal property that for any cocomplete stable  $\infty$-category $\bD$ the pull-back along $y$ is an equivalence 
 \begin{equation}\label{ewqfwedqwededqdedq}  y^{*}:\Fun^{\colim}(\Ind(\bC),\bD)\stackrel{\simeq}{\to} \Fun^{\lex}(\bC,\bD)\ , \end{equation}
 where the supercript $\colim$ indicates small colimit preserving functors.
 If  $\bC$  has a bi-exact symmetric monoidal structure, then $  \Ind(\bC)$ has a natural symmetric monoidal structure and
 $y$ has a  symmetric monoidal refinement such that for any cocomplete bi-cocontinuous symmetric monoidal $\infty$-category  $\bD$ the pull-back along $y^{*}$ induces an equivalence 
 \begin{equation}\label{ewqfwedqwededqdedq1}y^{*}:\Fun^{\colim}_{\otimes/\lax}(\Ind(\bC),\bD)\stackrel{\simeq}{\to} \Fun^{\lex}_{\otimes/\lax}(\bC,\bD)\ .\end{equation}
The inverse of the restriction is given by left Kan-extension.
In the model \eqref{erguhiurfhirefrfwdewd}  the symmetric monoidal structure on the functor category  is the Day convolution structure on the functor category.

Let $!$ be in $\{\se,\exa\}$ and $?$ be in $\{\min,\max\}$. As before we allow  the  following combinations: \bigskip

\centerline{\begin{tabular}{c|c|c} ! $\backslash$ ?&min&max\\ \hline 
se&yes&yes\\
ex&no&yes
\end{tabular}}

For the moment we  use the abbreviations
\begin{equation}\label{vsddfkvopsdfvfsdvsdfvsfdvsdv}  \KK_{\sepa,!}:=L_{K}{\nCalg_{\sepa,h,!}}^{\group}\ , \quad  \kk_{\sepa,!}:=\Omega^{2}_{\sharp,!}\circ L_{\sepa,h,K,!}:\nCalg_{\sepa}\to \KK_{\sepa,!} \end{equation}   instead of $\KK_{\sepa}$ or $\EE_{\sepa}$ 
 order to discuss $KK$ and $E$-theory in a parallel manner.
%
%
 \begin{ddd}\label{iejgowergrfrfwerfref}
 We define $\KK_{!}:=\Ind(\KK_{\sepa,!})$ and the functor
 \[\kk_{!}:\nCalg\to \KK_{!}\] as the left Kan-extension 
 \[\xymatrix{\nCalg_{\sepa}\ar[dr]^{\incl}\ar[r]^{\kk_{\sepa,!}}&\KK_{\sepa,!}\ar[r]^{y}&\KK_{!}\\&\nCalg\ar@{..>}[ur]^{\kk_{!}}&}\]
 of $y\circ \kk_{\sepa,!}$ along $\incl$.
 \end{ddd}
 Since the inclusion functor $\incl$ is fully faithful the triangle commutes up to a natural equivalence. 

The following properties of the functor $\kk_{!}:\nCalg\to \KK_{!}$ are immediate from the definition.
 \begin{kor}\label{weoktgpwergwerfwrf}\mbox{}
\begin{enumerate}
\item \label{ergijowerfwerfwerf}  We have an equivalence $  \kk_{!}\circ \incl\simeq y\circ \kk_{\sepa,!}:\nCalg_{\sepa}\to \KK_{!}$.
\item $\KK_{!}$ is a large presentable stable $\infty$-category compactly generated by the image of $y$.
\item The functor  $\kk_{!}:\nCalg\to \KK_{!}$ is $s$-finitary.
\item $\KK_{!}$ admits a bi-cocontinuous symmetric monoidal structure $\otimes_{?}$ and $y^{*}$ has a symmetric monoidal refinement such that
\[y^{*}: \Fun^{\colim}_{\otimes/\lax}( \KK_{!},\bC)\to \Fun^{\otimes,\lex}_{\otimes/\lax}(\KK_{\sepa,!},\bD)\] is an equivalence  for any cocomplete bi-cocontinuous symmetric monoidal $\infty$-category with $\bD$.
\item \label{goijewrgoiwrfwfrefwrf} The tensor product with the  image  \begin{equation}\label{vewqcercerwce}b:\kk_{!}(S^{2}(\C))\to \kk_{!}(\C)  
\end{equation}of the equivalence $ \beta_{!,\C}$ in $\KK_{!}$  induces an equivalence $\Omega^{2}\stackrel{\simeq}{\to}\id_{\KK_{!}}$
of endofunctors of $\KK_{!}$
\end{enumerate}
\end{kor}

\begin{rem}
The main point of  \cref{weoktgpwergwerfwrf}.\ref{goijewrgoiwrfwfrefwrf}
is that the loop functor on $\KK_{!}$ is two-periodic, and that this periodicity is implemented
by the product with an (necessarily invertible) element $b$ in $\pi_{-2}\KK_{!}(\C,\C)\simeq \pi_{0}\KK_{!}(S^{2}(\C),\C)$,
where $\KK_{!}(\C,\C)$ is the commutative endomorphism ring spectrum of the tensor unit of $\KK_{!}$.
This will be used for the calculation of the ring spectrum in \cref{qeorigjiowergwerfwfrewfwreee}.
\hB\end{rem}

The following results prepare the verification of the universal property of the functor $\kk_{!}$.
%
%
%
%
%
We consider  a functor $F_{\sepa}:\nCalg_{\sepa}\to \bC$ and assume that it admits a 
left Kan-extension 
 \[\xymatrix{\nCalg_{\sepa}\ar[rr]^{F_{\sepa}}\ar[dr]_{\incl}&&\bC\\
 &\nCalg\ar@{..>}[ur]^{F}&}\ .\]
By  \cref{wergijwerogerwfewrferwfw}, the functor  $F$ is 
 $s$-finitary. 
     Recall the notions introduced in  \cref{wrogjpwrgrfrewf}.
 \begin{prop}\label{weropgkrpeogrfw}$F$ inherits the following properties from $F_{\sepa}$:
 \begin{enumerate}
 \item\label{gwtihgwtigowelrfre} homotopy invariance
 \item\label{gwtihgwtigowelrfre1}  stability
 \item\label{gwtihgwtigowelrfr2}  $!$-exactness for $!$ in  $\{\splt,\se,\exa\}$, provided  in $\bC$ filtered colimits preserve fibre sequences.
 \end{enumerate}
 \end{prop}
\begin{proof}  This is  shown in  \cite[Lemma 3.2]{KKG}, see also \cref{weroigwergfewrfrwef}.
The case of $!=\splt$ (not discussed in the reference) is analoguous to the case $!=\se$.
 \end{proof}

 \begin{theorem}\label{wergijweogrefwrferwfw}
 \mbox{}
 \begin{enumerate}
 \item \label{qrigjfoqrefwedqew9} The functor $\kk_{!}$ is homotopy invariant, stable, and $!$-exact.
 \item \label{qrigjfoqrefwedqew91}
 The restriction along $\kk_{!}$ induces for every cocomplete stable $\infty$-category $\bD$ an equivalence
 $$\kk_{!}^{*}:\Fun^{\colim}(\KK_{!},\bD)\stackrel{\simeq}{\to} \Fun^{h,s,!,sfin}(\nCalg,\bD)\ .$$
 \item \label{qrigjfoqrefwedqew92} $\kk_{!}$ has a natural symmetric monoidal refinement such that restriction anlong $\kk_{!}$
    induces for every cocomplete symmetric monoidal  stable $\infty$-category $\bD$ with bi-cocontinuous symmetric monoidal structure an equivalence
 $$\kk_{!}^{*}:\Fun_{\otimes/\lax}^{\colim}(\KK_{!},\bD)\stackrel{\simeq}{\to} \Fun^{h,s,!,sfin}_{(\otimes/\lax}(\nCalg,\bD)\ .$$
\end{enumerate}
 \end{theorem}
 \begin{proof}
 
In order to see  Assertion \ref{qrigjfoqrefwedqew9} note that  
$ \kk_{\sepa,!}$ is homotopy invariant, stable, and   !-exact by  \cref{wtgkwotpgkelrf09i0r4fwefwerf}.\ref{wetiogjwegferwferfwrefw883}. Since
$y$ is exact, the  composition $y\circ  \kk_{\sepa,!}$ has these properties, too. The Assertion now follows from \cref{weropgkrpeogrfw}.
  
 Assertion   \ref{qrigjfoqrefwedqew91} follows from the commutativity of
 \[\xymatrix{\Fun^{\colim}(\KK_{!},\bD)\ar[r]_{\simeq, \eqref{ewqfwedqwededqdedq}}^{y^{*}}\ar[d]^{\kk_{!}^{*}}&\Fun^{\lex}(\KK_{\sepa,!},\bD)\ar[d]^{\kk_{\sepa,!}^{*}}_{\simeq,\eqref{gwerfgewrferferwfwrefwref}}\\ \Fun^{ h,s,!,sfin} (\nCalg,\bD)\ar[r]_{\simeq}^{\incl^{*}}&\Fun^{ h,s,!} (\nCalg_{\sepa},\bD)}\ ,\]
 where by \cref{weropgkrpeogrfw} the inverse of the lower horizontal functor is the left Kan-extension functor along $\incl$.
 
 The functor $\kk_{!}$ is defined as a left-Kan extension of a symmetric monoidal functor $y\circ \kk_{\sepa,!}$ along another symmetric monoidal functor $\incl$. It  therefore (see \cite[Lem. 3.6]{KKG}) has a 
  lax symmetric monoidal refinement.   As  shown in \cite[Prop. 3.8]{KKG} this structure
 is actually symmetric monoidal.
 Assertion   \ref{qrigjfoqrefwedqew92} now follows  from the commutativity of
 \[\xymatrix{\Fun_{\otimes/\lax}(\KK_{!},\bD)\ar[r]_{\simeq,\eqref{ewqfwedqwededqdedq1}}^{y^{*}}\ar[d]^{\kk_{!}^{*}}&\Fun_{\otimes/\lax}^{\lex}(\KK_{\sepa,!},\bD)\ar[d]^{\kk_{\sepa,!}^{*}}_{\simeq,\eqref{vasdvasvdsavasdvasvsadvds}}\\ \Fun_{\otimes/\lax}^{ h,s,!,sfin} (\nCalg,\bD)\ar[r]_{\simeq}^{\incl^{*}}&\Fun_{\otimes/\lax}^{ h,s,!} (\nCalg_{\sepa},\bD)}\ .\]
The inverse of the lower horizontal morphism is the left-Kan extension functor.
It preserves symmetric monoidal functors by same argument as  for \cite[Prop. 3.8]{KKG}.
  \end{proof}

 \begin{ddd}\label{ojtowpgerwferfw1}\mbox{}\begin{enumerate}
 \item\label{ojtowpgerwferfw} We define the $KK$-theory for $C^{*}$-algebras by $$\KK:=\KK_{\se} \ , \qquad \kk:=\kk_{\se}:\nCalg\to \KK\ .$$
  \item We define the $E$-theory for $C^{*}$-algebras by $$\EE:=\KK_{\exa} \ , \qquad \ee:=\kk_{\exa}:\nCalg\to \EE\ .$$
 \end{enumerate}
 \end{ddd}

 \begin{rem}The universal properties of $KK$ and $E$-theory are given by Theorem \ref{wergijweogrefwrferwfw}.

 Thus $\kk:\nCalg\to \KK$ is the universal functor to a cocomplete stable $\infty$-category which is
 homotopy invariant, stable, semiexact and s-finitary. The category $\KK$ has  presentably  symmetric monoidal structures 
 $\otimes_{?}$ for $?$ in ${\min,\max}$, and the functor $\kk$ has corresponding symmetric monoidal refinements
 which have an analoguous universal property for cocomplete stable  test categories  with bi-cocontinuous symmetric monoidal structures.

The functor   $\ee:\nCalg\to \EE$ is the universal functor to a cocomplete stable $\infty$-category which is
 homotopy invariant, stable,  exact and s-finitary. 
  The category $\EE$ has a presentably  symmetric monoidal structure
 $\otimes_{\max}$, and the functor $\ee$ has a corresponding symmetric monoidal refinement 
 which has an analoguous universal property for cocomplete stable  test categories  with bi-cocontinuous symmetric monoidal structures. 
 
Since exactness is a stronger condition than semiexactness we have a  canonical  comparison functor fitting in to a triangle
\[\xymatrix{&\nCalg\ar[dr]^{\ee}\ar[dl]_{\kk}\\ \KK\ar[rr]&&\EE}\]
 which commutes up to a natural transformation.
Under certain conditions it induces an equivalence on mapping spaces, see \cref{wtrkgoerfrefwerfwerf} for a detailed statement.

   If $A, B$ are separable $C^{*}$-algebras, then by \cref{weoktgpwergwerfwrf}.\ref{ergijowerfwerfwerf} we have equivalences 
 \begin{equation}\label{fvwiuehviuervwvc}\KK_{\sepa}(A,B)\simeq \KK(A,B)\ , \quad \EE_{\sepa}(A,B)\simeq \EE(A,B)\ .
\end{equation} 
  \hB
\end{rem}


\begin{rem}\label{weroigwergfewrfrwef}
We apologize for introducing an incompleteness of the presentation by  deferring the proof of \cref{weropgkrpeogrfw} to the reference
 \cite[Lemma 3.2]{KKG}.
But let us point out that the argument for \cite[Lemma 3.2]{KKG}  only employs elementary facts about $C^{*}$-algebras and their tensor products
and not any deeper parts from $KK$-theory. It therefore should be directly accessible for readers
having reached this point of the present paper.
The same applies to the argument that $\kk_{!}$ is actually symmetric mononidal (in contrast to being lax symmetric monoidal) which is deferred to \cite[Prop. 3.8]{KKG}.
This argument is also by elementary $C^{*}$-algebra theory, but the case of the maximal tensor product is more involved since it uses \cite[Lemma 7.18]{KKG} which does not seem to be so standard. \hB
 \end{rem}

 \section{$K$-theory and the  stable group of unitaries} \label{wreijgowertgwrgrwfrefwrf}

 Using $E$-theory we can give  a simple construction of a highly structured version of the topological $K$-functor for $C^{*}$-algebras. The main goal of this section is to relate this $K$-theory functor with the stable 
unitary group functor from  \eqref{vfevsdfvrqfccscsdc}.   
 
 In order to construct the $K$-theory functor we shall use the following general facts.
 \begin{rem}
 If $\bC$ is a stable symmetric monoidal $\infty$-category with tensor unit $\beins$,
 then the functor $\map_{\bC}(\beins,-):\bC\to \Sp$ is lax symmetric monoidal. 
Since $\beins$ is naturally a commutative algebra object  in $\bC$  we get a commutative ring spectrum
$R:=\map_{\bC}(\beins,\beins)$. The $\infty$-category  $\bC$ has then a natural $R$-linear structure.
In particular, its mapping spectra $\map_{\bC}(C,D)$ naturally  refine to  objects of $\Mod(R)$ such that the composition is $R$-bilinear. \hB
 \end{rem}

We will apply this to the symmetric monoidal $\infty$-category $\EE$. Its tensor unit is given by 
 $\beins_{\EE}:=\ee(\C)$. \begin{ddd}\label{refjweoirgwrefffwerfwrf} We define the commutative ring spectrum $\KU:=\EE(\C,\C)$ in $\CAlg(\Sp)$. \end{ddd} 
 We refer to \cref{qeorigjiowergwerfwfrewfwreee} for a justification of the notation. The stable $\infty$-category 
 $\EE$ becomes a $\KU$-linear stable $\infty$-category.  In particular, its mapping spectra $\EE(A,B)$ naturally belong to $\Mod(\KU)$.
 
 For the moment   we consider the maximal tensor product on $\nCalg$.
In order to incorporate the minimal tensor product see \cref{dbsibjiowferfewrfref}.

 \begin{ddd}\label{wtgijwoergferwfrefw}
 The lax symmetric monoidal {\em topological $K$-theory functor} for $C^{*}$-algebras is defined by
$$K:=\EE(\C, -):\nCalg\to \Mod(\KU)\ .$$
 \end{ddd}
 By  construction, $K$ is homotopy invariant, stable and  exact.
Since $\C$  is separable, the object $\ee(\C)$ in $\EE$ is compact. Hence, $s$-finitaryness of  $\ee:\nCalg\to \EE$ implies
that the
$K$-theory functor  is also $s$-finitary.

Recall  the stable unitary group functor  $\cU^{s}$ from \eqref{vfevsdfvrqfccscsdc}.   
\begin{prop}\label{trgjkowergferfweferfw}
We have a canonical equivalence  of functors
\begin{equation}\label{regfwerfrefweff}\  \cU^{s}\simeq \Omega^{\infty-1}K:\nCalg \to  \CGroups(\Spc)\ .\end{equation} 
\end{prop}
\begin{proof}
Using \cref{wtgijwoergferwfrefw}, stability of the $\infty$-category $\EE$ and  \cref{qerigqerfewfq9}
 we get an equivalence
\begin{equation}\label{ferwfrefrefefwrfrfwf}\Map_{\EE}(S(\C),- ) \simeq \Omega^{\infty-1}   \EE(\C,-) \simeq \Omega^{\infty-1}K(-) \ .
\end{equation} 
We furthermore have a transformation of $\CGroups(\Spc)$-valued functors
\begin{equation}\label{regerwfrefrefwerf}\hspace{-0.5cm} \cU^{s}_{\sepa}(-)\stackrel{ \eqref{vfevsdfvrqfccscsdc}}{\simeq} \Map_{L_{K}\nCalg_{\sepa}} (S(\C),-) \xrightarrow{\Omega^{2}_{\sepa,\exa}\circ L_{\sepa,\exa}}  \Map_{\EE_{\sepa}} (S(\C),-)\stackrel{\eqref{fvwiuehviuervwvc}}{\simeq}    \Map_{\EE}(S(\C),-)_{|\nCalg_{\sepa}}\ ,
\end{equation} 
where $\cU^{s}_{\sepa}$ is the restriction of $\cU^{s}$ to separable algebras.
 We now employ the following facts.
\begin{lem}\label{wetgijowgrfrefwf}
The composition \eqref{regerwfrefrefwerf} is an equivalence.
\end{lem}
\begin{lem}\label{weojgiwoerfrefweferf}
The functor $\cU^{s}$ preserves small filtered colimits and is in particular  $s$-finitary.
\end{lem}
Combining both results we get the desired equivalence \eqref{regfwerfrefweff}
by left-Kan extending the equivalence \eqref{regerwfrefrefwerf} 
along $\nCalg_{\sepa}\to \nCalg$ and composing with  \eqref{ferwfrefrefefwrfrfwf}.
\end{proof}


  \begin{kor}\label{erigoeggerfwerf9}
The $K$-theory functor $K:\nCalg\to \Sp$ preserves small filtered colimits.
\end{kor}
\begin{proof}
We combine \cref{weojgiwoerfrefweferf} with   \cref{trgjkowergferfweferfw}
and two-periodicity. 
\end{proof}

\begin{rem}
Using that classical $E$-theory for separable $C^{*}$-algebras preserves countable sums  \cite[Prop. 7.1]{Guentner_2000} one can show using \cref{qiurhfgiuewrgwrfrefrfwrefw}
that $\ee_{\sepa}$ preserves countable sums. This implies by 
\cite[3.17]{Bunke:2024aa} that $\ee_{\sepa}$ preserves all countable filtered colimits.  Since $\ee(\C)$ is a compact object of $\EE$ this would give an alternative argument for the fact that $K$ preserves filtered small colimits. \hB
\end{rem}


\begin{rem}\label{weroijgwergrefwrfrwef}
In the proof of \cref{wetgijowgrfrefwf} we will employ the following general fact about mapping spaces in a Dwyer-Kan localization  $\ell:\bC\to \bC[W^{-1}]$ of $\infty$-categories.  We call an object   $C$  of $\bC$  {\em colocal for $W$}
if the functor $\Map_{\bC}(C,-)$ sends the elements of $W$ to equivalences. 
The following assertion is an easy consequence of the Yoneda lemma.
If $C$ is colocal for $W$, then 
$\ell:\Map_{\bC}(C,-)\to \Map_{\bC[W^{-1}]}(\ell(C),\ell(-))$ is an equivalence of functors from
$\bC$ to $\Spc$.
\hB
\end{rem}

\begin{proof}[Proof of   \cref{wetgijowgrfrefwf}]
We must show that the composition  
\begin{equation*}\Map_{L_{K}\nCalg_{\sepa}} (S(\C),-) \xrightarrow{  L_{\sepa,\exa}} 
 \Map_{L_{K}\nCalg_{\sepa,h,\exa}} (S(\C),-) \xrightarrow{ \Omega^{2}_{\sepa,\exa}} 
 \Map_{\EE_{\sepa}} (S(\C),-)\end{equation*}
 is an equivalence. Since $S(\C)$ represents  a group in $L_{K}\nCalg_{\sepa,h,\exa}$
 and $\Omega^{2}_{\sepa,\exa}$ is by \cref{wtgkwotpgkelrf09i0r4fwefwerf} a right Bousfield localization at the groups  the second morphism is an equivalence.
 By \cref{weroijgwergrefwrfrwef}, in order to show that the Dwyer Kan localization  $ L_{\sepa,\exa}$ induces  an equivalence  of mapping spaces
 it suffices to show that $L_{\sepa,h,K}(S(\C))$  is colocal for $W_{\sepa,\exa}$.
 Since $\Map_{L_{K}\nCalg_{\sepa}} (S(\C),-)$ is left-exact, by \eqref{erfwerferfwef} it suffices to show that
 $L_{\sepa,h,K}(S(\C))$ is colocal for $\hat W_{\sepa,\exa}$ from \eqref{erijoglerfqwfwerfrefwrf}.
 By \cref{wetigowergferferferwfw} it suffices to show that
 $\Map_{L_{K}\nCalg_{\sepa}} (S(\C),-)$  sends exact sequences to fibre sequences.
 In view of \eqref{vfevsdfvrqfccscsdc} 
 this follows from the following Lemma since
 $\ell$ sends Serre fibre sequences to fibre sequences.
 
 \begin{lem}\label{wejiogwergwerfrrewf}
 If $0\to A\to B\to C\to 0$ is an exact sequence in $\nCalg$, then
 $U^{s}(B)\to U^{s}(C)$ is a  Serre fibration with fibre 
 $U^{s}(A)$.
 \end{lem}
 \begin{proof}
 This lemma is surely well-known in $C^{*}$-algebra theory. For completeness of the presentation we   add a proof.
 
 It is clear from the definition \eqref{urfuherufhwerjkfekrf89} that $U^{s}(A)$ is the fibre of the map $U^{s}(B)\to U^{s}(C)$. In order to show that this map is a Serre fibration
  we will solve lifting problem 
\[\xymatrix{X\ar[r]\ar[d]_{x\mapsto (0,x)}&U^{s}(B)\ar[d]\\ [0,1]\times X\ar[r]\ar@{..>}[ur]&U^{s}(C)}\]
for all compact $X$. By \eqref{fqwfewfdwedewdwdq} and \eqref{qefeoijoi4r34r34}  this lifting problem is equivalent to 
\[\xymatrix{\{0\}\ar[r]\ar[d] &U^{s}(C(X)\otimes B)\ar[d]\\ [0,1 ]\ar[r]\ar@{..>}[ur]&U^{s}(C(X)\otimes C)}\ .\]
It thus suffices to solve 
 the
 path lifting problems \begin{equation}\label{eqwfq98whuqwe9dqedqd}\xymatrix{\{0\}\ar[r]^{u}\ar[d] &U^{s}(B)\ar[d] \\ [0,1] \ar[r]^{\gamma}\ar@{..>}[ur]^{\tilde \gamma}&U^{s}(C)}
\end{equation}
 for all surjections $B\to C$.

We call 
a  path $\sigma:[0,1]\to U^{s}(C)$  
short if  $\sigma(0)=1$ and
 $\|\sigma(t)-1\|<1$ for all $t$.
For the moment we 
assume that we can lift short paths to paths that  start in $1$ in $U^{s}(B)$.
 
Let $\gamma:[0,1]\to U^{s}(C)$  be a general path.  
 Then we can find  $n$  in $\nat$ such that the segment $\gamma(i/n)^{-1}\gamma_{|[i/n,(i+1)/n]}$ is short for all $i=0,\dots,n-1$  (we implicitly reparametrize).
We can now lift $\gamma$ inductively.
We are given the lift  
 $u$ of $\gamma_{|[0,0]}$. Assume  that 
have found a lift  $\tilde \gamma$  of $\gamma_{|[0,i/n]}$.
 We choose a lift  $\tilde \sigma$ of the  short path $\gamma^{-1}(i/n) \gamma_{|[i/n,(i+1)/n]}$ and
define an extension of $\tilde \gamma$ on $[i/n,(i+1)/n]$ by $\tilde \gamma(i/n) \tilde \sigma$.

 It remains to  solve the  lifting problem for short paths. To this end we observe that 
%
%
%
%
%
%
%
%
%
%
%
%
the exponential map of $(K\otimes A)^{u}$ restricts to $\exp:i(K\otimes A)^{\sa}\to U^{s}(A)$ with the  partial inverse 
$\log:\{U\in U^{s}(A)\mid \|U-1\|<1\}\to i(K\otimes A)^{\sa}$.
Since the tensor product preserves surjections,
the
  map  
 $C_{0}((0,1])\otimes K\otimes B\to C((0,1])\otimes K\otimes C$  and its restriction to anti-selfadjoint elements are surjective. 
 We 
interpret $\log(\sigma)$ as  element in $i(C_{0}((0,1])\otimes K\otimes C)^{\sa}$
and can thus 
choose a lift $\widehat{\log(\sigma)}$ in  $ i(C_{0}((0,1])\otimes K\otimes B)^{\sa}$.
Then
$\tilde \sigma:=\exp(\widehat{\log(\sigma)}):[0,1]\to U^{s}(B)$ is the desired lift of $\sigma$.
We have thus shown \cref{wejiogwergwerfrrewf}.
  \end{proof}
  This finishes the proof of \cref{wetgijowgrfrefwf}
  \end{proof}
  
  \begin{rem}
  Using that $S(\C)$ is a semi-projective $C^{*}$-algebra on could deduce the  path lifting  in \eqref{eqwfq98whuqwe9dqedqd} from \cite[Thm. 5.1]{Blackadar_2016}. \hB \end{rem}

 \begin{proof}[Proof of  \cref{weojgiwoerfrefweferf}]
 In view of \cref{wergijwerogerwfewrferwfw} it suffices to show that 
 the functor $\cU^{s}:\nCalg\to \CGroups(\Spc)$ preserves small filtered colimits.
Since the forgetful functor $\CGroups(\Spc)\to \Spc$ preserves small filtered colimits and is conservative it suffices to show that  the underlying $\Spc$-valued functor of $\cU^{s}$ preserves small filtered colimits.
 
 Let  $I$ be a small filtered poset,   $(B_{i})_{i\in I}$ be  an $I$-indexed family in $\nCalg$, and set 
 $B:=\colim_{i\in I} B_{i}$ in $\nCalg$. Then we must show that the canonical maps
\[ \pi_{n}(\colim_{i\in I}  \ell U^{s}(B_{i}))\to  \pi_{n}(\ell U^{s}(B))\] are isomorphisms at all choices of base points and for all $n$ in $\nat$. 
Since taking      
  homotopy groups/sets on $\Spc$ commute with filtered colimits 
 it suffices to 
 show that the canonical maps  $\colim_{i\in I}  \pi_{n}(  U^{s}(B_{i}))\to  \pi_{n}(U^{s}(B))$ are isomorphisms.
 This will be an immediate consequence of \cref{wtgijrtogwrefrefr} below applied to
 the inclusions $S^{n}\to D^{n+1}$ or $\emptyset\to S^{n}$.
 
 For   $i,j$ in $I$ with  $i\le j$
let
 $\phi_{j,i}:(K\otimes B_{i})^{u}\to (K\otimes B_{j})^{u}$  and
 $\phi_{i}: (K\otimes B_{i})^{u}\to  (K\otimes B )^{u}$  denote   the   connecting map and the canonical homomorphism.
 
 Let  $X$  be any  compact metrizable  space and
  $Y$ be a closed subspace.   
  We fix $i_{0}$ in $I$  and assume that we are given  a square
  $$ \xymatrix{Y\ar[r]^-{f}\ar[d] &U^{s}(B_{i_{0}}) \ar[d]^{\phi_{i_{0}}} \\X \ar[r]^-{g} & U^{s}(B) } \ .$$

 \begin{lem}\label{wtgijrtogwrefrefr}
 There exists  $i$ in $I$ with $i\ge i_{0}$  and $h:X\to  U^{s}(B_{i})$ such that $h_{|Y}=\phi_{i,i_{0}}\circ f$ and
 $\phi_{i}\circ h$ is homotopic to $g$  rel $Y$.    \end{lem}
\begin{proof}\mbox{}
For any $C^{*}$-algebra $A$ we set  $G^{s}(A):=\{a\in GL_{1}((K\otimes A)^{u})\mid  1-a\in K\otimes A\}$. Then 
$U^{s}(A)\subseteq G^{s}(A)$ and 
the polar decomposition provides a retraction $W:G^{s}(A)\to U^{s}(A)$.
There exists a 
  $c$ in $(0,1)$   such that $\max\{\|a^{*}a-1\|, \|aa^{*}-1\|\}\|\le c $ implies $\|W(a)-a \|\le 1/10$.
We interpret a map $X\to G^{s}(A)$ as a point in $G^{s}(C(X)\otimes A )$.
We will write $\phi_{ij}$ and $\phi_{i}$ instead of $\id_{C(X)}\otimes \phi_{ij}$ and $\id_{C(X)}\otimes \phi_{i}$.
We use the general fact  that a 
filtered colimit  in $\nCalg$  is formed by taking the completion of the pre-$C^{*}$-algebra  given by  the  filtered colimit of underlying sets   equipped with the induced algebraic  structures, see \eqref{fvrvfevefcvsdfvsfdvsfdvsv}. We furthermore use that 
 the (maximal) tensor product preserves filtered colimits, and 
that we can calculate norms in a filtered colimit as limits of norms.  The last statement says e.g.  that
  for $h$ in $C(X)\otimes K\otimes B_{i}  $ we have  $  \|\phi_{i}(h)\|=\lim_{j\in I, i\le j} \|\phi_{j,i}(h)\|$.
 
 We use Dugundji's extension theorem  in order to   find an extension $f_{0}$ of $f$ in $ (C(X)\otimes K \otimes B_{i_{0}} )^{u}$ such that $  f_{0}-1\in C(X) \otimes K\otimes  B_{i_{0}} $. Then $g_{0}:=g-\phi^{u}_{i_{0}}(f_{0}) \in C_{0}(X\setminus Y)\otimes K\otimes B$. 
 We can now find 
 $i_{1}$ in $I$ such that there exists
 $r$ in $ C_{0}(X\setminus Y)\otimes K\otimes B_{i_{1}}$  with
 $\|\phi_{i_{1}}(r)- g_{0}\|\le c/100$.
 We set
 $f_{1}:=\phi^{u}_{i_{1},i_{0}}(f_{0})+r$.
 Then $\|\phi_{i_{1}}(f_{1})-g\|\le c/100$ and hence
   $\max\{\|\phi_{i_{1}}(f_{1}f_{1}^{*})-1\|,\|\phi_{i_{1}}(f_{1}^{*}f_{1})-1\|\} \le c/10$.
 We can then find   $i$ in $I$ with $i\ge  i_{1}$ such that
  $\max\{\|\phi_{i,i_{1}}(f_{1})\phi_{i,i_{1}}(f_{1}^{*})-1\|,\|\phi_{i,i_{1}}(f_{1}^{*})\phi_{i,i_{1}}(f_{1})-1\|\} \le c/3$.
 We define
   $h:=W(\phi_{i,i_{1}}(f_{1}))$
  in $ U^{s}(C(X)\otimes B_{i})$.
 Then 
  $h_{|Y}=\phi_{i,i_{0}}(f)$
 and 
 $\|\phi_{i}(h)-g\|\le 1/2$. We get a  homotopy 
  $(W((1-s)\phi_{i_{0}}(h)+sg))_{s\in [0,1]}$  from $\phi_{i}(h)$ to $g$ rel $Y$
  in $U^{s}(B)$.
  This finishes the proof of \cref{wtgijrtogwrefrefr}.
 \end{proof}
 
 \begin{rem}
 Using that $S(\C)$ is semi-projective we could deduce \cref{wtgijrtogwrefrefr} directly
 from the proof of \cite[Prop. 3.8]{Bunke:2024aa}, in particular from the existence of the lift in 
  \cite[Eq. (3.6)]{Bunke:2024aa}. \hB
 \end{rem}

We have now finished the proof of \cref{weojgiwoerfrefweferf}.
\end{proof}

 The reason for  using $E$-theory in order to construct the $K$-theory functor for 
$C^{*}$-algebras  was that then exactness of the latter is true by construction.
We have a canonical  transformation
\[\Omega^{\infty-1}\KK(\C,-)\to \Omega^{\infty-1}\EE(\C,-)\simeq  \Omega^{\infty-1}K(-)\ .\]
The arguments above work equally well for  $KK$-theory (just replace $\exa$ by $\se$) and show that the canonical transformation
$\cU^{s}(-)\to \Omega^{\infty-1}\KK(\C,-)$ is an equivalence.
Using Bott periodicity we can then conclude an equivalence of functors 
\begin{equation}\label{fewqifjeqwfewfqewfd}\KK(\C,-)\stackrel{\simeq}{\to} K(-):\nCalg\to \Sp
\end{equation}  
As a consequence, the stable $\infty$-category  $\KK$  also naturally acquires a  $KU$-linear structure. Furthermore we conclude:
\begin{kor}\label{dbsibjiowferfewrfref}
The $K$-theory functor for $C^{*}$-algebras $K:\nCalg\to \Mod(\KU)$
has a lax symmetric monoidal refinement for the minimal tensor product on $\nCalg$.
\end{kor}
The following is true for both  tensor products on $\KK$.
We  have a limit-preserving  symmetric monoidal  functor \[\cK:=\KK(\C,-):\KK\to \Mod(\KU)\]
such that $K\simeq \cK\circ \kk$.  Since  $\kk(\C)$ is compact in $\KK$ this functor also preserves colimits. It is the right-adjoint of a symmetric monoidal right Bousfield localization
\begin{equation}\label{wefrfwgergegrfwrefefwrefweferfw}\cB:\Mod(\KU)  \leftrightarrows \KK  :\cK\ .
\end{equation} 
\begin{ddd}
The essential image of the left-adjoint $\cB$ in \eqref{wefrfwgergegrfwrefefwrefweferfw} is called the {\em $\UCT$ class}. 
 \end{ddd}
Equivalently, the  $\UCT$-class is the localizing subcategory of $\KK$ generated by the tensor unit $\kk(\C)$.  \begin{kor}\label{wtrkgoerfrefwerfwerf}\mbox{}
If $B$ is in the  $\UCT$-class, then we have  the following assertions:
\begin{enumerate}
\item \label{qrfijqiofdewdqwed} The natural transformation
$-\otimes_{\max}B\to -\otimes_{\min} B$ of endofunctors on $\KK$ is an equivalence.
\item \label{qrfijqiofdewdqwed1}  The  transformation
$\KK(B,-)\to \EE(B,-)$ of functors $\nCalg\to \Mod(\KU)$ is an equivalence.
\item (UCT) We have an equivalence $\KK(B,-)\simeq \map_{\Mod(\KU)}(\cK(B),\cK(-))$ of functors from $\KK$ to $\Mod(\KU)$.
\item (Künneth formula)
We have an equivalence $\cK(-)\otimes_{\KU} \cK(B)\simeq \cK(-\otimes B)$
of functors from $\KK$ to $\Mod(\KU)$.
\end{enumerate}
\end{kor}
\begin{proof}
In all cases the equivalence is induced by an obvious natural transformation. 
One argues that the full subcategory of objects $B$ in 
$\KK$ for which is transformation is an equivalence is localizing,  and that it contains  $\kk(\C)$.
  \end{proof}

\begin{rem}
The
  $\kk$-functor from \cref{ojtowpgerwferfw1}.\ref{ojtowpgerwferfw} is not compatible with filtered colimits on the level of $C^{*}$-algebras. It does not even preserve countable sums. The reason is that the functor $y:\KK_{\sepa}\to \KK$ does not preserve countable sums. 
One could improve on this point  by observing  that  $\kk_{\sepa}$     preserves countable sums \cref{weijotgwegferfwrefw}.\ref{fhqdewjfioqwefqdwedqwed1}, and then
working with the $\aleph_{1}$-Ind-completion instead of the Ind-completion in \cref{iejgowergrfrfwerfref}. We refer to  
\cite[Sec. 3.4]{Bunke:2024aa} where the details of such a construction have been worked out in the case of $E$-theory, see also \cite[Rem. 3.4]{KKG}

In the context of the present paper, in order to discuss the relation of the $\UCT$-class with the classical definition, it is better to consider
the separable version $\UCT_{\sepa}$ defined as a smallest countably cocomplete stable $\infty$-category of $\KK_{\sepa}$ (which is known to be countably cocomplete by   \cref{weijotgwegferfwrefw}.\ref{fhqdewjfioqwefqdwedqwed}) containing the tensor unit, see also \cite[Sec. 5.5]{Bunke:2023ab}.
It is a famous question whether for every nuclear separable algebra $A$ we have  $\kk_{\sepa}(A)\in \UCT_{\sepa}$.

The analogs of the statements of \cref{wtrkgoerfrefwerfwerf} in the separable case hold true.
\hB
\end{rem}

 \begin{rem}
 In the  statement \cref{wtrkgoerfrefwerfwerf}.\ref{qrfijqiofdewdqwed}
we can replace the condition that $B$ is in the $\UCT$-class by the condition that $B$ is represented by a separable and nuclear $C^{*}$-algebra. Indeed, by definition of nuclearity 
 the transformation $-\otimes_{\max}B\to \otimes_{\min}B$ of endofunctors
of $\nCalg_{\sepa}$ is an isomorphism.

 Classically it is known that for separable algebras the map from $KK$ to $E$-theory
 is an isomorphism if the first argument is a nuclear algebra 
 \cite[Thm. 3.5]{MR1068250}.
By stability and \cref{qiurhfgiuewrgwrfrefrfwrefw} we can conclude that in 
   \cref{wtrkgoerfrefwerfwerf}.\ref{qrfijqiofdewdqwed1} we
 can replace the $\UCT$-condition on $B$ by the condition that $B$ is represented by a separable and nuclear $C^{*}$-algebra. 
\hB
     \end{rem}

\begin{rem}\label{qeorigjiowergwerfwfrewfwreee}
In this remark we justify \cref{refjweoirgwrefffwerfwrf}.
By a similar argument as in the proof of   \cref{weojgiwoerfrefweferf} we get  a weak equivalence 
\[\colim_{n\in \nat} \ell U(n)\cong \cU^{s}(\C)\]
in $\Groups(\Spc)$.
The left-hand side is $\Omega^{\infty-1}$ of one of the  classical versions of the $\KU$-spectrum.

We continue the justification of  \cref{refjweoirgwrefffwerfwrf} by calculating the ring $\pi_{*}\KU$.    Since the homotopy groups of $ \cU^{s}(\C)\simeq \Omega^{\infty-1}K(\C)$ are two-periodic, so are the homotopy groups on the left-hand side. We thus deduce the classical Bott periodicity theorem. For an explicit calculation
 we use the following  additional information from classical topology:
  \[\pi_{i}(\colim_{n\in \nat} U(n))\cong \left\{\begin{array}{cc} *&i=0 \\ \Z&i=1\\0&i=2  \end{array} \right.\]  in order to conclude that 
 \begin{equation}\label{qefoihwefioewdqewdqwedwqedwqedwdwed} \pi_{i} \Omega^{\infty}K(\C)\cong \left\{\begin{array}{cc} \Z&i\in 2\nat \\ 0&i\in 2\nat+1   \end{array} \right.\ .
 \end{equation}   We know further that $\pi_{*} \KU$ is a ring and that the Bott periodicity is implemented by the multiplication with
  the invertible element $b$ in $\pi_{-2} \KU\simeq \pi_{0}\KK(S^{2}(\C),\C)$ from \eqref{vewqcercerwce}.
  Consequently, $b^{-1}$ must be a generator of
  $\pi_{2}\KU$ and we get a ring  isomorphism
  \[\Z[b,b^{-1}]\stackrel{\cong}{\to} \pi_{*} \KU\ .\]
  \hB
 \end{rem}

\section{$K$-theory and the group completion of the space of projections}\label{reoijqoeirfewfqwdewd}
 
For a unital  $C^{*}$-algebra $A$ the abelian  group   $K_{0}(A)$ is classically defined
as the Grothendieck group of the monoid of unitary equivalence classes of projections in $K\otimes A$, where the unitaries belong to the multiplier algebra $U(K\otimes A)$. Thereby the  monoid operation is induced by the block sum.
  One then observes that the relation of   unitary equivalence between projections is equivalent to homotopy. 
  Using the notation from \cref{poregjopewrgerg} we thus get an isomorphism
  \begin{equation}\label{fdsvdsvfsvwevdfsvfdv} \pi_{0} (\cProj^{s}(A))^{\group} \cong  K_{0}(A)\ .
\end{equation} 
  In this section   we will show a space-level refinement of  this isomorphism.
  Unfolding definitions we obtain  a   natural map
  \[\cProj^{s}(A)\to \Omega^{\infty}K(A)\]
   of commutative monoids in spaces. 
 Then \cref{erijgoerwgregffwfrefw}    asserts that 
  this map presents its target as a group completion.
 Note that the unitality assumption on $A$ is crucial for this statement, see \cref{qvpojsdfdvfdvsv}.
 The modification for general algebras 
 is formulated as \cref{wogkpwgerfwrefwf}. Though it looks like an obvious  $K$-theoretic statement 
 its detailed verification is surprisingly long.

\begin{ex} \label{qvpojsdfdvfdvsv}
Let $X$ be a locally compact Hausdorff space which is connected and not compact. 
 Then we have $\Proj^{s}(C_{0}(X))=\{0\}$. Indeed, a projection 
 $p$ in $\Proj^{s}(C(X) )$ can be interpreted as a  
 function $p:X\to \Proj^{s}(\C)$. The function
 $x\mapsto \|p(x)\|$ is continuous and takes values in $\{0,1\}$. Since it vanishes at infinity, the assumptions on $X$ imply that it vanishes identically.
 
 We know from \cref{qeorigjiowergwerfwfrewfwreee} that $K_{0}(C_{0}(\R^{2}))\cong \Z$, and this contradicts
 \eqref{fdsvdsvfsvwevdfsvfdv} whose left-hand side would be the zero group.
 \hB\end{ex}

The following two statements enable us to study the space of projections within the   homotopy theory developed in the present notes.   They are the analogs of \cref{wejiogwergwerfrrewf} and \cref{weojgiwoerfrefweferf}.
Recall from \cref{poregjopewrgerg} that $\Proj(A):=\underline{\Hom}(\C,A)$ denotes the topological space of projections in $A$ and that $\cProj(A):=\ell\Proj(A)$ is the associated space.
\begin{prop}\label{wtogjgergregw9}
For every  surjective map   $  B\to C $ of $C^{*}$-algebras the map $\Proj(B)\to \Proj(C)$ of topological spaces is a Serre fibration. 
\end{prop}

\begin{prop}\label{wtogjgergregw91}
The functor $\cProj:\nCalg\to \Spc$ preserves small filtered colimits.
\end{prop}

We defer the technical proofs of these statements to the end of the section.

\begin{prop}\label{wejkgoowpfrefwfrefwef}
The functor $\cProj^{s}:\nCalg\to \CMon(\Spc)$ is homotopy invariant, stable,   Schochet exact, and $s$-finitary.
It furthermore sends cartesian squares 
\[\xymatrix{A\ar[r]\ar[d] & B\ar[d]^{f} \\C \ar[r] & D} \] in $\nCalg$ with the property that
  $f$ a surjection to cartesian squares.
\end{prop}
\begin{proof}
The functor $\cProj^{s}(-)\stackrel{\eqref{vwihviuewhviwevewre}}{\simeq} \Map_{L_{K}\nCalg_{h}}(\C,-)$ is homotopy invariant and stable
by definition. It furthermore sends Schochet fibrant cartesian squares to cartesian squares
since $L_{h,K}$ does so.

Since the functor $K\otimes -$ preserves   filtered colimits, it follows from   \cref{wtogjgergregw91}
that $\cProj^{s}(-)\simeq \cProj(K\otimes -)$  preserves   filtered colimits and is in particular   $s$-finitary.

Since
\[\xymatrix{K\otimes A\ar[r]\ar[d] & K\otimes B\ar[d]^{K\otimes f} \\K\otimes C \ar[r] & K\otimes D} \]
is again cartesian and $K\otimes f$ is still surjective,  the functor $\Proj:=\underline{\Hom}(\C,-)$  sends this square to a  cartesian square in $\Top$ which is in addition Serre fibrant by \cref{wtogjgergregw9}.
We now apply 
 $\ell$ and   get the desired cartesian square 
\[\xymatrix{\cProj^{s}(A)\ar[r]\ar[d] & \cProj^{s}(B)\ar[d]^{\cProj^{s}(f)} \\\cProj^{s}(C) \ar[r] & \cProj^{s}(D)} \ .\]
\end{proof}

  The group completion functor $(-)^{\group}$ is defined as the left-adjoint of a 
  Bousfield localization 
\begin{equation}\label{fdvdsvsfdvdfv3ffewerfreferwf}(-)^{\group}:\CMon(\Spc)\leftrightarrows \CGroups(\Spc):\incl
\ .\end{equation} 
Recall that to any $C^{*}$-algebra $A$ we can functorially associate the split unitalization sequence
\begin{equation}\label{afoiahjfoiaf}0\to A\to A^{u}\to \C\to 0\ .
\end{equation}
\begin{ddd} We define the functor  
\begin{equation}\label{vfdvdfsvq3fdfvfvvfvdfsvsvfvs}\rProj^{s}:\nCalg \to \CGroups(\Spc)\ , \qquad A\mapsto \Fib(\cProj^{s} (A^{u})^{\group}\to \cProj^{s} (\C)^{\group})
\end{equation} 
\end{ddd}

\begin{rem}
Observe that in the definition of $\rProj^{s}$ we  take the group completion  first  and then the fibre.
We can not reverse the order since the group completion does not preserve fibre sequences in general. \hB
 \end{rem}

We define the natural transformation
\begin{equation}\label{vsdfhuiwhiuvhjfuivsdfvsfdvs}e_{h}:\cProj^{s}(-)\simeq \Map_{L_{K}\nCalg_{ h,K}}(\C,-)\xrightarrow{\Omega^{2}_{\exa}\circ L_{ \exa}}  \Map_{\EE}(\C,-)\simeq \Omega^{\infty}K(-)
\end{equation} 
of functors from $\nCalg $ to $\CMon(\Spc)$.
Since $\Omega^{\infty}K$ takes values in groups,   by the universal  property of the  group completion   we get the dotted arrow in the commutative diagram
of functors $\nCalg\to \CMon(\Spc)$
\[\xymatrix{\cProj^{s} \ar[dr]\ar[rr]^{\ee_{h}}&&\Omega^{\infty}K  \\&\cProj^{s,\group}\ar@{..>}[ur]^{\hat \ee_{h}}& }\ ,\]
where the down-right arrow is  the  counit of the adjunction \eqref{fdvdsvsfdvdfv3ffewerfreferwf}.

We next 
construct  a natural transformation \begin{equation}\label{gwegergergreffewrwf}\tilde \ee_{h}:\rProj^{s} \to \Omega^{\infty}K \ .\end{equation}  
Applying the exact   functor 
 $\Omega^{\infty}K$  to the split-exact unitalization sequence \eqref{afoiahjfoiaf}     we get the split fibre sequence \[\Omega^{\infty}K(A) \to   \Omega^{\infty}K(A^{u})\to \Omega^{\infty}K(\C)\]
  in $\CGroups(\Spc)$. 
We now form the 
 diagram of vertical fibre sequences of functors from $\nCalg$ to $\CGroups(\Spc)$ 
\begin{equation}\label{vervwerwrfwerfrefrewfefw}\xymatrix{\rProj^{s}(-)\ar[d]\ar@{..>}[r]^{\tilde \ee_{h}}&\Omega^{\infty}K(-)\ar[d]\\  \cProj((-)^{u})^{\group}\ar[r]^{\hat \ee_{h}} \ar[d]&\ar[d]\Omega^{\infty}K((-)^{u})\\ \cProj(\C)^{\group}\ar[r]^{\hat \ee_{h}}&\Omega^{\infty}K(\C)}
\end{equation} 
 defining $\tilde \ee_{h}$ as the natural extension of the lower square to a map of fibres.

The following is the main theorem of the present section. 

 \begin{theorem}\label{wogkpwgerfwrefwf}
  The natural transformation $\tilde \ee_{h}:\rProj^{s}\to \Omega^{\infty}K$
  is an equivalence.  
 \end{theorem}

 Before we start with the proof we consider the specialization to unital algebras.
 
 \begin{kor}\label{erijgoerwgregffwfrefw} For every unital $C^{*}$-algebra $A$
 the map
 \begin{equation}\label{bdsbfsfvfvsfdvrferfe}\hat \ee_{h,A}:\cProj^{s}(A)\to \Omega^{\infty}K(A)
\end{equation} 
 presents its target as a group completion.
 \end{kor}
 \begin{proof}
 Since the composition \[\cProj^{s}(A)^{\group}\to \cProj^{s}(A^{u})^{\group}\to \cProj^{s}(\C)^{\group}\] vanishes, by the universal property of the fibre in \eqref{vfdvdfsvq3fdfvfvvfvdfsvsvfvs} we get a canonical morphism
\begin{equation}\label{qewfiuhquiewfqedqwdewd}i_{A}:\cProj^{s}(A)^{\group}\to \rProj^{s}(A)\ .
\end{equation}   If $A$ is unital, then the identity of $A$ canonically extends to a homomorphism $A^{u}\to A$. The composition $\rProj^{s}(A)\to  \cProj^{s}(A^{u})^{\group} \to
  \cProj^{s}(A)^{\group}$ provides an inverse of $i_{A}$. 
  The map in \eqref{bdsbfsfvfvsfdvrferfe} is then equivalent to
  \[\cProj^{s}(A)\to \cProj^{s}(A)^{\group}\stackrel{i_{A}}{\simeq} \rProj^{s}(A)\stackrel{\tilde e_{h,A}}{\simeq} \Omega^{\infty}K(A)\ ,\]
  where the last equivalence is given by \cref{wogkpwgerfwrefwf}.
  This shows the assertion of \cref{erijgoerwgregffwfrefw}.
\end{proof}

All of the above  has a version for separable algebras. The following is the separable version of  \cref{wogkpwgerfwrefwf}.

\begin{prop}\label{wogkpwgerfwrefwf1}
 The natural transformation 
  $\tilde \ee_{\sepa,h}:\rProj^{s}_{\sepa}\to \Omega^{\infty} K_{\sepa}$  is an equivalence.
 \end{prop}

\begin{proof}[Proof of \cref{wogkpwgerfwrefwf} assuming \cref{wogkpwgerfwrefwf1}]
We claim that $\rProj^{s}$ is $s$-finitary. Since $\Omega^{\infty}K$ is also $s$-finitary
we then obtain the equivalence in \cref{wogkpwgerfwrefwf} as a left Kan-extension of the equivalence 
in \cref{wogkpwgerfwrefwf1} along the inclusion $\nCalg_{\sepa}\to \nCalg$.

In order to see the claim  we note that $(-)^{\group}$ in \eqref{fdvdsvsfdvdfv3ffewerfreferwf}  is left adjoint and preserves all colimits. By \cref{wejkgoowpfrefwfrefwef}
the functor $\cProj^{s,\group} $ is also $s$-finitary.
Finally we use that 
the fibre of a filtered colimit of maps in $\CGroups(\Spc)$  is the filtered colimit of the fibres.
 \end{proof}

The following result prepares the proof of \cref{wogkpwgerfwrefwf1}.
 \begin{prop}\label{wiergo0wergferwferwf9}
 The functor $\rProj^{s} :\nCalg \to \CGroups(\Spc)$ is  homotopy invariant, stable and 
 exact.
  \end{prop}
  \begin{proof}
  Homotopy invariance and stability are obvious from the definition. Exactness is much deeper. It can not be concluded simply from the exactness properties of $\cProj^{s}$ stated in \cref{wejkgoowpfrefwfrefwef} since
the group completion functor does not preserve  fibre sequences in general.
The basic insight in our special case is that for unital algebras $A$  the  group completion of $\cProj^{s}(A)$ can be expressed by a specific filtered colimit.

 We first recall  some generalities on group completions following \cite{niko}, \cite{rwo}.
We consider a commutative monoid $X$
  in $\CMon(\Spc)$. For an element $s$ in $X$ we can form  
 \[X_{s}:=\colim (X\xrightarrow{+s}X\xrightarrow{+s}X\xrightarrow{+s}\dots )\]
in $X$-modules. For any 
  finite  ordered set  $\{s_{1},\dots,s_{n}\}$  of elements in  $X$ we define inductively $X$-modules
\[X_{\{s_{1},\dots,s_{n}\}}:=(X_{\{s_{1},\dots,s_{n-1}\}})_{s_{n}}\ .\]
We choose a 
 well-ordering on $\pi_{0}(X)$ and a representative $s$ in $X$ for any component. We then define the space 
  \begin{equation}\label{vafaewqdcadscqewfwefewdewdewdq}X_{\infty}:=\colim_{S\subseteq\pi_{0}(X)} X_{S} .
\end{equation}

\begin{prop}[{\cite[Prop. 6]{niko}}]\label{weriojgogewrweg9}
If the fundamental group of every component of $X_{\infty}$  is abelian, then
$X\to X_{\infty}$ is equivalent to the underlying map of  $X\to X^{\group}$.
\end{prop}

An element 
   $t$ in   $\pi_{0}(X)$ is called 
{\em cofinal } if for every $s$ in $\pi_{0}(X)$ there exists $s'$ in $\pi_{0}(X)$  and $n$ in $\nat $ such that $s+s'=nt$.
 One easily checks that if $t$ is cofinal, then 
the canonical map $X_{t}\to X_{\infty}$ is an equivalence.
In particular if the fundamental groups of all components of $X_{t}$ are abelian, then
 $X\to X_{t}$ is equivalent to the underlying map of the group completion $X\to X^{\group}$.
 
The following results enable us to apply \cref{weriojgogewrweg9} to our problem.
For $C^{*}$-algebras  $A$ and $B$  we have the commutative monoid
 $\Map_{L_{K}\nCalg_{h}}(A,B)$ and can form the space 
  $\Map_{L_{K}\nCalg_{h}}(A,B)_{\infty}$ as in \eqref{vafaewqdcadscqewfwefewdewdewdq}
  \begin{lem}\label{ewrgijwerogrwewfwe9}
If $\pi_{0}\Map_{L_{K}\nCalg_{h}}(A,B)$ contains a cofinal element, then
the fundamental groups of the components of $\Map_{L_{K}\nCalg_{h}}(A,B)_{\infty}$ are abelian.
\end{lem}
\begin{proof}
In the following we use that $\ell\underline{\Hom} (A,K\otimes B)\simeq \Map_{L_{K}\nCalg_{h}}(A,B)$.
 Let $[t]$ be a cofinal element in $\pi_{0}\Map_{L_{K}\nCalg_{h}}(A,B)$ represented by a map 
$t$ in $\underline{\Hom} (A,K\otimes B)$.
Then \[\Map_{L_{K}\nCalg_{h}}(A,B)_{\infty}\simeq \colim (\Map_{L_{K}\nCalg_{h}}(A,B) \xrightarrow{[t]+-}\Map_{L_{K}\nCalg_{h}}(A,B)\xrightarrow{[t]+-} \dots )\ .\]
 We consider a component $x$ in $\pi_{0}\Map_{L_{K}\nCalg_{h}}(A,B)_{\infty}$. Then
there exists $k$ in $\nat$ and $f$ in  the topological space $\underline{\Hom} (A,K\otimes B)$   contributing to $\Map_{L_{K}\nCalg_{h}}(A,B)$
 in the $k$-th stage of the $\nat$-indexed  diagram above which represents  $x$.

 We now  consider elements   $[\gamma]$ and $[\sigma]$ in $\pi_{1}(\Map_{L_{K}\nCalg_{h}}(A,B)_{\infty}, x)$.  We must show that $[\gamma]\circ[\sigma]= [\sigma]\circ [\gamma]$.  After going further in the diagram we can assume that $[\gamma]$ and $[\sigma]$ are represented by loops $\gamma$, $\sigma$ at $f$ in  $\underline{\Hom} (A,K\otimes B)$. 
  
  By cofinality of $[t]$ there exists a map $f'$ and an integer $n$ such that
  $[f]+[f']=n[t]$ in $\pi_{0}\Map_{L_{K}\nCalg_{h}}(A,B)$.
  The $+$-sign in the following denotes choices of block sums.
We know that   $\gamma+nt$ and $\sigma+nt$ viewed as points in 
 $\underline{\Hom} (A,K\otimes B)$  contributing to  $\Map_{L_{K}\nCalg_{h}}(A,B)$ in the $(k+n)$-th stage of the diagram  also represent   $[\gamma]$ and $[\sigma]$.
We now have homotopies
$\gamma+nt\sim \gamma+f+f'$ and $\sigma+nt\sim \sigma+f+f'$.
 It thus suffices to show that $(\gamma+f)\sharp (\sigma+f)\sim (\sigma+f)\sharp (\gamma+f)$, where $\sharp$ denotes concatenation.  

Conjugating $\sigma+f$ with a two-dimensional rotation of blocks we get a homotopy between 
$(\gamma+f)\sharp (\sigma+f)$ and $ (\gamma+f)\sharp (f+\sigma)$.  Now
$  (\gamma+f)\sharp (f+\sigma)$ is homotopic to 
$\gamma+\sigma$.
Using the commutativity of $+$ up to homotopy we get a homotopy 
$ \gamma+\sigma \sim \sigma+\gamma$. 
By  reversing the first part  we finally get  a homotopy from $ \gamma+\sigma$ to   $(\sigma+f)\sharp (\gamma+f)$.
 \end{proof}

\begin{kor}\label{erijgowergfrefwrfwref}If $\pi_{0}\Map_{L_{K}\nCalg_{h}}(A,B)$ contains a cofinal element, then
the map of spaces \[\Map_{L_{K}\nCalg_{h}}(A,B) \to \Map_{L_{K}\nCalg_{h}}(A,B)_{\infty}\] is equivalent to the underlying  map of the group completion \[\Map_{L_{K}\nCalg_{h}}(A,B) \to \Map_{L_{K}\nCalg_{h}}(A,B)^{\group}\ .\]
\end{kor}

Let now $A$ be a unital $C^{*}$-algebra with unit $1_{A}$. 
In the following we show that we can apply  \cref{erijgowergfrefwrfwref} to $\cProj^{s}(A)\simeq \Map_{L_{K}\nCalg_{h}}(\C,A)$ by exhibiting a cofinal component.
Let   $e$ be a minimal projection in $K$.
 Then we define  $t_{A}:=[e\otimes 1_{A} ]$ in $\pi_{0} \cProj^{s}(A )$.
The following lemma is well-known.
\begin{lem}\label{giwjrgoergerfrewferwfwrefw}
 The  element $t_{A}$  is cofinal.
\end{lem}
\begin{proof}\mbox{}
We consider the separable Hilbert space  $H:=L^{2}(\nat)$ and let 
 $K:=K(H)$. For $n$ in $\nat$ we let $e_{n}$ in $K(H)$  denote the   projection onto the $n$th basis vector of $H$. We further consider the projection $P_{n}=\sum_{i=0}^{n}e_{n} $.
 
 We consider a component  $[p]$ in $\pi_{0}\cProj^{s}(A )$ with $p$ in $\Proj(K\otimes A)$. Then 
 there exists $n$ in $\nat$ such that $\|p- ( P_{n} \otimes 1_{A} )p( P_{n}\otimes 1_{A}  )\|<1/2$.
Using function calculus we get a homotopy between $p$  and a projection  $p'$ with $p'= (P_{n}\otimes 1_{A})p'(P_{n} \otimes 1_{A}  )$. We then  have $[p]=[p']$ in $\pi_{0}\cProj^{s}(A )$
 and with 
  $q':= (  P_{n} \otimes 1_{A} )-p'$
 we get  $[p']+[q']=[P_{n} \otimes 1_{A}  ]=nt_{A}$.
 \end{proof}

 To any unital $C^{*}$-algebra $A$ we can functorially (for unital morphisms) associate the  $\nat$-indexed diagram
\[\hat \cF(A):\quad \cProj^{s}(A)\xrightarrow{-+t_{A}}\cProj^{s}(A)\xrightarrow{-+t_{A}}\cProj^{s}(A)\xrightarrow{-+t_{A}}\cProj^{s}(A)\xrightarrow{-+t_{A}}\dots\] of spaces.
We furthermore define the functor \begin{equation}\label{erg2regwregwefwrefw}  \cF:=\colim_{\nat}\hat \cF:\Calg\to \Spc\ .\end{equation} 
We conclude that the natural transformation
$\cProj^{s}(-)\to \cF(-)$ of $\Spc$-valued functors  is equivalent to the transformation
$\cProj^{s}\to \cProj^{s,\group}$ of $\CMon(\Spc)$-valued functors after forgetting the commutative monoid structure.

We can now finally show the asserted exactness of the functor $\rProj^{s}$.
We must show that this functor sends 
  an exact sequence \[0\to A\xrightarrow{i} B\xrightarrow{p} C\to 0\] of $C^{*}$-algebras 
  to a fibre sequence in $\CGroups(\Spc)$.
 We first  
  form the square \begin{equation}\label{vkjnwkelfnlvfevfdvsfvsvsfvs}
\xymatrix{ A^{u} \ar[r]^{\pi}\ar[d]^{i^{u}} & \C \ar[d]^{j}\\ B^{u} \ar[r]^{p^{u}}& C^{u} } \end{equation}
in $\Calg_{/\C}$,  
 where $\pi$ is induced by  the canonical projection of the unitalization sequence  \eqref{afoiahjfoiaf}, the homomorphism $j:\C\to C^{u}$ is induced by the identity of $C^{u}$, and we do not write the  structure maps to $\C$
 given by the projections of the unitalization sequences of $A$, $B$ and $C$, respectively, and the identity of $\C$. This square is cartesian and $p^{u}$ is surjective. Applying the functor $\cProj^{s}$ we get a   diagram in $ \Spc_{/\cProj^{s}(\C)}$, and by  \cref{wejkgoowpfrefwfrefwef}  a cartesian square 
  \begin{equation}\label{sdgsgwrgweergwergwr}  \xymatrix{\cProj^{s}(A^{u})\ar[rr]^{\cProj^{j}(\pi)}\ar[d]^{\cProj^{s}(i^{u})}&&\cProj^{s}(\C)\ar[d]^{\cProj^{s}(j)}\\\cProj^{s}(B^{u})\ar[rr]^{\cProj^{s}(p^{u})}&&\cProj^{s}(C^{u})}\ .\end{equation} 
%

We now apply the functor $\cF$ from \eqref{erg2regwregwefwrefw} to the square \eqref{vkjnwkelfnlvfevfdvsfvsvsfvs}. This amounts to forming a filtered colimit of a diagram of squares of the form \eqref{sdgsgwrgweergwergwr}.  
 Since a filtered colimit of cartesian squares in $ \Spc$ is again a cartesian square and since the forgetful functor $\CGroups(\Spc)\to \Spc$ detects limits we can conclude that
   \[\xymatrix{\cProj^{s,\group}(A^{u})\ar[rr]^{\cProj^{j}(\pi)}\ar[d]^{\cProj^{s,\group}(i^{u})}&&\cProj^{s,\group}(\C)\ar[d]^{\cProj^{s,\group}(j)}\\\cProj^{s,\group}(B^{u})\ar[rr]^{\cProj^{s,\group}(p^{u})}&&\cProj^{s,\group}(C^{u})}\]
   is a cartesian square in $\CGroups(\Spc)$. Together with its non-written structure maps 
   it is also a diagram in 
  $\CGroups(\Spc)_{/\cProj^{s,\group}(\C)}$. 
   We finally take the fibre of the structure maps to $\cProj^{s,\group}(\C)$ and get 
   the desired cartesian square (or fibre sequence) 
    \[\xymatrix{\rProj^{s}(A)\ar[rr] \ar[d]^{\rProj^{s}(i)}&&0\ar[d] \\\rProj^{s}(B)\ar[rr]^{\rProj^{s}(p)}&&\rProj^{s }(C)}\ .\]

This finishes the proof of \cref{wiergo0wergferwferwf9}.
  \end{proof}

\begin{proof}[Proof of  \cref{wogkpwgerfwrefwf1}]
In order to define an inverse transformation $\Omega^{\infty}K\to \rProj^{s}$ we plan to apply the Yoneda lemma for $\EE_{\sepa}$.  We therefore need a factorization  of $ \rProj^{s}$ through a functor $\bProj^{s}$ defined on $\EE_{\sepa}$. Since $\CGroups(\Spc)$ is left-exact and additive, the universal property 
\eqref{vdfvqr3ferffafafds1} of $\ee_{\sepa}$   and \cref{wiergo0wergferwferwf9} together provide the dotted arrow 
in 
\begin{equation}\label{vfdsvdfvhosvsdfvfsdv}\xymatrix{\nCalg_{\sepa}\ar[rr]^{\rProj_{\sepa}^{s}}\ar[dr]^{\ee_{\sepa}}&&\CGroups(\Spc)\\&\EE_{\sepa}\ar@{..>}[ur]_{\bProj}&}\ .
\end{equation}  Furthermore, using \cref{wtgijwoergferwfrefw} the pull-back along $\ee_{\sepa}$  induces an equivalence   \begin{equation}\label{vwercerfrefdeqe423r44}\ee_{\sepa}^{*}:\Nat(\bProj(-) , \Map_{\EE_{\sepa}}(\C,-) ) \stackrel{\simeq}{\to}\Nat(\rProj^{s}_{\sepa}(-)  , \Omega^{\infty}K_{ \sepa}(-))\ .
\end{equation}

We define a
point $a_{*}$ in $\bProj(\C) $  as the image of $\id_{\C}$ under
\begin{equation}\label{fvefvwevvwefrwfrefefwr}\Map_{L_{K}\nCalg_{\sepa,h}}( \C , \C )\simeq \cProj^{s}_{\sepa}(\C)\to \cProj^{s}_{\sepa}(\C)^{\group}\xrightarrow{i_{\C}, \eqref{qewfiuhquiewfqedqwdewd}} \rProj_{\sepa}^{s}(\C)\stackrel{\eqref{vfdsvdfvhosvsdfvfsdv}}{\simeq} \bProj(\C) \ .
\end{equation} 
Via the Yoneda Lemma this point determines a natural transformation 
\[\tilde a:  \Map_{\EE_{\sepa}}(\C,-)\to \bProj(-)\]  
of functors  from $\EE_{\sepa}$ to $\CGroups(\Spc)$ characterized by $\tilde a_{\C}(\ee_{\sepa}(\id_{\C}))\simeq a_{*}$.
 Its pull-back along $ \ee_{\sepa}$ is a natural transformation 
\[a:=\ee^{*}_{\sepa}\tilde a:\Omega^{\infty}K_{\sepa}\to \rProj_{\sepa}\]
of functors from $\nCalg_{\sepa}$ to $\CGroups(\Spc)$.
 We have already  a natural transformation \[b:=\tilde \ee_{\sepa,h }:\rProj_{\sepa} \to \Omega^{\infty}K_{\sepa}\]
defined by means of the diagram \eqref{vervwerwrfwerfrefrewfefw}. In view of \eqref{vwercerfrefdeqe423r44}
 there is an essentially unique natural transformatiom \[\tilde b:\bProj\to  \Map_{\EE_{\sepa}}( \C,-)\] 
such that   $\ee_{\sepa}^{*}\tilde b\simeq b$.
 
 The following proposition implies    \cref{wogkpwgerfwrefwf1} asserting that $b$ is an equivalence.
   \begin{prop}\label{wreigowegregwre9} The natural transformations 
$a$ and $b$ are mutually inverse to each other.
\end{prop}
\begin{proof}
The assertion follows from the following two lemmas.
\begin{lem}\label{wekjgowopergwfwrefrefewrf}
We have $b\circ a\simeq \id_{\Omega^{\infty}K_{\sepa}}$.
\end{lem}
\begin{proof}
It suffices to show \[\tilde b\circ \tilde a:  \Map_{\EE_{\sepa}}( \C ,-)\to   \Map_{\EE_{\sepa}}( \C ,-)\] is  equivalent to identity. Via the Yoneda Lemma this map is determined by 
the point $\tilde b_{ \C}(a_{*})$ in $   \Map_{\EE_{\sepa}}( \C,\C)$.
We therefore must show that $ \tilde b_{ \C}(a_{*})\simeq  \ee_{\sepa}(\id_{\C})$.
 In order to  verify this equivalence we consider the diagram
\[\xymatrix{\cProj^{s}_{\sepa}(\C)\ar[r]\ar@{..>}[drr]^{\cProj^{s}(j)}_{\id_{\C}\mapsto \tilde a_{*}}&\cProj^{s}_{\sepa}(\C)^{\group}\ar[r]^{i_{\C}} & \bProj^{s}_{\sepa}(\C) \ar[d]\ar[r]^-{ \tilde b_{\C}}&    \Map_{\EE_{\sepa}}( \C , \C )\ar[d]^{j_{*}}\\&&\cProj^{s}_{\sepa}(\C^{u})^{\group} \ar[r]^-{\hat \ee_{\sepa,h,\C}}\ar[d]& \Map_{\EE_{\sepa}}( \C , \C^{u})\ar[d]\\&& \cProj_{\sepa}(\C)^{\group} \ar[r]^-{\hat \ee_{\sepa, h,\C}}& \Map_{\EE_{\sepa}}( \C , \C )}\]
obtained by merging \eqref{vervwerwrfwerfrefrewfefw} with \eqref{fvefvwevvwefrwfrefefwr},
where  $j:\C\to \C^{u}$ is the inclusion.

The 
  two upper left horizontal arrows  send $ \id_{\C} $ to $a_{*}$.
  We let $\tilde a_{*}$ denote its image  in $ \cProj^{s}_{\sepa}(\C^{u})^{\group}$.
Since  $\hat \ee_{\sepa,h,\C}(\tilde a_{*}) \simeq j_{*}(\ee_{\sepa}(\id_{\C}))$ and $j_{*}$ is a monomorphism  we can conclude that $\tilde b_{\C}(a_{*})\simeq \ee_{\sepa}(\id_{\C})$.
   \end{proof}

\begin{lem}
We have $a\circ b\simeq \id_{\rProj^{s}_{\sepa}}$
\end{lem}
\begin{proof} 
 It suffices to show for every $A$ in $\nCalg_{\sepa}$ that \[ \pi_{0} \rProj_{\sepa}(A)\xrightarrow{b_{A}}  \pi_{0}   \Omega^{\infty }K_{\sepa}(A) \xrightarrow{a_{A}}  \pi_{0}\rProj_{\sepa}(A)\] is an  isomorphism.
In fact, in order to deduce the isomorphism for $\pi_{i}$ with $i>0$ we apply this result
for $A$ replaced by $S^{i}(A)$ and use the left-exactness of the functors and the fact that they take values in groups.

For the calculation in $\pi_{0}$ introduce the following simplified notation (borrowed from \cite[5.3]{blackadar}):
$$K_{0}:=\pi_{0}\Omega^{\infty}K_{\sepa}\ , \quad \tilde K_{0}:=\pi_{0}\rProj^{s}_{\sepa}$$
and
$$V_{00}:=\pi_{0}\cProj^{s}_{\sepa}\ , \quad 
\tilde K_{00}:=\pi_{0} \cProj_{\sepa}^{s,\group}\ .$$
We form the commutative diagram 
 \[\xymatrix{&\tilde K_{0}(A)\ar[r]^-{b_{A}}\ar[d]&\ar[d] K_{0}(A)\ar[r]^-{a_{A}}&\tilde K_{0}(A)\ar[d]\\ V_{00}(A^{u})\ar[r]\ar[d]^{(3)} & \tilde K_{00}(A^{u}) \ar[d] \ar[r]^-{\hat e_{h,\sepa}}& K_{0}(A^{u})\ar[r]^-{ a_{A^{u}}} &\tilde K_{0}(A^{u})\ar[d]^{(2)}&\\V_{00}((A^{u}\otimes K)^{u})\ar[r]^{(1)}& \tilde K_{00}((A^{u}\otimes K)^{u})\ar@{=}[rr]&
&\tilde K_{00}((A^{u}\otimes K)^{u})&  }\]
The vertical maps are 
 inclusion of summands and
all cells except the lower right commute obviously.
We will show that this cell also commutes.
We can then conclude that
$a_{A}\circ b_{A}=\id_{\tilde K_{0}(A)}$.

The transformation 
  $V_{00}\to \tilde K_{00}$ is the algebraic group completion.
  In view of its universal property it
   suffices to show that the composition of the two  lower cells commutes.

 We consider a point    $[p]$ in $V_{00}(A^{u})$  
 given by map $p: \C\to A^{u}\otimes K$. 
 We first calculate its image $k_{rd}$ under the right-down composition. The horizontal map sends it to the point in
 $\tilde K_{0}(A^{u})$  given by  
 $a_{A^{u}}(\ee_{\sepa}(p))\simeq \tilde K_{0}(p)(a_{*})$.  
 The element $k_{rd}$ is the image of $\tilde K_{0}(p)(a_{*})$   in $\tilde K_{00}((A^{u}\otimes K)^{u})$ under the   map $(2)$.
As seen in the proof of \cref{wekjgowopergwfwrefrefewrf} the map 
 $\tilde K_{0}(\C)\to \tilde K_{00}(\C^{u})$ sends $a_{*}$ to image of $i_{\C}:\C\to \C^{u}$ in $V_{00}(\C^{u})$ under group completion map
$V_{00}(\C^{u})\to  \tilde K_{00}(\C^{u})$.
The image $k_{rd}$ of $\tilde K_{0}(p)(a_{*})$ in $\tilde K_{00}((A^{u}\otimes K)^{u})$ is then the image of 
$p^{u}\circ i_{\C} :\C\to \C^{u}\to (A^{u}\otimes K)^{u}$ in $V_{00}((A^{u}\otimes K)^{u})$ under the  map $(1)$.

We now calculate the image $k_{dr}$ of $p$ under the down-right composition.
The  image of $p$   under $(3)$   
  is $i_{A^{u}\otimes K}\circ p$ in $V_{00}((A^{u}\otimes K)^{u})$. We now use that
 the  following square commutes:
\[\xymatrix{\C\ar[r]^{i_{\C}}\ar[d]^{p} &\C^{u} \ar[d]^{p^{u}} \\ A^{u}\otimes K \ar[r]^{i_{A^{u}\otimes K}} &(A^{u}\otimes K)^{u} } \ .\]
Consequently  we have $i_{A^{u}\otimes K}\circ p\simeq p^{u}\circ i_{\C}$ in $V_{00}((A^{u}\otimes K)^{u})$.
This implies that $k_{rd}\simeq k_{dr}$.
\end{proof}
This finishes the proof of \cref{wreigowegregwre9}.
\end{proof}
We finally have completed the proof of \cref{wogkpwgerfwrefwf1}.
\end{proof}

%
%
%
%
%
%
%
%
%
We finish this section with the proofs of the two technical results \cref{wtogjgergregw9} and
\cref{wtogjgergregw91}.

 \begin{proof}[Proof of \cref{wtogjgergregw9}]
  The  statement of \cref{wtogjgergregw9} is surely a known fact in $C^{*}$-algebra theory. But for the sake of completeness we will provide an argument.
 
 We start with two facts taken from \cite[Sec. 4.3]{blackadar}.
 We consider a $C^{*}$-algebra $B$  and a projection $p$ in $B$. Note that projections are always assumed to be selfadjoint.
Assume that $b$ is an invertible element in the multiplier algebra $M(B)$ and $b=ur$ its polar decomposition.

\begin{lem}\label{rihugoewrgrwefrf9}
If $bpb^{-1}$ is a projection, then $bpb^{-1}=upu^{*}$.
\end{lem}
\begin{proof}
 By assumption
 $q:=bpb^{-1}$ is a projection, so it is in particular selfadjoint.
 Writing 
   $q=urpr^{-1}u^{*}$ we see that 
$u^{*}qu=rpr^{-1}$ is selfadjoint, too. This implies
 $rpr^{-1}=r^{-1}pr$. We 
 multiply with $r$ from left and right
and get the  equality  $r^{2}p=pr^{2}$.
We now use that $r$ is  positive and therefore  $r=\sqrt{r^{2}}$.
We conclude that
also 
$rp=pr$ holds, and this implies  
$q=upu^{*}$.
\end{proof}

In the following we let $B(b,r)$ denote the open $r$-ball at $b$ in $B$.
\begin{lem}\label{qeirjgoqerfqewfqwf9}
There exists a constant  $c$ in $(0,1/2)$ and a map
\[w_{p}:\Proj(B)\cap B(p,c)\to U(B^{u})\cap B(1,1)\]    such that
\[w_{p}(q)p w_{p}(q)^{-1}=q\] for all $q$ in $\Proj(B)\cap B(p,c)$.
\end{lem}
\begin{proof}\mbox{}
We define  $v_{p}(q):=\frac{1}{2}((2q-1)(2p-1)+1)$
and note that 
 $v_{p}(q)\in B^{u}$. If  $\|p-q\|<1/2$, then 
$\|v_{p}(q)-1\|<1$ and $v_{p}(q)$ is invertible.
Furthermore we have  $v_{p}(p)=1$ and 
$v_{p}(q)pv_{p}^{-1}(q)=q $.
We now form the  polar decomposition $v_{p}(q)=w_{p}(q) r_{p}(q)$ in $B^{u}$.
By  \cref{rihugoewrgrwefrf9} we then have $w_{p}(q)p w_{p}(q)^{-1}=q$.
By continuity of $w_{p}$ and $w_{p}(p)=1$ we can find a constant $c$ in $(0,1)$ such that
$w_{p}(B(p,c))\subseteq B(1,1)$.
 \end{proof}

We consider  a surjection  $B\to C$ of $C^{*}$-algebras.
 We must show that the induced map
 $ \Proj(B)\to \Proj(C)$ is a Serre  fibration. 
 To this end 
we will solve the lifting problem
\[\xymatrix{X\ar[r]\ar[d]_{x\mapsto (0,x)}&\Proj(B)\ar[d]\\ [0,1]\times X\ar[r]\ar@{..>}[ur]&\Proj(C)}\]
for all compact spaces $X$.
This problem is equivalent to the lifting problem
\[\xymatrix{\{0\}\ar[r]\ar[d] &\Proj(C(X)\otimes B)\ar[d]\\ [0,1] \ar[r]\ar@{..>}[ur]&\Proj(C(X)\otimes C)}\ .\]
We must therefore solve the path lifting problem
\[\xymatrix{\{0\}\ar[r]\ar[d]&\Proj(B)\ar[d]\\ [0,1] \ar[r]^{\gamma}\ar@{..>}[ur]^{\tilde \gamma}&\Proj(C)}\]
for all surjections $B\to C$.

 A path  $\sigma: [0,1]\to \Proj(C)$  is called short if $\|\sigma(t)-\sigma(0)\|< c$ and $\|w_{\sigma(0)}(\sigma(t))-1\|<1$ for all $t$ in $[0,1]$ (with $c$ and $w_{-}(-)$ from \cref{qeirjgoqerfqewfqwf9})  .
  For the moment we assume that  we have  a solution for the lifting problem   \begin{equation}\label{qewfqedewdwedewdqdqdwed}\xymatrix{\{0\}\ar[r]^{\tilde \sigma(0)}\ar[d] &\Proj(B )\ar[d]\\ [0,1] \ar[r]^{\sigma}\ar@{..>}[ur]^{\tilde \sigma}&\Proj(  C)}
\end{equation}  for  all short  paths $\sigma$.

By continuity and compactness of the interval there exists $n$ in $\nat$  such that $\gamma_{|[i/n,(i+1)/n]}$ is short for all $i=0,\dots,n-1$.
We then solve lifting problem  inductively by solving the lifting problem for the short paths $\gamma_{|[i/n,(i+1)/n]}$  with initial $\tilde \gamma(i/n)$.

We finally solve the lifting problem \eqref{qewfqedewdwedewdqdqdwed} for short paths  $\sigma$.
 We set   $u(t):=w_{\sigma(0)}(\sigma(t)$. Then  
 $u(0)=1$,   $\sigma(t)=u(t)\sigma(0) u(t)^{*}$, and 
$\|u(t)-1\|<1$.
 We get a 
  path $\log u:[0,1]\to iC^{\sa}$ with $\log u(0)=0$.
  We interpret 
$\log u$ as an element $  i(C_{0}((0,1])\otimes C)^{\sa}$.
Since  $ C_{0}((0,1])\otimes B \to C_{0}((0,1])\otimes C$ is surjective
 we can find a preimage    $b$ in $i(C_{0}((0,1])\otimes B)^{\sa}$.
 We then set   $v:=\exp(b):[0,1]\to U(B^{u})$. Then 
 $\tilde \sigma(t):=v(t)\tilde \sigma(0) v(t)^{*}$ is the desired lift of $\sigma$ with initial $\tilde \sigma(0)$ in \eqref{qewfqedewdwedewdqdqdwed}.
\end{proof}

  \begin{rem}
  Using that $\C$ is a semi-projective $C^{*}$-algebra on could deduce the  path lifting  in \eqref{qewfqedewdwedewdqdqdwed} from \cite[Thm. 5.1]{Blackadar_2016}. \hB \end{rem}

\begin{proof}[Proof of \cref{wtogjgergregw91}]
Similarly as in the proof of \cref{weojgiwoerfrefweferf} we reduce the assertion to a consideration of homotopy groups and eventually to the following lemma.

 We consider a   small filtered family  $(B_{i})_{i\in I}$ of $C^{*}$-algebras indexed by a poset
and set $B:=\colim_{i\in I} B_{i}$. 
 For $i,j$ in $I$, $i\le j$ we ket
$\phi_{j,i}:B_{i} \to B_{j} $ be the structure map and   
 $\phi_{i}: B_{i} \to  B  $ be the canonical homomorphism.

Let $X$ be a   compact  metrizable space and
 $Y$ be a closed subspace. Let 
  $i_{0}$ be in $I$  and assume that we are
given a square 
\[ \xymatrix{Y\ar[r]^-{f}\ar[d] &\Proj(B_{i_{0}}) \ar[d]^{\phi_{i_{0}}} \\X \ar[r]^-{g} & \Proj(B) } \ .\]

 \begin{lem}\label{wtgijrtogwrefrefr22222}
 There exists  $i$ in $I$ with $i\ge i_{0}$  and $h:X\to  \Proj(B_{i})$ such that $h_{|Y}=\phi_{i,i_{0}}\circ f$ and
 $\phi_{i}\circ h$ is homotopic to $g$  rel $Y$.    \end{lem}
\begin{proof} 

Let $A$ be any $C^{*}$-algebra. For a selfadjoint element $q$ in $A$ we let $\sigma(q)$ denote the spectrum of $q$. We fix a number
$c$ in $(0,\frac{1}{2})$ and 
  consider the subspace
 \[P(A):=\{q\in  A^{\sa}\mid d(\sigma(q),\frac{1}{2})>c\}  \] of    selfadjoints in $A$   with a spectral gap at $\frac{1}{2}$. We observe that 
 it contains the space of projections  $\Proj(A)$.
We fix a 
 function $\chi\in C(\R)$ with $\chi_{|(-\infty,\frac{1}{2}-c]}\equiv 0$ and  $\chi_{|[\frac{1}{2}+c,\infty)}\equiv 1$.
 The map $q\mapsto \chi(q)$  defined using the function calculus is a  retraction $W:P(A)\to \Proj(A)$.
By continuity we can choose a 
  constant $c_{1}$ in $(0,\infty)$ such that $\|q^{2}-q \|\le c_{1} $ implies $\|W(q)-q \|\le  c$.

 We  interpret  $ f$ as a function $Y\to B_{i_{0}}^{\sa}$. 
 Using Dugundji's extension theorem 
 we find an extension $h_{0}:X\to B_{i_{0}}^{\sa}$.
 
 We then have $g-\phi_{i_{0}}(h_{0})\in (C_{0}(X\setminus Y)\otimes B)^{\sa}$.
We can thus find  $i_{1}$ in $I$ with $i_{1}\ge i_{0}$  such that
there exists $r$ in $(C_{0}(X\setminus Y)\otimes B_{i_{1}})^{\sa}$  
with 
 $\|g-\phi_{i_{0}}(h_{0})-\phi_{i_{1}}(r)\|\le c_{1}/20$.
We set $h_{1}:=\phi_{i_{1},i_{0}}(h_{0})+r$ in $(C(X)\otimes B_{i_{1}})^{\sa}$.
Since $g$ is a projection, we have
 $  \|\phi_{i_{1}}(h_{1})^{2}-\phi_{i_{1}}(h_{1})\|  \le c_{1}/2$.
 We now find 
  $i$ in $I$ with $i\ge  i_{1}$ such that 
   such that
  $ \| \tilde h^{2}-\tilde h  \|  \le c_{1}$, where $\tilde h:=\phi_{i,i_{1}}(h_{1})$.
 
We finally  define     
 $h:=W(\tilde h)$ in 
  $   \Proj^{s}(C(X)\otimes B_{i})$. By construction we have 
 $h_{|Y}=\phi_{i,i_{0}}(f)$ and 
 $\|\phi_{i}(h)-g\|\le c$. Then 
  $(W((1-s)\phi_{i}(h)+sg))_{s\in [0,1]}$ is a homotopy from $\phi_{i}(h)$ to $g$ rel $Y$.
 This finishes the proof of \cref{wtgijrtogwrefrefr22222}
 \end{proof}
 
  \begin{rem}
 Using that $\C$ is semi-projective we could deduce \cref{wtgijrtogwrefrefr22222} directly
 from the proof of \cite[Prop. 3.8]{Bunke:2024aa}, in particular from the existence of the lift in 
  \cite[Eq. (3.6)]{Bunke:2024aa}. \hB
 \end{rem}

Hence we have completed the proof of \cref{wtogjgergregw91}.
\end{proof}

 \section{The $q$-construction}\label{efgijoerfreferfwrf}
 
 The {\em $q$-construction} introduced by  Cuntz  \cite{MR899916} is an effective tool to capture  Kasparov modules in terms of homomorphisms, see \cref{wtokgwopgfrefreferfw}. Using the $q$-construction one can express the classical $KK$-theory groups in  terms of homotopy classes of maps.  
The crucial formula states that 
  for  two separable $C^{*}$-algebras
 $A$ and $B$ we have \begin{equation}\label{vefdvvfdvsdvsfdvdfcl} \pi_{0} \underline{\Hom}(qA,K\otimes B)\cong \KK^{\class}_{\sepa,0}(A,B) \ .
\end{equation}  
In \cite[Def. 1.5]{MR899916}  this isomorphism is actually   the definition  of the right-hand side. From this formula  the composition
\[ \KK^{\class}_{\sepa,0}(A,B) \otimes  \KK^{\class}_{\sepa,0}(B,C)\to  \KK^{\class}_{\sepa,0}(A,C)\]
is not obvious. But one can show that the left-hand side in \eqref{vefdvvfdvsdvsfdvdfcl} is naturally isomorphic to 
$\pi_{0}\underline{\Hom}(K\otimes qA,K\otimes qB)$ and this makes the composition obvious.

The final goal of the present section and \cref{weijgiowergfrfreferfwr} together
is to give a selfcontained proof that for all separable $C^{*}$-algebras $A$ and $B$
 \begin{equation}\label{vefdvvfdvsdvsfdvdf} \pi_{0} \underline{\Hom}(qA,K\otimes B)\cong \pi_{0}\KK_{\sepa}(A,B) \ .
\end{equation}  
\begin{rem}The comparison of \eqref{vefdvvfdvsdvsfdvdfcl} and \eqref{vefdvvfdvsdvsfdvdf} provides a proof of the $KK$-theory version of \cref{qiurhfgiuewrgwrfrefrfwrefw} below
which does not depend on the knowledge of the universal property of $\kk^{\class}_{\sepa}$.\hB
\end{rem}

We will start this section with recalling the  $q$-construction. 
 We then continue to study those of its homotopical properties
that are easily accessible without going deeper into $C^{*}$-algebra theory. 
We shall see that inverting the images in $L_{K}\nCalg_{h}$ of the canonical morphisms $\iota_{A}:qA\to A$ for all separable $C^{*}$-algebras
produces a Dwyer-Kan localization of \[L_{q}:L_{K}\nCalg_{h}\to L_{K}\nCalg_{h,q}\] which is equivalent to composition of the localizations $L_{\splt}$ from \cref{wegjwoergerwfrferfrwefwf}   (enforcing  split-exactness) and the right Bousfield localization at the subcategory of group objects, see 
\cref{wetgjowoegrewfwref} and \cref{rthogijprtgfrefrfw}.

All of the above has a separable version.
 At the end of the present section we go deeper into $C^{*}$-algebra theory. In \cref{erigjwoergerfwerfwref} we reproduce the proof of \cite[Thm. 1.6]{MR899916}.
As a consequence, for separable $C^{*}$-algebras $A$ and $B$  we can simplify the formula for the  mapping spaces in $L_{K}\nCalg_{\sepa,h,q}$ to
\begin{equation}\label{vweriuvbvjbjkfvfvfsfvsdfv} \ell \underline{\Hom}(qA, K\otimes B) \simeq
\Map_{L_{K}\nCalg_{\sepa,h,q}}(A,B)   \end{equation} 
which is already very close to \eqref{vefdvvfdvsdvsfdvdf}.
The final step  towards this formula, discussed in   \cref{weijgiowergfrfreferfwr}, is to show that the canonical functor $L_{K}\nCalg_{\sepa,h,q}\to \KK_{\sepa}$ is an equivalence.

 We now start with the description of Cuntz'  $q$-construction.
 To every $C^{*}$-algebra $A$ we can functorially associate
  a diagram \begin{equation}\label{wrtbvwvwvsdfvfv}\xymatrix{0\ar[r]&qA\ar[r]^{i}&A*A\ar[r]^{d }&A\ar@{=}[d]\ar[r]&0\\0\ar[r]&qA\ar[u]^{\bar \sigma }\ar[dr]^{\iota_{A}}\ar[r]^{i }&\ar[d]^{p_{0}}A*A\ar[u]^{\sigma }\ar[r]^{d }&\ar@/^-1cm/@{..>}[l]^{\iota_{i}}A\ \ar[r]&0\\&&A&&}\ ,
\end{equation}
  where the horizontal sequences are exact.
  Recall that the free product $A*A$ together with the two canonical maps $\iota_{i}:A\to A*A$, $i=0,1$ represents the coproduct in $\nCalg$. The map $d$ (often called the fold map) is determined via the universal property of the free product by $d\circ \iota_{i}=\id_{A}$ for $i=0,1$. The $C^{*}$-algebra $qA$ is defined as the kernel of $d$.
 The two maps $\iota_{i}$  determine splits of the exact sequence. 
 The  maps
 $p_{i}:A*A\to A $ are determined by the conditions
 $p_{i}\circ \iota_{i}  =\id_{A}$ and $p_{i}\circ \iota_{1-i}=0$. We can then define   the  map
 $\iota_{A}:=p_{0}\circ i :qA\to A$. 
    The flip of the two factors of the free product defines an automorphism $\sigma:A*A\to A*A$. Since $d \circ \sigma=d$
it restricts to an  involutive automorphism $\bar \sigma  :qA\to qA$.
 In principle we should add an index $A$ also to the notation for the maps $d,i,\sigma,\dots$ as they are all components of natural transformations but we refrain from doing so in order to shorten the notation. 

Since $qA$ is an ideal in $A*A$ we have a canonical map $m:A*A\to M(qA)$, where $M(qA)$ denotes the multiplier algebra of $qA$.
We define $m_{i}:=m\circ \iota_{i}$.

\begin{rem}\label{4tjigortgwrefrwefref}
Let $B$ be any $C^{*}$-algebra. In order to give a map
$f:qA\to B$ we could give a map $\hat f:A*A\to M(B)$ and set
 $f:=\hat f\circ i$.  We must ensure that this composition takes values in the  ideal $B$ of $M(B)$. To this end 
we consider the {\em components} $\hat f_{i}:=\hat f\circ \iota_{i}$  of $\hat f$.
 We must require that  $\hat f_{1}(a)-\hat f_{0}(a)\in B$ for all $a$ in $A$. Under this condition  $f$ is a well-defined homomorphism with values in $B$. Moreover, if   
 $\hat f_{0}=\hat f_{1}$, then it follows that $f=0$. 
 We will call $\hat f$ also the {\em associated homomorphism}.
 
  This construction can be reversed. Assume that $f:qA\to B$ is a homomorphism. We define  $B':=f(qA)$  and the map   $\hat f:A*A\to M(qA)\xrightarrow{M(f)}  M(B')$, where  we must restrict the codomain of $f$  to $B'$ in order be able to apply the multiplier algebra functor $M$ which is  only functorial for non-degenerate morphisms. The components of $\hat f$ are then given by   $\hat f_{i}:=\hat f\circ \iota_{i}:A\to M(B')$.
 The datum \[A\rightrightarrows^{\hat f_{0}}_{\hat f_{1}} M(B')\triangleright B'\to B\]
 is called a pre-quasihomomorphism in \cite{MR899916}.
 \hB
\end{rem}

The functor $q:\nCalg\to \nCalg$ is continuous with respect to the topological enrichment of $\nCalg$. 
 It therefore preserves homotopy equivalences and descends to a functor  $q:\nCalg_{h}\to \nCalg_{h}$. 
Since $q$ does not preserve $K$-stability we are led to define the  functor
\begin{equation}\label{fqwefiojiofjowqefqewdwdqwed}q^{s}: \nCalg_{h}\to  L_{K}\nCalg_{h}\ , \quad A\mapsto K\otimes qA\ .
\end{equation} 
In the following we use the notation \begin{equation}\label{gugwerg34wrefwferf}(-)^{s}:=L_{K}(-):\nCalg_{h}\to L_{K}\nCalg_{h}
\end{equation}
for the stabilization functor from \cref{wreogkpwegrewferf}.\ref{wegergoijiowerfgwerferfrewfrfwefrf}. 
%
%
%
We  define the natural transformation $\iota^{s}:q^{s}\to  (-)^{s}$  such that its component at $A$ in $\nCalg_{h}$ is given by \begin{equation}\label{ewfqewwdweqdewdqwdeqdwdewdeqd}\iota^{s}_{A}: q^{s}A=K\otimes qA \xrightarrow{\id_{K}\otimes \iota_{A}}   K\otimes A= A^{s}\ .
\end{equation} 

Recall from \cref{weotgpwergferferfwref} that $L_{K}\nCalg_{h}$ is semi-additive. By \cref{erkgowegergrewfwerf9}   every object in this category is naturally a  commutative monoid object.
For  $A$  in $\nCalg_{h}$   the object $q^{s}A=K\otimes qA$  is thus a commutative monoid object of $L_{K}\nCalg_{h}$. We furthermore
 define $\bar \sigma^{s}:=\id_{K}\otimes \bar \sigma:q^{s}A\to q^{s}A$, where $\bar \sigma$ is as in \eqref{wrtbvwvwvsdfvfv}.

The following is a version of \cite[Prop. 1.4]{MR899916}.
\begin{lem}\label{wtgijoweoifrefwerferfwrf} For $A$ in $\nCalg_{h}$ the object
$q^{s}A$ is a commutative group in $L_{K}\nCalg_{h}$ whose inversion map is given by $ \bar \sigma^{s}$.
\end{lem}
\begin{proof}
We must show that $\id_{q^{s}A}+ \bar \sigma^{s}\simeq 0$. We have the following chain of equivalences
  \begin{eqnarray*}
\Map_{L_{K}\nCalg_{h}}(q^{s}A,q^{s}A)
&\simeq&
\Map_{\nCalg_{h}}( K\otimes qA,K\otimes qA)\\& \simeq&
\Map_{\nCalg_{h}}(qA,K\otimes qA)\\&
\stackrel{\cref{erigjeroferferfrefwerf}}{\simeq} &  \ell\underline{\Hom}(qA,K\otimes qA) \ .
\end{eqnarray*}
The first  reflects the definition of $q^{s}A$ and that 
$L_{K}$ is a left Bousfield localization \cref{wreogkpwegrewferf}.\ref{wegergoijiowerfgwerferfrewfrfwefrf}. The second equivalence is induced by 
 left upper corner inclusion $qA\to K\otimes qA$ which induces an equivalence
since $K\otimes qA$ is a local object in this localization. 
Under this equivalence the sum
is determined by the block sum of morphisms $qA\to K\otimes qA$.
Hence $\id_{q^{s}A}+\bar \sigma^{s}$ is induced by the composition
\[\diag(\id_{qA},\bar \sigma):qA\to \Mat_{2}(qA)\to K\otimes qA\ .\]
It suffices to show that
$\diag(\id_{qA},\bar \sigma):qA\to \Mat_{2}(qA)$ is homotopic to zero.

The composition $m\circ i:qA\to M(qA)$ is the inclusion, hence it has the associated homomorphism $\widehat{\id_{qA}}=m:A*A\to M(qA)$ (using the notation introduced in \cref{4tjigortgwrefrwefref}) and the components
 ${\widehat{\id_{qA}}}_{\ i}=m_{i}:A\to M(qA)$. 
 Furthermore, the associated homomorphism  $\widehat{\bar \sigma}=i\circ \sigma$  of $\bar \sigma$ has the components
${\widehat{\bar \sigma}}_{\ i}=m_{1-i}$.
We identify $M(\Mat_{2}(qA))\cong \Mat_{2}(M(qA))$ in the natural way.
Then $\diag(\id_{qA},\bar \sigma)$  has the associated homomorphism \[ \widehat{\diag(\id_{qA},\bar \sigma)}=\diag(m,m\circ \sigma):A*A\to  M(\Mat_{2}(qA))\]
 and the components
$ {\widehat{\diag(\id_{qA},\bar \sigma)}}_{i}=\diag(m_{i},m_{1-i} ):A\to  \Mat_{2}(M(qA))$.

For $t$ in $[0,1]$ we consider the scalar unitary  \[U_{t}:=\left( \begin{array}{cc} \cos( \pi t/2)&-\sin(\pi t/2) \\ \sin(\pi t/2) &\cos(\pi t/2) \end{array} \right) \] in 
$\Mat_{2}(M(qA))$. Since $m_{0}(a)-m_{1}(a)\in qA$ for all $a$ in $A$ 
we have
\[U_{t}\diag(m_{1}(a),m_{0}(a) )U_{t}^{*}-\diag(m_{0}(a),m_{1}(a) ) \in \Mat_{2}(qA)\]
for all $t$ in $[0,1]$.
Hence for every $t$ in $[0,1]$, as explained in  \cref{4tjigortgwrefrwefref}, the pair 
\[\diag(m_{0},m_{1}),U_{t}\diag(m_{1},m_{0} )U_{t}^{*}:A\to \Mat_{2}(M(qA))\]  gives the components of
 a map   $h_{t}:qA\to \Mat_{2}(qA)$. We have
$h_{0}=\diag(\id_{qA},\bar \sigma)$ and
$h_{1}$ is the map with equal components
$(\hat h_{1})_{i}=
  \diag(m_{0},m_{1})$ for $i=0,1$.
Hence $h_{1}=0$.
 \end{proof}

 Let $E:L_{K}\nCalg_{h}\to \bC$ be  a left-exact functor to a semi-additive $\infty$-category.
We say that   that $E$ is split-exact if it sends the images under $L_{h,K}$ of split-exact sequences of $C^{*}$-algebras to fibre sequences. 
 
   The following lemma is a version of \cite[Prop. 3.1.b]{MR899916}. Let   $A$ be a $K$-stable $C^{*}$-algebra. We will use the same symbol for the corresponding object of  $L_{K}\nCalg_{h}$, i.e., we will write $A$ instead of $L_{h,K}(A)$ in order to simplify the notation.
\begin{lem}\label{wergkooprefrwefrefre} 
If $E$ is split-exact and $E(A)$ is a group, then
the map  $\iota^{s}_{A}: q^{s}A\to A^{s} $ from \eqref{ewfqewwdweqdewdqwdeqdwdewdeqd}  induces an equivalence $E(\iota^{s}_{A}):E(q^{s}A)\stackrel{\simeq}{\to} E(A^{s})$.
\end{lem}
\begin{proof}    
   The middle horizontal exact sequence in 
\eqref{wrtbvwvwvsdfvfv} is split and induces the fibre sequence
\begin{equation}\label{fadscdscasdascd} E (q^{s}A)\xrightarrow{E(i^{s})} E( (A*A)^{s})\xrightarrow{E (d^{s})} E ( A^{s})\ . \end{equation} 
Since $A$ is $K$-stable we have $A\simeq A^{s}$ and hence $E(A)\simeq E(A^{s})$. 
In view of our assumptions 
 $E(A^{s}) $ is a   group. Since $q^{s}A$ is a group by \cref{wtgijoweoifrefwerferfwrf} and
$E$ preserves products (since it is split-exact) we see that also $E(q^{s}A)$ is a group. Finally, by \cref{weroigjowerfrefw} we have an equivalence
\[E((A*A)^{s})\xrightarrow{E (p^{s}_{0})\times E (p^{s}_{1})} E(A^{s})\times E( A^{s})\] whose inverse is
$E (\iota^{s}_{0})\circ \pr_{0}+E(\iota^{s}_{1})\circ \pr_{1}$.
This implies that  that $E((A*A)^{s})$ is a group as well. 
Using that the objects in the sequence \eqref{fadscdscasdascd} are commutative groups we see that 
 \[E(q^{s}A)\times E(A^{s})\xrightarrow{E(i^{s}_{A})+E(\iota_{1})}E( (A*A)^{s})\]  is an equivalence.
Hence
\[\hspace{-0.2cm}\left( \begin{array}{cc} E(\iota^{s}_{A})&0 \\ E(p^{s}_{1}\circ i^{s})  &\id_{E(A^{s})} \end{array} \right): E(q^{s}A)\times E(A^{s})  \xrightarrow{E(i^{s})+E(\iota^{s}_{1})}E((A*A)^{s})\xrightarrow{E (p^{s}_{0})\times E (p^{s}_{1})} E(A^{s})\times E(A^{s})\] is an equivalence.
Again using that the factors are   groups this implies that $E(\iota^{s}_{A})$ is an equivalence.
\end{proof}

We consider the set of morphisms \begin{equation}\label{wergerwgerferfrefwerf}  \hat W_{q}:=\{\iota^{s}_{A}:q^{s}A\to A^{s}\mid A\in  \nCalg_{h}\} \end{equation}  in $L_{K}\nCalg_{h}$.
 \begin{ddd} We define the   Dwyer-Kan localization  \begin{equation}\label{eqwfwdewqdewdq}L_{q}:L_{K}\nCalg_{h}\to L_{K}\nCalg_{h,q}
\end{equation} 
at the set $\hat W_{q}$.
\end{ddd}

We consider the composition \begin{equation}\label{vsdfvdfsjvposdvsfdvsfdvsfdvsfdv} 
   L_{h,K,q}:\:\:\: \nCalg\xrightarrow{L_{h}}\nCalg_{h}\xrightarrow{L_{K}}L_{K}\nCalg_{h}\xrightarrow{L_{q}} L_{K}\nCalg_{h,q}\ .  \end{equation}

\begin{prop}\label{eigjohwergerfgerwfrwef}
The functor  $L_{h,K,q}:\nCalg\to L_{K}\nCalg_{h,q}$ is a Dwyer-Kan localization.
\end{prop}
\begin{proof}
This follows from the fact that $L_{h,K,q}$ is a composition of Dwyer-Kan localizations which are
all determined by images of collections of morphisms in $\nCalg$, namely homotopy equivalences, left upper corner inclusions $\kappa_{A}$ from \eqref{oiwjveoivjowevsvdfvsfdv}, and the morphisms $\id_{K}\otimes \iota_{A}  :K\otimes qA \to K\otimes A$ 
for all $A$ in $\nCalg$.
\end{proof}

In \cref{wetgjowoegrewfwref}.\ref{wetigjwoerfwerfwrefwrf} we will show that the  mapping spaces in $  L_{K}\nCalg_{h,q}$   can be easily understood   in terms of a calculus of fractions. In a sense this result is of intermediate nature since 
in  \cref{wejogiwerfgerwfrewfwerfw} and  \cref{erugheiwugwefrefwerfewr}  we will state  a much better result at the cost of using more of $C^{*}$-algebra theory.

For every $A$ in $L_{K}\nCalg_{h}$ we consider the diagram $W(A):\nat^{\op}\to   L_{K}\nCalg_{h}$
\begin{equation}\label{gregfqweewdweqdqwed} \dots  \xrightarrow{\iota^{s}_{q^{3}A} } (q^{3}A)^{s}\xrightarrow{\iota^{s}_{q^{2}A}} (q^{2}A)^{s}\xrightarrow{\iota^{s}_{qA}} (qA)^{s}\xrightarrow{\iota^{s}_{A}} A^{s}\ .\end{equation} 
 We furthermore  let 
  $W_{q}$ be the subcategory of $L_{K}\nCalg_{h}$ generated by $\hat W_{q}$.
The diagram $W(A)$ is a {\em putative right calculus of fractions at $A$} for $W_{q}$ in the sense of \cite[Def. 7.2.2]{Cisinski:2017}. 
In fact, $A\simeq A^{s}$ is the final object of the diagram and all morphisms belong to $W_{q}$.

 We fix $A$ in $L_{K}\nCalg_{h}$  and
  consider the functor \begin{equation}\label{qewfwdwededwqed}H_{A}:L_{K}\nCalg_{h}\to \CGroups(\Spc)\ , \quad B\mapsto \colim_{n\in \nat}\Map_{L_{K}\nCalg_{h}}((q^{n}A)^{s},B) \ , \end{equation}
i.e., we insert the diagram \eqref{gregfqweewdweqdqwed} into the first argument of the mapping space and take the colimit.
The following result is a version of \cite[Prop. 2.1]{MR899916}.
\begin{prop}\label{weiojgowergfwerfwerf}
The functor $H_{A}$ is split-exact.
\end{prop}
\begin{proof}
 We consider a split-exact sequence 
\begin{equation}\label{sfdvsdfvdfvfdcsdc}0\to I\xrightarrow{j} B\xrightarrow{\pi} Q\to 0
\end{equation}  of $C^{*}$-algebras 
  with split $s:Q\to B$. We must show that
  \[H_{A}(I^{s})\oplus H_{A}(Q^{s})\xrightarrow{H_{A}(j^{s})+ H_{A}(s^{s})} H_{A}(B^{s})\]
  is an equivalence, where the superscript $s$ stands for $K$-stabilization as in \eqref{gugwerg34wrefwferf}.
  
  We first observe that it suffices to show that this map induces an isomorphism of groups of connected components.
  For $i$ in $\nat$ we write $H_{A,i}(-):=\pi_{i}H_{A}(-)$ for the corresponding abelian group valued functor.  For every $i$ in $\nat$ we have 
 a canonical isomorphism  $H_{A,i}( -)\cong H_{A,0}(S^{i}(-))$, where $S^{i}(-):=C_{0}(\R^{i})\otimes-$ is the $i$-fold suspension functor. Using the fact that the functor
  $S^{i}(-)$ preserves exact sequences we see that
  it suffices to show
  that \begin{equation}\label{sdbdfverwvfdvsdfvdfvsfv}H_{A,0}(I^{s})\oplus H_{A,0}(Q^{s})\xrightarrow{H_{A,0}(j^{s})\oplus H_{A,0}(s^{s})} H_{A,0}(B^{s})
\end{equation} 
  is an isomorphism for all split-exact sequences \eqref{sfdvsdfvdfvfdcsdc}.
    
    We now use that $L_{K}$ is a left Bousfield localization
    and  \cref{erigjeroferferfrefwerf},  or  directly \eqref{vwihviuewhviwevewre},
    in order see that  \begin{equation}\label{ervekjvhefjkvnefiuvevfvs}\Map_{L_{K}\nCalg_{h}}((q^{n}A)^{s},(-)^{s})\simeq \Map_{\nCalg_{h}}(q^{n}A,(-)^{s}) \simeq \ell \underline{\Hom}(q^{n}A,(-)^{s})
\end{equation}       as a functors from $\nCalg$ to  $\Spc$.  Combining this with \eqref{qewfwdwededwqed}
    we get an equivalence
    \[H_{A}(B)\simeq \colim \left(\ell\underline{\Hom}(A,B^{s})\xrightarrow{\iota_{A}^{*}} \ell\underline{\Hom}(qA,B^{s})\xrightarrow{\iota_{qA}^{*}} \ell\underline{\Hom}(q^{2}A,B^{s})  \xrightarrow{\iota_{q^{2}A}^{*}} \dots\right)\ .\]
     We define a map  
   \[l_{n} :\underline{\Hom}(q^{n}A,B^{s})\to \underline{\Hom}(q^{n+1}A,I^{s})\] as follows.
    Let   $f:q^{n}A \to B^{s}$ be an element of  $\underline{\Hom} ( q^{n}A ,B^{s})$. Then we  get an associated map
      $(f, s^{s}\pi^{s}f):	q^{n}A*q^{n}A\to B^{s}$ with the components as indicated, see    \cref{4tjigortgwrefrwefref}. We observe that
    $\pi^{s}\circ (f, s^{s}\pi^{s}f)\circ i_{q^{n}A}=0$, where $i_{q^{n}A}:q^{n+1}A\to q^{n}A*q^{n}A$ is the canonical inclusion. Therefore we  can define
    $$ l_{n}(f):=(f, s^{s}\pi^{s}f)\circ i_{q^{n}A}:q^{n+1}A\to I^{s}\ .$$
    
   For $g:q^{n}A\to I^{s}$ we have
    $l_{n}(j^{s}\circ g)= g\circ \iota_{q^{n}A}$.
  Finally, for $h:q^{n}A\to Q^{s}  $ we have
  $l_{n}(s^{s}\circ h)=  (s^{s}\circ h,s^{s}\circ h)  \circ i_{q^{n}A}    =0$.
    
   We next show that   $H_{A,0}(j^{s})$ is injective.
   Let $g:q^{n}A\to I^{s}$ represent an element $[g]$ in $H_{A,0}(I^{s})$ such that $H_{A,0}(j^{s})([g])=0$.
   Then there exists $m$ in $\nat$ with $n\le m$ such that
   $[j^{s}\circ g\circ \iota_{q^{n}A}\circ \dots \iota_{q^{m-1}A}]=0$ in $\pi_{0}\underline{\Hom}(q^{m}A,B^{s})\ .$ But then
   \[0=[l_{m}(j^{s}\circ g\circ \iota_{q^{n}A}\circ \dots \iota_{q^{m-1}A})]
       =[ g\circ \iota_{q^{n}A}\circ \dots \iota_{q^{m-1}A}\circ  \iota_{q^{m}A}]
    \]
    in $\pi_{0}\underline{\Hom}(q^{m+1}A,I^{s})$. This implies that $[g]=0$.

    We now show that the family of  maps $(l_{n})_{n}$ determines a well-defined map
    $l:H_{A,0}(B^{s})\to H_{A,0}(I^{s})$. Let  $[f]$  in $H_{A,0}(B^{s})$ be represented by a map 
    $f:q^{n}A\to B^{s}$. Then \[H_{A,0}(j^{s})([l_{n}(f)])=[f]-H_{A,0}(s^{s}\circ \pi^{s})([f])\ .\] The right-hand side does not depend on the choice of the representative. By the injectivity of statement above we conclude that 
   $ [l_{n}(f)]$ is well-defined.
   
 We thus have \[H_{A,0}(j^{s})\circ l+ H_{A,0}(s^{s})\circ H_{A,0}(\pi^{s})=\id_{H_{A,0}}(B) \ ,\] \[l\circ H_{A,0}(j^{s})=\id_{H_{A,0}(I^{s})}\ , \qquad 
 l\circ  H_{A,0}(s^{s})=0\ ,  \qquad  H_{A,0}(p^{s})\circ H_{A,0}(s^{s})=\id_{H_{A,0}(Q^{s})}\ .\]  These equalities imply  that \eqref{sdbdfverwvfdvsdfvdfvsfv} is an isomorphism.
  \end{proof}

  The putative right calculus of fractions is called a {\em right calculus of fractions} in the sense of \cite[Def. 7.2.6]{Cisinski:2017} 
   if 
  the functor $H_{A}(-)$ from \eqref{qewfwdwededwqed} sends the morphisms from $W_{q}$ to equivalences.
   
 \begin{prop}\label{wetgjowoegrewfwref}\mbox{}
\begin{enumerate}
\item \label{werogjpegrrewwferwfrefw}
$W(A)$ is a right-calculus of fractions.  
 \item\label{wetigjwoerfwerfwrefwrf} We have
\begin{equation}\label{vsdfvfsvqrferf}H_{A}(B)\simeq \Map_{L_{K}\nCalg_{h,q}}(A,B)\ . \end{equation} 
\item\label{woitgjwoerfrfwfrefefw} The $\infty$-category  $L_{K}\nCalg_{h,q}$ is additive.
\item\label{wetigjwoerfwerfwrefwrf1}  The localization $L_{q}$ is left-exact.
\item \label{wetigjwoerfwerfwrefwrf2} 
We have an essentially unique  commutative diagram \begin{equation}\label{qewfdqewdqewkjnjkrfruiw}\xymatrix{&\nCalg\ar[d]^{L_{h,K}}\ar[ddl]_{L_{h,K,\splt}}\ar[ddr]^{L_{h,K,q}}&\\&L_{K}\nCalg_{h}\ar[dr]^{L_{q}}\ar[dl]_{L_{\splt}}&\\ L_{K}\nCalg_{h,\splt}\ar@{..>}[rr]^{L}&&L_{K}\nCalg_{h,q}}\ ,
\end{equation}  
where $L$ is a left exact functor.

\end{enumerate}
 \end{prop}
\begin{proof}
The Assertion \ref{werogjpegrrewwferwfrefw} follows from 
\cref{weiojgowergfwerfwerf} and \cref{wergkooprefrwefrefre}.

The Assertion \ref{wetigjwoerfwerfwrefwrf} follows from \ref{werogjpegrrewwferwfrefw} and  
the   general formula \cite[Def. 7.2.8]{Cisinski:2017} 
 for the mapping spaces  in a localization in the presence of a right calculus of fractions.

For Assertion \ref{woitgjwoerfrfwfrefefw} note that by \eqref{vsdfvfsvqrferf} the mapping spaces in $L_{K}\nCalg_{h,q}$ are commutative groups.

Assertion \ref{wetigjwoerfwerfwrefwrf1}  is a consequence of  the formula (combine \eqref{qewfwdwededwqed} and \eqref{vsdfvfsvqrferf})  \begin{equation}\label{wervfweoivjroivjworecmdlcmsldffd}\Map_{L_{K}\nCalg_{h,q}}(A,B)\simeq \colim_{\nat}\Map_{L_{K}\nCalg_{h}}((q^{n}A)^{s},B)
\end{equation}
for the mapping space as a colimit over the filtered  poset $\nat$ and the fact that filtered colimits in $\Spc$ commute with finite limits.

 We finally show Assertion \ref{wetigjwoerfwerfwrefwrf2}. The two upper triangles in \eqref{qewfdqewdqewkjnjkrfruiw} reflect the definitions of the maps. 
  By Assertion \ref{wetigjwoerfwerfwrefwrf} the functor $H_{A}(-)$ represents the mapping space in $L_{K}\nCalg_{h,q}$. Since it is split-exact by \cref{weiojgowergfwerfwerf}
the localization  $L_{q}$ sends (the images under $L_{h,K}$ of)  split-exact sequences
to fibre sequences.  The composition $L_{h,K,q}$ is thus homotopy invariant, stable, and Schochet- and split-exact.  By the universal property of $L_{h,K,\splt}$ stated in  \cref{wreoigwtegwtgw9}.\ref{wergiuhweriugrewf2} we get the map $L$ and the two-morphism  filling of the outer triangle 
 in \eqref{qewfdqewdqewkjnjkrfruiw}. We finally use
 the universal property \eqref{qwfdqedewdqwedd} of $L_{h,K}$ in order to define the two-morphism filling  the lower triangle.
\end{proof}

\begin{prop}\label{rthogijprtgfrefrfw}
The functor $L$ from \eqref{qewfdqewdqewkjnjkrfruiw} is  equivalent to  the right adjoint of a right Bousfield localization  $$\incl: L_{K}{\nCalg_{h,\splt}}^{\group} \leftrightarrows L_{K}\nCalg_{h,\splt}:R\ .$$  \end{prop}
\begin{proof} 
We call a functor on $L_{K}\nCalg_{h}$ split-exact if it sends the elements of $W_{\splt}$ to equivalences.
We first show that $L_{\splt}\circ q^{s}$ is split-exact and left-exact
 and therefore by \cref{wreoigwtegwtgw9}.\ref{wergiuhweriugrewf} descends to a left-exact functor 
\[R:L_{K}\nCalg_{h,\splt}\to L_{K}\nCalg_{h,\splt}\ .\]

The functor $L_{\splt}:L_{K}\nCalg_{h}\to L_{K}\nCalg_{h,\splt}$ is split-exact  by definition,  and it is left exact by 
\cref{wreoigwtegwtgw9}.\ref{wergiuhweriugrewf}.
  The functor
   $L_{\splt}\circ q^{s}$ sends $A$ in $L_{K}\nCalg_{h}$  to the fibre of
$L_{\splt}((A*A)^{s})\xrightarrow{d_{A}^{s}} L_{\splt}(A^{s})$, where $d_{A}:A*A\to A$ is the fold map.
Since $L_{\splt}((A*A)^{s})\simeq L_{\splt}(A^{s})\times L_{\splt}(A^{s})$ by semi-additivity
we can identify $L_{\splt}\circ q^{s}$ with the functor which sends $A$ to the fibre of  some natural map
$ L_{\splt}(A^{s})\times L_{\splt}(A^{s})\to  L_{\splt}(A^{s})$ between split-exact and left exact functors. 
We conclude that $L_{\splt}\circ q^{s} $ itself  is split-exact and left exact.
%

As a consequence of \cref{wtgijoweoifrefwerferfwrf} and the fact that $L_{\splt}$ preserves products  the functor $R$ takes values in $L_{K}{\nCalg_{h,\splt}}^{\group}$. We have a  natural transformation 
\begin{equation}\label{gbsbkmopsfbsbdffbsb}\kappa:=L_{\splt}(\iota^{s}):R\to \id_{L_{K}{\nCalg_{h,\splt}}}:L_{K}{\nCalg_{h,\splt}}\to L_{K}{\nCalg_{h,\splt}}\ .
\end{equation}
If $A$  in $ L_{K}\nCalg_{h}$  has the property that $L_{\splt}(A)$ is a group, then  $\kappa_{L_{\splt}(A)}\simeq L_{\splt}(\iota^{s}_{A})$ is an equivalence by  \cref{wergkooprefrwefrefre}
since 
$L_{\splt}$ is split-exact.
Since $L_{\splt}$, being a Dwyer-Kan localization,  is essentially surjective
we can conclude that  the essential image of $R$ is  $L_{K}{\nCalg_{h,\splt}}^{\group}$.

In particular we see that  $\kappa_{R(A)}$ is an equivalence for very $A$ in $ L_{K}\nCalg_{h,\splt}$.
Since
$R\circ L_{\splt}\simeq L_{\splt}\circ q^{s}$ is split-exact and takes values in groups we can conclude again by \cref{wergkooprefrwefrefre} that
$R(\kappa_{A})$ is also  an equivalence for every $A$.
 
As explained at the beginning of \cref{werigjowergerwg9}  this implies  that $R$
is the right-adjoint of a right Bousfield localization
\begin{equation}\label{vsfv3fvfsvrqfwfsvfwervwerfwerfrw}\incl:L_{K}{\nCalg_{h,\splt}}^{\group}\leftrightarrows L_{K}\nCalg_{h,\splt}:R
\end{equation} 
with counit $\kappa$.

The functor $L$ inverts the morphisms $\kappa_{A}$ since
$L(\kappa_{A})\simeq L_{q}(\iota^{s}_{A})$. This gives the factorization $L'$ in the diagram 
 \[\xymatrix{&L_{K}\nCalg_{h}\ar[dr]^{L_{q}}\ar[dl]_{L_{\splt}}&\\ L_{K}\nCalg_{h,\splt}\ar[rr]^{L}\ar[dr]^{R}&&L_{K}\nCalg_{h,q}\ar@/^0.4cm/@{..>}[dl]^{L^{\prime,-1}}\\&L_{K}{\nCalg_{h,\splt}}^{\group}\ar@/^0.4cm/@{..>}[ur]^{L'} &}\]
 Since $R\circ L_{\splt}$ sends    the morphisms $\iota^{s}_{A}$ to the equivalences $R(\kappa_{A})$  for all $A$ in $L_{K}\nCalg_{h}$ 
  we get an inverse to  $L'$
 from the universal property of $L_{q}$. 
\end{proof}

  Recall the definition  \eqref{vsdfvdfsjvposdvsfdvsfdvsfdvsfdv} of $ L_{h,K,q}: \nCalg\to L_{K}\nCalg_{h,q}$. 
\begin{prop}\label{erfjrgergfregrgw9}
For any left-exact and additive $\infty$-category $\bD$ we have an equivalence
\begin{equation}\label{fwefwefhhuhdiqweuhdwedwqdd} L_{h,K,q}^{*}:\Fun^{\lex}(L_{K}\nCalg_{h,q},\bD)\stackrel{\simeq}{\to} \Fun^{h,s,splt+Sch}(\nCalg,\bD)\ .\end{equation}
\end{prop}
\begin{proof}
This follows from the following chain of equivalences
\begin{eqnarray*}
\Fun^{\lex}(L_{K}\nCalg_{h,q},\bD)&\stackrel{L^{\prime,*}}{\simeq}&\Fun^{\lex}(L_{K}{\nCalg_{h,\splt}}^{\group},\bD)\\
&\stackrel{R^{*}}{\simeq}&
\Fun^{\lex}(L_{K}\nCalg_{h,\splt},\bD)\\
&\stackrel{L_{h,K,\splt}^{*}, \eqref{ewfwqedqewdedqwd}}{\simeq}&\Fun^{h,s,splt+Sch}( \nCalg ,\bD)\ .
\end{eqnarray*}
In order to see that $R^{*}$ is an equivalence note that 
the  components $\kappa_{A}:R(A)\to A$ of the natural transformation \eqref{gbsbkmopsfbsbdffbsb} generate the  Dwyer-Kan localization
$R$. As a 
 consequence of \cref{wergkooprefrwefrefre} any left exact functor 
$ L_{K}\nCalg_{h,\splt}\to \bD$ to an additive $\infty$-category $\bD$ sends theses components   to equivalences. 
%
\end{proof}

\begin{prop} For $?$ in $\{\min,\max\}$
the localization $L_{q}$ has a symmetric monoidal refinement and the tensor product $\otimes_{?}$ on
$L_{K}\nCalg_{h,q}$ is bi-left exact.
\end{prop}
\begin{proof}
Since $L_{\splt}$ has a symmetric monoidal refinement with a bi-left exact tensor product by \cref{wtgkwotpgkelrf09i0r4fwefwerf} it suffices to show that the functor $L$ has one.
As seen in the proof of \cref{rthogijprtgfrefrfw}
 we have a  functor which sends $A$ in $L_{K}\nCalg_{h,\splt}$ to the diagram
\[\xymatrix{R( A ) \ar[r]\ar[dr]^{\kappa_{A}}&  A  \times  A  \ar[r]\ar[d]^{\pr_{0}}& A   \\& A  &}\ ,\]
where the upper sequence is a fibre sequence. Since $\otimes$ is bi-exact on $L_{K}\nCalg_{h,\splt}$, for $B$ in $L_{K}\nCalg_{h,\splt}$ we get a similar diagram
\[\xymatrix{R( A )\otimes  B  \ar[r]\ar[dr]^{\kappa_{A}\otimes  B }&  (A   \times A )  \otimes  B \ar[r]\ar[d]^{\pr_{0}\otimes B}& A\otimes  B\\& A \otimes  B &}\ .\]
 We can conclude that
$\kappa_{A\otimes B}\simeq \kappa_{A}\otimes B$. 
In particular, $-\otimes B$ preserves the generators of the Dwyer-Kan localization $L$ which therefore has a symmetric monoidal refinement. Furthermore, $-\otimes B$ descends to a left-exact functor on $L_{K}\nCalg_{h,q}$.
\end{proof}

For every symmetric monoidal additive $\infty$-category we thus get an equivalence
\begin{equation}\label{fwefwefhhuhdiqweuhdwedwqdd1} L_{h,K,q}^{*}:\Fun_{\otimes/\lax}^{\lex}(L_{K}\nCalg_{h,q},\bD)\stackrel{\simeq}{\to} \Fun_{\otimes/\lax}^{h,s,splt+Sch}(\nCalg,\bD)\ .\end{equation}

\begin{rem}\label{wtokgwopgfrefreferfw}
In this remark we provide the bridge to Kasparov modules. We refer  to  \cite{kasparovinvent}, \cite{blackadar}
for a detailed theory.
Let $f:qA\to B\otimes K$ be a homomorphism with components $\hat f_{i}:A\to M(B\otimes K)$. The corresponding 
 $(A,B)$-bimodule $(H,\phi,F)$ is the $\Z/2\Z$-graded $B$-Hilbert $C^{*}$-module $L^{2}(B)\oplus L^{2}(B)$ 
 with the  odd endomorphism \[F:=\left(\begin{array}{cc}0&1\\1&0 \end{array}\right)\]
 and \begin{equation}\label{feoifjwioerfrff32}\phi:=\left(\begin{array}{cc}\hat f_{0}&0\\0& \hat f_{1} \end{array}\right):A\to \Mat_{2}(B(L^{2}(B))\ ,
\end{equation}  
 where in order to interpret \eqref{feoifjwioerfrff32}  we identify $B(L^{2}(B))\cong M(B\otimes K)$ in the canonical way.
   For separable $C^{*}$-algebras $A,B$ we
 interpret 
 $ \underline{\Hom}(qA,K\otimes B)$ as the topological space of $(A,B)$-Kasparov modules in the Cuntz picture.
Its underlying space \begin{equation}\label{tgihjoi4gfrewgrefrefw}\ell \underline{\Hom}(qA,K\otimes B)\simeq \Map_{L_{K}\nCalg_{\sepa,h}}(qA,B)\simeq \Map_{L_{K}\nCalg_{\sepa,h}}(q^{s}A,B)
\end{equation}
has a natural refinement to a commutative monoid in spaces. This monoid structure reflects the direct sum of Kasparov modules. Since groups and cogroups in a semi-additive $\infty$-category coincide  
 \cref{wtgijoweoifrefwerferfwrf} implies that $\ell\underline{\Hom}(qA,K\otimes B)$ it is actually a commutative group.
We have a natural map of commutative groups
\begin{equation}\label{fqwefiuqhduiijqqwedqwedqewdeq}\ell\underline{\Hom}(qA,K\otimes B)\xrightarrow{}  \Map_{L_{K}\nCalg_{\sepa,h,q}}(A,B)
\ ,\end{equation}
given by the composition
of \eqref{tgihjoi4gfrewgrefrefw}
  with the canonical map from the right-hand side of this equivalence to the second stage of the colimit in \eqref{wervfweoivjroivjworecmdlcmsldffd}.
  By \cref{wejogiwerfgerwfrewfwerfw}  below 
 we see that this   map is actually an equivalence presenting the commutative mapping groups in $L_{K}\nCalg_{\sepa,h,q}$  in terms of spaces of Kasparov modules. 
 \hB
\end{rem}

%
%
%
 
 For completeness of the presentation  we now  discuss  \cite[Thm. 1.6]{MR899916}. All of the above has a version for separable algebras which we will indicate by an additional subscript $\sepa$.
Let $A$ be a separable $C^{*}$-algebra.
\begin{theorem}[{\cite[Thm. 1.6]{MR899916}}]\label{erigjwoergerfwerfwref}
There exists a homomorphism $\phi:qA\to \Mat_{2}(q^{2}A)$ such that $\Mat_{2} (\iota_{qA})\circ \phi:qA\to  \Mat_{2}(qA)$
and $ \phi\circ \iota_{qA}:q^{2}A\to \Mat_{2}(q^{2}A)$ are homotopic to the left upper corner inclusions.
\end{theorem}

Before we sketch the proof we derive the consequences of \cref{erigjwoergerfwerfwref}.

\begin{kor}\label{wejogiwerfgerwfrewfwerfw} 
For separable $C^{*}$-algebras $A$ and $B$ the morphism \eqref{fqwefiuqhduiijqqwedqwedqewdeq} is an equivalence.
  \end{kor}
\begin{proof}
This is an immediate consequence of \eqref{wervfweoivjroivjworecmdlcmsldffd}, 
\eqref{ervekjvhefjkvnefiuvevfvs} and \cref{erigjwoergerfwerfwref} which implies that the colimit in
\eqref{wervfweoivjroivjworecmdlcmsldffd} stabilizes from $n=1$ on.
\end{proof}

\begin{prop}\label{erugheiwugwefrefwerfewr}
The functor $q_{\sepa}^{s}:L_{K}\nCalg_{\sepa,h}\to q_{\sepa}^{s}L_{K}\nCalg_{\sepa,h}$ is the right-adjoint of a Bousfield localization
\[\incl:q_{\sepa}^{s}L_{K}\nCalg_{\sepa,h}\leftrightarrows  L_{K}\nCalg_{\sepa,h}:q_{\sepa}^{s}\]
and \[q_{\sepa}^{s}:L_{K}\nCalg_{\sepa,h}\to q_{\sepa}^{s}L_{K}\nCalg_{\sepa,h}\] represents its target as the Dwyer-Kan localization
at the set $\hat W_{\sepa,q}$ from the separable version of \eqref{wergerwgerferfrefwerf}.
\end{prop}
\begin{proof}
Let $\incl:q_{\sepa}^{s}L_{K}\nCalg_{\sepa,h}\to L_{K}\nCalg_{\sepa,h}$ denote   the inclusion of the full subcategory of $L_{K}\nCalg_{\sepa,h}$ on the image of $q_{\sepa}^{s}$. We have a natural transformation \begin{equation}\label{woierfjweriofwiofwerfwref}\iota^{s}:\incl\circ q_{\sepa}^{s}\to \id_{L_{K}\nCalg_{\sepa,h}}\ .
\end{equation} 
For $A$ in $q_{\sepa}^{s}L_{K}\nCalg_{\sepa,h}$ and $B$ be in $L_{K}\nCalg_{\sepa,h}$ the binatural
transformation
\[\hspace{-0.4cm}\Map_{q^{s}_{\sepa}L_{K}\nCalg_{\sepa,h}}(A,q^{s}_{\sepa}B)\xrightarrow{\incl}
\Map_{L_{K}\nCalg_{\sepa,h}}(\incl A,\incl(q^{s}_{\sepa}B))\xrightarrow{\iota^{s}_{B,*}}  \Map_{L_{K}\nCalg_{\sepa,h}}( \incl A, B)\] is an equivalence. To this end we set 
$A=q_{\sepa}^{s}A'$ for some $A'$ in $ L_{K}\nCalg_{\sepa,h}$ and factorize the map as a composition of equivalences
 \begin{eqnarray*}
\Map_{q^{s}_{\sepa}L_{K}\nCalg_{\sepa,h}}(A,q^{s}_{\sepa}B)&\stackrel{\incl}{\simeq}&
\Map_{L_{K}\nCalg_{\sepa,h}}(q_{\sepa}^{s}A', q_{\sepa}^{s}B)\\&\stackrel{\cref{wejogiwerfgerwfrewfwerfw}}{\simeq}&
\Map_{L_{K}\nCalg_{\sepa,h,q}}(A', q_{\sepa}^{s}B)\\&\stackrel{\iota^{s}_{B,*}}{\simeq}&
\Map_{L_{K}\nCalg_{\sepa,h,q}}(A', B)\\& \stackrel{\cref{wejogiwerfgerwfrewfwerfw}}{\simeq}&
\Map_{L_{K}\nCalg_{\sepa,h}}(\incl A, B)
\end{eqnarray*}
We conclude that  \eqref{woierfjweriofwiofwerfwref} is the counit of a right Bousfield localization.  
Since the right  Bousfield localization is a Dwyer Kan localization at the set of the components of its counit
we conclude the second assertion by a comparison with \eqref{wergerwgerferfrefwerf}.
\end{proof}

 Recall the construction \eqref{fqoiwjefioqwejfoiqwjioewdjq} of a sum of a family of $C^{*}$-algebras.

\begin{kor}\label{roijgferiogerfwerffwref}
The category $L_{K}\nCalg_{\sepa,h,q}$ admits countable coproducts which are represented by the free product and also by 
the sum in $\nCalg_{\sepa}$.
\end{kor}
\begin{proof}
Since $L_{K}\nCalg_{\sepa,h,q}$ is  a right Bousfield localization  of $L_{K}\nCalg_{\sepa,h}$ it inhertits all colimits from the latter category and the inclusion functor, being a left adjoint, preserves them. The assertion now follows from  \cref{wreogkpwegrewferf1}.\ref{wetokgpwerferwfwref1}.
\end{proof}

 \cref{erigjwoergerfwerfwref} is crucial for understanding the nature of the localization 
$L_{q}$ which in turn  implies the important categorical property  of $L_{K}\nCalg_{\sepa,h,q}$ of being countably cocomplete   and of course the  
 simple formula   \eqref{vweriuvbvjbjkfvfvfsfvsdfv} for the mapping spaces. Because of its relevance, for completeness  of the presentation
 we  decided to   
  repeat the proof of  \cref{erigjwoergerfwerfwref} from \cite{MR899916}.

  \begin{proof}[Proof of \cref{erigjwoergerfwerfwref}]
  We will use as a fact:
  \begin{lem}\label{wefojgiwergerfwerf9}
  If $I\to B$ is the inclusion of an ideal, then
  $I*I\to B*B$ is injective. 
  \end{lem}
%
%
%
%

We abbreviate $QA:= A*A$.
By \cref{wefojgiwergerfwerf9}  the map $qA*qA\to QA*QA$ is injective.    We construct the following diagram of exact sequences
  \[\xymatrix{&0\ar[r]&\Mat_{2}(qQA)\ar[r]&\Mat_{2}(Q^{2}A)\ar[r]&\Mat_{2}(QA)\ar[r]&0\\
  0\ar[r]&J\ar[r]\ar[ur]&D\ar[r]^(0.6){\pi}\ar[ur]&S\ar[ur]\ar[r]&0&\\&0\ar[r] &\Mat_{2}(q^{2}A)\ar@/^{-1.5cm}/[uu]\ar[r] &\Mat_{2}(QqA)\ar[r]\ar@/^{-1.5cm}/[uu]&\ar@/^{-1.5cm}/[uu]\Mat_{2}(qA)\ar[r]&0\\ 0\ar[r]&J\ar@{=}@/^{1.5cm}/[uu]\ar[ur]\ar[r]&R\ar@/^{1.5cm}/[uu]\ar[ur]\ar[r]&\Mat_{2}(qA)\ar@/^{1.5cm}/[uu]\ar@{=}[ur]\ar[r]&0&}\ .\]
  All vertical maps are injective. The $C^{*}$-algebra 
   $R$ is defined as  the subalgebra of $\Mat_{2}(QqA)$ generated by
  matrices of the form
  \[\left(\begin{array}{cc}\eta_{0}(qA)&\eta_{0}(qA)  \eta_{1}(qA)\\  \eta_{1}(qA)  \eta_{0}(qA)& \eta_{1}(qA)\end{array}\right)\ ,\]
  where we use the notation $\eta_{0}$ and $ \eta_{1}$ for the canonical inclusions $\iota_{0}$ and $\iota_{1}$
  of $qA$ into $QqA$.
  We observe that the projection $R\to \Mat_{2}(qA)$ is surjective and defines the ideal $J$ as its kernel.
  We let $D$ be the subalgebra of $\Mat_{2}(Q^{2}A)$ generated by the image of $R$ and the elements 
  $$\left(\begin{array}{cc}\eta_{0} (\iota_{0}(a) )&0 \\ 0&   \eta_{1}   (\iota_{0}(a))\end{array}\right)\ , \quad a\in A\ .$$
   
 One checks by an explicit calculation that $R$ is an ideal in $D$, and hence $J$ is also an ideal in $D$.
  Then $S$ is the subalgebra of $\Mat_{2}(QA)$ generated by $\Mat_{2}(qA)$ and the diagonal elements 
\[\left(\begin{array}{cc} \iota_{0}(a) &0 \\ 0&  \iota_{0}(a)\end{array}\right)\ , \quad a\in A\ .\]
We let $U_{t}$ be the rotation matrix from \eqref{qrfqfoiue9foqwefqwedwqewdqewd}. We note that
conjugation by $U_{t}$ on $\Mat_{2}(QA)$ preserves the subalgebra $S$.  The derivative
of this action is a bounded derivation $\bar \delta $ of $S$.

By Pedersen's derivation lifting theorem  \cite{peder}, \cite[Thm 8.6.15]{pedersen} there exists a
derivation $\delta$ of $D$ such that $\pi\circ \delta=\bar \delta\circ \pi$. It is at this point
where separability of $A$ is important. There are counterexamples
to the derivation lifting theorem for non-separable algebras.

We define the family  $\sigma_{t}:=e^{t\delta}$ of automorphisms of $D$ and set $\sigma:=\sigma_{1}$.

We now define $\phi:qA\to  \Mat_{2}(q^{2}A)$ as the homomorphism with the components
$\hat \phi_{i}:A\to   \Mat_{2}(QqA)\to M( \Mat_{2}(q^{2}A))$ given by 
\begin{equation}\label{fqweewqfdewdewqdwedqwd} \hat \phi_{0}:=\left(\begin{array}{cc}\eta_{0}\circ  \iota_{0}   &0 \\ 0&  \eta_{1}\circ \iota_{0} \end{array}\right)\ , \quad 
\hat \phi_{1}:=\sigma \left(\begin{array}{cc}\eta_{0} \circ \iota_{1}   &0 \\ 0&  \eta_{1}\circ \iota_{0} \end{array}\right)\ . \end{equation}
In order to see that the application of $\sigma$ is well-defined we 
rewrite
\[\left(\begin{array}{cc}\eta_{0}\circ  \iota_{1}   &0 \\ 0&  \eta_{1}\circ \iota_{0} \end{array}\right)=
\left(\begin{array}{cc}\eta_{0} \circ \iota_{0}   &0 \\ 0&  \eta_{1}\circ \iota_{0} \end{array}\right)-\left(\begin{array}{cc} \eta_{0}\circ (\iota_{1} -\iota_{0}) &0 \\ 0&  0 \end{array}\right)\]
which obviously takes values in $D$. Using in addition a   similar rewriting  of $\hat \phi_{0}$
one then 
checks that  $\hat \phi_{0}-\hat \phi_{1}$ takes values in $\Mat_{2}(q^{2}A)$.  
 
The homotopy $\gamma_{t}$ from the left upper corner inclusion $q^{2}A\to \Mat_{2}(q^{2}A)$  to $\phi\circ \iota_{qA}$  has the components $\hat \gamma_{i}:qA\to \Mat_{2}(Q^{2}A)$ given by the map (again given by a pair of components)
\[\hat \gamma_{0}:= \left(\left(\begin{array}{cc}\eta_{0} \circ \iota_{0}   &0 \\ 0&  \eta_{1} \circ\iota_{0} \end{array}\right)\ ,  
 \sigma_{t} \left(\begin{array}{cc}\eta_{0} \circ \iota_{1}   &0 \\ 0&  \eta_{1} \circ\iota_{1} \end{array}\right)\right)\]
 and the map 
 \[\hat \gamma_{1}:= \left(\left(\begin{array}{cc}\eta_{1}\circ  \iota_{0}   &0 \\ 0&  \eta_{1}\circ \iota_{1} \end{array}\right)\ ,  
 U_{t} \left(\begin{array}{cc}\eta_{1}\circ  \iota_{1}   &0 \\ 0&  \eta_{1}\circ \iota_{0} \end{array}\right) U_{t}^{*}\right)\ .\]
 Similarly,   a homotopy $\lambda_{t}$ from the left upper corner inclusion
 $qA\to  \Mat_{2}(qA) $   to $\iota_{qA}\circ \phi$    has the components $\Mat_{2}(p_{0})\circ \tilde \lambda_{t}:qA \to \Mat_{2}(qA)$,
 where $p_{0}:  QqA \to  qA$ and $\tilde \lambda_{t}:qA\to R\to \Mat_{2}(QqA)$ is given by
 \[  \left(\left(\begin{array}{cc}\eta_{0} \circ \iota_{0}   &0 \\ 0&  \eta_{1}\circ \iota_{1} \end{array}\right)\ ,  
 \sigma_{t}\left(\begin{array}{cc}\eta_{0} \circ \iota_{1}   &0 \\ 0&  \eta_{1}\circ \iota_{0} \end{array}\right)  \right)\ .\]
  We leave the  justifications for these formulas to the interested reader or refer to the proof of  \cite[Thm. 1.6]{MR899916}.
    \end{proof}

\section{The automatic semiexactness theorem}\label{weijgiowergfrfreferfwr}

Since the symmetric monoidal functor $\kk_{\sepa}$ is homotopy invariant, stable and split-exact,  it
  belongs to the right-hand side of  the separable version of the  equivalence \eqref{fwefwefhhuhdiqweuhdwedwqdd1} describing the universal property of $L_{\sepa,h,K,q}$ for $\bD:=\KK_{\sepa}$.   
Its preimage under this equivalence 
is the  left-exact and symmetric monoidal functor $h$ depicted by the lower horizontal arrow in the commutative triangle \begin{equation}\label{frewfijofifqedwedqd}\xymatrix{&\nCalg_{\sepa}\ar[dl]_{L_{\sepa,h,K,q}}\ar[dr]^{\kk_{\sepa}}&\\L_{K}\nCalg_{\sepa,h,q}\ar@{..>}[rr]^{h}&&\KK_{\sepa}}\ .
\end{equation}

The functor $h$ will be called the {\em comparison functor}.

%
 In  \cref{wgokjweprgrefwrefwrfw} we claim that this comparison functor is an equivalence.   We will give two immediate proofs 
 which   at least implicitly   assume the formulas  \eqref{vefdvvfdvsdvsfdvdfcl}, \eqref{vefdvvfdvsdvsfdvdf} and \cref{qiurhfgiuewrgwrfrefrfwrefw}. They therefore involve more than just the simple homotopy theoretic considerations from the present notes.  In order to provide a selfcontained proof  we will  formulate two equivalent statements \cref{tkohprtggrtgetrg} and  \cref{rijoirwtertretwetwretetwretw}.  
Note that  \cref{tkohprtggrtgetrg} is just an assertion about functors defined on the category of separable $C^{*}$-algebras and does not require any $K$-theoretic element at all. On the other hand,  the argument for 
\cref{rijoirwtertretwetwretetwretw} is quite accessible to the methods developed here so that we
write out the details of the argument for this version. This finally also verifies \eqref{vefdvvfdvsdvsfdvdf} in a non-circular manner.

The fact that the comparison functor is an equivalence has the important consequence that $\KK_{\sepa}$
admits countable colimits and is idempotent complete, see \cref{weijotgwegferfwrefw}.
We do not have a direct proof of this fact just from the construction of $\KK_{\sepa}$.

We start with formulating the main result of the present section.
\begin{theorem}\label{wgokjweprgrefwrefwrfw}
The comparison functor $h:L_{K}\nCalg_{\sepa,h,q}\to \KK_{\sepa}$ is an equivalence.
\end{theorem}
 As said above, we will give two proofs which should convince the reader that the assertion is true. On the other hand both  involve deep facts from the classical $KK$-theory which are not easily provable on the basis of the approach taken in the present paper. 

\begin{proof}[1. Proof via universal properties]
We will show that $\kk_{\sepa}:\nCalg_{\sepa}\to \KK_{\sepa}$   has the same universal property as
$L_{\sepa,h,K,q}:\nCalg_{\sepa}\to L_{K}\nCalg_{\sepa,h,q}$ stated in the separable version of  \eqref{fwefwefhhuhdiqweuhdwedwqdd}. It appears as the upper horizontal equivalence in the diagram below, where $\bD$
 is any  left-exact and additive $\infty$-category:
\begin{equation}\label{vewrfcwdcsdcsdca}\xymatrix{\Fun^{\lex}(\KK_{\sepa},\bD)\ar@{-->}[r]^-{\simeq}_-{\kk_{\sepa}^{*}}\ar[d] &  \Fun^{h,s,splt+Sch}(\nCalg_{\sepa},\bD) \ar[d] \\ \Fun^{\coprod}(\KK_{\sepa},\bD)\ar[d]\ar[r]_-{\kk_{\sepa}^{*}}^-{\simeq}& \Fun^{h,s,splt}(\nCalg_{\sepa},\bD) \ar[d]^{!}\\\Fun (\KK_{\sepa},\bD) \ar[r]_-{\kk_{\sepa}^{*}}^-{\simeq} & \Fun^{\tilde W_{\sepa,\se}}(\nCalg_{\sepa},\bD)  }  ,
\end{equation} 
where the superscript $\coprod$ indicates finite coproduct preserving functors.
The lower square has been discussed in the proof of \cite[Thm.2.23]{KKG}, see \cite[(2.31)]{KKG}. 
The lower horizontal equivalence reflects the fact (see \cref{ergiojeroigwerfwrefrefdvs}) that $\kk_{\sepa}$ is the Dwyer-Kan localization 
at the set $\tilde W_{\sepa,\se}$ of morphisms in $\nCalg_{\sepa}$ inverted by $\kk_{\sepa}$. 
The crucial point is the existence of the arrow marked by $!$. To see that it exists in \cite{KKG} we used the comparison of
$\ho\KK_{\sepa}$ with the classical theory  and the fact that the latter has a universal property  involving the condition of split-exactness \cite[Cor. 2.4]{KKG}.
  The  middle
 horizontal  equivalence has been discussed in  the  proof of \cite[Thm.2.23]{KKG}. For our present purpose
 we need the dashed equivalence which is obtained by an analogous argument explained in the subsequent paragraph.
  
  All vertical arrows in the diagram above are fully faithful functors.
Since $\kk_{\sepa}$ is homotopy invariant, stable, and  Schochet- and  semiexact, it is Schochet- and split-exact. Therefore  the dashed arrow exists. We must show that it is essentially surjective. Thus  consider a functor  $F$  in $ \Fun^{h,s,splt+Sch}(\nCalg_{\sepa},\bD)$. It gives rise to a functor $\hat F$ in $ \Fun (\KK_{\sepa},\bD)$ such that $\kk_{\sepa}^{*}\hat F\simeq F$.
It remains to show that $\hat F$ is left-exact. It is clearly reduced.  Every cartesian square in $\KK_{\sepa}$ can be represented as the image under $\kk_{\sepa}$ of  a Schochet fibrant cartesian square, and by assumption  $F$ sends this square to a cartesian square in $\bD$. This implies left-exactness of $\hat F$.
See the proof of \cref{weokjgpwerferfwef}.\ref{elrkijgowlferfwffr9} for an analogous argument.
  \end{proof}

\begin{proof}[2.Proof based on \eqref{vefdvvfdvsdvsfdvdf}]
For any two separable $C^{*}$-algebras $A$ and $B$ the comparison map induces the second map in 
\[\pi_{0} \underline{\Hom}(qA,K\otimes B)\stackrel{\cref{wejogiwerfgerwfrewfwerfw}}{\simeq}\pi_{0} \Map_{L_{K}\nCalg_{\sepa,h,q}}(A,B)\stackrel{h}{\to}
\KK_{\sepa,0}(A,B)\ .\] In view of \eqref{vefdvvfdvsdvsfdvdf}
it is a bijection. 
%
%
%
Using the left-exactness of the comparison functor we can upgrade this
to obtain an isomorphism between the higher homotopy groups of mapping spaces by 
 inserting the suspension $S^{i}(B)$ in the place of  $B$ for $i$ in $\nat$. 
Since the comparison  functor is clearly essentially surjective
 it is an equivalence.
 \end{proof}

 \begin{rem}
 We note that the two proofs are not independent.  In order to obtain the marked arrow in \eqref{vewrfcwdcsdcsdca}
 we  used \cite[Cor.2.4]{KKG}  which is based on the universal property 
 of $\ho\KK_{\sepa}$  as the initial functor to an additive category
 which is homotopy invariant, stable and split-exact. The verification of this universal property
   \cite[Thm. 4.5]{higsondiss} 
 also uses the formula  \eqref{vefdvvfdvsdvsfdvdf}. 
\hB 
 \end{rem}
 

\begin{kor}\label{weijotgwegferfwrefw}\mbox{}
\begin{enumerate}
\item  \label{fhqdewjfioqwefqdwedqwed}The category $\KK_{\sepa}$ admits all countable colimits.
\item  \label{fhqdewjfioqwefqdwedqwed1} For a countable family of separable $C^{*}$-algebras $(B_{i})_{i\in I}$ we have an equivalence
\[\bigsqcup_{i\in I} \kk_{\sepa}(B_{i})\simeq \kk_{\sepa}(\bigoplus_{i\in I} B_{i})\ .\]
\item  \label{fhqdewjfioqwefqdwedqwed2} $\KK_{\sepa}$ is idempotent complete and the inclusion $\KK_{\sepa}\to \KK$ identifies
$\KK_{\sepa}$ with the full subcategory of compact objects of $\KK$.
\end{enumerate}
\end{kor}
 \begin{proof}
 Since $\KK_{\sepa}$ is stable by \cref{wtgkwotpgkelrf09i0r4fwefwerf}.\ref{wetiogjwegferwferfwrefw881} it admits all finite colimits. For Assertion \ref{fhqdewjfioqwefqdwedqwed} it thus suffices to show that $\KK_{\sepa}$ admits countable coproducts. But this immediately follows from  \cref{wgokjweprgrefwrefwrfw} and \cref{roijgferiogerfwerffwref}.
 
The same results imply  also Assertion \ref{fhqdewjfioqwefqdwedqwed1}.

Assertion  \ref{fhqdewjfioqwefqdwedqwed2} is a general fact about $\Ind$-completions of stable $\infty$-categories admitting countable colimits and thus an immediate consequence of \cref{iejgowergrfrfwerfref}
and Assertion \ref{fhqdewjfioqwefqdwedqwed}.
 \end{proof}

By comparing the universal properties of $\kk_{\sepa}$ and $L_{\sepa,h,K,q}$ stated in the separable version of \eqref{fwefwefhhuhdiqweuhdwedwqdd} and \eqref{vdfvqr3ferffafafds} we see that
\cref{wgokjweprgrefwrefwrfw} is equivalent to  the {\em automatic semiexactness theorem}.

\begin{theorem}\label{tkohprtggrtgetrg}
For every left-exact and additive $\infty$-category $\bD$ the canonical inclusion is an equivalence
 $$  \Fun^{h,s,se+Sch}(\nCalg_{\sepa},\bD)\stackrel{\simeq}{\to} \Fun^{h,s,splt+Sch}(\nCalg_{\sepa},\bD)\ .$$
  \end{theorem}
   Apriori semiexactness is a much stronger  condition than  split-exactness.

 Recall from \eqref{tbrtbbgdbfert} that for every exact sequence of $C^{*}$-algebras  \begin{equation}\label{qewfiojoifqefqdewdqqwedewdq}0\to A\to B\xrightarrow{f} C\to 0
\end{equation}
 we have defined a map $\iota_{f}:A\to C(f)$ from $A$ to the mapping cone $C(f)$ of $f$.
The following theorem was shown in \cite{zbMATH03973625}.
 \begin{theorem}\label{rijoirwtertretwetwretetwretw} For every semi-split exact sequence \eqref{qewfiojoifqefqdewdqqwedewdq} of separable $C^{*}$-algebras the morphism
$L_{\sepa,h,K,q}(\iota_{f})$ in $L_{K}\nCalg_{\sepa,h,q}$ is an equivalence. 
 \end{theorem} 
In view of \cref{wetigowergferferferwfw} the \cref{rijoirwtertretwetwretetwretw}  is  equivalent to  the automatic semiexactness \cref{tkohprtggrtgetrg} and hence to 
  \cref{wgokjweprgrefwrefwrfw}.
 Since the automatic semiexactness theorem is absolutely crucial in order to see that our construction of
 $\KK_{\sepa}$ coincides with the classical constructions and the proof 
 of \cref{rijoirwtertretwetwretetwretw} in  \cite{zbMATH03973625} implicitly already  uses 
this comparison   we must give an independent argument in order to avoid a circularity.

The remainder of the present section is devoted to the proof of \cref{rijoirwtertretwetwretetwretw}.
We will closely follow the outline given in
 the appendix of \cite{zbMATH03973625}.  We reduce the argument to a single calculation \cref{fkjgogerfwreferfw} in
 $\Map_{L_{K}\nCalg_{\sepa,h,q}}(\C,\C)$ verifying that the composition of two explicit
candidates for the Bott element and its inverse is really the identity.   We use this argument also as a chance to present a calculus 
 which allows to manipulate morphisms in $L_{K}\nCalg_{\sepa,h,q}$ by constructions with 
 semi-split exact sequences.

In the following discussion it is important to remember, in which categories the morphisms live. We will therefore
be more precise with the notation.
We will abbreviate $L:=L_{\sepa,h,K,q}$. In  contrast to the conventions in the rest of the text,  e.g. a $C^{*}$-algebra considered as an object of $L_{K}\nCalg_{\sepa,h,q}$ will be denoted by $L (A)$ instead of simply by $A$.
By abusing the notation, for a morphism $f:qA\to K\otimes B$  we will also use the notation $L(f)$ for the induced element in
$\Map_{L_{K}\nCalg_{\sepa,h,q}}(A,B)$ under the map \eqref{fqwefiuqhduiijqqwedqwedqewdeq}.

Following \cite[Sec. 1]{skand12} and the appendix of \cite{zbMATH03973625} we  start with a construction  which associates to every   semi-split exact sequence  \begin{equation}\label{qrwrdedeqwdewdqwdw}\cS:\qquad 0\to I\to A\xrightarrow{q}  Q\to 0 \end{equation}  of separable $C^{*}$-algebras a   morphism  \begin{equation}\label{wgerfrefcdfsfdvsfvsfvre}
f_{\cS}:L (Q)\to L (S(I))
\end{equation} in $L_{K}\nCalg_{\sepa,h,q}$. 
This construction is necessarily of analytic nature since it must take the existence of the cpc split and separability into account. We provide the  details since we merge the approaches of \cite[Sec. 1]{skand12} and  \cite{zbMATH03973625}. In particular we want to work out in detail that the   morphism $f_{\cS}$ is independent of the choices.


 \begin{construction}\label{weijrogreferwfwerf}
We fix a cpc split $s:Q\to A$. By Kasparov's version of Stinespring's theorem  there exists a countably generated $A^{u}$-Hilbert-$C^{*}$-module $E_{0}$
and a homomorphism $\phi:Q\to B(A^{u}\oplus E_{0})$ such that
$s(x)=P\phi(x)P$ for all $x$ in $Q$, where $P$ in $B(A^{u}\oplus E_{0})$ is the projection
onto $A^{u}$ and we consider $A$ as a subset of $B(A^{u}\oplus E_{0})$ in the canonical way. 
The point of taking $A^{u}$ instead of $A$ is that $P$ becomes a compact operator on $A^{u}\oplus E_{0}$. 

After making $E_{0}$ smaller if necessary we can assume that $E_{0}$ is generated as a $A^{u}$-Hilbert $C^{*}$-module by the elements of the form $(1-P)\phi(x)Pa$ for all $x$ in $Q$ and $a$ in $A^{u}$. We refer to  \cref{wgergiwerjgiofwfwrf} for a sketch of a direct construction of $E_{0}$ and $\phi$ which explains  the essence of the proof of the Stinespring theorem mentioned above.  

The pair $(E_{0},\phi)$ is uniquely determined up to canonical isomorphism.
Let $(E'_{0},\phi')$ be another choice. Then we define a map
$E_{0}\to E_{0}'$ sending the generator $(1-P)\phi(x)Pa$ to the generator $(1-P')\phi'(x)P'a$.
In order to see that this map is well-defined we note that
\begin{eqnarray}\label{qijrgoqrfqrfqffqwe}
\langle \sum_{i}(1-P)\phi(x_{i})P a_{i},\sum_{j}(1-P)\phi(x_{j})P a_{j}\rangle&=&
\sum_{i,j} a_{i}^{*} P\phi(x_{i})^{*}(1-P)\phi(x_{j})P a_{j}\\&=&
\sum_{i,j}a_{i}^{*} \left( s(x_{i}^{*}x_{j})- s(x_{i})^{*}s(x_{j}) \right) a_{j} \nonumber
\end{eqnarray}
does only depend on the split $s$, but not on $E_{0}$ and $\phi$.
In addition we observe that 
  the right-hand side takes values in the ideal $I$. Hence the $A^{u}$-Hilbert $C^{*}$-module 
  $E_{0}$ becomes    an $I$-Hilbert $C^{*}$-module $E_{0|I}$ when  we restrict the right $A^{u}$-action to $I$.
  \begin{rem}\label{wgergiwerjgiofwfwrf}Here is a direct construction of $E_{0}$ and $\phi$ starting from the datum of the split $s$.
  One can consider the  right $A^{u}$-module $Q\otimes A^{u}$
  with the  $A^{u}$-valued (actually $I$-valued) scalar product
  \[\langle  x\otimes  a,x'\otimes a'\rangle:=
a^{*} \left( s(x^{*}x')- s(x)^{*}s(x') \right) a'\ .\]
Using that $s$ is completely positive one checks that this is non-negative.
We then let $E_{0}$ be the completion of $Q\otimes A^{u}$ with respect to the induced seminorm.  Note that this involves factoring out vectors with zero norm. We write suggestively $(1-P)\phi(x)Pa$ for the image of $x\otimes a$ in $E_{0}$.
For $y$ in $Q$ we then define $\phi(y)$ in  $B(A^{u}\oplus E_{0})$   by
\[\phi(y)\left(\begin{array}{c}a\\ (1-P)\phi(x)Pb\end{array}\right)=\left(\begin{array}{c}s(y)a+ s(yx)b-s(y)s(x)b\\ (1-P)\phi(y)Pa+(1-P)\phi(yx)Pb+(1-P)\phi(y)Ps(x)b\end{array}\right)\ .\]
One checks that this is a $*$-homomorphism.  \hB
     \end{rem}

We let $B$ denote the unital subalgebra of $B(A^{u}\oplus E_{0})$ generated by $\phi(Q)$ and $P$. We further let $J$ denote the ideal in $B$ generated by $[\phi(Q),P]$.  
We finally let $E_{1}$ be the sub-$A^{u}$-Hilbert $C^{*}$ module of $A^{u}\oplus E_{0}$  generated by $J (A^{u}\oplus E_{0}) $. We then have a canoncial homomorphism
$B\to M(J)\to B(E_{1})$.

Note that  
$[\phi(x),P]=(1-P)\phi(x)P-P\phi(x)(1-P)$. Combining this formula with \eqref{qijrgoqrfqrfqffqwe}
we see that  
for $j$ in $J$  and all $e$ and $e'$ in $A^{u}\oplus E_{0}$ we have
$\langle je,e'\rangle \in I$. 
Hence the   scalar product of $E_{1}$ takes values in $I$ and  $E_{1}$ becomes an   $I$-Hilbert $C^{*}$-module $E_{1|I}$ after restricting the right-module structure to $I$. Furthermore, since $J\subseteq K(A^{u}\oplus E_{0})$ 
(because $P$ was compact), we can conclude that 
  $J\subseteq K(E_{1|I})$.
  In detail, consider   an element $j$ in $J$. It can be approximated by    finite sums
    $\sum_{i}\theta_{\xi_{i},\eta_{i}}$ of one-dimensional operators on $A^{u}\oplus E_{0}$.
  We can find members $u$ and $u'$ of an approximate identity of $J$ such that $uju'$ approximates  $j$.
  But then $j$ is also approximated by the finite sums $\sum_{i}\theta_{u\xi_{i},u^{\prime,*}\eta_{i}}$  of one-dimensional operators on $E_{1|I}$.

  In the following we identify the suspension $S(A)$ of a $C^{*}$-algebra with $C_{0}(S^{1}\setminus\{1\},A)$ and $C_{b}(S^{1}\setminus\{1\}, M(A))$ with a subalgebra of the multiplier algebra  $M(S(A))$.
We consider the family $F:S^{1}\mapsto B(E_{1|I})$ given by $F(u):=P+u(1-P)$.   Since $\phi$ takes values in $B$  we can consider $\phi$ as a homomophism $Q\to B(E_{1|I})$. 
We then define a homomorphism
\begin{equation}\label{vfvsfvdverfcds}
f:qQ\to S(K(E_{1|I}))
\end{equation} whose associated homomorphism (see \cref{4tjigortgwrefrwefref}) has the  components 
\[\hat f_{0},\hat f_{1}:Q\to  C_{b}(S^{1}\setminus \{1\},B(E_{1|I})) \subseteq M(S(K(E_{1|I})))\]  given by \[\hat f_{0}(x)(u):=\phi(x)\ , \qquad \hat f_{1}(x)(u):=F(u) \phi(q)  F(u)^{*}\ .\]
Then $\hat f_{0}(x)(u)-\hat f_{1}(x)(u)$ belongs to  $K(E_{1|I})$ for every $u$ and
$\hat f_{0}(x)(1)-\hat f_{1}(x)(1)=0$. Thus $\hat f_{0}(x)-\hat f_{1}(x)$ belongs to $S(K(E_{1|I}))$ and $f$ is well-defined. This homomorphism does not yet take values in the desired target $S(K\otimes I)$.
We will employ Kasparov's stabilization theorem in  order to produce a homomorphism
$K(E_{1|I})\to K\otimes I$ which is unique up to homotopy.

We let $H_{I}:=\bigoplus_{\nat}I$ denote the standard $I$-Hilbert $C^{*}$-module.
We then have a canonical isomorphism $K(H_{I})\cong K\otimes I$.
 Using that that $E_{1|I}$ is countably generated (it is here where we use separability) and Kasparov's stabilization theorem \cite[Thm. 2]{kaspstine} we can choose an isomorphism $E_{1|I}\oplus H_{I}\cong  H_{I}$ which is unique up to homotopy since the unitary group of $B(H_{I})$ is connected, even contractible.
Using this isomorphism 
 we get an   embedding $E_{1|I}\to  H_{I}$ of $I$-Hilbert $C^{*}$-modules which is also well-defined   up to homotopy. It induces a homomorphism $K(E_{1|I})\to K(H_{I})\cong K\otimes I$ and hence 
$S(K(E_{1|I}))\to  S(K\otimes I)$. Postcomposing \eqref{vfvsfvdverfcds} with this map 
we get a map \begin{equation}\label{rgfqerfdewdqewddacds}
f'_{\cS}:qQ\to S(K\otimes I)
\end{equation} which represents the desired map \eqref{wgerfrefcdfsfdvsfvsfvre}. 
Up to homotopy it only depends on choice of the cpc split.
We finally see that $f_{\cS}$ is independent of the choice of the cpc split 
since any two splits can be joined by a path. \hB
\end{construction}

We now  interpret the pre- or post-composition of $f_{\cS}$   with a homomorphism in $\nCalg_{\sepa}$ 
and its tensor product with an auxiliary $C^{*}$-algebra 
in terms of   operations with 
semi-split exact sequences \cite[Lem. 1.5]{skand12}.
We  consider a map of semi-split exact sequences of separable $C^{*}$-algebras
\[\xymatrix{\tilde \cS&&0\ar[r]&I\ar[r]\ar@{=}[d]&\tilde A\ar[r]^{\tilde q}\ar[d]&\tilde Q\ar[d]^{e}\ar[r]&0\\\cS&&0\ar[r]&I\ar[r]&A\ar[r]^{q}&Q\ar[r]&0}\]
where the right square is a pull-back.
\begin{lem}\label{woptkgpwtgttheth9}We have an equivalence $ f_{\tilde \cS}\simeq f_{\cS}\circ L (e)$.
\end{lem}
\begin{proof}
The split $s:Q\to A$ canonically induces a split $\tilde s:\tilde Q\to \tilde A$.
Working with this split, by an inspection of the constructions we see that the resulting
$I$-Hilbert $C^{*}$-module $\tilde E_{1|I}$ is canonically isomorphic to $E_{1|I}$. 
With this identification we get an equality $f\circ q(e)=\tilde f:q\tilde Q\to S(K(E_{1|I}))$ of maps in \eqref{vfvsfvdverfcds}.
%
%
%
%
 This implies that the desired equivalence.
\end{proof}

We now consider a diagram of semi-split exact sequences 
$$\xymatrix{\cS&&0\ar[r]&I\ar[r]\ar[d]^{h}&A\ar[r]\ar[d]^{\tilde h}&Q\ar@{=}[d]\ar[r]&0\\\tilde \cS&&0\ar[r]&\tilde I\ar[r]&\tilde A\ar[r]&Q\ar[r]&0}\ .$$

\begin{lem} \label{woptkgpwtgttheth91} We have an equivalence $f_{\tilde \cS}\simeq  L (S(h))\circ f_{\cS} $.
\end{lem}
\begin{proof}
The split $s:Q\to A$ induces a split $\tilde s:=\tilde h\circ s:Q\to \tilde A$. Working with this split
we get  a canonical isomorphism $\tilde E_{i|\tilde I}:=E_{i|I}\otimes_{I}\tilde I$. 
Then the resulting map $qQ\to S(K(\tilde E_{1|\tilde I}))$ in  \eqref{vfvsfvdverfcds}
is $qQ
\to S(K( E_{1|I})) \xrightarrow{\id\otimes \tilde I}  S(K( \tilde E_{1|\tilde I}))$.
%
%
%
 This implies the desired equivalence.
\end{proof}

In the following   $\otimes$ can be the minimal or the maximal tensor product.
Recall that $\cS$ denotes a semi-split exact sequence \eqref{qrwrdedeqwdewdqwdw}.
If $B$ is any $C^{*}$-algebra, then
\[\cS\otimes B\qquad 0\to I\otimes B\to I\otimes A\to Q\otimes B\to 0\ .\]
is semi-split exact again.

\begin{lem}\label{woptkgpwtgttheth2}We have an  equivalence $f_{\cS\otimes B}\simeq f_{\cS}\otimes L (B)$.
\end{lem}
\begin{proof} The split $s$ induces a split $s\otimes B$. With this choice 
 the  composition 
\[qQ\otimes B\xrightarrow{\can} q(Q\otimes B) \xrightarrow{!}  S(K\otimes I\otimes B)\cong S(K\otimes I)\otimes B\] 
with the marked map constructed from $\cS\otimes B$
is equal to the map $f\otimes B$ constructed from $s$, and the morphism  $\can$ induces an equivalence in $L_{K}\nCalg_{\sepa,h,q}$ since $L $ is symmetric monoidal.
This implies the assertion. 
\end{proof}

The semi-split exact exact sequence 
\begin{equation}\label{hfquiewfhwedkewmldq}
\cR:\qquad 0\to S(\C)\to C(\C)\to \C\to 0\ .
\end{equation}
gives rise to a  map $f_{\cR}:L(\C)\to L (S^{2}(\C))$ in $L_{K}\nCalg_{\sepa,h,q}$.
The crucial fact  is that it admits
a left inverse, a posteriori even an inverse.
This left inverse $L(\beta): L (S^{2}(\C))\to L (\C)$, called the {\em Bott element},
will be given by an explicit   homomorphism  $\beta:q S^{2}(\C)\to \Mat_{2}(K)$ in \cref{qwoifjqowfwedqewdewdqwedwedq}.
\begin{rem}
In the reference \cite{zbMATH03973625} the map $f_{\cR}$ is called the Bott element. We prefer to call $\beta$ the Bott element because of its role in \cref{wtgkwotpgkelrf09i0r4fwefwerf}.
It is actually the crucial point that the Bott element $\beta$ constructed in \cref{qeorigjiowergwerfwfrewfwr}
in the semiexact situation  has a lift to the split-exact world considered in the present section. 
This fact is the heart of the automatic semiexactness. \hB
\end{rem}

\begin{prop}\label{fkjgogerfwreferfw}
We have $L(\beta)\circ f_{\cR}\simeq \pm \id_{L(\C)}$.
\end{prop}
\begin{rem}
\cref{fkjgogerfwreferfw}  is shown in \cite{zbMATH03973625} by calculating a Kasparov product.  
This argument is therefore not part of the theory developed in the present note. 
%

Of course, the  obvious approach would be to calculate the composition of the two representatives explicitly. 
This would result in a map $q^{2}\C\to K\otimes \C$ which would have  to be compared with the composition of a left upper corner inclusion $\C\to K$ with $  \iota_{\C}\circ \iota_{q\C}:q^{2}\C\to \C$. To go  this path seems to be quite tricky. 

Further below we will therefore provide argument for \cref{fkjgogerfwreferfw}  which avoids to go through Kasparov products or making the homomorphism $\phi$ in  \cref{erigjwoergerfwerfwref} explicit.
  \hB \end{rem}
  
  For the moment we assume \cref{fkjgogerfwreferfw}. For simplicity we will  adjust the sign of $\beta$ such that
  $L(\beta)\circ f_{\cR}\simeq \id_{L(\C)}$.

  \begin{proof}[Proof of  \cref{rijoirwtertretwetwretetwretw}]
We reproduce the argument from the appendix of \cite{zbMATH03973625}.
We consider a semi-split exact sequence \[\cS:\qquad 0\to I\to A\xrightarrow{q} Q\to 0\] and the map $\iota_{q}:I\to C(q)$ as in \eqref{tbrtbbgdbfert}.
We want to show that $L(\iota_{q})$ is an equivalence.

We have a semi-split exact sequence
\[\cT:\qquad 0\to S(I)\to C(A)\to C(q)\to 0\ ,\]
where the second map sends an element of $C(A)$ given by a path $\sigma$ in $A$ with $\sigma(1)=0$  to
the pair $(\sigma(0),q\circ \sigma)$ in $C(q)$, see \cref{werigowrgerfewrfwerfwerf}. The kernel of the map consists of paths $\sigma$ in $I$ with $\sigma(0)=0=\sigma(1)$ and is hence isomorphic to $S(I)$. 
We let $f_{\cT}:L (C(q))\to  L (S^{2}(I))$ be the associated morphism.
Then we define \begin{equation}\label{wergerwffrfrffwe}u:= (L(\beta)\otimes L(I))\circ f_{\cT}:L (C(q))\to  L (I)\ .
\end{equation} 
We now calculate the composition
$u\circ L ( \iota_{q}):L (I)\to L (I)$.
Using \eqref{hfquiewfhwedkewmldq} 
we have a map of exact sequences
\[\xymatrix{\cR\otimes I:&&0\ar[r]&S(I)\ar[r]\ar@{=}[d]&C(I)\ar[r]\ar[d]&I\ar[d]^{\iota_{q}}\ar[r]&0\\\cT&&0\ar[r]&S(I)\ar[r]&C(A)\ar[r]&C(q)\ar[r]&0}\]
where the right square is a pull-back. 
Then
by \cref{woptkgpwtgttheth9} and \cref{woptkgpwtgttheth2} 
we have \begin{equation}\label{wergpokpregwfewff}f_{\cT}\circ L ( \iota_{q})\simeq f_{\cR\otimes I}\simeq f_{\cR}\otimes L (I):L (I)\to  L (S^{2}(I))\ .
\end{equation}
 Using \cref{fkjgogerfwreferfw} we conclude that 
\begin{eqnarray*}
u\circ L ( \iota_{q})
&\stackrel{\eqref{wergerwffrfrffwe}}{\simeq} &
 (L(\beta)\otimes L(I))\circ f_{\cT}\circ  L ( \iota_{q})
\\& \stackrel{\eqref{wergpokpregwfewff}}{\simeq} & (L(\beta)\otimes L(I)) \circ (f_{\cR}\otimes L (I))
 \\&\simeq & (L(\beta)\circ f_{\cR})\otimes L (I)\\&\stackrel{\cref{fkjgogerfwreferfw}}{\simeq}& \id_{L(I)}
\ .\end{eqnarray*}

Hence $\iota_{q}$ is a split monomorphism with left-inverse $u$.

We have the semi-split exact exact sequence
\[ \cU:\qquad 0\to C(\iota_{q})\to C(C(q))\xrightarrow{\phi} C(Q)\to 0\ .\]
Using  \cref{werigowrgerfewrfwerfwerf}
recall that  $C(C(q))$ consists of pairs $(\sigma,\gamma)$ with $\sigma$ a path in $A$, $\gamma=(\gamma(-,t))_{t}$ a path of paths in $Q$
such that $\gamma(s,1)=0$ for all $s$, $\gamma(1,t)=0$ for all $t$, $q(\sigma (s))=\gamma(s,0)$ for all $s$ and $\sigma(1)=0$.
The map $\phi$ sends $(\sigma,\gamma)$ to $ \gamma(0,-)$. Its kernel consists
pairs $(\sigma,\gamma)$ where $\sigma$ is a path in $A$ with $\sigma(0)\in I$ and $\sigma(1)=0$, and $\gamma$ is path of paths in $Q$ with $\gamma(0,t)=0$, $\gamma(1,t)=0$, $q(\sigma (s))=\gamma(s,0)$, and $\gamma(s,1)=0$ for all $s,t$. This is precisely the description of a point in $C(\iota_{q})$.

%
%

Applying the above to the semi-split exact sequence $\cU$
we conclude that $L (\iota_{\phi}):L (C(\iota_{q}))\to L (C(\phi))$ is a split monomorphism.
Since $L$ is Schochet  exact we see that $L (C(\phi))\simeq 0$  since it is the fibre of a map between objects which are equivalent to zero (cones are contractible).
This implies that $L (C(\iota_{q}))\simeq 0$.

Again using that  $L$ is Schochet  exact  we can conclude that $L (S(\iota_{q})):L (S(I))\to L (S(C(q)))$ is an equivalence. 
Then also $L (S^{2}(\iota_{q}))$ is an equivalence. But $L (\iota_{q})$ is a retract of $L (S^{2}(\iota_{q}))$ by
\[\xymatrix{ L (I)\ar@/^1cm/[rrrr]^{\id_{L (I)}}\ar[rr]^{f_{\cR}\otimes L ( I)}\ar[d]^{L(\iota_{q})} &&L (S^{2}(I))\ar[rr]^{\beta\otimes L (I)}\ar[d]^{L (S^{2}(\iota_{q}))}&&L (I) \ar[d]^{L (\iota_{q})}   \\C(q)\ar@/^-1cm/[rrrr]_{\id_{L (C(q))}} \ar[rr]^-{f_{\cR}\otimes L (C(q))} &&L (S^{2}(C(q)))\ar[rr]^-{\beta\otimes L (C(q))}&&  L (C(q)) } \]
 and hence an equivalence, too.
\end{proof}

 Assume that \[\cS:\qquad 0\to I\to A\stackrel{q}{\to} Q\to 0\] is  semi-split exact.   By semiexactness of $ \kk_{\sepa}$
  we get a boundary map $\partial_{\cS}^{\kk}:\kk_{\sepa}(S(Q))\to \kk_{\sepa}(I)$. By the automatic 
  semiexactness theorem we know  that there is a boundary map $\partial_{\cS}:L (S(Q))\to L  (I)$ such that $h(\partial_{\cS})\simeq \partial_{\cS}^{\kk}$, where $h$ is the left-exact comparison functor in \eqref{frewfijofifqedwedqd}.
 The following proposition clarifies its relation with  $f_{\cS}:L (Q)\to L (S(I))$ stated in 
the last sentence of   \cite{zbMATH03973625}.
 
 Recall the exact sequence \eqref{hfquiewfhwedkewmldq}. 
 \begin{prop}\label{weriogjwoierpgreferfrewferwf}
 We have an equivalence
 $$f_{\cS}\simeq S(\partial_{\cS})\circ  f_{\cR}\otimes L(Q):L (Q)\to L (S(I))\ .$$
 \end{prop}
\begin{proof}
We consider the following diagram
\[
\xymatrix{
\cS&&0\ar[r]&I\ar[d]^{\iota_{q}}\ar[r]&A\ar[r]^{q}\ar[d] &Q\ar[r]\ar@{=}[d]&0\\
\cU&&0\ar[r]&C(q)\ar[r]&Z(q)\ar[r]^{\tilde q}&Q\ar[r]&0\\
\cR\otimes Q&&0\ar[r]&S(Q)\ar[u]_{\partial_{q}} \ar[r]&C(Q)\ar[r]\ar[u]&Q\ar[r]\ar@{=}[u]&0}\ .\]
It implies by \cref{woptkgpwtgttheth91} that
\[L (S(\iota_{q}))\circ f_{\cS}\simeq f_{\cU}\simeq L (S(\partial_{q}))\circ (f_{\cR}\otimes L(Q))\ .\]
We now use that $L(\iota_{q})$ is an equivalence and that 
$L (\iota_{q})\circ \partial_{\cS}\simeq  L (\partial_{q})$ in order to deduce the 
desired equivalence.
\end{proof}

We now prepare the proof of \cref{fkjgogerfwreferfw}.
 We will freely
use results from \cref{wreijgowertgwrgrwfrefwrf} and \cref{reoijqoeirfewfqwdewd}.
 We start with  showing a partial case of \cref{wgokjweprgrefwrefwrfw}.
   
\begin{lem}\label{qerijgoqrfqdfewfqef}
For any separable $C^{*}$-algebra $B$ the  comparison functor $h$ in \eqref{frewfijofifqedwedqd} induces an equivalence 
\[ \Map_{L_{K}\nCalg_{\sepa,h,q}}(\C,B) \xrightarrow{h}  \Omega^{\infty}\KK_{\sepa}(\C,B)\ .\] 
\end{lem}
\begin{proof}
Replacing $\ee :\nCalg \to E $ by $L_{\sepa,h,K,q}:\nCalg_{\sepa}\to L_{K}\nCalg_{\sepa,h,q} $ and correspondingly 
$\Omega^{\infty}K(-)$ by $\Map_{L_{K}\nCalg_{\sepa,h,q}}(\C,-)$ we can construct
a natural transformation
\[\tilde q_{\sepa,h}:\rProj_{\sepa}^{s}(-)\to \Map_{L_{K}\nCalg_{\sepa,h,q}}(\C,-)\]
in complete analogy to the construction of $\tilde \ee_{h}$ in \eqref{gwegergergreffewrwf}.
We  start from \[q_{h}:= L_{\sepa,q}: \cProj^{s}(-)\to   \Map_{L_{K}\nCalg_{\sepa,h,q}}(\C,-)\] in place of $e_{h}$ in \eqref{vsdfhuiwhiuvhjfuivsdfvsfdvs}, use that $L_{\sepa,h,K,q}$ is split-exact and that the unitalization sequence  \eqref{afoiahjfoiaf}   is split-exact, and that  $\Map_{L_{K}\nCalg_{\sepa,h,q}}(\C,-)$ takes values in groups.
The proof of \cref{wogkpwgerfwrefwf1} goes through word by word and shows that
$\tilde q_{\sepa,h}$ is an equivalence. Here instead of the \eqref{vdfvqr3ferffafafds1} we use of course the   universal property  of $L_{\sepa,h,K,q}$ given by the separable version of \eqref{fwefwefhhuhdiqweuhdwedwqdd}
in order to construct the factorization $\bProj: L_{K}\nCalg_{\sepa,h,q}\to \CGroups(\Spc)$ in the analogue of \eqref{vfdsvdfvhosvsdfvfsdv}. The commutativity of 
\[\xymatrix{&\rProj^{s}_{\sepa}(B)\ar[dr]_{\simeq}^{\tilde \ee_{\sepa}}\ar[dl]^{\simeq}_{\tilde q_{\sepa,h}}&\\\Map_{L_{K}\nCalg_{\sepa,h,q}}(\C,B)\ar[rr]^{h}& &\Omega^{\infty}\KK_{\sepa}(\C,B)}\] 
implies that $h$ is also an equivalence.
 \end{proof}
 
 \begin{ex}\label{ewrgoihweriogefrwefw}
 By the calculation of the spectrum $KU\simeq \KK_{\sepa}(\C,\C)$ in \cref{qeorigjiowergwerfwfrewfwreee} (and implicitly using \cref{wtrkgoerfrefwerfwerf}.\ref{qrfijqiofdewdqwed1} in order to go from $E$- to $KK$-theory)
 we know that for any $i$ in $\nat$
\[\pi_{0} \Map_{L_{K}\nCalg_{\sepa,h,q}}(\C,S^{i}(\C)) \stackrel{\cref{qerijgoqrfqdfewfqef}}{\cong}    \KK_{\sepa,0}(\C,S^{i}(\C))\cong \pi_{i}\KU\stackrel{\eqref{qefoihwefioewdqewdqwedwqedwqedwdwed}}{\cong} 
\left\{\begin{array}{cc} \Z&i\in 2\nat\\0 &else  \end{array} \right. \ . \] 
The proof of \cref{qerijgoqrfqdfewfqef} together with the fact that $\pi_{0}$ sends the homotopy theoretic group completion to the algebraic group completion  implies that 
\[\pi_{0}\Map_{L_{K}\nCalg_{\sepa,h}}(\C,\C)\to \pi_{0}\Map_{L_{K}\nCalg_{\sepa,h,q}}(\C,\C)\] is the algebraic group completion. This will allow  us to detect elements in  the group $\pi_{0}\Map_{L_{K}\nCalg_{\sepa,h,q}}(\C,\C)$
represented by maps $\C\to K$ or $q\C\to K$ in a simple manner. 

We identify maps $p:\C\to K$ with the projections  $p$  in $K$ given by the image of $1$.
 A map $p:\C\to K$ is  determined up to homotopy by the dimension  $\dim(p)$ of the range of $p$. We therefore have an isomorphism of monoids
\[\pi_{0}\Map_{L_{K}\nCalg_{\sepa,h}}(\C,\C)\to \nat\ , \quad  p\mapsto \dim(p)\ .\]
We now consider a map $(p_{0},p_{1}):q\C\to K$ given in terms of an associated map with components $p_{i}:\C\to M(K)=B$ for  $i=0,1$
such that $p_{0}-p_{1}\in K$. Then $p_{1}p_{0}:\im(p_{0})\to \im(p_{1})$ is a Fredholm operator 
and  we can define the relative index
\[I(p_{0},p_{1}):=\ind(p_{1}p_{0}:\im(p_{0})\to \im(p_{1}))\ .\] If $p_{i}$   is   compact for $i=0,1$, then  $I(p_{0},p_{1})=\dim(p_{0})-\dim(p_{1})$.

We now show that the relative index is homotopy invariant for homotopies of  pairs $(p_{0,t},p_{1,t})$ such that  $p_{0,t}-p_{1,t}$ is norm continuous. In particular there is no continuity  condition on the familes $p_{i,t}$ separately.
We follow \cite{Avron_1994} and define the norm continuous families of selfadjoint operators $A_{t}:=p_{0,t}-p_{1,t}$ and $B_{t}:=1-A_{t}$. Then $A^{2}_{t}+B^{2}_{t}=0$ and $A_{t}B_{t}+B_{t}A_{t}=0$ (see \cite[Thm. 2.1]{Avron_1994}). Since $A_{t}$ is selfadjoint and compact  the spectrum of $A_{t}$  away from $0$   is discrete and consists of eigenvalues of finite multiplicity. The relations above imply that
for $\lambda\not=1$ the operator $B_{t}$ induces an isomorphism between $\ker(A_{t}-\lambda)$ and $\ker(A_{t}+\lambda)$.  By \cite[Prop. 3.1]{Avron_1994} we have the first equality in 
\begin{eqnarray*}\ind(p_{0,t},p_{1,t})&=&\dim(\ker(A_{t}-1))-\dim(\ker(A_{t}+1))\\&=&\dim(E_{A_{t}}((1-\epsilon,1+\epsilon)))-\dim(E_{A_{t}}((-1-\epsilon,-1+\epsilon)))\end{eqnarray*} for any $\epsilon$ in $(0,1)$, where $E_{A_{t}}$ is the family of spectral projections for $A_{t}$. In order to see  the second equality  note that the contributions of the eigenspaces to the  eigenvalues different from $\pm 1$ cancel out by the consideration above.
 The right-hand side is  continuous in $t$ and hence constant. To this end, we consider  a point $t_{0}$ in $[0,1]$. Then we choose $\epsilon$  such that
$1\pm \epsilon$ do not belong to the spectrum of $A_{t_{0}}$. 
By the norm continuity of $t\mapsto A_{t}$   there exists $\delta$ in $(0,\infty)$ 
 for all $t$ in 
$[t_{0}-\delta,t_{0}+\delta]\cap [0,1]$ the points  $1\pm \epsilon$ do not belong to the spectrum 
of $A_{t}$. But then the right-hand side is constant on $[t_{0}-\delta,t_{0}+\delta]\cap [0,1]$.

Using the homotopy invariance and the  additivity of the relative index for block sums we can conclude that
the map  $(p_{0},p_{1})\mapsto I(p_{0},p_{1})$
induces a group homomorphism $\tilde \dim:\pi_{0}  \underline{\Hom}(q\C,K)\to \Z$ such that the bold part of 
\[\xymatrix{\ar@/^-2cm/[dd]_{\pi_{0}L}\pi_{0}\underline{\Hom}(\C,K)\ar[r]_-{\cong}^-{\dim}\ar[d]^{\iota_{\C}^{*}}&\nat \ar[d]^{i}\\ \pi_{0}  \underline{\Hom}(q\C,K)\ar[r]^-{\tilde \dim}\ar[d]_{!}^{\eqref{fqwefiuqhduiijqqwedqwedqewdeq}}&\Z\\\pi_{0}\Map_{L_{K}\nCalg_{\sepa,h,q}}(\C,K)\ar@{-->}[ur]_{\cong}&}\] commutes.
The dashed arrow is obtained from the universal property of the arrow denoted by $\pi_{0}L$ as a group completion, since 
the right-down map $i\circ \dim$ is a homomorphism to a group. 
It remains to show that the lower triangle commutes. 

We claim that the arrow marked by $!$ is an isomorphism. 
Assuming the claim we know that $\iota^{*}_{\C}$ also represents  a group completion. 
We can then argue
that the two ways to go from $\pi_{0}  \underline{\Hom}(q\C,K)$ to $\Z$ must agree since 
 group completions are initial in homomorphisms to groups.

In order to see the claim we can appeal to \cref{wejogiwerfgerwfrewfwerfw}.
But as this implicitly uses \cref{erigjwoergerfwerfwref} one could alternatively
show directly that $\tilde \dim$ is an isomorphism and then conclude that $!$ is an isomorphism.
\hB
\end{ex}

\begin{construction} \label{qwoifjqowfwedqewdewdqwedwedq}
 In this construction we describe the Bott element $L(\beta)$ in the mapping space $\Map_{L_{K}\nCalg_{\sepa,h,q}}(S^{2}(\C),\C)$. Instead of reproducing the construction from    \cite{zbMATH03973625} we describe a version which is more amenable to explicit calculations.

We consider the closed smooth manifold   $\C\P^{1}\cong S^{2}$. It will be   equipped
  with
a constant scalar curvature Riemannian metric and the orientation determined by the complex structure. We consider $\C\P^{1}$ as a
  Riemannian spin manifold and let $\Dirac$ be the spin Dirac operator.
 It acts as a  first order elliptic differential operator on the sections of  the $\Z/2\Z$-graded  spinor  bundle $S \cong S^{+}\oplus S^{-}$ which is odd with respect to the grading and therefore represented by a matrix
 \begin{equation}\label{qwefqwofjqwekofqweldewdqwedqwedq}  \left(\begin{array}{cc}0&\Dirac^{-}\\\Dirac^{+}&0 \end{array}\right)\ . \end{equation}  The Schr\"odinger-Lichnerowicz formula states that  $\Dirac^{2}=\Delta+\frac{s}{4}$, where $\Delta$ is the canonical Laplacian on the spinor bundle  associated to the connection and $s$ is the scalar curvature. Since the Laplacian is non-negative and $s$ is positive we see that $\Dirac^{2}$ is positive and hence invertible as an unbounded operator on $H:=L^{2}(\C\P^{1},S)$ with domain $C^{\infty}(\C\P^{1},S)$. Using function calculus we obtain the odd and  zero order pseudodifferential  unitary operator
  \[U:= \Dirac |\Dirac|^{-1} \] in $B(H)$.   
  The principal symbol of $U$ is the unitary part of the polar decomposition of the principal symbol of $\Dirac$.
  If $f$ in $C(\C\P^{1})$ acts as multiplication operator on $H$, then $[f,U]$ is compact.
  Indeed, let $f$ be smooth for the moment and consider it as  a zero order pseudodifferential operator. Then  the principal symbols of $f$ and $U$ commute and the commutator is  a pseudodifferential 
  operator of order $-1$ and hence compact. Since $C^{\infty}(\C\P^{1})$ is dense in $C(\C\P^{1})$ in the norm and the compact operators are closed in norm we see that $[f,U]$ is compact for all $f$ in $C(\C\P^{1})$.
  
   The grading of $S$ gives a decomposition $H=H^{+}\oplus H^{-}$ and we represent $U$ as a matrix \[\left( \begin{array}{cc} 0&U^{-} \\U^{+} &0 \end{array} \right)  \ .\]
  We have two homomorphisms $\hat \phi_{\pm}:C(\C\P^{1})\to  B(H^{\pm})$ such that for $f$ in $C(\C\P^{1})$ the operator $\hat \phi_{\pm}(f)$ is  the multiplication operator by $f$ on $H^{\pm}$. 
  
  We  define two homomorphisms $\hat \phi_{i}:C(\C\P^{1})\to B(H^{+})$  for $i=0,1$ by \begin{equation}\label{oihrfqiweoewdqwdwqdqwdewdedq}\hat \phi_{0}:=\hat \phi_{+}\ , \quad   \hat \phi_{1}:=U^{-}\hat \phi_{-}U^{+}\ .
\end{equation}  Then we have $\hat \phi_{0}(f)-\hat \phi_{1}(f)\in K(H^{+})$ 
  for all $ f$ in $C(\C\P^{1})$. The homomorphisms $\hat \phi_{i}$ for $i=0,1$ are therefore the components of the associated homomorphism of a homomorphism
  \[\hat \beta:qC(\C\P^{1})\to K(H^{+})\ .\]
  It represents a point  $L(\hat \beta)$ in $\Map_{L_{K}\nCalg_{\sepa,h,q}}(C(\C\P^{1}),\C)$.
 
 We fix a base point $*$ in $\C\P^{1}$. Using an orientation-preserving diffeomorphism $\R^{2}\cong \C\P^{1}\setminus\{*\}$ we  identify $S^{2}(\C)$ with the subalgebra $C_{0}(\C\P^{1}\setminus \{*\})$
 of $C(\C\P^{1})$ of functions vanishing at $*$. We let $\iota:S^{2}(\C)\to C(\C\P^{1})$ denote the inclusion.
  We define the Bott element   as the composition
  \[\beta:  qS^{2}(\C)\xrightarrow{q(\iota)} qC(\C\P^{1})\xrightarrow{\hat \beta} K(H^{+})\ .\]  
  Then $L(\beta)\simeq L(\hat \beta)\circ L(\iota)$  is a point in $\pi_{0}\Map_{L_{K}\nCalg_{\sepa,h,q}}(S^{2}(\C),\C)$ which is our candidate for the Bott element.
 \hB 
      \end{construction}

\begin{proof}[Proof of \cref{fkjgogerfwreferfw}]
 We have the  following commutative diagram
\begin{equation}\label{ewfoijqwioefjqiowjfiowqejdfqwedqweefewqfqewf} \xymatrix{\Map_{L_{K}\nCalg_{\sepa,h,q}}(S^{2}(\C),\C)\times \Map_{L_{K}\nCalg_{\sepa,h,q}}(\C,S^{2}(\C)) \ar[r]^-{\circ}\ar[d] &\Map_{L_{K}\nCalg_{\sepa,h,q}}( \C,\C) \ar[d]^{\simeq} \\ \Omega^{\infty}\KK_{\sepa}(S^{2}(\C),\C)\times  \Omega^{\infty}\KK_{\sepa}(\C,S^{2}(\C))\ar[r]^-{\circ} &  \Omega^{\infty}\KK_{\sepa}(\C,\C)}\  ,\end{equation}
 where the vertical morphisms are induced by $h$ and the horizontal morphisms are given by composition.
The right vertical morphism is an equivalence by \cref{qerijgoqrfqdfewfqef}.

We have
$  \KK_{\sepa,0}(S^{2}(\C),\C))\cong \pi_{-2}\KU\cong 
\Z$. We furthermore know that under the identification
 $\KK_{\sepa,0}(S^{i}(\C),S^{j}
(\C))\cong \pi_{j-i}\KU$ the composition 
 \[\KK_{\sepa,0}(S^{j}(\C),S^{k}
(\C))\times \KK_{\sepa,0}(S^{i}(\C),S^{j}(\C))\to \KK_{\sepa,0}(S^{i}(\C),S^{k}(\C))\] is identified with the 
multiplication in the ring $\pi_{*}\KU\cong \Z[b,b^{-1}]$.

  So in order to show that $L(\beta)$ is a left-inverse inverse of $f_{\cR}$ up to sign  it suffices to show
that the images $h(L(\beta))$ and $h(f_{\cR})$ of these elements in $ \pi_{-2}\KU$ and $\pi_{2}\KU$ are generators.
This is the content of  \cref{qriofjoqfwedqweddedqe} and \cref{geriowgjeoiferfwerfwef}.



%

       \begin{lem}\label{qriofjoqfwedqweddedqe}
The class $h(L(\beta))$ in $\KK_{\sepa,0}(S^{2}(\C),\C)\cong \Z$ is a generator.
\end{lem}
\begin{proof}
It suffices to provide  
  an element
 $L(p)$ in $\Map_{L_{K}\nCalg_{\sepa,h,q}}(\C, S^{2}(\C))$ such that $L(\beta) \circ L(p)$ represents a generator 
  of $\pi_{0}\Map_{L_{K}\nCalg_{\sepa,h,q}}(\C, \C)\cong \Z$.

  We keep the conventions from \cref{qwoifjqowfwedqewdewdqwedwedq}.
 We  have a tautological line bundle $L\to \C\P^{1}$  which is naturally a subbundle of the trivial bundle $\C\P^{1}\times \C^{2}$. We let $L^{\perp}$ be the orthogonal complement so that
 $L\oplus L^{\perp}\cong \C\P^{1}\times \C^{2}$. We consider the  projection $P$
 in $  \Mat_{2}(C(\C\P^{1}))$ such that the value $P_{x}$ is the orthogonal projection onto the fibre of $L^{\perp}_{x}$ of $L^{\perp}$  for all  $x$ in $\C\P^{1}$. We interpret $P$ as a homomorphism $P:\C\to \Mat_{2}(C(\C\P^{1}))$ such that $1\mapsto P $.
 We get $L(P)$ in $\Map_{L_{K}\nCalg_{\sepa,h,q}}(\C,C(\C\P^{1}))$.
 We can now calculate the composition
 $L(\hat \beta)\circ L(P)$ which is represented by 
 \begin{equation}\label{dqweihuifqwedewdewdqewd}q\C\xrightarrow{qP}  q\Mat_{2}(C(\C\P^{1}))   \xrightarrow{\Mat_{2}(\hat \beta)} \Mat_{2}(K(H^{+}))\ ,
\end{equation} 
where $\Mat_{2}(\hat \beta)$ is the map whose associated homomorphism has the components $\Mat_{2}(\hat \phi_{i})$  with $\hat \phi_{i}$ as in  \eqref{oihrfqiweoewdqwdwqdqwdewdedq}.
  We set $\hat H:=H^{+}\otimes \C^{2}$ and identify $\Mat_{2}(K(H^{+}))\cong K(\hat H)$.
 The components of the map \eqref{dqweihuifqwedewdewdqewd} are then given by the projections
 $Q_{i}:=\Mat_{2}(\hat \phi_{i})( P) $ in $  B(\hat H)$.
Note that    the difference $Q_{0}-Q_{1}$ is compact. By \cref{ewrgoihweriogefrwefw}  the class of the composition in \eqref{dqweihuifqwedewdewdqewd} is detected by the relative index
 $I(Q_{0},Q_{1})$ on $\Z$, i.e., the index of the Fredholm operator $Q_{1}Q_{0}:Q_{0}\hat H\to Q_{1}\hat H$.
 We note that $Q_{0}=P_{+}$ and $Q_{1}=\Mat_{2}(U^{-}) P_{-}\Mat_{2}(U^{+})$, where we consider $P_{\pm}$ as a multiplication operator $P$ on $H^{\pm}\otimes \C^{2}\cong L^{2}(\C\P^{1},S^{\pm}\otimes \C^{2})$.
 Multiplying with the unitary $\Mat_{2}(U^{-})$ from the left we can thus identify  the  Fredholm operator  $Q_{1}Q_{0}$ with 
 \[ P_{-}\Mat_{2}(U^{+})P_{+}:L^{2}(\C\P^{1},S^{+}\otimes \C^{2})\to L^{2}(\C\P^{1},S^{-}\otimes \C^{2})\ .\]
 This is a zero order pseudodifferential operator whose symbol is the 
 symbol of the twisted Dirac operator
 $\Dirac^{+}_{L}$ (see \eqref{qwefqwofjqwekofqweldewdqwedqwedq}) made unitary.  In particular
we have $ I(Q_{0},Q_{1})=\ind(\Dirac^{+}_{L})$.
  By the Atiyah-Singer index theorem we have $$\ind(\Dirac^{+}_{L})=\int_{\C\P^{1}} \hat A(S^{2})\ch(L^{\perp})=-\int_{\C\P^{1}} c_{1}(L^{\perp})=1\ .$$

The  base point of  $\C\P^{1}$ gives a decomposition $L_{*}\oplus L_{*}^{\perp}\cong \C\P^{1}\otimes \C^{2}$.
We let $P_{*}$ in $ \Mat_{2}(C(\C\P^{1}))$ be the corresponding constant projection onto $L_{*}^{\perp}$.
 The same calculations as above show that the composition $L(\hat \beta)\circ L(P_{*})$ represented by 
 \[q\C\xrightarrow{qP_{*}}  q\Mat_{2}(C(S^{2})) \xrightarrow{\Mat_{2}( \hat \beta)} \Mat_{2}(K(H^{+}))\]
 represents the zero element.
 
 The  projections  $P$ and $P_{*}$  in  $\Mat_{2}(C(S^{2}))\cong M(\Mat_{2}(S^{2}(\C)))$ can be considered
 as components of the associated homomorphism of a homomorphism
$p:q\C\to \Mat_{2}(S^{2}(\C))$ since $P-P_{*}\in \Mat_{2}(S^{2}(\C))$.
 The composition
 \[q\C\xrightarrow{p}  \Mat_{2}(S^{2}(\C))\xrightarrow{\Mat_{2}(\iota)}  \Mat_{2}( C(S^{2}))\] represents the difference, i.e., 
 $L(\iota)\circ L(p)\simeq   L(P)-L(P_{*} )$.   
 This implies that
 \[L(\beta)\circ L(p)\simeq
  L(\hat \beta)\circ L(\iota)\circ L(p)\simeq  L(\hat \beta) \circ   (L(P)-L(P_{*})) \simeq  L(\hat \beta) \circ   L(P) \simeq  L(\id_{\C})\]
  is a generator. 
   \end{proof}

\begin{lem}  \label{geriowgjeoiferfwerfwef} The class $h(f_{\cR})$ in $\KK_{\sepa,0}(\C,S^{2}(\C) )\cong \Z$ is a generator.
\end{lem}
\begin{proof} We first 
  make \cref{weijrogreferwfwerf} explicit in order to
 describe an explicit representative of  the map $f_{\cR}$.
Let $t$ be the coordinate on $[0,1]$. We identify the cone over $\C$ as $C(\C)\cong C_{0}((0,1])$. We define the cpc map
$s:\C\to C(\C)$ such that  $s(1)=t$, where $t$ is the coordinate function of the interval acting by multiplication on $C_{0}((0,1])$. 
We consider $E_{0}:=S(\C)\cong C_{0}((0,1))$ as a $C(\C)^{u}$-Hilbert $C^{*}$-module.
We use the identifcation $C(\C)^{u}\cong C([0,1])$ and
 define
$\phi:\C\to B(C([0,1])\oplus E_{0})$ such that
\[\phi(1)=\left(\begin{array}{cc} t& \sqrt{t(1-t)}\\ \sqrt{t(1-t)}&(1-t )\end{array} \right)\ .\] Note that the right-hand side is a projection.
We set \[P:=\left(\begin{array}{cc} 1&0\\ 0&0\end{array} \right)\ .\]
Then $(1-P)\phi(1)Pa= a\sqrt{t(1-t)}$ for $a$ in $C([0,1]) $. 
These elements span a dense subset of $E_{0}$.
 The unital algebra $B$  spanned by $P$ and $\phi(1)$  is given by 
 \[ \left( \begin{array}{cc}C([0,1]) &S(\C) \\ S(\C)&C(S^{1}) \end{array} \right) \subseteq B (C([0,1])\oplus S(\C))\ .\]
 The ideal $J$ generated by $[P,\phi(1)]$ is $\Mat_{2}(S(\C))$, and $E_{1}=S(\C)\oplus S(\C)$.
 We get a map $q\C\to S(\Mat_{2}(S(\C)))$ whose associated map has the  components
 \[\hat f_{0}(1)(u)=\left(\begin{array}{cc} t& \sqrt{t(1-t)}\\ \sqrt{t(1-t)}&(1-t )\end{array} \right)\ , \quad 
 \hat f_{0}(1)(u)=\left(\begin{array}{cc} t& u^{-1}\sqrt{t(1-t)}\\ u\sqrt{t(1-t)}&(1-t )\end{array} \right)\ .\]
 This map represents $f_{\cR}$.

 In order to check that $h(f_{\cR})$ is a generator, using Swan's theorem we will translate the problem into  a calculation in usual topological $K$-theory 
 of compact spaces defined in terms of vector bundles.
 We interpret $t$ and $u$ as longitude and latitude coordinates on $S^{2}$ such that $t=0$ is the south pole and $t=1$ is the north pole. By  Swan's theorem  we have an isomorphism between
 $\KK_{\sepa,0}(\C,C(S^{2}))\cong K_{0}(C(S^{2}))$ and the  $K$-theory $  K^{0}(S^{2})$  of the sphere defined in terms of vector bundles as usual. 
 Under this isomorphism and with $S^{2}(\C)\cong C_{0}(S^{2}\setminus  \{n\})$ we   identify $\KK_{\sepa,0}(\C,S^{2}(\C))$ with the reduced $K$-theory $\tilde K^{0}(S^{2})$ relative to the north pole  $\{n\}$. 

Since $\hat f_{0}(1)$ does not depend on $u$ it is obviously a projection $P_{0}$ in $\Mat_{2}(C(S^{2}))$. We now observe that $\hat f_{1}(1)$ does not depend on $u$ if $t=0$ or $t=1$. We can therefore also  interpret $\hat f_{1}(1)$ as a projection $P_{1}$ in $ \Mat_{2}(C(S^{2}))$. Identifying projections with vector bundles (actually subbundles of the trivial bundle $S^{2}\times \C^{2}$) we get a class $[P_{0}]-[P_{1}]$ in $K^{0}(S^{2})$. 
Since $P_{0}$ and $P_{1}$ coincide at the north pole this difference is actually a reduced class
in $\tilde K^{0}(S^{2})\cong \Z$. Our task is to show that it is a generator. 

Since $P_{0}$ does not depend on the $u$-coordinate
it comes from a projection  in $C([0,1])$. Since $[0,1]$ is contractible
we can conclude that $P_{0}$ is homotopic to a constant projection and the corresponding vector bundle can be trivialized. The matrix function
\[(t,u)\mapsto U(t,u):=\left( \begin{array}{cc} 1&0 \\ 0& u\end{array} \right)\]
defines an isomorphism of vector bundles $P_{0}\to P_{1}$ (considered as subbundles of the two-dimensional trivial bundle)
on the subspace $\{t\not=0\}$. Since \[P_{0,n}=\C\left( \begin{array}{c} 1 \\ 0 \end{array} \right)\]
this isomorphism extends across the north pole.  Away from the north pole  this isomorphism sends the section
\[(u,t)\mapsto \left( \begin{array}{c} \sqrt{t(1-t)} \\ (1-t) \end{array} \right)\] of $P_{0}$
to the section
\[(u,t)\mapsto \left( \begin{array}{c} \sqrt{t(1-t)} \\ u(1-t) \end{array} \right)\ .\]
Note that if $t$ becomes small this is essentially multiplication by $u$.
So $P_{1}$ is obtained from the trivial bundle
by cutting at the equator $S^{1}\subseteq S^{2}$ and regluing with a map $S^{1}\to U(1)$ of degree one. 
This implies that $[P_{0}]-[P_{1}]$ generates $\tilde K^{0}(S^{2})$.
  \end{proof}
This finishes the proof of \cref{fkjgogerfwreferfw}.
 \end{proof}

 \section{Half-exact functors}
 
 In  classical $KK$- and $E$-theory universal properties are  formulated in terms of half-exact functors to ordinary additive categories.  In the present section we will recall this language and state the universal properties of the homotopy category versions of the functors constructed in the previous sections in these terms. This will be used to    show that they are equivalent to the classical $KK$- and $E$-theory functors. This comparison is the main objective of this section.
  For simplicity we will restrict to separable algebras.

 Recall that an additive $1$-category is an ordinary category  which is pointed, admits finite coproducts and products such  for any two objects $D,D'$ the canonical morphism $D\sqcup D'\to D\times D'$ is an isomorphism, and which has the property that the commutative monoids $\Hom_{\bD}(D,D')$  are abelian groups for all objects $D,D'$. An additive category is automatically  enriched in abelian groups.
 A functor $\bD\to \bD'$ between additive categories is additive if it preserves coproducts and products.
 It is then compatible with the enrichments in abelian groups. We let $\Fun^{\add}(\bD,\bD')$ be the category of additive functors.

\begin{ex}
 If $\bC$ is an additive $\infty$-category, then its homotopy category  $\ho\bC$ is an additive $1$-category. 
\hB \end{ex}

We next introduce the notion of a half-exact   additive category. A half-exact structure looks like a
glimpse of a triangulated structure. We will use this notion mainly in order to  
match the  formulation of the universal properties of $KK$- and $E$-theory in the classical literature,  in particular 
in \cite{MR1068250}.

Let $A\to B\to C$ be a sequence of maps in an additive $1$-category $\bD$
and $\cS$ be a set of objects of $\bD$.

\begin{ddd}
We say that this sequence is {\em half-exact with respect to $\cS$} if the induced sequences  \[\Hom_{\bD}(D,A)\to \Hom_{\bD}(D,B) \to \Hom_{\bD}(D,C)\]
 and
 \[\Hom_{\bD}(C,D)\to \Hom_{\bD}(B,D)\to \Hom_{\bD}(A,D)\]
 of abelian groups are exact for every $D$ in $\cS$.
   \end{ddd}


\begin{ddd}A {\em marking} on an additive $1$-category is a subset of objects $\cS$ which is closed under isomorphisms. 
A  {\em half-exact additive category} is an additive $1$-category with a marking and a collection of distinguished sequences 
which are required to be half-exact with respect to the given marking .\end{ddd}
  Let $\bD$ and $\bD'$ be   half-exact additive  categories  with markings $\cS$ and $\cS'$, respectively.

 \begin{ddd} A {\em half-exact functor} $\phi:\bD\to \bD'$      is an additive functor  such that $\phi(\cS)\subseteq \cS'$ and $\phi$ sends the distinguished  half-exact sequences in $\bD$ to   distinguished    half-exact sequences in $\bD'$. \end{ddd}
  We let $\Fun^{\add,\frac12\exa}(\bD,\bD')$ be the category of half-exact additive functors between half-exact additive categories.

  \begin{ex}\label{werijogrgerwgweg}
Every additive category   has the {\em  canonical  half-exact structure}  with  $\cS=\Ob(\bD)$ and
  the collection of all sequences which are half-exact with respect to this marking.
 In this case we denote the half-exact additive category by $\bD_{\can}$. \hB
\end{ex}

\begin{ex} Recall that a sequence $A\xrightarrow{i} B\xrightarrow{p} C$   in an additive  $1$-category  $\bD$ is {\em split-exact} if $p$ admits a left inverse $s:C\to B$ such that $(i,s):A\oplus C\to B$ is an isomorphism. The   split-exact sequences  are half-exact for the maximal marking.  
We write $\bD_{\splt}$ for $\bD$ equipped with the maximal marking and the collection of
split-exact sequences.
An additive functor  $\bD\to \bD'$ automatically preserves split-exact  sequences. 
 It therefore belongs to  
$\Fun^{\add,\frac12 \exa}(\bD_{\splt},\bD'_{\splt})$. \hB\end{ex}
\begin{rem}\label{righowgwrgrfrefwrf} Let $\bD$ be a half-exact additive category with marking $\cS$. Then a sequence
\[D_{0}\to D_{1}\to D_{2}\to D_{3}\to D_{4} \] in $\bD$ is called half-exact exact  if each segment $D_{i-1}\to D_{i}\to D_{i+1}$ is half-exact. If $D_{0}\to D_{1}$ and $D_{3}\to D_{4}$ are isomorphisms and $D_{2}$ belongs to $\cS$, then we can conclude that $D_{2}\cong 0$ by showing that $\id_{D_{2}}=0$.

If $D_{1}\cong 0$ and $D_{4}\cong 0$ and $D_{2}$ and $D_{3}$ belong to $\cS$, then we can conclude that
$D_{2}\to D_{3}$ is an isomorphism by constructing left and right inverses. \hB
\end{rem}
  
 \begin{ex} In order to have a non-trivial half-exact structure  at hand consider the category of abelian groups and let $\cS$ be the  set of uniquely divisible abelian groups. We distinguish all sequences which are half-exact with respect to $\cS$.
  Then e.g. the sequence $0\to \Z\xrightarrow{5}\Z\to 0$ is half-exact.
  Note that in this case we can not conclude that $5$ is an isomorphism since $\Z$ does not belong to the marking.
   \hB
  \end{ex}

Let $\bD$ be a half-exact additive category.
 A functor $F:\bC\to \bD$ from a left-exact $\infty$-category will be called {\em half-exact}  if the marking of $\bD$ contains $F(\Ob(\bC))$ 
  and $F$ 
   sends
 fibre sequences to distinguished  half-exact sequences.  \begin{ddd} A  functor
 $F:\nCalg_{\sepa}\to \bD$  is called {\em  half-exact (half-semiexact)} if the marking of $\bD$ contains    $F(\Ob(\nCalg))$, and if $F$  sends
 exact (semi-split exact) sequences to distinguished half-exact sequences. A functor $\nCalg_{\sepa}\to \bD$ is {\em split-exact} if it sends split-exact sequences of $C^{*}$-algebras to  split-exact sequences.  \end{ddd}We indicate such functors by superscripts $\frac12 ex$, $\frac12 se$, or $splt$.

\begin{ex}
  If 
  $\bC$   is   a stable $\infty$-category, then $\ho \bC$ is a triangulated  $1$-category.
  The sequences $A\to B\to C$ for any distinguished triangle $A\to B\to C\to \Sigma A$ are 
  half-exact with respect to the maximal marking. 
  We call the half-exact structure consisting of the maximal marking and   these sequences the {\em triangulated half-exact structure}. 
The corresponding half-exact additive category will be denoted by    
  $\ho\bC_{\Delta}$.
    
  If we consider $\bC$ as a left-exact  $\infty$-category, then according to our conventions the functor $\ho:\bC\to \ho \bC_{\Delta}$ is half-exact. \hB
 \end{ex}

We consider $!$ in $\{\exa,\se,q\}$. For $!\in \{\exa, \se\}$  the functor
$\kk_{\sepa,!}:\nCalg_{\sepa}\to \KK_{\sepa,!}$ is as in \eqref{vsddfkvopsdfvfsdvsdfvsfdvsdv}. For $!=q$
we set
\[\kk_{\sepa,q}:=L_{\sepa,h,K,q}:\nCalg_{\sepa}\to \KK_{\sepa,q}:=L_{K}\nCalg_{\sepa,h,q}\]
  using the separable versions of \eqref{vsdfvdfsjvposdvsfdvsfdvsfdvsfdv}.
We note that the targets of these functors  are a left-exact additive $\infty$-category for $!=q$ (separable version of \cref{wetgjowoegrewfwref}.\ref{woitgjwoerfrfwfrefefw}) or even stable $\infty$-categories (\cref{wtgkwotpgkelrf09i0r4fwefwerfsepa}.\ref{wetiogjwegferwferfwrefw881sepa}) for $!\in \{\exa ,\se\}$. In particular, their homotopy categories are additive $1$-categories.
We consider the functor
\[\ho\kk_{\sepa,!}:=\ho\circ \kk_{\sepa,!}:\nCalg_{\sepa}\to  \ho\KK_{\sepa,!}\ .\]
We will equip the additive $1$-category $   \ho\KK_{\sepa,!}$ with  the triangulated half-exact structure in case $!\in \{\exa,\se\}$, and the split 
half-exact structure if $!=q$.  In all cases the marking is the maximal one.
\begin{kor}\label{weroigjowergerfgwref9}
The functor $\ho\kk_{\sepa,!}$ is homotopy invariant and stable. Moreover, it is split-exact in case $!=q$, half-semiexact in case $!=\se$, and half-exact in case $!=\exa$.
\end{kor}
\begin{proof}
 Homotopy invariance and stability are clear by construction.


  In case $!=\exa$ ($!=\se$) we use that $\kk_{\sepa,!}$ sends
 exact (semi-split exact) sequences of $C^{*}$-algebras to fibre sequences, and the half-exactness of $\ho$.
 
 In case $!=q$ we use that  $\kk_{\sepa,q}$  and $\ho$ preserve split-exact sequences.
\end{proof}

In order to make uniform statements in all three cases we will call a functor $F:\nCalg_{\sepa}\to \bD$ to a half-exact additive  category $!$-exact if it is split-exact in case $!=q$, half-semiexact in case $!=\se$, or exact in case $!=\exa$. We furthermore call the functor {\em suspension stable} if for every morphism $f:A\to B$ in $\nCalg_{\sepa}$ the fact that $F(f)$ is an isomorphism is equivalent to the fact  that $F(S(f))$ is an isomorphism.  

\begin{ex}\label{wiotghjogrefwfwf} For $!$ in $\{\se,\exa\}$
the functors $\ho\kk_{\sepa,!}:\nCalg_{\sepa}\to \ho\KK_{\sepa,!}$ are suspension stable.
This fact is due to  the triangulated structure on $ \ho\KK_{\sepa,!}$ which is a consequence if the stability of $\KK_{\sepa,!}$. We further use \cref{qerigqerfewfq9} in order to identify looping in $ \KK_{\sepa,!}$ with suspension on the level of algebras. \hB
\end{ex}

We let $\tilde W_{\sepa,!}$ denote the set of morphisms in $\nCalg_{\sepa}$ which are sent to equivalences by $\kk_{\sepa,!}$.

%

We consider a half-exact additive category $\bD$   and a functor $F:\nCalg_{\sepa}\to \bD$.

\begin{prop}\label{wergijowergwerfefwerfw}
If $F$ is 
 homotopy invariant, stable, (suspension stable in the cases $!\in  \{\exa,\se)\}$),  and $!$-exact, then $F$ sends $\tilde W_{\sepa,!}$ to isomorphisms
\end{prop}
\begin{proof} We first consider the cases $!\in \{\se, \exa\}$. In order to show that $F$ sends $\tilde W_{\sepa,!}$ to equivalences it suffices to show that it admits a sequence of 
 factorizations 
  \[\xymatrix{\nCalg_{\sepa}\ar[drr]^{F}\ar[d]_{L_{h,K}}\ar@/^-2cm/[ddd]_{\kk_{\sepa,!}}&& \\L_{K}\nCalg_{\sepa,h}\ar@{..>}[rr]^-{\bar F}\ar[d]_{L_{\sepa,!}}&&\bD\\ 
 L_{K}\nCalg_{\sepa,h,!}\ar@{..>}[urr]^{\tilde F}\ar[d]_{\Omega_{\sepa,!}}&&\\{L_{K}\nCalg_{\sepa,h,!}}^{\group}\ar@{..>}[uurr]&&}\]
 Since $F$ is homotopy invariant and stable,  it has a factorization $\bar F$ as indicated.
 The $!$-exactness of $F$ implies that the functor $\bar F$ is actually half-exact.
 We now claim that $\bar F$ sends the morphisms  in the separable version $\hat W_{\sepa,!}$ 
 of \eqref{erijoglerfqwfwerfrefwrf} to equivalences.
 In case $!=\exa$ this is precisely \cite[Prop. 21.4.1]{blackadar}. In the case 
 $!=\se$ the proof of  \cite[Prop. 21.4.1]{blackadar} goes though word by word since all exact sequences
 used in that proof are then semi-split exact.

 We now claim that $\bar F$ also inverts the closures $W_{\sepa,!}$ of $\hat W_{\sepa,!}$ under $2$-out-of-$3$
 and pull-back. It suffices to show that the collection of morphisms inverted by $\bar F$ is preserved by pull-backs.
 We will use \cite[Thm. 21.4.4]{blackadar} saying that  $F$ admits    long  half-exact   sequences 
 \begin{equation}\label{vsdfvsfvrqwfvs}  \cdots\to  F(S(I))\to F(S(A))\xrightarrow{F(S(f))} F(S(B))\to F(I)\to F(A)\xrightarrow{F(f)} F(B)  
\end{equation}
 associated to 
 exact (or semi-split exact, respectively) sequences $0\to I\to A\to B\to 0$.
    For the semiexact case again note that all exact sequences appearing in the proof are semi-split exact.
 Alternatively  in   both cases, this also follows from the half-exactness of $\bar F$ by applying it to the image under $L_{K}$ of the Puppe sequence \eqref{oijiosfeqrfgerfefsfdfgsgsdfgdsfg} associated to the  map $A\to B$.
Here we use that $F(I)\xrightarrow{\cong} F(C(f))$  which follows from the fact that
$F$ inverts $\hat W_{\sepa,!}$ and the analog of \cref{wetigowergferferferwfw} for half-exact functors.


     We consider the case $!=\se$. The case of $!=\exa$ is simpler and obtained from the  $!=\se$ by removing all  mentionings  of cpc-splits.
 We consider a  diagram
 \[\xymatrix{  & L_{\sepa,h,K}(A)\ar[d]^{\bar f} \\ L_{\sepa,h,K}(B')\ar[r]^{\bar g} &L_{\sepa,h,K}(B) } \] in $L_{K}\nCalg_{\sepa,h}$.
  We can assume (see e.g. the proof of \cref{weokjgpwerferfwef}) that up to equivalence the diagram 
is the image under $ L_{\sepa,h,K}$ of the bold part of a cartesian diagram
\[\xymatrix{ A \ar@{..>}[r]^{g'}\ar@{..>}[d]^{f'}&  A \ar[d]^{  f} \\  B' \ar[r]^{   g} & B  } \] 
in $\nCalg_{\sepa}$ where $f$   admits   cpc split.
  The map $f'$ again admits a cpc split.
 We can extend the vertical maps to exact sequences
 \[0\to I\to A\xrightarrow{f} B\to 0\ , \quad 0\to I'\to A'\xrightarrow{f} B'\to 0\] in $\nCalg_{\sepa}$
 such that the induced map $I\to I'$ is an isomorphism.    Since $F(f)$ is an isomorphism also   $F(S(f))$ and $F(S^{2}(f))$ are   isomorphisms by suspension stability. 
  By the long  half-exact  sequence \eqref{vsdfvsfvrqwfvs}  for $f$
 we conclude that $F(I)\cong 0$ and $F(S(I))\cong 0$, see \cref{righowgwrgrfrefwrf}. Using the version of this   long  half-exact    sequence for $f'$    and $F(S^{2}(I'))\cong 0 $ and $F(S(I'))\cong 0$ we conclude (again by    \cref{righowgwrgrfrefwrf}) that
  $F(S(f'))$ is an isomorphism. 
  Finally,  again using suspension stability we see that  $F(f')=\bar F(\bar f')$ is an isomorphism.

We thus get a factorization $\tilde F$.
In $ L_{K}\nCalg_{\sepa,h,!}$ we have the morphisms \[\beta_{\sepa,!,A}:L_{\sepa,h,K,!}(S^{2}(A))\to L_{\sepa,h,K,!}( A)\ .\] Since $F$ takes values in groups, using similar arguments as for \cref{wtriojgopwrtgrwewfref222} (replacing left-exactness by the existence of the long half-exact sequences) we know that 
 $\tilde F(\beta_{\sepa,!,A})$ is an isomorphism.
 In detail,   we consider the functor $F_{A}(-):=F(-\otimes_{\max}A)$. By
 \cref{wtriojgopwrtgrwewfref} we can conclude that $F_{A}(\cT_{0})\cong 0$.
 Then the boundary map $F_{A}(S^{2}(\C))\to F_{A}(\C)$ of the long half-exact sequence for
 $0\to K\to \cT_{0}\to S(\C)\to 0$ is an isomorphism.
 But this map is precisely  $\tilde F(\beta_{\sepa,!,A})$.

 Consequently we get the last factorization as indicated.
 This finishes the proof in the cases $!=\se$ and $!=\exa$.

 In the case of $!=q$ we construct a sequence of factorizations
  \[\xymatrix{\nCalg_{\sepa}\ar[drr]^{F}\ar[d]_{L_{h,K}}\ar@/^-2cm/[dd]_{\kk_{\sepa,q}}&& \\L_{K}\nCalg_{\sepa,h}\ar@{..>}[rr]^-{\bar F}\ar[d]_{L_{\sepa,!}}&&\bD\\ 
 L_{K}\nCalg_{\sepa,h,q} \ar@{..>}[urr]&&}\]
 Since $F$ is split-exact and takes values in groups, by \cref{wergkooprefrwefrefre} the functor $\bar F$ sends the morphisms 
  $\iota^{s}_{A}:q^{s}A\to A^{s}$ to isomorphisms for all $A$ in $L_{K}\nCalg_{\sepa,h}$.
  This yields the last factorization in this case.
   \end{proof}

   In the following we remove the assumption of suspension stability in \cref{wergijowergwerfefwerfw}.
   Consider a half-exact additive category $\bD$.
 
 \begin{prop}\label{wtrioghjowgrwefrefrewfwref}
 A homotopy invariant, stable, and half-exact (or half-semiexact) functor $F:\nCalg_{\sepa}\to \bD$
inverts $\tilde W_{\sepa,\exa}$ (or $\tilde W_{\sepa,\se}$, respectively).
  \end{prop}
 \begin{proof}
 We start with the case $!=\exa$.
 We consider the classical $E$-theory functor $\ee^{\class}_{\sepa}\to \EE_{\sepa}^{\class}$  constructed in  \cite{MR1068250},
 where we equip the additive $1$-category $\EE_{\sepa}^{\class}$ with the canonical   half-exact structure from \cref{werijogrgerwgweg}.
 The functor   $\ee^{\class}_{\sepa}$  is homotopy invariant, stable,  half-exact, and suspension stable.
 In view of \cref{wergijowergwerfefwerfw} the functor $\ee_{\sepa}^{\class}$ inverts $\tilde W_{\sepa,\exa}$.
 
 Let now $F:\nCalg_{\sepa}\to \bD$ be a homotopy invariant, stable and half-exact functor.
 By the universal property  of $\ee_{\sepa}^{\class}$  stated  in \cite[Thm. 3.6]{MR1068250}
 we get a  factorization
 \[\xymatrix{\nCalg_{\sepa}\ar[dr]^{\ee^{\class}_{\sepa}}\ar[rr]^{F}&& \bD\\& \EE_{\sepa}^{\class}\ar[ur]^{\hat F}&}\]
This implies that $F$ also inverts $\tilde W_{\sepa,\exa}$.

We now consider the case $!=\se$. It would be natural to argue as in the exact case using a corresponding universal property of $\KK_{\sepa}^{\class}$ involving half-semiexactness. But since we do not know a reference for this we 
will argue differently invoking the automatic semicontinuity theorem.
 Being a half-semiexact functor, $F$ is in particular split-exact.
By \cref{wergijowergwerfefwerfw}  it inverts $ \tilde W_{\sepa,q}$. As a consequence of the automatic semicontinuity theorem \cref{wgokjweprgrefwrefwrfw} we have  $ \tilde W_{\sepa,q}= \tilde W_{\sepa,\se}$.
\end{proof}

   \uli{bis hier}
    
We can now state the universal property of $\ho\kk_{\sepa,!}$.
 \begin{prop}\label{qrifgjoergrfref9}
Pull-back along $ \ho\kk_{\sepa,!}$ induces an equivalence
\[\Fun^{\add,\frac12 \exa}(\ho\KK_{\sepa,!},\bD)\stackrel{\simeq}{\to} \Fun^{h,s,\frac12 !}(\nCalg_{\sepa},\bD)\ , \quad  
 !\in \{\exa,\se\}\]  for any additive  half-exact category $\bD$ or 
\[\Fun^{\add }(\ho\KK_{\sepa,q}, \bD)\stackrel{\simeq}{\to} \Fun^{h,s,  splt}(\nCalg_{\sepa},\bD  ) \ , \quad 
 !=q
\] for any additive $1$-category.
 
\end{prop}
\begin{proof}
 
It follows from \cref{weroigjowergerfgwref9} that the pull-back along  $\ho\kk_{\sepa,!}$ takes values in the indicated category of functors. 

We first discuss the case $!\in \{\exa,\se\}$. We consider the diagram
 \[\xymatrix{&\Fun^{\add,\frac12 \exa }(\ho\KK_{\sepa,!},\bD)\ar[r]^{(2)}\ar[d]& \Fun^{h,s,\frac12 !}(\nCalg_{\sepa},\bD)\ar[d]\\\Fun(\KK_{\sepa,!},\bD)\ar@/_2cm/[rr]_{\cong}^{\kk_{\sepa,!}^{*}}\ar[r]^{\ho^{*}}_-{\simeq}& \Fun (\ho\KK_{\sepa,!},\bD)\ar[r]^{(1)}_{\simeq}&\Fun^{\tilde W_{\sepa,!}}(\nCalg_{\sepa},\bD)}\ .\]
The lower    functor is an equivalence by the universal property of the Dwyer-Kan localization $\kk_{\sepa,!}$, see \cref{ergiojeroigwerfwrefrefdvs}.
The lower left equivalence  is the universal property of $\ho$.
As a consequence  we see that the functor marked by $(1)$ is an equivalence.

The vertical functors are fully faithful. Therefore the functor marked by $(2)$ is fully faithful, too.
It remains to show that it is essentially surjective. If $F$ is any functor in  
 $\Fun^{h,s,\frac12 !}(\nCalg_{\sepa},\bD)$, then by \cref{wtrioghjowgrwefrefrewfwref}
  there exists a functor $\hat F$ in $\Fun(\KK_{\sepa,!},\bD)$ such that $\hat F\circ \kk_{\sepa,!}\simeq F$.
We furthermore have a functor $\bar F$ in $  \Fun (\ho\KK_{\sepa,!},\bD)$ such that
$\ho^{*}\bar F\simeq \hat F$. It remains to show that
$\bar F$ is additive and half-exact.

Since $F$ sends finite sums to products we can conclude that $\bar F$ is additive.
 Since all triangles in $\ho\KK_{\sepa,!}$ come from exact (semi-split exact) sequences  
in $\nCalg_{\sepa}$ and $F$ is half-exact (or half-semiexact)
the functor $\bar F$ sends the half-exact sequences in $\ho\KK_{\sepa,!}$ (with the triangulated half-exact structure)
to half-exact sequences in $\bD$.
Hence $\bar F$ is also half-exact.

In the case $!=q$ we argue similarly with 
 \[\xymatrix{&\Fun^{\add  }(\ho\KK_{\sepa,q},\bD)\ar[r]^{(2)}\ar[d]& \Fun^{h,s,splt}(\nCalg_{\sepa},\bD)\ar[d]\\\Fun(\KK_{\sepa,q},\bD)\ar@/_2cm/[rr]_{\cong}^{\kk_{\sepa,q}^{*}}\ar[r]^{\ho^{*}}_-{\simeq}& \Fun (\ho\KK_{\sepa,q},\bD)\ar[r]^{q}_{\simeq}&\Fun^{\tilde W_{\sepa,q}}(\nCalg_{\sepa},\bD)}\ .\]
\end{proof}

Let $\kk_{\sepa}^{\class}:\nCalg_{\sepa}\to \KK_{\sepa}^{\class}$ denote the classical additive category valued
functor described by the universal property   \cite[Thm. 3.4]{MR1068250}.
We equip $ \KK_{\sepa}^{\class}$ with the split half-exact structure. Then 
$\kk_{\sepa}^{\class}$ is split-exact.
Since $\kk_{\sepa}^{\class}$ is also homotopy invariant and stable,  by 
\cref{wergijowergwerfefwerfw}  we get a dotted factorization
\[\xymatrix{\nCalg_{\sepa}\ar[r]^{\kk_{\sepa}^{\class}}\ar[d]^{\kk_{\sepa,q}}&  \KK_{\sepa}^{\class}\\ \KK_{\sepa,q}\ar[r]^{\ho}\ar@{..>}[ur] &\ho\KK_{\sepa,q}\ar@{-->}[u]^{\psi}}\ .\]
The dashed factorization is induced by the universal property of $\ho$ since $\KK^{\class}_{\sepa}$ is an ordinary category. Since  $\kk_{\sepa}^{\class}$  is split-exact, the dotted arrow is half-exact, and the dashed arrow is additive.

 Since $\ee^{\class}_{\sepa}$ inverts $\tilde W_{\sepa,\exa}$ (as seen in the proof of \cref{wtrioghjowgrwefrefrewfwref}) we also have a factorization 
 \[\xymatrix{\nCalg_{\sepa}\ar[r]^{\ee_{\sepa}^{\class}}\ar[d]^{\ee_{\sepa}}&  \EE_{\sepa}^{\class}\\ \EE_{\sepa}\ar[r]\ar@{..>}[ur]&\ho\EE_{\sepa}\ar@{-->}[u]^{\psi}}\ .\]
 Since $\ee^{\class}_{\sepa}$ is half-exact we can conclude that the dashed arrow
 is additive and half-exact.
\begin{theorem}\label{qiurhfgiuewrgwrfrefrfwrefw}
The comparison functors
$\psi:\ho\KK_{\sepa,q}\to  \KK_{\sepa}^{\class}$ and 
$\psi:\ho\EE_{\sepa}\to  \EE_{\sepa}^{\class}$ are equivalences. 
\end{theorem}
\begin{proof}
By  \cref{ergiojeroigwerfwrefrefdvs} or \cref{eigjohwergerfgerwfrwef}, respectively, we know that
$\ee_{\sepa}$ and $\kk_{\sepa,q}$ are Dwyer-Kan localizations. The composition of a Dwyer-Kan localization with the canonical functor to the homotopy category is again a localization, in this case  in the sense of ordinary categories. We conclude that $\ho \ee_{\sepa}$ and $\ho\kk_{\sepa,q}$ are  localizations.
Note that such a localization is determined uniquely up to  equivalence under $\nCalg_{\sepa}$, and that  two choices of such equivalences   under $\nCalg_{\sepa}$
  are isomorphic by a unique isomorphism.

Since the universal properties of  $\kk_{\sepa}^{\class}$ \cite[Thm. 3.4]{MR1068250}
and $\ee_{\sepa}^{\class}$ \cite[Thm. 3.6]{MR1068250} are formulated in terms of equalities of functors (instead of isomorphisms), 
it will be useful to choose $\ho\kk_{\sepa,q}$  and $\ho\ee_{\sepa}$  such that these functors are bijective on objects.

We write down the details of the argument for $KK$-theory.
Since $\ho\kk_{\sepa,q}$ is homotopy invariant, stable and split-exact,
the universal property   \cite[Thm. 3.4]{MR1068250} provides an additive factorization 
 \[\xymatrix{\nCalg\ar[dr]^{\kk^{\class}_{\sepa}}\ar[rr]^{\ho\kk_{\sepa,q}}&&\ho\KK_{\sepa,q} \\& \KK^{\class}_{\sepa} \ar@{..>}[ur]^{\phi}&}\]
 which strictly commutes.

 The pull-back along $\ho\kk_{\sepa,q}$ of
  the composition $\ho\KK_{\sepa,q}\xrightarrow{\psi}  \KK_{\sepa}^{\class} \xrightarrow{\phi} \ho\KK_{\sepa,q}$ is  equivalent to $\ho\kk_{\sepa,q}$. Therefore by 
\cref{qrifgjoergrfref9} this composition $\phi\circ \psi$  itself is an equivalence.

We now show that the composition $ \psi\circ \phi$  is also an equivalence invoking the uniqueness statement of  \cite[Thm. 3.4]{MR1068250}. 
This requires an equality $\psi\circ \phi\circ  \kk_{\sepa}^{\class}=\kk_{\sepa}^{\class}$.
By the construction of $\phi$ using  \cite[Thm. 3.4]{MR1068250}  we have an equality $\phi\circ \kk_{\sepa}^{\class}=\kk_{\sepa,q}$.
But the construction of $\psi$ only ensures a natural isomorphism
$f:\psi\circ\ho\kk_{\sepa,q} \stackrel{\cong}{\to} \kk_{\sepa}^{\class}$ which is not necessarily an equality.
We now use the freedom to replace $\psi$ by an isomorphic functor and our special choice of $\ho\kk_{\sepa,q}$ to be    bijective on objects.

For every $A$ in $\nCalg_{\sepa}$ we have an isomorphism $f_{A}:\psi(\ho\kk_{\sepa,q}(A)) \stackrel{\cong}{\to}  \kk_{\sepa}^{\class}(A)$.
We   define $\psi':  \ho\KK_{\sepa,q}\to  \KK^{\class} $ on objects
such that $\psi'(\ho\kk_{\sepa,q}(A)):=\kk_{\sepa}^{\class}(A)$.  For a morphism $h:A\to B$ in $ \ho\KK_{\sepa,q}$ we then define $\psi'(h):=f_{B}\phi(h) f_{A}^{-1}$.
The family $(f_{A})_{A}$ also implements an isomorphism $\psi\cong \psi'$. 
Now $\psi'\circ \phi\circ  \kk_{\sepa}^{\class}=\kk_{\sepa}^{\class}$ which implies that
$\psi'\circ \phi=\id$. 

In particular, $\phi$ has a right and a left inverse  equivalence and is hence itself an equivalence.   
But then also $\psi$ is an equivalence.

The case of $E$-theory is completely analogous. We use the universal property  \cite[Thm. 3.6]{MR1068250} 
of $\ee^{\class}_{\sepa}$.
  \end{proof}

Note that the proof of \cref{qiurhfgiuewrgwrfrefrfwrefw}  does not use the case of  \cref{wtrioghjowgrwefrefrewfwref} for 
$!=\se$ and is therefore independent of the automatic semicontinuity theorem.

\section{Asymptotic morphisms in $E$-theory}\label{ewriogjowegregrewf99}

The first construction of an additive $1$-category representing $E$-theory  was given in \cite{MR1068250}
 by enforcing universal properties. This construction was the blueprint for the $\infty$-categorical version considered in the present note.
  Shortly after in  \cite{zbMATH04182148} the $E$-theory groups were  represented as equivalence classes of asymptotic morphisms, see also  \cite{Guentner_2000}. 
Recall that we construct $\KK$-theory for separable algebras
by a sequence of Dwyer-Kan localizations applied to $\nCalg_{\sepa}$. 
In view of  \cite{zbMATH04182148}, \cite{Guentner_2000} a natural idea would be to apply a similar construction to the category of $C^{*}$-algebras and asymptotic morphisms. 
The first obstacle one encounters in this approach  is that the  composition of asymptotic morphisms  
is only well-defined after going over to homotopy classes. By now\footnote{2024} we think that the correct way to relate $E$-theory with asymptotic  morphisms is the one worked out recently in \cite[Sec. 3.5]{Bunke:2024aa}. 
It is based on the shape theory of \cite{zbMATH03996430}, \cite{Dadarlat_1994} and goes beyond the scope of the present paper. In the present  section we will therefore just  show that 
 asymptotic morphisms  also give rise to morphisms in our version $E$-theory in a way which is
compatible with the composition.


We consider the endofunctors
\[T,F:\nCalg\to \nCalg\] defined by
\[T(A):= C_{b}([0,\infty),A)\ , \quad  F(A):=C_{b}([0,\infty),A)/C_{0}([0,\infty),A)\ .\]
We have a natural transformation
$\alpha :T\to F$ such that $\alpha_{A}:T(A)\to F(A)$  is the projection onto the quotient.
We furthermore have the natural transformations \[\beta:\id_{\nCalg}\to T\ , \quad  \ev_{0}:T\to \id_{\nCalg}\] such that $\beta_{A}:A\to T(A)$ sends $a$ in $A$ to the constant function
with value $a$, and $\ev_{0,A}:T(A)\to A$ evaluates the function $t\to f(t)$ in $T(A)$ at $t=0$. We finally define the natural transformation \[\gamma:=\alpha\circ \beta:\id_{\nCalg}\to F(A)\ .\]
 Note that the sequence
\[0\to C_{0}([0,\infty),A) \to T(A)\xrightarrow{\alpha_{A}} F(A)\to 0\] is exact and that 
$C_{0}([0,\infty),A) $ is contractible. Since $\ee:\nCalg\to \EE$ is reduced, homotopy invariant  and exact we see that  $\ee(\alpha_{A})$ is an equivalence for every $A$ in $\nCalg$.
We define a natural transformation  $\delta:\ee\circ F\to \ee$ by \[\delta_{A}:=\ee(\ev_{0,A})\circ \ee(\alpha_{A})^{-1}:\ee(F(A))\to \ee(A)\ .\]

%
%
%
Following \cite[Sec. 2]{Guentner_2000} we adopt the following definition.
\begin{ddd}\label{werigjowergerwf}
For $n$ in $\nat$ an asymptotic morphism $f:A \leadsto_{n} B$ is a morphism $f:A\to F^{n}(B)$ in $\nCalg$.
\end{ddd}

\begin{rem}
Note that asymptotic morphisms for $n=0$ are usual morphisms, and the case of $n=1$ corresponds to the notion of an asymptotic morphism in \cite{zbMATH04182148}. As in  \cite[Sec. 2]{Guentner_2000} 
we include the case of bigger $n$ in order to have a simple definition of a composition  of  asymptotic morphisms which also works for non-separable algebras. 
\hB\end{rem}

If $f:A \leadsto_{n} B$ is an asymptotic morphism, then we define
\[\ee_{n}(f):= \delta_{B}\circ \dots\circ \delta_{F^{n-1}(B)}\circ  \ee(f)\ .\]
If $n=0$, then this formula is interpreted as $\ee_{0}(f):=\ee(f)$.

Let $f': A\leadsto_{n+1} B$ be given  by $\gamma_{F^{n}(B)}\circ f$.
Then we say that $f'$ and $f$ are {\em related}.
\begin{lem}
If $f'$ is related  to $f$, then $\ee_{n}(f)\simeq \ee_{n+1}(f')$.
\end{lem}
\begin{proof}
This follows from $\ee(\gamma_{F^{n}(B)})\simeq \ee(\alpha_{F^{n}(B)})\circ \ee(\beta_{F^{n}(B)})$
and $\ee(\ev_{0})\circ \ee(\beta_{F^{n}(B)})\simeq \id_{\ee(F^{n}(B))}$.
 \end{proof}
 
 The argument implies that $\ee_{1}(\gamma_{A})\simeq \id_{A}$ for every $C^{*}$-algebra $A$.

We define the {\em composition} of  two asymptotic morphisms 
$f:A \leadsto_{n} B$ and $g:B\leadsto_{m} C$ as
$g\sharp f:A\leadsto_{n+m}C$ given by $F^{n}(g)\circ f$.

\begin{lem}
We have $\ee_{n+m}(g\sharp f)\simeq \ee_{m}(g)\circ \ee_{n}(f)$.
\end{lem}
\begin{proof}
We consider the following diagram
\[\xymatrix{&&\ee(F^{n+m}(C))\ar[dr]^{\delta_{F^{m}(C)}\dots \delta_{F^{n+m-1}(C)}}&&\\&\ee(F^{n}(B))\ar[ur]^{\ee(F^{n}(g))}\ar[dr]^{\delta_{B}\dots \delta_{F^{n-1}(B)}}&&\ee( F^{m}(C))\ar[dr]^{\delta_{C}\dots\delta_{F^{m-1}(C)}}&\\\ee(A)\ar[ur]^{\ee(f)}\ar[rr]^{\ee_{n}(f)}& &\ee(B)\ar[ur]^{e(g)}\ar[rr]^{\ee_{m}(g)}&&\ee(C)}\ .\]
The square commutes since $\delta$ is a natural transformation. The lower triangles reflect the definitions of the lower horizontal arrows.
\end{proof}

We say that $f_{0},f_{1}:A \leadsto_{n} B$ are {\em homotopic} if there exists
$f:A \leadsto_{n} C([0,1],B)$ such that $F^{n}(\ev_{i})\circ f=f_{i}$.

\begin{lem}
If $f_{0}$ and $f_{1}$ are homotopic, then $\ee_{n}(f_{0})\simeq \ee_{n}(f_{1})$.
\end{lem}
\begin{proof}
We have 
$\ee_{n}(f_{i})\simeq  \ee_{0}(\ev_{i})\sharp \ee_{n}(f)$.
The assertion now follows  since $e_{0}=e$ and $e$ is homotopy invariant.   
\end{proof}

In the remainder of this section we relate the $E$-theory constructed in the present note with the version from 
 \cite{Guentner_2000}, called the classical $E$-theory $\EE^{\class}$.
In  \cite[Def. 2.13]{Guentner_2000}  (even in the equivariant case)  a {\em homotopy category $\mathfrak{A}$ of asymptotic morphisms} is introduced. Its objects are $C^{*}$-algebras, and its morphisms are
equivalence classes of asymptotic morphisms, where the equivalence relation is generated by the relations of being related and homotopy introduced above.  The results above show that the functor $\ho\circ \ee:\nCalg\to \ho\EE$ factorizes over $\mathfrak{A}$.

\begin{kor}
We have a commutative triangle
$$\xymatrix{&\nCalg\ar[dr]^{\ho\circ \ee}\ar[dl]&\\\mathfrak{A}\ar[rr]^{c}&&\ho\EE}\ .$$
\end{kor}
\begin{proof}
The down-left arrow sends a morphism $f:A\to B$ to the equivalence class represented by $f\leadsto_{0}B$,
and the lower horizontal {\em comparison arrow} $c$ sends the $C^{*}$-algebra $A$ to $\ee(A)$ and the class of an asymptotic morphism
$f\leadsto_{n}B$ to $\ee_{n}(f)$.
\end{proof}

In  \cite[Def. 2.13]{Guentner_2000} the {\em classical $E$-theory category} $\EE^{\class}$ is defined
as the category whose objects are $C^{*}$-algebras, and whose morphisms are
given by \begin{equation}\label{eifjwiqoedjoqwedewdqdeqwdq}\Hom_{\EE^{\class}}(A,B):=\Hom_{\mathfrak{A}}(K\otimes S(A),K\otimes S(B))\ .
\end{equation}
It should not be confused with the separable version $E^{\class}_{\sepa}$ from \cite{MR1068250}.
There is a canonical functor \[ i: \mathfrak{A}\to  \EE^{\class}\] which is the identity on objects and
 sends the class of an asymptotic morphism
$f:A\leadsto_{n}B$ to the class of $i(f):K\otimes S(A)\leadsto_{n} K\otimes S(B)$ given by the composition
\[K\otimes S(A)\xrightarrow{f}K\otimes S(F^{n}(B))\to F^{n}(K\otimes S(B))\ .\]
(note that the second map is not an isomorphism).

\begin{kor}
We have a commutative triangle \begin{equation}\label{feqwfdwedcdcasc}\xymatrix{&\mathfrak{A}\ar[dr]^{c}\ar[dl]_{i}&\\\EE^{\class}\ar[rr]^{\hat c}&&\ho\EE}\ .
\end{equation}
\end{kor}
\begin{proof}
The lower horizontal map sends
a $C^{*}$-algebra $A$ to $\ee(A)$ and   the class of a morphism
$f:K\otimes S(A)\to F^{n}(K\otimes S(B))$ to the image under
\[\Hom_{\mathfrak{A}}(K\otimes S(A),K\otimes S(B))\xrightarrow{c}
\EE_{0}(K\otimes S(A),K\otimes S(B))\cong \EE_{0}( A, B)\ ,\]
where we use stability of the functor $\ee$ and stability of the $\infty$-category $\EE$ for the second isomorphism.
\end{proof}

\begin{rem}The functor $\hat c$  in \eqref{feqwfdwedcdcasc} is not an equivalence. In fact   
the classical $E$-theory functor preserves countable sums by  \cite[Prop. 7.1]{Guentner_2000}. In contrast,
the functor $\ee$ does not preserve countable sums, since $y:\EE_{\sepa}\to \EE$ does not preserve countable sums.  

But note that it is shown in \cite{Bunke:2024aa} that
the restriction of $\hat c$ to the full subcategory of separable algebras  induces an equivalence
$\hat c_{\sepa}:\EE^{\class}_{\sepa}\to \ho \EE_{\sepa}$. In particular the formula \eqref{eifjwiqoedjoqwedewdqdeqwdq} gives an explicit description of the morphism groups in $ \ho \EE_{\sepa}$ in terms of homotopy classes of asymptotic morphisms.
\hB
   \end{rem}
 
\begin{rem}\label{fewfiuzhefiqweddqewd}
Let \begin{equation}\label{fewdqwedewdedqqewded}\cS:\quad 0\to A\to B\to C\to 0
\end{equation}  be an exact sequence of separable $C^{*}$-algebras. The $E$-theory analogue of \cref{weijrogreferwfwerf} is 
  \cite[Prop. 5.5]{Guentner_2000},    where a morphism 
$\sigma_{\cS}$ in $\Hom_{\mathfrak{A}}(S(C),A)$ was constructed.
It follows from  \cite[Prop. 5.15]{Guentner_2000} (this is an analogue of \cref{weriogjwoierpgreferfrewferwf})
that the image of $\sigma_{\cS}$ in $\EE_{0}(S(C),A)$ is the boundary map
$\partial_{\cS}$ of the fibre sequence in $\EE$ associated to the exact sequence \eqref{fewdqwedewdedqqewded}.
This shows that the comparison functor $\hat c$ is compatible with the long exact sequences associated to exact sequences of separable $C^{*}$-algebras.
\hB\end{rem}

\bibliographystyle{alpha}
\bibliography{forschung2021}

\end{document}